\documentclass[10pt]{amsart}
\usepackage{amsmath, amsthm, amsfonts, amssymb, graphicx, bm,verbatim,mathrsfs,siunitx}
\usepackage{hyperref}
\usepackage{stmaryrd,enumerate,tikz}
\theoremstyle{plain}
\newtheorem{thrm}{Theorem}[section]
\newtheorem{lmm}[thrm]{Lemma}
\newtheorem{prpstn}[thrm]{Proposition}
\newtheorem{crllry}[thrm]{Corollary}

\newtheorem{dfntn}{Definition}

\newtheorem*{rmk}{Remark}

\newtheorem{nthrm}{Theorem}

\numberwithin{equation}{section}

\DeclareMathOperator{\Kl}{Kl}

\newcommand{\Mod}[1]{\ (\mathrm{mod}\ #1)}

\renewcommand{\hat}[1]{\widehat #1}
\usepackage[title]{appendix}

\allowdisplaybreaks[1]
\begin{document}

\title[Primes in arithmetic progressions to large moduli I]{Primes in arithmetic progressions to large moduli I: Fixed residue classes}
\author{James Maynard}
\address{Mathematical Institute, Radcliffe Observatory quarter, Woodstock Road, Oxford OX2 6GG, England}
\email{james.alexander.maynard@gmail.com}
\begin{abstract}
We prove new mean value theorems for primes in arithmetic progressions to moduli larger than $x^{1/2}$. Our main result shows that the primes are equidistributed for a fixed residue class over all moduli of size $x^{1/2+\delta}$ with a `convenient sized' factor. As a consequence, the expected asymptotic holds for all but $O(\delta Q)$ moduli $q\sim Q=x^{1/2+\delta}$ and we get results for moduli as large as $x^{11/21}$.

Our proof extends previous techniques of Bombieri, Fouvry, Friedlander and Iwaniec by incorporating new ideas inspired by amplification methods. We combine these with techniques of Zhang and Polymath tailored to our application. In particular, we ultimately rely on exponential sum bounds coming from the spectral theory of automorphic forms (the Kuznetsov trace formula) or from  algebraic geometry (Weil and Deligne style estimates).
\end{abstract}
\begin{samepage}
\maketitle
\tableofcontents
\end{samepage}
\parskip 7.2pt
\newpage
%
%
%
%
%
%
%
%
%
%
\section{Introduction}\label{sec:Introduction}
The Siegel--Walfisz theorem states that uniformly for $q\le (\log{x})^A$ and $(a,q)=1$
\begin{equation}
\pi(x;q,a)=\Bigl(1+O_A\Bigl(\frac{1}{(\log{x})^A}\Bigr)\Bigr)\frac{\pi(x)}{\phi(q)}.
\label{eq:PNT}
\end{equation}
Without progress on the notorious problem of Siegel zeros we cannot hope to unconditionally improve the range of $q$, despite it being very desirable to do so. Under the assumption of the Generalized Riemann Hypothesis, however, the range when \eqref{eq:PNT} holds can be extended to $q\le x^{1/2}/(\log{x})^B$ for $B=B(A)$ sufficiently large in terms of $A$. Moreover, it is conjectured \cite{Montgomery} that \eqref{eq:PNT} should hold for all $q\le x^{1-\epsilon}$.

For the purpose of many applications, an adequate substitute for the Generalized Riemann Hypothesis is the Bombieri--Vinogradov Theorem \cite{Bombieri,Vinogradov}. This states that for every $A>0$ and $B=B(A)$ sufficiently large in terms of $A$ we have
\begin{equation}
\sum_{q\le x^{1/2}/(\log{x})^B}\,\sup_{(a,q)=1}\Bigl|\pi(x;q,a)-\frac{\pi(x)}{\phi(q)}\Bigr|\ll_A\frac{x}{(\log{x})^A}.
\label{eq:BV}
\end{equation}
If $Q\le x^{1/2}/(\log{x})^B$ then this implies that, for all but $O_A(Q/(\log{x})^A)$ moduli $q\le Q$, we have \eqref{eq:PNT} for all $(a,q)=1$. In particular, `most' moduli $q$ satisfy \eqref{eq:PNT}, even when the moduli are essentially as large as can be handled with the Generalized Riemann Hypothesis.

As with the individual estimate \eqref{eq:PNT}, it is conjectured \cite{ElliottHalberstam} that the range of $q$ in \eqref{eq:BV} can be extended to $q\le x^{1-\epsilon}$. However, it remains an important outstanding problem in analytic number theory just to extend the range of moduli in \eqref{eq:BV} to $q\le x^{1/2+\epsilon}$, thereby going `beyond the square-root barrier' and producing estimates which are not directly implied by the Generalized Riemann Hypothesis. By expanding the summand via Dirichlet characters and the explicit formula, we see that this would imply some cancellation for sums over zeroes of \emph{different} $L$-functions.

In a series of papers \cite{Fouvry,Fouvry2,Fouvry3,Fouvry4,FouvryIwaniec,FouvryIwaniec2,BFI1,BFI2,BFI3} both separately and in collaboration  Bombieri, Fouvry, Friedlander and Iwaniec produced weakened variants of \eqref{eq:BV} which \emph{were} valid for moduli larger than $x^{1/2}$. More recently, the key estimate in the work of Zhang \cite{Zhang} on bounded gaps between primes was a variant of \eqref{eq:BV} valid for moduli slightly larger than $x^{1/2}$ provided the moduli only had small prime factors. Zhang's work was extended by the Polymath project \cite{Polymath}, and related recent results were obtained by Fouvry and Radziwi\l\l \,\cite{FouvryRadziwill,FouvryRadziwill2}, Drappeau \cite{Drappeau}, Drappeau, Radziwi\l\l \,and Pratt \cite{DrappeauPratt} and Assing, Blomer and Li \cite{Blomer}.

In this paper we also consider variants of \eqref{eq:BV} which are valid for moduli beyond $x^{1/2}$, extending some of the previous estimates, particularly those of Bombieri, Friedlander and Iwaniec \cite{BFI1,BFI2,BFI3}. As with these previous works, we will make use of estimates of Deshouillers--Iwaniec \cite{DeshouillersIwaniec} for sums of Kloosterman sums coming from the Kuznetsov trace formula which requires us to only consider the situation when $a$ is fixed (or a small power of $x$). Crucially we refine some estimates via an amplification method, allowing us to handle several of the critical cases which were previously inaccessible. We combine these ideas with ideas originating in the work  Zhang \cite{Zhang} and Polymath \cite{Polymath} based on exponential sum estimates from algebraic geometry which also require some divisibility conditions on the moduli. Together, our methods enable us to prove qualitatively new variants of \eqref{eq:BV} for larger moduli as well as giving quantitative improvements.

Our main result proves an extension of \eqref{eq:BV} for moduli which have a conveniently sized divisor.
%
%
%
%
%
%
%
%
\begin{thrm}\label{thrm:MainTheorem}
Let $a\in\mathbb{Z}$, let $\epsilon>0$ and let $Q_1,Q_2$ satisfy
\begin{align}
Q_1Q_2^2&<x^{1-100\epsilon},\label{eq:Cons1}\\
Q_1^{12}Q_2^7&<x^{4-100\epsilon},\label{eq:Cons2}\\
Q_1^{20}Q_2^{19}&<x^{10-100\epsilon}.\label{eq:Cons3}
\end{align}
Then for every  $A>0$ we have
\[
\sum_{\substack{q_1\le  Q_1\\ (q_1,a)=1}}\sum_{\substack{q_2\le  Q_2\\ (q_2,a)=1}}\Bigl|\pi(x;q_1q_2,a)-\frac{\pi(x)}{\phi(q_1q_2)}\Bigr|\ll_{a,\epsilon,A}\frac{x}{(\log{x})^A}.
\]
\end{thrm}
%
%
%
%
%
%
%
%
The main qualitative feature of Theorem \ref{thrm:MainTheorem} is that it bounds the error term for primes in arithmetic progressions with absolute values and a strong error term, whilst simultaneously applying to `most' moduli.

The conditions on $Q_1$, $Q_2$ should be thought of as restricting to moduli $q<x^{11/21}$ with a `convenient sized factor', where the restriction on the possible size of this factor is weak when considering moduli of size $x^{1/2+\delta}$ with $\delta$ small, but more restrictive as $\delta$ grows. In particular, setting $Q_1Q_2=x^{1/2+\delta}$ and allowing $Q_1$ to vary, we see that Theorem \ref{thrm:MainTheorem} can be rephrased as follows.
%
%
%
%
%
%
%
%
\begin{crllry}\label{crllry:BVDivisor}
Let $a\in\mathbb{Z}$, $0<\delta<1/42$, and $0<\eta<(1-42\delta)/4$ and 
\[
\mathcal{Q}_{\delta,\eta}:=\Bigl\{q\le x^{1/2+\delta}:\, \exists\, d|q\text{ s.t. }x^{2\delta+\eta}<d<\min\Bigl(\frac{x^{1/10}}{x^{7\delta/5+\eta}},\frac{x^{1/2}}{x^{19\delta+\eta}}\Bigr)\Bigr\}.
\]
Then for every $A>0$ we have
\[
\sum_{\substack{q\in\mathcal{Q}_{\delta,\eta}\\ (q,a)=1}}\Bigl|\pi(x;q,a)-\frac{\pi(x)}{\phi(q)}\Bigr|\ll_{a,\eta,A}\frac{x}{(\log{x})^A}.
\]
\end{crllry}
%
%
%
%
%
%
%
%
One should think of $\eta$ as a small constant, so $\mathcal{Q}_{\delta,\eta}$ essentially counts moduli $q$ of size $x^{1/2+\delta}$ which have a `conveniently sized factor' in the interval $[x^{2\delta},x^{1/10-7\delta/5}]$.

When $\delta$ is small `most' moduli of size $x^{1/2+\delta}$ have a divisor in the range $[x^{2\delta},x^{1/20}]$, and so most moduli are contained in $\mathcal{Q}_{\delta,\eta}$. In particular, we can extend the asymptotic \eqref{eq:PNT} to `most' moduli of size $x^{1/2+\delta}$. This is made precise by the following result.
%
%
%
%
%
%
%
%
\begin{crllry}\label{crllry:MostModuli}
Let $a\in\mathbb{Z}$, $0<\delta<1/55$, $A>0$ and $Q\le x^{1/2+\delta}$. Then for all but at most $18\delta Q\phi(a)/a$ moduli $q\in [Q,2Q]$ with $(q,a)=1$ we have
\[
\pi(x;q,a)=\Bigl(1+O_{a,\delta,A}\Bigl(\frac{1}{(\log{x})^A}\Bigr)\Bigr)\frac{\pi(x)}{\phi(q)}.
\]
\end{crllry}
%
%
%
%
%
%
%
%
For example, if $Q\le x^{1/2+1/2000}$, we see that $99\%$ of moduli $q\in[Q,2Q]$ with $(a,q)=1$ have the expected asymptotic count for the number of primes.

When $Q_1\approx x^{1/21}$, Theorem \ref{thrm:MainTheorem} allows us to obtain non-trivial result for moduli as large as $x^{11/21-\epsilon}$; this compares quite favourably with previous results. Explicitly, we have the following variant.
%
%
%
%
%
%
%
%
\begin{crllry}\label{crllry:1121}
Let $a\in\mathbb{Z}$, $\epsilon>0$. Then we have for every $A>0$
\[
\sum_{\substack{q_1\le  x^{1/21}\\ (q_1,a)=1}}\sum_{\substack{q_2\le x^{10/21-\epsilon} \\ (q_2,a)=1}}\Bigl|\pi(x;q_1q_2,a)-\frac{\pi(x)}{\phi(q_1q_2)}\Bigr|\ll_{a,\epsilon,A}\frac{x}{(\log{x})^A}.
\]
\end{crllry}
%
%
%
%
%
%
%
%
\subsection{Comparison with previous results}
The key qualitative features of Theorem \ref{thrm:MainTheorem} are that it gives good savings over the trivial bound for `most' moduli of size $x^{1/2+\delta}$ and treats the error terms with absolute values. This is the first such result.

We first record three key results for comparison. The first is a combination of the main results of \cite{BFI2} and \cite{BFI3}.
%
%
%
%
%
%
%
%
\begin{nthrm}[Bombieri, Friedlander, Iwaniec]\label{nthrm:BFI2}
Let $a\in\mathbb{Z}$, $\delta>0$ and $Q=x^{1/2+\delta}$. Then we have
\[
\sum_{\substack{q\in[Q,2Q]\\(q,a)=1}}\Bigl|\pi(x;q,a)-\frac{\pi(x)}{\phi(q)}\Bigr|\ll_a\delta^2\frac{x}{\log{x}}+\frac{x(\log\log{x})^{O(1)}}{(\log{x})^3}.
\]
\end{nthrm}
%
%
%
%
%
%
%
%
The second is \cite[Theorem 10]{BFI1}.
%
%
%
%
%
%
%
%
\begin{nthrm}[Bombieri, Friedlander, Iwaniec]\label{nthrm:BFI}
Let $A>0$ and $\lambda_d$ be a well-factorable sequence of level $Q<x^{4/7-\epsilon}$. Then we have
\[
\sum_{\substack{q\le Q\\ (q,a)=1}}\lambda_q\Bigl(\pi(x;q,a)-\frac{\pi(x)}{\phi(q)}\Bigr)\ll_{a,A,\epsilon}\frac{x}{(\log{x})^A}.
\]
\end{nthrm}
%
%
%
%
%
%
%
%
The final result is Polymath's refinement \cite{Polymath} of Zhang's work \cite{Zhang}.
%
%
%
%
%
%
%
%
\begin{nthrm}[Zhang, Polymath]\label{nthrm:Zhang}
Let $a\in\mathbb{Z}$, $0<\delta<7/300$ and $Q=x^{1/2+\delta}$. Then there is a constant $\eta=\eta(\delta)>0$ such that for every $A>0$
\[
\sum_{\substack{q\le Q \\(q,a)=1\\ p|q\Rightarrow p\le x^\eta}}\Bigl|\pi(x;q,a)-\frac{\pi(x)}{\phi(q)}\Bigr|\ll_A\frac{x}{(\log{x})^A}.
\]
\end{nthrm}
%
%
%
%
%
%
%
%
Theorem \ref{nthrm:BFI2} is non-trivial only when $\delta$ is small, and wins only a small amount over the trivial bound, but involves \emph{all} moduli of size $Q$. By comparison, Theorem \ref{thrm:MainTheorem} gives a non-trivial result for significantly larger moduli, with better savings, but only holds for a subset of moduli. When $\delta$ is small, `most' moduli are included, and so it automatically implies a result of the sort of Theorem \ref{nthrm:BFI2}. Indeed, the method of proof is an extension and refinement of the ideas underlying Theorem \ref{nthrm:BFI2}, and the proof contains the proof of Theorem \ref{nthrm:BFI2} as a special case. (One should be able to obtain a version of Theorem \ref{nthrm:BFI2} with $\delta^2$ replaced by $\delta^4$, for example.)

Theorem \ref{nthrm:BFI} has the advantage over Theorem \ref{thrm:MainTheorem} that it can handle significantly larger moduli (of size $x^{4/7-\epsilon}$ compared with $x^{11/21-\epsilon}$). However, Theorem \ref{thrm:MainTheorem} has the advantage that it applies to a much larger set of moduli (the moduli appearing in Theorem \ref{nthrm:BFI} must have many flexible factorizations, which is a small proportion of $q\le Q$), and it has the advantage that the error term is handled with absolute values rather than a well-factorable weight. In the second paper of this series we will build upon ideas in this paper to produce a variant of Theorem \ref{nthrm:BFI} which can handle moduli larger than $x^{4/7}$.

Theorem \ref{nthrm:Zhang} has a distinct advantage over Theorem \ref{thrm:MainTheorem} in that it is uniform with respect to the residue class $a$. It only applies to special moduli, however, which are quite rare, whereas Theorem \ref{thrm:MainTheorem} holds for a fairly large proportion of moduli (depending on the size of $\delta$). Moreover, Theorem \ref{thrm:MainTheorem} holds for moduli as large as $x^{11/21-\epsilon}$, which is slightly larger than the largest moduli covered by Theorem \ref{nthrm:Zhang} (of size $x^{157/300-\epsilon}$). Indeed, for a fixed residue class, Theorem \ref{thrm:MainTheorem} implies that Theorem \ref{nthrm:Zhang} holds with moduli of size up to $x^{11/20-\epsilon}$. In the third paper of this series we will produce an estimate which is completely uniform with respect to the residue class, as in \eqref{eq:BV}.
%
%
%
%
%
%
%
%
%
%
%
%
%
%
%
%
%
%
%
%
\section{Acknowledgements}

I would like to thank Etienne Fouvry, John Friedlander, Andrew Granville, Ben Green, Henryk Iwaniec, Kyle Pratt and Maksym Radziwi\l\l\, for useful discussions and suggestions. JM is supported by a Royal Society Wolfson Merit Award, and this project has received funding from the European Research Council (ERC) under the European Union’s Horizon 2020 research and innovation programme (grant agreement No 851318).
%
%
%
%
%
%
%
%
%
%
%
%
%
%
%
%
%
%
%
%
\section{Outline}\label{sec:Outline}
\subsection{High-level overview}
The overall structure of our argument follows all previous results; we use a combinatorial decomposition of the primes (such as Vaughan's identity or Heath-Brown's identity) to transform the problem into that of estimating various convolutions of sequences in arithmetic progressions. We then produce various different estimates depending on the sizes of the factors in the convolution. In each case we use a version of the Linnik dispersion method and Fourier expansion to reduce the problem to estimating complicated exponential sums. After some significant manipulation, these exponential sums are then estimated using results based on ideas from algebraic geometry (the Weil bound or Deligne-style estimates) or from automorphic forms (Deshouillers--Iwaniec-style estimates for sums of Kloosterman sums via the Kuznetsov trace formula).

On a basic level, this paper refines ideas of Bombieri--Friedlander--Iwaniec and combines them with those of Zhang to produce a hybrid result. Indeed, Zhang's work particularly made use of the flexible factorizations of the moduli, and we produce a hybrid result which requires some, but weaker, divisibility conditions for our moduli by combining it with previous ideas coming from `Kloostermania' estimates. 

Unfortunately, na\"ively combining the two methods produces very unsatisfactory results, since the `worst-case' is the same for both approaches. To overcome this issue it is vital that we refine some of the estimates of Bombieri--Friedlander--Iwaniec to handle what was previously the hardest case. This refinement isn't quite enough to cover all cases on its own, but after this refinement the new worst case scenario for the Bombieri--Friedlander--Iwaniec-style estimates is one of the best case scenarios for the Zhang-style estimates, and so the methods combine effectively.
\subsection{Critical case in the Bombieri--Friedlander--Iwaniec argument}
Let us recall the situation in the proof of \cite{BFI3}, where we consider primes in arithmetic progressions to moduli of size $Q=x^{1/2+\delta}$ with $\delta>0$ fixed but small. The ideas in the proof of \cite{BFI3} show that one can handle all convolutions coming from, for example, the Heath-Brown identity, except for essentially two bad cases:
\begin{enumerate}
\item Convolutions of 5 factors each supported on numbers of size $x^{1/5+O(\delta)}$.
\item Convolutions of 4 factors each supported on numbers of size $x^{1/4+O(\delta)}$.
\end{enumerate}
(This is actually an  over-simplification as there are other related bad cases, such as 4 factors with one of size $x^{2/5+O(\delta)}$ and three of size $x^{1/5+O(\delta)}$, but the cases above are the key ones. The proof in \cite{BFI3} also gives a wider variety of bad cases, but a small variation on the techniques can handle these other cases without new ideas.)

In the bad cases above all factors are restricted to relatively short ranges on a logarithmic scale, and so one can simply discard the contributions from these terms. This produces an error term of size roughly $\delta^3$ times the trivial bound, which then automatically gives Theorem \ref{nthrm:BFI2}.

Zhang's work (and its refinements by Polymath) give a method to cover any given combination of convolution lengths, provided the moduli under consideration factorize in a suitable manner. Unfortunately the `hardest' situation is dealing with convolutions of five factors each of length $x^{1/5}$. The argument for this case would produce a result with severe constraints on the factorizations of the moduli, and would almost degenerate to Theorem \ref{nthrm:Zhang}, producing little qualitatively (or quantitatively) new.
\subsection{The original BFI argument}
Let us quickly recall why the argument that Bombieri--Friedlander--Iwaniec use for handling convolutions of five factors fails. Grouping into variables $n_1\sim N_1,n_2\sim N_2$ and $n_3\sim N_3$ with $N_1,N_2\approx x^{2/5}$ and $N_3\approx x^{1/5}$, one wants to estimate
\begin{equation}
\sum_{q\sim Q}c_q\sum_{n_1\sim N_1}\alpha_{n_1}\sum_{n_2\sim N_2}\beta_{n_2}\sum_{n_3\sim N_3}\Bigl(\mathbf{1}_{n_1n_2n_3\equiv a\Mod{q}}-\frac{\mathbf{1}_{(n_1n_2n_3,q)=1}}{\phi(q)}\Bigr)
\label{eq:OverallTarget}
\end{equation}
(for suitable divisor-bounded coefficients $c_q,\alpha_{n_1},\beta_{n_2}$.) Applying Cauchy's inequality in $n_1$ and $q$, the key term one needs to estimate is (a smoothed variant of)
\begin{equation}
\sum_{q\sim Q}\sum_{n_3,n_3'\sim N_3}\sum_{\substack{n_2,n_2'\sim N_2\\ n_2n_3\equiv n_2'n_3'\Mod{q}}}\beta_{n_2}\beta_{n_2'}\sum_{\substack{n_1\sim N_1\\ n_1n_2n_3\equiv a\Mod{q}}}1.
\label{eq:OverallTarget2}
\end{equation}
Fourier completing the inner sum, using Bezout's identity to flip the moduli in the resulting exponential and substituting $qr=n_2n_3-n_2'n_3'$ leaves us to estimate a sum like
\begin{equation}
\sum_{r\sim N_2N_3/Q}\sum_{n_3,n_3'\sim N_3}\sum_{\substack{n_2,n_2'\sim N_2\\ n_2n_3\equiv n_2'n_3'\Mod{r}}}\beta_{n_2}\beta_{n_2'}\sum_{h\sim Q/N_1}e\Bigl(\frac{ah\overline{n_2'n_3'}}{n_2n_3}\Bigr).
\label{eq:OverallTarget3}
\end{equation}
The approach of Bombieri--Friedlander--Iwaniec is then to apply Cauchy's inequality to $r,n_2,n_3,n_2'$ and combine $n_2,n_3$ together, leading one to bound
\begin{equation}
\sum_{r\sim N_2N_3/Q}\,\sum_{n\sim N_2N_3}\sum_{n_2'\sim N_2}\Bigl|\sum_{h\sim Q/N_1}\sum_{\substack{n_3'\sim N_3\\ n_2'n_3'\equiv n\Mod{r}}}\beta_{n_2'}e\Bigl(\frac{ah\overline{n_2'n_3'}}{n}\Bigr)\Bigr|^2.
\label{eq:OverallTarget4}
\end{equation}
Unfortunately the diagonal terms are acceptable only if $N_3/r>(\log{x})^{2A}Q/N_1$ (so that the the savings outweigh the losses from completion of sums), which requires that $N_1>(\log{x})^{2A}N_2$. This condition (just) fails if $N_1\approx N_2$, and this means that this approach cannot handle the case of 5 equally sized factors. It doesn't seem like any other rearrangement of the sums can handle this situation either.
\subsection{A modified argument} \label{sec:OutlineModified}
To address the issue raised above we use an idea inspired by the `amplification' method of Friedlander--Iwaniec (see \cite{FriedlanderICM}). Normally the amplification method enables one to obtain a non-trivial bound by considering a larger average over a family, but amplifying the contribution from the sum of interest, which increases the size of the diagonal terms. In our situation we have an almost opposite situation - we wish to reduce (or `de-amplify') the contribution from the diagonal terms. We do this by first artificially introducing a congruence condition $n_2n_3\equiv b\Mod{c}$ which we average over at the outset, so we consider in place of \eqref{eq:OverallTarget}
\begin{align*}
\frac{1}{C}\sum_{c\sim C}\sum_{b\Mod{c}}\sum_{q\sim Q}c_q\sum_{n_1\sim N_1}\alpha_{n_1}\sum_{n_2\sim N_2}\beta_{n_2}\sum_{\substack{n_3\sim N_3\\ n_2n_3\equiv b\Mod{c}}}\Delta,
\end{align*}
where $\Delta$ is the same inner sum as before. We now follow the previous approach, applying Cauchy in $c,b,q$ and $n_1$, leaving us to estimate the equivalent of \eqref{eq:OverallTarget2}
\[
\sum_{c\sim C}\sum_{q\sim Q}\sum_{n_3,n_3'\sim N_3}\sum_{\substack{n_2,n_2'\sim N_2\\ n_2n_3\equiv n_2'n_3'\Mod{q c}}}\beta_{n_2}\beta_{n_2'}\sum_{\substack{n_1\sim N_1\\ n_1n_2n_3\equiv a\Mod{q}}}1.
\]
We can Fourier complete the inner sum and apply Bezout as before (with no additional losses) but now our substitution becomes $q r' c=n_2n_3-n_2'n_3'$, and the equivalent of \eqref{eq:OverallTarget3} becomes
\begin{equation*}
\sum_{r'\sim N_2 N_3/QC}\,\sum_{c\sim C}\sum_{n_3,n_3'\sim N_3}\sum_{\substack{n_2,n_2'\sim N_2\\ n_2n_3\equiv n_2'n_3'\Mod{r'c}}}\beta_{n_2}\beta_{n_2'}\sum_{h\sim Q/N_1}e\Bigl(\frac{ah\overline{n_2'n_3'}}{n_2n_3}\Bigr).
\end{equation*}
This has essentially forced the $r$ variable to have a factor $c\sim C$ compared with what we had before, and we can choose any $C\in[1,N_2N_3/Q]$. Thinking of $C=x^\epsilon$, we can then apply Cauchy in $r,n_2,n_3,n_2'$, leaving us to estimate
\begin{equation*}
\sum_{r\sim N_2N_3/QC}\,\sum_{n\sim N_2N_3}\sum_{n_2'\sim N_2}\Bigl|\sum_{c\sim C}\sum_{h\sim Q/N_1}\sum_{\substack{n_3'\sim N_3\\ n_2'n_3'\equiv n\Mod{rc}}}\beta_{n_2'}e\Bigl(\frac{ah\overline{n_2'n_3'}}{n_2n_3}\Bigr)\Bigr|^2.
\end{equation*}
This has effectively reduced the size of the diagonal terms by a factor of $C$, so they are now admissible. The off-diagonal terms will be larger by a factor polynomial in $C$, which is acceptable since for $C$ a small enough power of $x$ since we had a power-saving bound beforehand. This therefore allows us to treat the critical situation of five factors of size $x^{1/5}$.

To make this argument work we rely on progress towards the Selberg eigenvalue conjecture for Maass forms, and in particular we use the work of Kim--Sarnak \cite{KimSarnak} coming from automorphy of symmetric fourth power $L$-functions in the estimates of Deshouillers--Iwaniec \cite{DeshouillersIwaniec} on sums of Kloosterman sums.
\subsection{Combining with ideas of Zhang}
Unfortunately, the estimates of Bombieri--Friedlander--Iwaniec type are still not useful for handling products of 4 factors of size $x^{1/4+O(\delta)}$. To handle these terms we use ideas based on Zhang's work, and it is these ideas which require us to assume that our moduli have a convenient sized prime factor (if we could handle these quadruple convolutions for all $q$, then we could get a corresponding version of Theorem \ref{thrm:MainTheorem} valid for all $q\sim x^{1/2+\delta}$). 

One can combine two of the factors together to get a factor very close to $x^{1/2}$. Provided the moduli factorize such that one factor is a small power of $x$ and the other is close to $x^{1/2}$, the Zhang-style estimates can cover these cases, and moreover it is one of the most efficient parts of Zhang's argument. This means for moduli with a `convenient sized' prime factor, we can handle these terms.

Together with our previous refinement, this enables us to handle all the terms not handled by \cite{BFI3}, and so would give a variant of Theorem \ref{thrm:MainTheorem} when $\delta$ is sufficiently small.
\subsection{Quantitative bounds}
To get a good quantitative estimate, we refine several different auxiliary estimates in other regimes, making use of the fact that we assume that our moduli have a conveniently sized divisor to enhance the previous estimates. In particular, we refine work of Fouvry \cite{Fouvry} for estimates with a factor close to $x^{1/7}$, refine work of Zhang \cite{Zhang} and Polymath \cite{Polymath} for estimates with a factor close to $x^{1/2}$, produce a new small divisor estimate for estimates with a factor close to $x^{1/21}$, and refine estimates on triple convolutions of smooth sequences by producing estimates tailored to our situation. These refinements enable us to handle moduli as large as $x^{11/21}$. Until the recent work of Fouvry-Kowalski-Michel \cite{FKMDivisor} and Polymath \cite{Polymath}, the much easier problem of showing that triple divisor function $d_3(n)$ satisfied a Bombieri-Vinogradov type theorem to moduli larger than $x^{11/21}$ was not known, for example. We have also swept several additional technical details under the carpet to emphasize to the reader what we view as the most important new ideas.
%
%
%
%
%
%
%
%
%
%
%
%
%
%
%
%
%
%
%
%
\section{Corollaries \ref{crllry:BVDivisor}, \ref{crllry:MostModuli} and \ref{crllry:1121} from Theorem \ref{thrm:MainTheorem}}
Before properly embarking on the proof of Theorem \ref{thrm:MainTheorem}, we first show how Corollaries \ref{crllry:BVDivisor}-\ref{crllry:1121} follow from Theorem \ref{thrm:MainTheorem}.
%
%
%
%
%
%
%
%
\begin{proof}[Proof of Corollary \ref{crllry:BVDivisor} assuming Theorem \ref{thrm:MainTheorem}]
For notational convenience set $D_1:=x^{2\delta+\eta}/2$ and $D_2:=\min(x^{1/10-7\delta/5-\eta},x^{1/2-19\delta-\eta})$. By considering the $d$ in dyadic ranges we see that
\begin{align*}
\mathcal{Q}_{\delta,\eta}&\subseteq \bigcup_{\substack{D=2^j\\ D_1\le D\le D_2}}\Bigl\{q\le x^{1/2+\delta}:\,\exists d|q \text{ s.t. }d\in[D,2D]\Bigr\}\\
&=\bigcup_{\substack{D=2^j\\ D_1\le D\le D_2}}\Bigl\{q_1q_2\le x^{1/2+\delta}:\,q_1\in[D,2D]\Bigr\}.
\end{align*}
We now set $Q_1:=2D$ and $Q_2:=x^{1/2+\delta}/D$, noting that if $q_1q_2\le  x^{1/2+\delta}$ and $q_1\in[D,2D]$ then $q_1\le Q_1$ and $q_2\le Q_2$. If $D_1\le D \le D_2$ we have
\begin{align*}
Q_1Q_2^2&=2\frac{x^{1+2\delta}}{D}\ll x^{1-\eta},\\
Q_1^{12}Q_2^7&=2^{12}D^5 x^{7/2+7\delta}\ll x^{4-5\eta},\\
Q_1^{20}Q_2^{19}&=2^{20}D x^{19/2+19\delta}\ll x^{10-\eta}.
\end{align*}
Thus, by the union bound and applying Theorem \ref{thrm:MainTheorem} (with $\epsilon=\eta/101$) we see that
\begin{align*}
\sum_{\substack{q\in \mathcal{Q}_{\delta,\eta}\\ (q,a)=1}}\Bigl|\pi(x;q,a)-\frac{\pi(x)}{\phi(q)}\Bigr|&\ll (\log{x})\sup_{D_1\le D\le D_2}\sum_{q_1\le 2D}\sum_{q_2\le  x^{1/2+\delta}/D}\Bigl|\pi(x;q_1q_2,a)-\frac{\pi(x)}{\phi(q_1q_2)}\Bigr|\\
&\ll_{a, A,\eta} \frac{x}{(\log{x})^A}.\qedhere
\end{align*}
\end{proof}
%
%
%
%
%
%
%
%
\begin{proof}[Proof of Corollary \ref{crllry:MostModuli} assuming Theorem \ref{thrm:MainTheorem}]
We let $\eta=\epsilon\delta$. Since $0<\delta<1/55$, we see that any $q\sim Q$ with $q\notin \mathcal{Q}_{\delta,\eta}$ must be the product $m_1m_2$ with $m_1<x^{(2+\epsilon)\delta}$ and $P^-(m_2)\ge x^{1/10-(7/5+\epsilon)\delta}\ge (Q/m_1)^{1/7}$, where we use $P^-(m)$ to denote the smallest prime factor of $m$. We recall that as $t\rightarrow\infty$
\[
\sum_{\substack{n\le t\\ P^-(n)\ge t^\gamma}}1=\Bigl(\omega\Bigl(\frac{1}{\gamma}\Bigr)+o_\gamma(1)\Bigr)\frac{t}{\gamma\log{t}} ,
\]
where $\omega(u)$ is the Buchtab function, which is piecewise continuous and satisfies $\omega(7)< 4/7$. Thus we see that for $\epsilon$ small enough and $x$ large enough
\begin{align*}
\sum_{\substack{q\in [Q,2Q]\backslash\mathcal{Q}_{\delta,\eta}\\ (q,a)=1}}1\le \sum_{\substack{m_1\le x^{(2+\epsilon)\delta} \\ (m_1,a)=1}}\sum_{\substack{Q/m_1\le m_2\le 2Q/m_1\\ P^-(m_2)\ge (Q/m_1)^{1/7}}}1&= \sum_{\substack{m_1\le x^{(2+\epsilon)\delta}\\ (m_1,a)=1}}\frac{(7\omega(7)+o(1))Q}{m_1\log(Q/m_1)}\\
&\le \frac{\phi(a)}{a}\frac{(4+o(1))Q(2+\epsilon)\delta}{1/2-(2+\epsilon)\delta}\le 18\frac{\phi(a)\delta Q}{a}.
\end{align*}
By Corollary \ref{crllry:BVDivisor}, for all but $O(Q/\log^A{x})$ moduli $q\in\mathcal{Q}_{\delta,\eta}$ with $(q,a)=1$ we have 
\[
\pi(x;q,a)=\frac{\pi(x)}{\phi(q)}+O_{a,A,\delta}\Bigl(\frac{x}{(\log{x})^{A}}\Bigr).
\] 
Putting this all together gives the result.
\end{proof}
%
%
%
%
%
%
%
%
\begin{proof}[Proof of Corollary \ref{crllry:1121} assuming Theorem \ref{thrm:MainTheorem}]
This is immediate from Theorem \ref{thrm:MainTheorem} taking $Q_1=x^{1/21}$, $Q_2=x^{10/21-100\epsilon}$, and then redefining $\epsilon$.
\end{proof}
%
%
%
%
%
%
%
%
Thus we are left to establish Theorem \ref{thrm:MainTheorem}.
%
%
%
%
%
%
%
%
%
%
%
%
%
%
%
%
%
%
%
%
%
\section{Paper structure}
The dependency diagram between the main propositions in the paper is as follows.
  \begin{center}
\makebox[\textwidth]{\parbox{1.5\textwidth}{
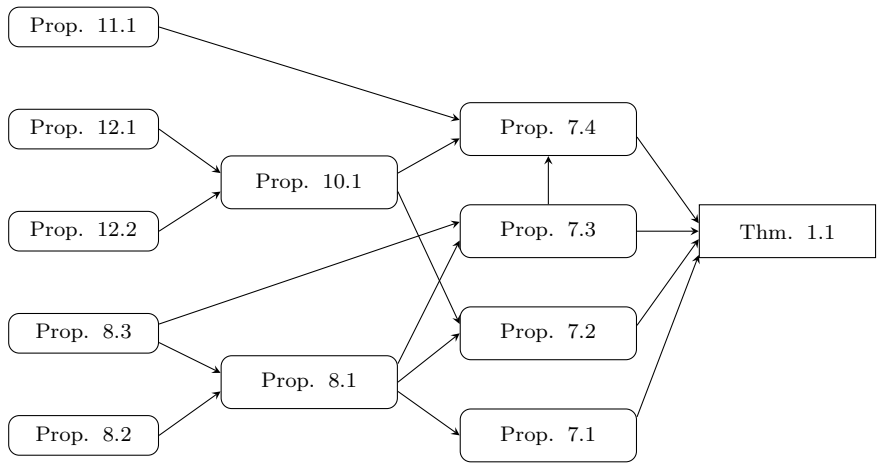
\begin{center}
   \tikzstyle{interface}=[draw, text width=6em,
      text centered, minimum height=2.0em]
   \tikzstyle{daemon}=[draw, text width=6em,
      minimum height=2em, text centered, rounded corners]
   \tikzstyle{lemma}=[draw, text width=5em,
      minimum height=1.5em, text centered, rounded corners]
   \tikzstyle{dots} = [above, text width=6em, text centered]
   \tikzstyle{wa} = [daemon, text width=6em,
      minimum height=2em, rounded corners]
   \tikzstyle{ur}=[draw, text centered, minimum height=0.01em]
   \def\blockdist{1.3}
   \def\edgedist{0.}
   \begin{tikzpicture}
\node (d1)[daemon] {\footnotesize Prop. \ref{prpstn:3Primes}};
\path (d1.south)+(0,-1) node (d6)[daemon] {\footnotesize Prop. \ref{prpstn:4Primes}};
\path (d6.south)+(0,-1) node (d2)[daemon] {\footnotesize Prop. \ref{prpstn:SieveAsymptotic}};
\path (d2.south)+(0,-1) node (d3)[daemon] {\footnotesize Prop. \ref{prpstn:TypeII2}};
\path (d6.east)+(2,0) node (t1)[interface] {\footnotesize Thm. \ref{thrm:MainTheorem}};
\path (d6.west)+(-2,0.65) node (m1)[daemon] {\footnotesize Prop. \ref{prpstn:VBounds}};
\path (d2.west)+(-2,-0.65) node (m2)[daemon] {\footnotesize Prop. \ref{prpstn:TypeII}};
\path (d1.north west)+(-5,1) node (l1)[lemma] {\footnotesize Prop. \ref{prpstn:TripleDivisor}};
\path (d1.west)+(-5,0) node (l2)[lemma] {\footnotesize Prop. \ref{prpstn:Fouvry}};
\path (d6.west)+(-5,0) node (l3)[lemma] {\footnotesize Prop. \ref{prpstn:SmallDivisor}};
\path (d2.west)+(-5,0) node (l5)[lemma] {\footnotesize Prop. \ref{prpstn:TripleRough}};
\path (d2.south west)+(-5,-1) node (l6)[lemma] {\footnotesize Prop. \ref{prpstn:Zhang}};
\path [draw, ->,>=stealth] (l1.east) -- node [above] {} ([yshift=3.5]d1.west) ;
\path [draw, ->,>=stealth] (l2.east) -- node [above] {} ([yshift=3.5]m1.west) ;
\path [draw, ->,>=stealth] (l3.east) -- node [above] {} ([yshift=-3.5]m1.west) ;
\path [draw, ->,>=stealth] ([yshift=-3.5]l5.east) -- node [above] {} ([yshift=3.5]m2.west) ;
\path [draw, ->,>=stealth] ([yshift=3.5]l5.east) -- node [above] {} ([yshift=3.5]d6.west) ;
\path [draw, ->,>=stealth] (l6.east) -- node [above] {} ([yshift=-3.5]m2.west) ;
\path [draw, ->,>=stealth] ([yshift=3.5]m1.east) -- node [above] {} ([yshift=-3.5]d1.west) ;
\path [draw, ->,>=stealth] ([yshift=-3.5]m1.east) -- node [above] {} ([yshift=3.5]d2.west) ;
\path [draw, ->,>=stealth] (m2.east) -- node [above] {} (d2.west) ;
\path [draw, ->,>=stealth] ([yshift=-3.5]m2.east) -- node [above] {} ([yshift=0]d3.west) ;
\path [draw, ->,>=stealth] ([yshift=7]m2.east) -- node [above] {} ([yshift=-3.5]d6.west) ;
\path [draw, ->,>=stealth] (d6.north) -- node [above] {} (d1.south) ;
\path [draw, ->,>=stealth] ([yshift=-3]d1.east) -- node [above] {} ([yshift=3]t1.west) ;
\path [draw, ->,>=stealth] ([yshift=3]d2.east) -- node [above] {} ([yshift=-3]t1.west) ;
\path [draw, ->,>=stealth] ([yshift=7]d3.east) -- node [above] {} ([yshift=-9]t1.west) ;
\path [draw, ->,>=stealth] (d6.east) -- node [above] {} (t1.west) ;
   \end{tikzpicture}
\end{center}}}
   \end{center}
In the first half of the paper, Sections \ref{sec:Decomposition}-\ref{sec:SmallTypeII}, we reduce the proof of Theorem \ref{thrm:MainTheorem} to the task of establishing the technical Propositions \ref{prpstn:TripleDivisor}, \ref{prpstn:Fouvry}, \ref{prpstn:SmallDivisor}, \ref{prpstn:Zhang} and \ref{prpstn:TripleRough}. The second half of the paper, Sections \ref{sec:Lemmas}-\ref{sec:TripleDivisor} is then spent establishing these propositions.

In Section \ref{sec:Decomposition}, we reduce the proof of Theorem \ref{thrm:MainTheorem} to establishing Propositions \ref{prpstn:3Primes}, \ref{prpstn:SieveAsymptotic}, \ref{prpstn:TypeII2} and \ref{prpstn:4Primes}. This is done by a combinatorial sieve decomposition, relying on Harman's sieve.

In Section \ref{sec:TypeII} we first reduce the proof of our Type II result (Proposition \ref{prpstn:TypeII2}) to the more general Proposition \ref{prpstn:TypeII}. We then reduce Proposition \ref{prpstn:TypeII} to establishing the more technical Proposition \ref{prpstn:Zhang} and Proposition \ref{prpstn:TripleRough}, which establish Type II estimates for convolutions which are effective either close to, or further from the balanced length case.

In Section \ref{sec:4Primes}, we deduce Proposition \ref{prpstn:4Primes} handling numbers with four or more prime factors from Proposition \ref{prpstn:TripleRough} and Proposition \ref{prpstn:TypeII} by performing more combinatorial manipulations.

In Section \ref{sec:Asymptotic} we establish Proposition 
\ref{prpstn:SieveAsymptotic} from Proposition \ref{prpstn:TypeII} and the more technical Proposition \ref{prpstn:VBounds}. These arguments are based on estimates associated with the fundamental lemma of sieve methods and Harman's sieve.

In Section \ref{sec:3Primes} we deduce Proposition \ref{prpstn:3Primes} on numbers with three prime factors from Propositions \ref{prpstn:4Primes} and \ref{prpstn:VBounds}, and the more technical Proposition \ref{prpstn:TripleDivisor}.

In Section \ref{sec:SmallTypeII} we deduce Proposition \ref{prpstn:VBounds} on convolutions with a short factor from the more technical Proposition \ref{prpstn:Fouvry} and \ref{prpstn:SmallDivisor}.

In Section \ref{sec:Lemmas} and Section \ref{sec:Dispersion} we establish various preparatory lemmas and a general estimate for using the Linnik dispersion method to reduce the problems considered  to exponential sums.

In Section \ref{sec:Fouvry} we establish Proposition \ref{prpstn:Fouvry}, refining ideas of Fouvry \cite{Fouvry}.

In Section \ref{sec:SmallDivisor}, we establish Proposition \ref{prpstn:SmallDivisor}, which is a new estimate similar to ideas of Fouvry \cite{Fouvry}.

In Section \ref{sec:Zhang} we establish Proposition \ref{prpstn:Zhang}, which is a refinement of an estimate of Zhang \cite{Zhang}.

In Section \ref{sec:BFI} we establish Proposition \ref{prpstn:TripleRough}, refining previous work of Bombieri, Friedlander and Iwaniec \cite[Theorem 4]{BFI2}.

In Section \ref{sec:Smoothing}, we establish a couple of preparatory lemmas for the later divisor function estimates of Section \ref{sec:TripleDivisor}.

In Section \ref{sec:TripleDivisor} we establish Proposition \ref{prpstn:TripleDivisor} on the triple divisor function, developing ideas of Friedlander, Iwaniec \cite{FIDivisor}, Heath-Brown \cite{HBDivisor}, Fouvry, Kowalski, Michel \cite{FKMDivisor} and Polymath \cite{Polymath}.
%
%
%
%
%
%
%
%
%
%
%
%
%
%
%
%
%
%
%
%
\section{Notation}
We will use the Vinogradov $\ll$ and $\gg$ asymptotic notation, and the big oh $O(\cdot)$ and $o(\cdot)$ asymptotic notation. $f\asymp g$ will denote the conditions $f\ll g$ and $g\ll f$ both hold. Dependence on a parameter will be denoted by a subscript.

We will view $a$ (the residue class we count arithmetic functions in to different moduli $q$) as a fixed positive integer throughout the paper, and any constants implied by asymptotic notation will be allowed to depend on $a$ from this point onwards. Similarly, throughout the paper, we will let $\epsilon$ be a single fixed small real number; $\epsilon=10^{-100}$ would probably suffice. Any bounds in our asymptotic notation will also be allowed to depend on $\epsilon$.

The letter $p$ will always be reserved to denote a prime number. We use $\phi$ to denote the Euler totient function, $e(x):=e^{2\pi i x}$ the complex exponential, $\tau_k(n)$ the $k$-fold divisor function, $\mu(n)$ the M\"obius function. We let $P^-(n)$, $P^+(n)$ denote the smallest and largest prime factors of $n$ respectively, and $\hat{f}$ denote the Fourier transform of $f$ over $\mathbb{R}$ - i.e. $\hat{f}(\xi)=\int_{-\infty}^{\infty}f(t)e(-\xi t)dt$. We use $\mathbf{1}$ to denote the indicator function of a statement. For example,
\[
\mathbf{1}_{n\equiv a\Mod{q}}=\begin{cases}1,\qquad &\text{if }n\equiv a\Mod{q},\\
0,&\text{otherwise}.
\end{cases}
\]
We will use $(a,b)$ to denote $\gcd(a,b)$ when it does not conflict with notation for ordered pairs. For $(n,q)=1$, we will use $\overline{n}$ to denote the inverse of the integer $n$ modulo $q$; the modulus will be clear from the context. For example, we may write $e(a\overline{n}/q)$ - here $\overline{n}$ is interpreted as the integer $m\in \{0,\dots,q-1\}$ such that $m n\equiv 1\Mod{q}$. Occasionally we will also use $\overline{\lambda}$ to denote complex conjugation; the distinction of the usage should be clear from the context.  For a complex sequence $\alpha_{n_1,\dots,n_k}$, $\|\alpha\|_2$ will denote the $\ell^2$ norm $\|\alpha\|_2=(\sum_{n_1,\dots,n_k}|\alpha_{n_1,\dots,n_k}|^2)^{1/2}$.

Summations assumed to be over all positive integers unless noted otherwise. We use the notation $n\sim N$ to denote the conditions $N<n\le 2N$.

We will let $z_0:=x^{1/(\log\log{x})^3}$ and $y_0:=x^{1/\log\log{x}}$ two parameters depending on $x$, which we will think of as a large quantity. We will let $\psi_0:\mathbb{R}\rightarrow\mathbb{R}$ denote a fixed smooth function supported on $[1/2,5/2]$ which is identically equal to $1$ on the interval $[1,2]$ and satisfies the derivative bounds $\|\psi_0^{(j)}\|_\infty\ll (4^j j!)^2$ for all $j\ge 0$. (See \cite[Page 368, Corollary]{BFI2} for the construction of such a function.)
%
%
%
%
%
%
%
%
\begin{dfntn}[Siegel-Walfisz condition]
We say that a complex sequence $\alpha_n$ satisfies the \textbf{Siegel-Walfisz condition} if for every $d\ge 1$, $q\ge 1$ and $(a,q)=1$ and every $A>1$ we have
\begin{equation}
\Bigl|\sum_{\substack{n\sim N\\ n\equiv a\Mod{q}\\ (n,d)=1}}\alpha_n-\frac{1}{\phi(q)}\sum_{\substack{n\sim N\\ (n,d q)=1}}\alpha_n\Bigr|\ll_A \frac{N\tau(d)^{B_0}}{(\log{N})^A}.
\label{eq:SiegelWalfisz}
\end{equation}
\end{dfntn}
%
%
%
%
%
%
%
%
We note that $\alpha_n$ certainly satisfies the Siegel-Walfisz condition if $\alpha_n=1$, if $\alpha_n=\mu(n)$ or if $\alpha_n$ is the indicator function of the primes.
%
%
%
%
%
%
%
%
%
%
%
%
%
%
%
%
%
%
%
%
\section{Sieve decomposition}\label{sec:Decomposition}
In this section we prove Theorem \ref{thrm:MainTheorem} assuming five more technical propositions, which we will then gradually establish over the rest of the paper. We do this by applying a sieve decomposition to break the count of primes in arithmetic progressions into counts of integers with particular prime factorizations, which can then be estimated using the relevant proposition. The sieve decomposition is based on ideas based on Harman's sieve (see \cite{Harman}), but we could have used the Heath-Brown identity and some combinatorial lemmas as an alternative (we hope that this arrangement will be more useful for future work).

Define $S_n$ and $S_d(z)$ (depending on integers $a,q_1,q_2$ satisfying $(a,q_1q_2)=1$ which we suppress from the notation for convenience) for integers $n,d$ and a real $z$ by
\begin{align*}
S_n&:=\mathbf{1}_{n\equiv a\Mod{q_1q_2}}-\frac{1}{\phi(q_1q_2)}\mathbf{1}_{(n,q_1q_2)=1},\\
S_d(z)&:=\sum_{\substack{n\sim x/d\\ P^-(n)>z}}S_{d n}.
\end{align*}
With this notation, we are able to state our key propositions.
%
%
%
%
%
%
%
%
\begin{prpstn}[Type II estimate]\label{prpstn:TypeII2}
Let $Q_1,Q_2 $ satisfy \eqref{eq:Cons1} and \eqref{eq:Cons2}. Let $P_1,\dots, P_J\ge x^{1/7+10\epsilon}$ be such that $P_1\cdots P_J\ll x$ and
\[
x^{3/7+\epsilon}\le\prod_{j\in\mathcal{J}}P_j\le x^{4/7-\epsilon}
\]
for some subset $\mathcal{J}\subseteq\{1,\dots,J\}$. Then we have
\[
\sum_{q_1\sim Q_1}\sum_{\substack{q_2\sim Q_2\\ (q_1q_2,a)=1}}\Bigl|\mathop{\sideset{}{^*}\sum}_{\substack{p_1,\dots,p_J\\ p_i\sim P_i\,\forall i}}S_{p_1\cdots p_J}(p_J)\Bigr|\ll_A\frac{x}{(\log{x})^A}.
\]
Here $\sum^*$ indicates that the summation is restricted by $O(1)$ inequalities of the form $p_1^{\alpha_1}\cdots p_J^{\alpha_J}\le B$. The implied constant may depend on all such exponents $\alpha_i$, but none of the quantities $B$.
\end{prpstn}
%
%
%
%
%
%
%
%
\begin{prpstn}[Sieve asymptotics]\label{prpstn:SieveAsymptotic}
Let $A>0$. Let $x^{3/7+\epsilon}\ge P_1\ge \dots \ge P_r\ge x^{1/7+10\epsilon}$ be such that $P_1\cdots P_r\le x^{2/3}$ and such that either $r=1$ or $P_r\le x^{1/4+\epsilon}$. Let $Q_1,Q_2$ satisfy \eqref{eq:Cons1}, \eqref{eq:Cons2} and \eqref{eq:Cons3}.

Then we have
\[
\sum_{q_1\sim Q_1}\sum_{\substack{q_2\sim Q_2\\ (q_1q_2,a)=1}} \Bigl|\mathop{\sideset{}{^*}\sum}_{\substack{p_1,\dots,p_r\\ p_i\sim P_i\,\forall i}} S_{p_1\cdots p_r}(x^{1/7+10\epsilon})\Bigr|\ll_{A}\frac{x}{(\log{x})^A}.
\]
Here $\sum^*$ means that the summation is restricted to $O(1)$ inequalities of the form $p_1^{\alpha_1}\cdots p_r^{\alpha_r}\le B$ for some constants $\alpha_1,\dots \alpha_r$. The implied constant may depend on all such exponents $\alpha_i$, but none of the quantities $B$.

Moreover, we also have the related estimate
\[
\sum_{q_1\sim Q_1}\sum_{\substack{q_2\sim Q_2\\ (q_1q_2,a)=1}} \Bigl|S_{1}(x^{1/7+10\epsilon})\Bigr|\ll_{A}\frac{x}{(\log{x})^A}.
\]
\end{prpstn}
%
%
%
%
%
%
%
%
\begin{prpstn}[Numbers with 4 or more prime factors]\label{prpstn:4Primes}
Let $A>0$. Let $J\ge 4$ and $P_1\ge\dots\ge P_J\ge x^{1/7+10\epsilon}$ with $P_1\cdots P_J\asymp x$. Let $Q_1,Q_2$ satisfy \eqref{eq:Cons1} and \eqref{eq:Cons2}. Then we have
\[
\sum_{q_1\sim Q_1}\sum_{\substack{q_2\sim Q_2\\ (q_1q_2,a)=1}}\Bigl|\mathop{\sideset{}{^*}\sum}_{\substack{p_1,\dots,p_J\\ p_i\sim P_i\forall i}}S_{p_1\cdots p_J}\Bigr|\ll_{A}\frac{x}{(\log{x})^A}.
\]
Here $\sum^*$ indicates that the summation is restricted by $O(1)$ inequalities of the form $p_1^{\alpha_1} \cdots p_J^{\alpha_J}\le B$. The implied constant may depend on all such exponents $\alpha_i$, but none of the quantities $B$.
\end{prpstn}
%
%
%
%
%
%
%
%
\begin{prpstn}[Numbers with three prime factors]\label{prpstn:3Primes}
Let $A>0$ and let $P_1,P_2,P_3\in [x^{1/4},x^{3/7+\epsilon}]$ with $P_1P_2P_3\asymp x$. Then we have
\[
\sum_{q_1\sim Q_1}\sum_{\substack{q_2\sim Q_2\\ (q_1q_2,a)=1}}\Bigl|\mathop{\sideset{}{^*}\sum}_{\substack{p_1,p_2,p_3\\ p_i\sim P_i\forall i}}S_{p_1p_2p_3}\Bigr|\ll_A \frac{x}{(\log{x})^A}.
\]
Here $\sum^*$ means that the summation is restricted to $O(1)$ inequalities of the form $p_1^{\alpha_1}p_2^{\alpha_2} p_3^{\alpha_3}\le B$ for some constants $\alpha_1,\alpha_2, \alpha_3$. The implied constant may depend on all such exponents $\alpha_i$, but none of the quantities $B$.
\end{prpstn}
%
%
%
%
%
%
%
%
Proposition \ref{prpstn:TypeII2} will be established in Section \ref{sec:TypeII}, Proposition \ref{prpstn:SieveAsymptotic} will be established in Section \ref{sec:Asymptotic}, Proposition \ref{prpstn:4Primes} will be established in Section \ref{sec:4Primes}, and Proposition \ref{prpstn:3Primes} will be established in Section \ref{sec:3Primes}. 

Our decomposition will rely on suitable applications of the following elementary identity.
%
%
%
%
%
%
%
%
\begin{lmm}[Buchstab's identity]
Let $z_1<z_2$. Then
\[
S_d(z_2)=S_d(z_1)-\sum_{z_1<p\le z_2}S_{d p}(p)
\]
\end{lmm}
\begin{proof}
This is inclusion-exclusion based on the smallest prime factor of $n$ in the sum in $S_d(z)$.
\end{proof}
%
%
%
%
%
%
%
%
We are now in a position to prove Theorem \ref{thrm:MainTheorem}.
%
%
%
%
%
%
%
%
\begin{proof}[Proof of Theorem \ref{thrm:MainTheorem} assuming Propositions \ref{prpstn:TypeII2}, \ref{prpstn:SieveAsymptotic}, \ref{prpstn:4Primes} and \ref{prpstn:3Primes}]
We first establish some notation. Let $z_1,z_2,z_3$ be defined in terms of $x$ by
\begin{align*}
z_1&:= x^{1/7+10\epsilon},&\qquad z_2&:=x^{3/7+\epsilon},&\qquad z_3:=x^{4/7-\epsilon}.
\end{align*}
For the purposes of this section (but this section only), we assume that the symbols $r,s,t,u,v$ represent prime numbers, in the same manner that we use the symbol $p$. 
Our goal is to estimate
\begin{equation}
\sum_{q_1\sim Q_1}\sum_{\substack{q_2\sim Q_2\\ (q_1q_2,a)=1}}|S_1(2x^{1/2})|=\sum_{q_1\sim Q_1}\sum_{\substack{q_2\sim Q_2\\ (q_1q_2,a)=1}}\Bigl|\sum_{\substack{p\sim x\\ p\equiv a\Mod{q_1q_2}}}1-\frac{1}{\phi(q)}\sum_{\substack{p\sim x\\ (p,q_1q_2)=1}}1\Bigr|
\label{eq:Target}
\end{equation}
by decomposing $S_1(2x^{1/2})$ into terms which can be handled by one of Proposition \ref{prpstn:TypeII2}, Proposition \ref{prpstn:SieveAsymptotic}, Proposition \ref{prpstn:4Primes} or Proposition \ref{prpstn:3Primes}. 
By Buchstab's identity
\begin{align}
S_1(2x^{1/2})&=S_1(z_1)-\sum_{z_1< p\le z_2}S_p(p)-\sum_{z_2< p\le 2x^{1/2}}S_p(p).\label{eq:Decomp0}
\end{align}
By Proposition \ref{prpstn:SieveAsymptotic}, the first term on the right hand side of \eqref{eq:Decomp0} makes a negligible contribution to \eqref{eq:Target} (i.e. a contribution of size $O_{A}(x/(\log{x})^A)$ when summed with absolute values over $q_1$ and $q_2$). Similarly, by Proposition \ref{prpstn:TypeII2}, the last term also contributes negligibly. Thus we are left to consider the middle term. We note that if $p\in[2x^{1/3},z_2]$ then $S_p(p)$ just counts products of two primes, and so is equal to $S_p(x^{1/2+\epsilon}/p^{1/2})$. Thus, applying Buchstab's identity again, we have
\begin{align}
\sum_{z_1< p\le z_2}S_p(p)&=\sum_{z_1<p\le z_2}S_p\Bigl(\min\Bigl(p,\frac{x^{1/2+\epsilon}}{p^{1/2}}\Bigr)\Bigr)\nonumber\\
&=\sum_{z_1< p\le z_2}S_p(z_1)-\sum_{\substack{z_1< r\le p\le z_2\\ pr^2\le x^{1+2\epsilon}}}S_{p r}(r)\nonumber\\
&=\sum_{z_1< p\le z_2}S_p(z_1)-\sum_{\substack{z_1< r\le p\le z_2\\ p r\le z_2}}S_{p r}(r)-\sum_{\substack{z_1< r\le p\le z_2\\ z_2<p r\le z_3}}S_{p r}(r)\nonumber\\
&\qquad-\sum_{\substack{z_1< r\le p\le z_2\\ z_3<p r\\ r\le x^{1/4}}}S_{p r}(r)-\sum_{\substack{z_1< r\le p\le z_2\\ z_3<p r\\ p r^2\le x^{1+2\epsilon}\\ r>x^{1/4}}}S_{p r}(r).\label{eq:Decomp1}
\end{align}
(We have simplified some of the conditions of summation above.)

By Proposition \ref{prpstn:SieveAsymptotic}, the first term in \eqref{eq:Decomp1} contributes negligibly to \eqref{eq:Target}. By Proposition \ref{prpstn:TypeII2}, the third term also contributes negligibly. Thus we are left to consider the remaining three terms of \eqref{eq:Decomp1}. 

We first consider the final term of \eqref{eq:Decomp1}. Since $p r^3>2x$ this counts products of exactly three primes. Specifically
\[
\sum_{\substack{z_1< r\le p\le z_2\\ z_3<p r\\ p r^2\le x^{1+2\epsilon}\\ r>x^{1/4} }}S_{p r}(r)=\sum_{\substack{z_1< r\le p\le z_2\\ z_3<p r\\ p r^2\le x^{1+2\epsilon}\\ x^{1/4}< r< s\\ p r s\sim x}}S_{p r s}.
\]
Thus, by Proposition \ref{prpstn:3Primes}, this makes a negligible contribution to \eqref{eq:Target}.

We now consider the fourth term of \eqref{eq:Decomp1}. By Buchstab's identity again, we have
\begin{equation}
\sum_{\substack{z_1< r\le p\le z_2\\ z_3<p r\\  r\le x^{1/4}}}S_{p r}(r)=\sum_{\substack{z_1< r\le p\le z_2\\ z_3<p r\\  r\le x^{1/4}}}S_{p r}(z_1)-\sum_{\substack{z_1< s\le r\le p\le z_2\\ z_3<p r\\  r\le x^{1/4}}}S_{p r s}(s).
\label{eq:Decomp3}
\end{equation}
In the first term above we note that since $r\le x^{1/4}$ and $p\le z_2=x^{3/7+\epsilon}$ we have $p r< x^{19/28+\epsilon}<x^{2/3}$. Thus the first term of \eqref{eq:Decomp3} contributes negligibly to \eqref{eq:Target} by Proposition \ref{prpstn:SieveAsymptotic}. In the final term of \eqref{eq:Decomp3}, we see that since $p r\ge z_3$ and $s\ge z_1$, we have that $s\ge (2x/p r s)^{1/2}$, and so this counts products of exactly four primes. Thus this term also contributes negligibly to \eqref{eq:Target} by Proposition \ref{prpstn:4Primes}.

Finally, we consider the second term of \eqref{eq:Decomp1}. By Buchstab's identity, we have
\begin{align}
\sum_{\substack{z_1< r\le p\le z_2\\ p r\le z_2}}S_{p r}( r)&=\sum_{\substack{z_1< r\le p\le z_2\\ p r\le z_2}}S_{p r}(z_1)-\sum_{\substack{z_1< s\le r\le p\le z_2\\ p r\le z_2}}S_{p r s}(s).\label{eq:Decomp2}
\end{align}
The first term of \eqref{eq:Decomp2} makes a negligible contribution to \eqref{eq:Target} by Proposition \ref{prpstn:SieveAsymptotic}. For the final term of \eqref{eq:Decomp2} we apply Buchstab's identity once more,
 giving
\begin{align}
\sum_{\substack{z_1< s\le r\le p\le z_2\\ p r\le z_2}}S_{p r s}(s)&=\sum_{\substack{z_1< s\le r\le p\le z_2\\ p r\le z_2}}S_{p r s}(z_1)-\sum_{\substack{z_1< t\le s\le r\le p\le z_2\\ p r\le z_2}}S_{p r s t}(t)
\label{eq:Decomp4}
\end{align}
In the first term of \eqref{eq:Decomp4}, we have $p r s\le (p r)^{3/2}\le z_2^{3/2}\le x^{2/3}$ and $s\le x^{1/4}$. Therefore by Proposition \ref{prpstn:SieveAsymptotic} this term makes a negligible contribution to \eqref{eq:Target}. We see that the second term of \eqref{eq:Decomp4} counts products of exactly 5 or 6 primes. Specifically
\begin{equation}
\sum_{\substack{z_1< t\le s\le r\le p\le z_2\\ p r\le z_2}}S_{p r s t}(t)=\sum_{\substack{z_1< t\le s\le r\le p\le z_2\\ p r\le z_2\\ p r s t u\sim x}}S_{p r s t u}+\sum_{\substack{z_1< t\le s\le r\le p\le z_2\\ p r\le z_2\\ t< u\le v\\ p r s t u v\sim x}}S_{p r s t u v}.
\label{eq:Decomp6}
\end{equation}
Thus Proposition \ref{prpstn:4Primes} shows that both of the terms  in \eqref{eq:Decomp6} contribute negligibly to \eqref{eq:Target}. Thus all the terms in \eqref{eq:Decomp2} contribute negligibly to \eqref{eq:Target}.

Putting everything together, we find that
\[
\sum_{q_1\sim Q_1}\sum_{\substack{q_2\sim Q_2\\ (q_1q_2,a)=1}}|S_1(2x^{1/2})|\ll_A\frac{x}{(\log{x})^A},
\]
as required.
\end{proof}
%
%
%
%
%
%
%
%
To complete the proof of Theorem \ref{thrm:MainTheorem}, we are left to establish Propositions \ref{prpstn:TypeII2}, \ref{prpstn:SieveAsymptotic}, \ref{prpstn:4Primes} and  \ref{prpstn:3Primes}.
%
%
%
%
%
%
%
%
%
%
%
%
%
%
%
%
%
%
\section{Type II estimates}\label{sec:TypeII}
In this section we establish our Type II estimate Proposition \ref{prpstn:TypeII2} via a slight variant, namely Proposition \ref{prpstn:TypeII}, given below. We do this assuming two technical propositions, Propositions \ref{prpstn:Zhang} and Proposition \ref{prpstn:TripleRough} along with the preexisting results given by Lemmas \ref{lmm:BFITripleRough} \ref{lmm:TripleSmooth}.
%
%
%
%
%
%
%
%
\begin{prpstn}[Type II estimate]\label{prpstn:TypeII}
Let $A>0$ and let $Q_1,Q_2 $ satisfy \eqref{eq:Cons1} and \eqref{eq:Cons2}. Let $P_1,\dots, P_J\ge x^{1/7+10\epsilon}$ be such that $P_1\cdots P_J\asymp x$ and
\[
x^{3/7+\epsilon}\le\prod_{j\in\mathcal{J}}P_j\le x^{4/7-\epsilon}
\]
for some subset $\mathcal{J}\subseteq\{1,\dots,J\}$. Then we have 
\[
\sum_{q_1\sim Q_1}\sum_{\substack{q_2\sim Q_2\\ (q_1q_2,a)=1}}\Bigl|\mathop{\sideset{}{^*}\sum}_{\substack{p_1,\dots,p_J\\ p_i\sim P_i\forall i}}S_{p_1\cdots p_J}\Bigr|\ll_{A}\frac{x}{(\log{x})^A}.
\]
Here $\sum^*$ indicates that the summation is restricted by $O(1)$ inequalities of the form $p_1^{\alpha_1}\cdots p_J^{\alpha_J}\le B$. The implied constant may depend on all such exponents $\alpha_i$, but none of the quantities $B$.
\end{prpstn}
%
%
%
%
%
%
%
%
We recall that here (and throughout the paper) $S_n$ is defined by
\[
S_n:=\mathbf{1}_{n\equiv a\Mod{q_1q_2}}-\frac{1}{\phi(q_1q_2)}\mathbf{1}_{(n,q_1q_2)=1}.
\]
We begin by noting that Proposition \ref{prpstn:TypeII2} follows immediately from Proposition \ref{prpstn:TypeII}.
%
%
%
%
%
%
%
%
\begin{proof}[Proof of Proposition \ref{prpstn:TypeII2} assuming Proposition \ref{prpstn:TypeII}]
We note that $S_{p_1\cdots p_J}(p_J)$ is a weighted sum over integers $p_1\cdots p_J n\sim x$ with $P^-(p_1\dots p_J n)\ge x^{1/7+10\epsilon}$, and so with at most 6 prime factors. Expanding this into separate terms according to the exact number of prime factors, we see that each term is of the form considered by Proposition \ref{prpstn:TypeII}, and so the result follows immediately. 
\end{proof}
%
%
%
%
%
%
%
%
Thus we wish to establish Proposition \ref{prpstn:TypeII}. Our first key estimate is essentially an estimate in the work of Zhang \cite{Zhang}, and will be established in Section \ref{sec:Zhang}.
%
%
%
%
%
%
%
%
\begin{prpstn}[Zhang-style estimate]\label{prpstn:Zhang}
Let $A>0$. Let $N,M,Q_1,Q_2\ge 1$ with $NM\asymp x$ be such that
\begin{align*}
Q_1^{12}Q_2^{7}&<x^{4-19\epsilon},\qquad
x^\epsilon Q_2<N<\frac{x^{1-6\epsilon}}{Q_2}.
\end{align*}
Let $\beta_m,\alpha_n$ be complex sequences such that $|\alpha_n|,|\beta_n|\le \tau(n)^{B_0}$ and such that $\alpha_{n}$ satisfies the Siegel-Walfisz condition \eqref{eq:SiegelWalfisz} and $\alpha_n$ is supported on $n$ with all prime factors bigger than $z_0=x^{1/(\log\log{x})^3}$. Let
\[
\Delta(q):=\sum_{m\sim M}\sum_{\substack{n\sim N}}\alpha_n\beta_m\Bigl(\mathbf{1}_{m n\equiv a \Mod{q}}-\frac{\mathbf{1}_{(m n,q)=1}}{\phi(q)}\Bigr).
\]
Then we have
\[
\mathop{\sum_{q_1\sim Q_1}\sum_{q_2\sim Q_2}}\limits_{(q_1q_2,a)=1}|\Delta(q_1q_2)|\ll_{A,B_0} \frac{x}{\log^A{x}}.
\]
\end{prpstn}
%
%
%
%
%
%
%
%
Our next key proposition is a refinement of work of Bombieri, Friedlander and Iwaniec \cite[Theorem 4]{BFI2}, and will be established in Section \ref{sec:BFI}.
%
%
%
%
%
%
%
%
\begin{prpstn}[New estimate for triple convolutions]\label{prpstn:TripleRough}
Let $A,B_0>0$, $K L M\asymp x$, $\min(K,L,M)>x^\epsilon$, $a\ne 0$ and $x^{7/10-\epsilon}>Q>x^{1/2}(\log{x})^{-A}$. Let $L,K$ satisfy
\begin{align*}
Q x^\epsilon&<KL,\\
Q K&<x^{1-2\epsilon},\\
KL &< \frac{x^{153/224-10\epsilon}}{Q^{1/7}},\\
K L^{4}&<\frac{x^{57/32-10\epsilon}}{Q}.
\end{align*}
Let $\eta_k,\lambda_\ell,\beta_m$ be complex sequences such that $|\eta_n|,|\lambda_n|,|\beta_n|\le \tau(n)^{B_0}$ and such that $\eta_k$ satisfies the Siegel-Walfisz condition \eqref{eq:SiegelWalfisz}, and such that $\eta_k,\lambda_\ell$ be supported on integers with all prime factors bigger than $z_0$. Let 
\[
\Delta_{\mathscr{B}}(q):=\sum_{k\sim K}\sum_{\ell\sim L}\sum_{m\sim M}\eta_k\lambda_\ell\alpha_m\Bigl(\mathbf{1}_{k \ell m\equiv a\Mod{q}}-\frac{\mathbf{1}_{(k \ell m,q)=1}}{\phi(q)}\Bigr).
\]
Then we have
\[
\sum_{\substack{q\sim Q\\ (q,a)=1}}|\Delta_{\mathscr{B}}(q)|\ll_{A,B_0} \frac{x}{(\log{x})^A}.
\]
\end{prpstn}
%
%
%
%
%
%
%
%
We also require two results of Bombieri, Friedlander and Iwaniec.
%
%
%
%
%
%
%
%
\begin{lmm}[Triple convolution of rough sequences]\label{lmm:BFITripleRough}
Let $A,B_0>0$, $K L M\asymp x$, $\min(K,L,M)>x^\epsilon$, $a\ne 0$ and $x^{2/3-\epsilon}>Q>x^{1/2}(\log{x})^{-A}$. Let $L,K$ satisfy
\begin{align*}
Q x^\epsilon&<KL,\qquad
K^3 L^2<x^{3/2-\epsilon},\qquad
K^3 L^4< x^{2-2\epsilon},\qquad
K^2 L^5 <x^{2-2\epsilon}.
\end{align*}
Let $\eta_k,\lambda_\ell,\beta_m$ and $\Delta_{\mathscr{B}}(q)$ be as in Proposition \ref{prpstn:TripleRough}. Then we have
\[
\sum_{\substack{q\sim Q\\ (q,a)=1}}|\Delta_{\mathscr{B}}(q)|\ll_{A,B_0} \frac{x}{(\log{x})^A}.
\]
\end{lmm}
\begin{proof}
This is \cite[Theorem 3]{BFI2} (with the roles of $K$ and $L$ reversed).
\end{proof}
%
%
%
%
%
%
%
%
\begin{lmm}[Triple convolution involving a smooth sequence]\label{lmm:TripleSmooth}
Let $A,B_0>0$,  $ K L M\asymp x$ with $\min(K,L,M)>x^{\epsilon}$ and let $\beta_m,\eta_k$ be complex sequences with $|\beta_n|,|\eta_n|\le \tau(n)^{B_0}$. Let $\mathcal{I}_L\subseteq [L,2L]$ be an interval. Let $L,K,M$ satisfy
\begin{align*}
Q x^\epsilon&<KL,\qquad
K^3Q<L x^{1-\epsilon},\qquad
KQ^2<L x^{1-\epsilon}.
\end{align*}
Then we have
\[
\sum_{\substack{q\sim Q\\ (q,a)=1}}\Bigl|\mathop{\sum_{k\sim K}\sum_{m\sim M}\sum_{\substack{\ell\in\mathcal{I}_L }}}\limits_{P^-(\ell)\ge z_0}\eta_k\beta_m\Bigl(\mathbf{1}_{k m \ell\equiv a\Mod{q} }-\frac{\mathbf{1}_{(k \ell m,q)=1}}{\phi(q)}\Bigr)\Bigr|\ll_{A,B_0} \frac{x}{(\log{x})^A}.
\]
\end{lmm}
\begin{proof}
This is \cite[Theorem 5*]{BFI2}.
\end{proof}
%
%
%
%
%
%
%
%
Combining Proposition \ref{prpstn:TripleRough} with Lemma \ref{lmm:BFITripleRough} gives the following.
%
%
%
%
%
%
%
%
\begin{lmm}\label{lmm:TripleCombined}
Let $A,B_0>0$, $K L M\asymp x$, $\min(K,L,M)>x^\epsilon$, $a\ne 0$ and $x^{1/2-\epsilon/2}<Q<x^{127/224-10\epsilon}$. Let $L,K$ satisfy
\begin{align*}
Q x^\epsilon&\le KL\le x^{4/7-\epsilon},\qquad
L\le K\le (L K)^{3/4}.
\end{align*}
Let $\eta_k,\lambda_\ell,\beta_m$ and $\Delta_{\mathscr{B}}(q)$ be as in Proposition \ref{prpstn:TripleRough}. Then we have
\[
\sum_{\substack{q\sim Q\\ (q,a)=1}}|\Delta_{\mathscr{B}}(q)|\ll_{A,B_0} \frac{x}{(\log{x})^A}.
\]
\end{lmm}
%
%
%
%
%
%
%
%
\begin{proof}[Proof of Lemma \ref{lmm:TripleCombined} assuming Proposition \ref{prpstn:TripleRough}]
We consider separately the cases when $K\in [L,(L K)^{5/8}]$ and $K\in [(L K)^{5/8},(L K)^{3/4}]$. If $L\le K\le (L K)^{5/8}$ we see that
\begin{align*}
K^3L^4+K^2L^5&\ll K^3L^4\le (KL)^{7/2}\le x^{2-3\epsilon},\\
K^3L^2 &\le (KL)^{2+5/8}\le x^{(4/7-\epsilon)\cdot 21/8}< x^{3/2-2\epsilon}.
\end{align*}
Thus Lemma \ref{lmm:BFITripleRough} gives the result in this case. If instead $(L K)^{5/8}\le K\le (L K)^{3/4}$ then $L\le (L K)^{3/8}$ so
\begin{align*}
Q K&\le x^{4/7-2\epsilon} (LK)^{3/4}<x^{1-2\epsilon},\\
KLQ^{1/7}&<(x^{4/7})^{8/7}<x^{153/224-10\epsilon},\\
K L^{4} Q&<(KL)^{17/8}x^{127/224-10\epsilon}<(x^{4/7-\epsilon})^{17/8}x^{127/224-10\epsilon} <x^{57/32-10\epsilon}.
\end{align*}
Thus Proposition \ref{prpstn:TripleRough} applies, giving the result in this case. This gives the result.
\end{proof}
%
%
%
%
%
%
%
%
Before we establish Proposition \ref{prpstn:TypeII}, we require some preparatory lemmas.
%
%
%
%
%
%
%
%
\begin{lmm}[Divisor function bounds]\label{lmm:Divisor}
Let $|b|< x-y$ and $y\ge q x^\epsilon$. Then we have
\[
\sum_{\substack{x-y\le n\le x\\ n\equiv a\Mod{q}}}\tau(n)^C\tau(n-b)^C\ll \frac{y}{q} (\tau(q)\log{x})^{O_{C}(1)}.
\]
\end{lmm}
\begin{proof}
This follows from Shiu's Theorem \cite{Shiu}, and is given in \cite[Lemma 12]{BFI2}.
\end{proof}
%
%
%
%
%
%
%
%
\begin{lmm}[Heath-Brown identity]\label{lmm:HeathBrown}
Let $k\ge 1$ and $n\le 2x$. Then we have
\[
\Lambda(n)=\sum_{j=1}^k (-1)^j \binom{k}{j}\sum_{\substack{n=m_1\cdots m_j n_1\cdots n_{j}\\ m_1,\,\dots,\,m_j\le 2x^{1/k}}}\mu(m_1)\cdots \mu(m_j)\log{n_{1}}.
\]
\end{lmm}
\begin{proof}
See \cite{HBVaughan}.
\end{proof}
%
%
%
%
%
%
%
%
\begin{lmm}[Small sets contribute negligibly]\label{lmm:SmallSets}
Let $\delta>0$, $Q\le x^{1-\epsilon}$ and let $\mathcal{A}\subseteq [x,2x]$. Then we have
\[
\sum_{q\sim Q}\tau(q)\Bigl|\sum_{\substack{n\in \mathcal{A}\\ n\equiv a\Mod{q}}}1-\frac{1}{\phi(q)}\sum_{\substack{n\in\mathcal{A}\\ (n,q)=1}}1\Bigr|\ll x^\delta\#\mathcal{A}^{1-\delta}(\log{x})^{O_\delta(1)}.
\]
\end{lmm}
\begin{proof}
Trivially, we have that the sum is bounded by
\begin{align*}
\sum_{q\sim Q}\tau(q)\Bigl(\sum_{\substack{n\in \mathcal{A}\\ n\equiv a\Mod{q}}}1&+\frac{1}{\phi(q)}\sum_{n\in\mathcal{A}}1\Bigr)
\ll \sum_{n\in\mathcal{A}}\tau(n-a)^2+\#\mathcal{A}\sum_{q\sim Q}\frac{\tau(q)}{\phi(q)}\\
&\ll \#\mathcal{A}^{1-\delta}\Bigl(\sum_{n\sim x}\tau(n-a)^{2/\delta}\Bigr)^{\delta}+\#\mathcal{A}(\log{x})^2\\
&\ll x^\delta \#\mathcal{A}^{1-\delta}(\log{x})^{O_\delta(1)}+\#\mathcal{A}(\log{x})^2.
\end{align*}
Here we applied Holder's inequality in the second line, and Lemma \ref{lmm:Divisor} in the final line. Since $\#\mathcal{A}\ll x$  this gives the result.
\end{proof}
%
%
%
%
%
%
%
%
\begin{lmm}[Separation of variables from inequalities]\label{lmm:Separation}
Let $Q_1Q_2\le x^{1-\epsilon}$. Let $N_1,\dots, N_r\ge z_0$ satisfy $N_1\cdots N_r\asymp x$. Let $\alpha_{n_1,\dots,n_r}$ be a complex sequence with $|\alpha_{n_1,\dots,n_r}|\le (\tau(n_1)\cdots \tau(n_r))^{B_0}$. Then, for any choice of $A>0$ there is a constant $C=C(A,B_0,r)$ and intervals $\mathcal{I}_1,\dots,\mathcal{I}_r$ with $\mathcal{I}_j\subseteq [P_j,2P_j]$ of length $\le P_j(\log{x})^{-C}$ such that
\begin{align*}
\sum_{q_1\sim Q_1}\sum_{\substack{q_2\sim Q_2\\ (q_1q_2,a)=1}}&\Bigl|\mathop{\sideset{}{^*}\sum}_{\substack{n_1,\dots,n_r\\ n_i\sim N_i\forall i}}\alpha_{n_1,\dots,n_r}S_{n_1\cdots n_r}\Bigr|\\
&\ll_r \frac{x}{(\log{x})^A}+(\log{x})^{r C}\sum_{q_1\sim Q_1}\sum_{\substack{q_2\sim Q_2\\ (q_1q_2,a)=1}}\Bigl|\sum_{\substack{n_1,\dots,n_r\\ n_i\in \mathcal{I}_i\forall i}}\alpha_{n_1,\dots,n_r}S_{n_1\cdots n_r}\Bigr|.
\end{align*}
Here $\sum^*$ means that the summation is restricted to $O(1)$ inequalities of the form $n_1^{\alpha_1}\cdots n_r^{\alpha_r}\le B$ for some constants $\alpha_1,\dots \alpha_r$ and some quantity $B$.  The implied constant may depend on all such exponents $\alpha_i$, but none of the quantities $B$.
\end{lmm}
\begin{proof}
Clearly we may assume that $A$ is sufficiently large. We let implied constants in the proof of this lemma depend on $r$. By Lemma \ref{lmm:Divisor}, the set of $n_1,\dots,n_r$ with $n_i\sim N_i$ and $(\tau(n_1)\cdots \tau(n_r))^{B_0}>(\log{x})^{C_1}$ has size $\ll (\prod_i N_i)(\log{x})^{O_{B_1}(1)-C_1}$, and so by Lemma \ref{lmm:SmallSets}, if $C_1$ is chosen sufficiently large in terms of $A,B_0,r$ these terms contribute negligibly. After rescaling (and replacing $A$ with $A+C_1$), we see it suffices to consider the case when $\alpha_{n_1,\dots,n_k}$ is 1-bounded.

We split the summation $n_i\sim N_i$ into $\lceil (\log{x})^C\rceil $ disjoint short intervals $n_i\in \mathcal{I}_{i,j}$ where
\[
\mathcal{I}_{i,j}:=\Bigl[N_i+\frac{(j-1)N_i }{\lceil (\log{x})^{C}\rceil},N_i+\frac{j N_i }{\lceil (\log{x})^{C}\rceil}\Bigr)
\]
 for $1\le j\le \lceil(\log{x})^C\rceil$. We see that $\mathcal{I}_{i,j}=[N_{i,j}, N_{i,j}+N_i/\lceil (\log{x})^C\rceil)$ for suitable $N_{i,j}\sim N_i$. We do this for each $1\le i\le r$, and so in total there are $O(\log^{Cr}{x})$ choices of these intervals. 
 
 We clearly do not need to consider any tuple $(\mathcal{I}_{1,j_1},\dots,\mathcal{I}_{r,j_r})$ of intervals for which the inequalities implied by $\sum^*$ hold for no choice of $n_1,\dots,n_r$ with $n_i\in\mathcal{I}_{i,j_i}$. 
 
If the inequality  $n_1^{\alpha_1}\cdots n_r^{\alpha_r}\le B$ can only hold for some but not all elements $(n_1,\dots,n_r)\in (\mathcal{I}_{1,j_1},\dots,\mathcal{I}_{r,j_r})$ then we must have 
\[
N_{1,j_1}^{\alpha_1}\cdots N_{r,j_r}^{\alpha_r}=B\Bigl(1+O_{\alpha_1,\dots,\alpha_r}\Bigl(\frac{1}{(\log{x})^{C}}\Bigr)\Bigr).
\]
If $\alpha_k\ne 0$, then, given a choice of $j_i$ for $i\ne k$, there are $O_\alpha(1)$ choices of $j_k$  such that the above holds. Thus there are $O_\alpha( (\log{x})^{C(r-1)})$ choices of tuples of intervals for which the inequalities hold for some but not all elements of the intervals. By Lemma \ref{lmm:SmallSets}, the total contribution from these $O_\alpha(\log^{C(r-1)}{x})$ tuples of  intervals where the inequalities sometimes hold  is $\ll_\alpha x (\log{x})^{O_\delta(1)-C(1-r\delta)}\ll_\alpha x(\log{x})^{-A}$ if we choose $\delta=1/(2r)$ and $C=C(A,r)$ large enough. 

Thus we only need to consider tuples $(\mathcal{I}_{1,j_1},\dots,\mathcal{I}_{r,j_r})$ of intervals where the inequalities hold for all choices of elements in the intervals. In particular, the inequalities can be dropped, and taking the worst tuple of intervals then gives the result.
\end{proof}
%
%
%
%
%
%
%
%
\begin{lmm}[Type II estimate away from 1/2 for convolutions]\label{lmm:TypeIIAway}
Let $A>0$. Let $x^{1/2-\epsilon}\le Q\le x^{127/224-\epsilon}$, and let $M_1, \dots , M_r\ge 1$ be such that $\prod_{i=1}^r M_i\asymp x$ and
\[
x^{\epsilon}Q <\prod_{j\in\mathcal{J}} M_j<x^{4/7-\epsilon}
\]
for some set $\mathcal{J}\subseteq\{1,\dots,r\}$.  Let $\mathcal{I}_j\subseteq [M_j,2M_j]$ be intervals and $\beta^{(1)}_m,\dots,\beta_m^{(r)}$ be 1-bounded complex sequences satisfying the Siegel-Walfisz condition \eqref{eq:SiegelWalfisz}, supported on $m$ with $P^-(m)>z_0$, and such that $\beta_m^{(j)}=\mathbf{1}_{P^-(m)>z_0}$ for all $m>x^{1/15}$.

Then we have 
 \[
  \sum_{\substack{q\sim Q\\ (q,a)=1}}\Bigl|\sum_{\substack{m_{1},\dots, m_{r}\\ m_{i}\in \mathcal{I}_i\forall i}}\Bigl(\prod_{1\le j\le r}\beta^{(j)}_{m_j}\Bigr)\Bigl(\mathbf{1}_{m_1\cdots m_r\equiv a\Mod{q}}-\frac{\mathbf{1}_{(m_1\cdots m_r,q)=1}}{\phi(q)}\Bigr)\Bigr|\ll_{A,r}\frac{x}{(\log{x})^{A}}.
 \]
\end{lmm}
\begin{proof}[Proof of Lemma \ref{lmm:TypeIIAway} assuming Proposition \ref{prpstn:TripleRough}]
After reordering the indices, we may assume that $\mathcal{J}=\{1,\dots,k\}$ with $M_1\ge \dots \ge M_k$. Define $N^*$, $N$ and $M$ by
\begin{align*}
N^*:=x^{1/14}Q^{1/2},\qquad
N:=\prod_{j=1}^k M_j,\qquad 
M:=\prod_{j=k+1}^r M_j.
\end{align*}
Since $Q<x^{4/7-\epsilon}$ we have $N^*<x^{5/14}$, and since $\prod_{j=1}^r M_j\asymp x$ we have $N M\asymp x$.

If $M_1\ge N^*$ then $\beta^{(1)}_{n_1}=\mathbf{1}_{P^-(n_1)>z_0}$, and we wish to apply Lemma \ref{lmm:TripleSmooth} with $L=M_1$, $K\asymp N/M_1$, $\gamma_\ell:=\mathbf{1}_{\ell\in \mathcal{I}_1}$ and
\[
\eta_k:=\sum_{\substack{k=m_2\cdots m_{k}\\ m_i\in\mathcal{I}_i\forall 2\le i\le k}}\beta^{(2)}_{m_2}\cdots \beta^{(k)}_{m_{k}},\qquad \beta_m:=\sum_{\substack{m=m_{k+1}\cdots m_r\\ m_i\in \mathcal{I}_i\forall k+1\le i\le r}}\beta^{(k+1)}_{m_{k+1}}\cdots \beta^{(r)}_{m_r}.
\]
Since $M_1\ge N^*=x^{1/14}Q^{1/2}$ and $x^\epsilon Q< N\ll x^{4/7-\epsilon}$ and $Q\in [x^{1/2-\epsilon/2},x^{4/7-\epsilon}]$, we have
\begin{align*}
KL&\gg N> x^\epsilon Q,\\
K^3Q&\ll\frac{N^3Q L}{M_1^4}<\frac{x^{12/7} L}{x^{2/7}Q}<\frac{x^{10/7} L}{Q}<L x^{1-2\epsilon},\\
KQ^2&\ll \frac{N Q^2 L}{M_1^2}<\frac{x^{4/7-\epsilon} L Q}{x^{1/7}}<x^{1-2\epsilon}L,
\end{align*}
Thus Lemma \ref{lmm:TripleSmooth} applies, and this gives the result for $M_1 \ge N^*$.

We now consider the case when $M_1<N^*$. We claim we can find a product of some of $M_1,\dots,M_k$ which lies in the interval $[N^{1/4},N^{1/2}]$. Since $N^*<N^{3/4-\epsilon}$, if $M_1\ge N^{1/4}$ then either $M_1$ or $M_2\cdots M_k\asymp N/M_1$ will lie in this interval.
If instead $M_1<N^{1/4}$ then all of $M_1,\dots ,M_k$ are less than $N^{1/4}$, so choosing $j$ minimally such that $M_1\cdots M_j>N^{1/4}$ must give a suitable product (such a $j$ exists since $M_1\cdots M_k= N$). Thus in either case there is a product which lies in $[N^{1/4},N^{1/2}]$. After relabeling the indices $\{1,\dots,k\}$ (noting that this removes the property $M_1\ge \dots\ge M_k$) we may assume that $N^{1/4}\le M_1\cdots M_J\le N^{1/2}\le M_{J+1}\cdots M_k\le N^{3/4}$. We now consider $L\asymp M_1\cdots M_J$ and $K\asymp M_{J+1}\cdots M_k$ and
\begin{align*}
\alpha_m&:=\sum_{\substack{m=m_{k+1}\cdots m_r\\ m_i\in \mathcal{I}_i\forall k+1\le i\le r}}\beta^{(k+1)}_{m_{k+1}}\cdots \beta^{(r)}_{m_r},\qquad \lambda_\ell:=\sum_{\substack{\ell=m_1\cdots m_J\\ m_i\in\mathcal{I}_i\forall 1\le i \le J}}\beta^{(1)}_{m_1}\cdots \beta^{(J)}_{m_J},\\
\eta_k&:=\sum_{\substack{k=m_{J+1}\cdots m_k\\ m_i\in\mathcal{I}_i\forall J+1\le i \le k}}\beta^{(J+1)}_{m_{J+1}}\cdots \beta^{(k)}_{m_k}.
\end{align*}
We then see that since $L\le N^{1/2}\le K\le N^{3/4}$ and $L K\asymp N\ll x^{4/7-\epsilon}$ and $KL\gg N>x^{\epsilon} Q$, so Lemma \ref{lmm:TripleCombined} applies, giving the result for $M_1<N^*$. This finishes the proof.
\end{proof}
%
%
%
%
%
%
%
%
\begin{lmm}[Type II estimate away from $x^{1/2}$]\label{lmm:BasicTypeII}
Let $A>0$ and $Q_1Q_2\le x^{127/224-\epsilon}$, and let $P_1,\dots, P_J\ge x^{1/7+10\epsilon}$ be such that $P_1\cdots P_J\asymp x$ and
\[
x^{\epsilon}Q_1Q_2 <\prod_{j\in\mathcal{J}}P_j<x^{4/7-\epsilon}
\]
for some subset $\mathcal{J}\subseteq\{1,\dots,J\}$. 

Then we have
\[
\sum_{q_1\sim Q_1}\sum_{\substack{q_2\sim Q_2\\ (q_1q_2,a)=1}}\Bigl|\mathop{\sideset{}{^*}\sum}_{\substack{p_1,\dots,p_J\\ p_i\sim P_i\forall i}}S_{p_1\cdots p_J}\Bigr|\ll_{A}\frac{x}{(\log{x})^A}.
\]
Here $\sum^*$ indicates that the summation is restricted by $O(1)$ inequalities of the form $p_1^{\alpha_1}\cdots p_J^{\alpha_J}\le B$. The implied constant may depend on all such exponents $\alpha_i$, but none of the quantities $B$.
\end{lmm}
\begin{proof}[Proof of Lemma \ref{lmm:BasicTypeII} assuming Proposition \ref{prpstn:TripleRough}]
By Lemma \ref{lmm:Separation}, it suffices to show for $B=B(A)$ sufficiently large in terms of $A$
\[
\sum_{q_1\sim Q_1}\sum_{\substack{q_2\sim Q_2\\ (q_1q_2,a)=1}}\Bigl|\sum_{\substack{p_1,\dots,p_J\\ p_i\in \mathcal{I}_i\forall i}}S_{p_1\cdots p_J}\Bigr|\ll_{B}\frac{x}{(\log{x})^{B}}
\]
for all choices of intervals $\mathcal{I}_1,\dots,\mathcal{I}_J$ with $\mathcal{I}_j\subset[P_j,2P_j]$. By splitting these intervals into sub-intervals of length $P_j(\log{x})^{-2B}$, and taking the worst such subintervals it suffices to show for every choice of $\mathcal{I}_1,\dots, \mathcal{I}_J$ with $\mathcal{I}_j\subseteq[P_j,2P_j]$ of length $P_j(\log{x})^{-2B}$ that
\[
\sum_{q_1\sim Q_1}\sum_{\substack{q_2\sim Q_2\\ (q_1q_2,a)=1}}\Bigl|\sum_{\substack{p_1,\dots,p_J\\ p_i\in \mathcal{I}_i\forall i}}S_{p_1\cdots p_J}\Bigr|\ll_{B}\frac{x}{(\log{x})^{(2J+1)B}}.
\]
For $p_j\in [N_j,N_j+N_j(\log{x})^{-2B})$ we see that $\Lambda(p_j)=\log{N_j}+O((\log{x})^{-2B})$. Since $J\le 6$ and $B$ is sufficiently large, Lemma \ref{lmm:SmallSets} shows that the error term contributes
\[
\ll (\log{x})^{-2B}\frac{x}{(\log{x})^{2 J B(1-\epsilon)-O(1)}}\ll \frac{x}{(\log{x})^{(2J+1)B}},
\]
which is acceptable. Similarly, by Lemma \ref{lmm:SmallSets}, higher prime-powers contribute negligibly, and so we may replace the sum over primes with a sum of $\mathbf{1}_{P^-(n)>z_0}\Lambda(n)/\log{N_j}$ over integers $n\in\mathcal{I}_j$. Thus it suffices to show that
 \[
\sum_{q_1\sim Q_1}\sum_{\substack{q_2\sim Q_2\\ (q_1q_2,a)=1}}\Bigl|\sum_{\substack{n_1,\dots,n_J\\ n_i\in \mathcal{I}_i\forall i\\ P^-(n_i)>z_0\forall i}}\Lambda(n_1)\cdots \Lambda(n_J) S_{n_1\cdots n_j}\Bigr|\ll_{B}\frac{x}{(\log{x})^{(2J+1)B}}.
 \]
 We now apply Lemma \ref{lmm:HeathBrown} with $k=20$ to expand $\Lambda(n)$ into a sum over new variables, and we put the new variables into dyadic intervals $[M_{i,j},2M_{i,j})$.  Thus it suffices to show that for a suitable constant $C$ sufficiently large in terms of $B$ 
 \[
\sum_{q_1\sim Q_1}\sum_{\substack{q_2\sim Q_2\\ (q_1q_2,a)=1}}\Bigl|\sum_{\substack{m_{1,1},\dots, m_{J,40}\\ m_{i,j}\sim M_{i,j}\,\forall i,j\\ \prod_{j=1}^{40} m_{i,j}\in \mathcal{I}_i\forall i\\ P^-(m_{i,j})>z_0\forall i,j}}\Bigl(\prod_{\substack{1\le i\le J\\ 1\le j\le 40}}\alpha_{m_{i,j}}^{(j)}\Bigr) S_{m_{1,1}\cdots m_{J,40}}\Bigr|\ll_{C}\frac{x}{(\log{x})^{C}}.
 \]
 where $M_{i,j}\le x^{1/20}$ for $j\le 20$ and
 \[
 \alpha^{(j)}_m=\begin{cases}
 \mu(m),\qquad &\text{if }j\le 20,\\
 \log{m},&\text{if }j=21,\\
 1,&\text{if }22\le j\le 40.
 \end{cases}
 \]
By applying Lemma \ref{lmm:Separation} again, we can remove the conditions $\prod_{j=1}^{40}m_{i,j}\in\mathcal{I}_i$. Thus we see it suffices to show that for every choice of $M_{i,j}$ with $M_{i,j}\ll x^{1/20}$ if $j\le 20$ and $\prod_{j=1}^{40} M_{i,j}\asymp P_i$ and for every choice of intervals $\mathcal{I}'_{i,j}\subseteq[M_{i,j},2M_{i,j}]$ and every $C>0$ we have
 \[
 \sum_{q_1\sim Q_1}\sum_{\substack{q_2\sim Q_2\\ (q_1q_2,a)=1}}\Bigl|\sum_{\substack{m_{1,1},\dots, m_{J,40}\\ m_{i,j}\in  \mathcal{I}'_{i,j}\,\forall i,j\\ P^-(m_{i,j})>z_0\forall i,j }}\Bigl(\prod_{\substack{1\le i\le J\\ 1\le j\le 40}}\alpha_{m_{i,j}}^{(j)}\Bigr) S_{m_{1,1}\cdots m_{J,40}}\Bigr|\ll_{C}\frac{x}{(\log{x})^{C}}.
  \]
  By splitting $\mathcal{I}'_{i,21}$ into shorter sub-intervals, we may replace $\log{m_{i,21}}$ by a constant depending only on the subinterval in exactly the same way we saw that $\Lambda(p_i)$ was essentially constant on these short intervals. Therefore, it suffices to show for every $C>0$
   \[
\sum_{q_1\sim Q_1}\sum_{\substack{q_2\sim Q_2\\ (q_1q_2,a)=1}}\Bigl|\sum_{\substack{m_{1,1},\dots, m_{J,40}\\ m_{i,j}\in  \mathcal{I}'_{i,j}\,\forall i,j }}\Bigl(\prod_{\substack{1\le i\le J\\ 1\le j\le 40}}\tilde{\alpha}_{m_{i,j}}^{(j)}\Bigr) S_{m_{J,40}\cdots m_{1,k}}\Bigr|\ll_{C}\frac{x}{(\log{x})^{C}},
  \]
  where $M_{i,j}\le x^{1/20}$ for $j\le 20$ and
  \[
  \tilde{\alpha}^{(j)}_{m}=\begin{cases}
 \mathbf{1}_{P^-(m)>z_0}\mu(m),\qquad &\text{if }j\le 20,\\
 \mathbf{1}_{P^-(m)>z_0},&\text{otherwise}.
 \end{cases}
 \]
 The result now follows from Cauchy-Schwarz (to replace the sum over $q_1,q_2$ by a sum over $q\sim Q\asymp Q_1Q_2$) and Lemma \ref{lmm:TypeIIAway}.
\end{proof}
%
%
%
%
%
%
%
%
\begin{lmm}[Type II estimate near 1/2]\label{lmm:ZhangTypeII}
Let $A>0$ and let $Q_1,Q_2$ satisfy \eqref{eq:Cons1} and \eqref{eq:Cons2}. Let $P_1,\dots,P_J\ge x^{1/7+10\epsilon}$ be such that $P_1\cdots P_J\asymp x$ and
\[
\frac{x^{1-\epsilon}}{Q_1Q_2}\le \prod_{j\in \mathcal{J}}P_j\le x^\epsilon Q_1Q_2,
\]
for some subset $\mathcal{J}\subseteq\{1,\dots,J\}$.

Then we have
\[
\sum_{q_1\sim Q_1}\sum_{\substack{q_2\sim Q_2\\ (q_1q_2,a)=1}}\Bigl|\mathop{\sideset{}{^*}\sum}_{\substack{p_1,\dots,p_J\\ p_i\sim P_i\,\forall i}}S_{p_1\cdots p_J}\Bigr|\ll_A\frac{x}{(\log{x})^A}.
\]
Here $\sum^*$ indicates that the summation is restricted by $O(1)$ inequalities of the form $p_1^{\alpha_1}\cdots p_J^{\alpha_J}\le B$.  The implied constant may depend on all such exponents $\alpha_i$, but none of the quantities $B$.
\end{lmm}
\begin{proof}[Proof of Lemma \ref{lmm:ZhangTypeII} assuming Proposition \ref{prpstn:Zhang}]
This follows quickly from Proposition \ref{prpstn:Zhang}. Indeed, by Lemma \ref{lmm:Separation}, it suffices to show that
\[
\sum_{q_1\sim Q_1}\sum_{\substack{q_2\sim Q_2\\ (q_1q_2,a)=1}}\Bigl|\sum_{\substack{p_1,\dots,p_J\\ p_i\in \mathcal{I}_i\forall i}}S_{p_1\cdots p_J}\Bigr|\ll_{B}\frac{x}{(\log{x})^{B}}.
\]
for every $B>0$ and every choice of intervals $\mathcal{I}_1,\dots,\mathcal{I}_J$ with $\mathcal{I}_j\subset[P_j,2P_j]$. By reordering the indices, we may assume $\mathcal{J}=\{1,\dots,k\}$. By considering $N\asymp \prod_{j=1}^k P_j$, and $M\asymp \prod_{j=k+1}^{J}P_j$ and
\[
\alpha_n:=\sum_{\substack{n=p_1\cdots p_k\\ p_i\in \mathcal{I}_i}}1,\qquad \beta_m:=\sum_{\substack{m=p_{k+1}\cdots p_J\\ p_i\in \mathcal{I}_i}}1,
\]
we see that Proposition \ref{prpstn:Zhang} gives the result in the range
\[
Q_2x^\epsilon<N<\frac{x^{1-7\epsilon}}{Q_2}
\]
provided \eqref{eq:Cons2} holds. This covers the range $x^{1-\epsilon}/(Q_1Q_2)<N\le x^\epsilon Q_1 Q_2$ provided
\begin{align*}
Q_1Q_2^2&<x^{1-16\epsilon}.
\end{align*}
This is condition \eqref{eq:Cons1}, and so gives the result.
\end{proof}
%
%
%
%
%
%
%
%
Putting together Lemma \ref{lmm:BasicTypeII} and Lemma \ref{lmm:BasicTypeII}, we see Proposition \ref{prpstn:TypeII} follows immediately.
%
%
%
%
%
%
%
%
\begin{proof}[Proof of Proposition \ref{prpstn:TypeII} assuming Propositions \ref{prpstn:Zhang} and \ref{prpstn:TripleRough}]
This follows immediately from Lemma \ref{lmm:BasicTypeII} and Lemma \ref{lmm:ZhangTypeII}, noting that \eqref{eq:Cons1} and \eqref{eq:Cons2} imply that $Q_1Q_2=(Q_1Q_2^2)^{5/17}(Q_1^{12}Q_2^{7})^{1/17}<x^{9/17}<x^{127/224-\epsilon}$. 
\end{proof}
We are left to establish Propositions \ref{prpstn:SieveAsymptotic}, \ref{prpstn:4Primes}, \ref{prpstn:3Primes}, \ref{prpstn:Zhang} and \ref{prpstn:TripleRough}.
%
%
%
%
%
%
%
%
%
%
%
%
%
%
%
%
%
%
\section{Estimates for numbers with 4, 5 or 6 prime factors}\label{sec:4Primes}
In this section we deduce Proposition \ref{prpstn:4Primes} from Proposition \ref{prpstn:TypeII} and Proposition \ref{prpstn:TripleRough}. We do this by considering numbers with 4, 5 or 6 prime factors separately. %
We recall that $S_n$ is defined by
\[
S_n:=\mathbf{1}_{n\equiv a\Mod{q_1q_2}}-\frac{\mathbf{1}_{(n,q_1q_2)=1}}{\phi(q_1q_2)}.
\]
%
%
%
%
%
%
%
%
\begin{proof}[Proof of Proposition \ref{prpstn:4Primes} assuming Proposition \ref{prpstn:TypeII} and Proposition \ref{prpstn:TripleRough}]

By Lemma \ref{lmm:Separation} (and adjusting the constant $A$ suitably), it suffices to show for every choice of $\mathcal{I}_1,\dots,\mathcal{I}_J$ with $\mathcal{I}_j\subseteq[P_j,2P_j]$ and every choice of $A>0$ that
\[
\sum_{q_1\sim Q_1}\sum_{\substack{q_2\sim Q_2\\ (q_1q_2,a)=1}}\Bigl|\sum_{\substack{p_1,\dots,p_J\\ p_i\in \mathcal{I}_i\forall i}}S_{p_1\cdots p_J}\Bigr|\ll_{A}\frac{x}{(\log{x})^A}.
\]
We first note that from \eqref{eq:Cons1} and \eqref{eq:Cons2} we have
\[
Q_1Q_2=(Q_1Q_2^2)^{5/17}(Q_1^{12}Q_2^{7})^{1/17}<x^{9/17}.
\]
Since $P_i>x^{1/7+10\epsilon}$ and $P_1\cdots P_J\ll x$ we have that $J\le 6$. By assumption $J\ge 4$. We split the argument into 3 cases depending on whether $J=4,5$ or $6$.

\medskip

\textbf{Case 1: $J=4$.}

Since $P_i\ge x^{1/7+10\epsilon}$ for all $i$, we have that $P_2P_3P_4>x^{3/7+\epsilon}$, so the result follows from Proposition \ref{prpstn:TypeII} unless $P_2P_3P_4>x^{4/7-\epsilon}$. Thus we may assume $P_1<x^{3/7+\epsilon}$. Similarly, we have $P_1P_4>(P_1P_2P_3P_4)^{1/3}P_4^{2/3}\gg x^{1/3}x^{2/21+20\epsilon/3}>x^{3/7+\epsilon}$, so we are also done by Proposition \ref{prpstn:TypeII} unless $P_1P_4>x^{4/7-\epsilon}$.

We now wish to apply Proposition \ref{prpstn:TripleRough}. We let $K=\min(P_1,P_2P_3)$, $L=P_4$, $M\asymp x/KL\asymp\max(P_1,P_2P_3)$. If $P_1\le P_2P_3$ we choose
\[
\eta_k:=\sum_{\substack{p_1=k\\ p_1\in\mathcal{I}_1}}1,\qquad \lambda_\ell:=\sum_{\substack{p_4=\ell\\ p_4\in\mathcal{I}_4}}1,\qquad\alpha_m:=\sum_{\substack{m=p_2p_3\\ p_2\in\mathcal{I}_2\\ p_3\in\mathcal{I}_3}}1,
\]
whereas if $P_1>P_2P_3$ we swap the definitions of $\eta_k$ and $\alpha_m$.

Since $P_4\le \dots \le P_1\le x^{3/7+\epsilon}$, $P_1P_4>x^{4/7-\epsilon}$ and $Q_1Q_2<x^{9/13}<x^{4/7-5\epsilon}$, we have
\begin{align*}
K L&=\min(P_1 P_4,P_2P_3P_4)>x^{4/7-\epsilon}>x^{2\epsilon} Q_1 Q_2,\\
K Q_1Q_2&<x^{3/7+\epsilon}x^{4/7-5\epsilon}<x^{1-2\epsilon}.\\
(Q_1Q_2)^{1/7} K L&<x^{4/49} P_1^{3/5}(P_2P_3)^{2/5}P_4^{3/5} (P_2P_3)^{1/5}\ll x^{3/5+4/49}<x^{153/224-10\epsilon},\\
(Q_1Q_2) K L^{4}&<x^{9/17}P_1P_4(P_2P_3)^{3/2}\ll  x^{9/17+1+3/14}<x^{57/32-10\epsilon}.
\end{align*}
Thus Proposition \ref{prpstn:TripleRough} applies, giving the result.

\medskip

\textbf{Case 2: $J=5$.}

Since $P_i\ge x^{1/7+10\epsilon}$ for all $i$, we have that $P_3P_4P_5>x^{3/7+\epsilon}$, so the result follows from Proposition \ref{prpstn:TypeII} unless $P_3P_4P_5>x^{4/7-\epsilon}$. 

We now wish to apply Proposition \ref{prpstn:TripleRough} with $K\asymp P_4 P_5$, $L=P_3$ and $M\asymp P_1P_2$ and
\begin{align*}
\eta_k:=\sum_{\substack{k=p_4 p_5\\ p_4\in \mathcal{I}_4\\ p_5\in\mathcal{I}_5}}1,\qquad \lambda_\ell:=\sum_{\substack{p_3=\ell\\ p_3\in\mathcal{I}_3}}1,\qquad \alpha_m:=\sum_{\substack{m=p_1p_2\\ p_1\in\mathcal{I}_1\\ p_2\in\mathcal{I}_2}}1.
\end{align*}
We see that since $P_5\le P_4\le P_3\le P_2\le P_1$, $P_3 P_4 P_5>x^{4/7-\epsilon}$ and $Q_1Q_2<x^{9/17-5\epsilon}$ we have
\begin{align*}
Q_1Q_2 x^{2\epsilon}&<x^{4/7-\epsilon}<P_3 P_4 P_5 \ll  L K,\\
Q_1Q_2 K&<x^{4/7}(P_1P_2P_3P_4P_5)^{2/5}\ll x^{34/35}<x^{1-2\epsilon},\\
(Q_1Q_2)^{1/7} K L&<x^{4/49}(P_1P_2P_3P_4P_5)^{3/5}\ll x^{3/5+4/49}<x^{153/224-10\epsilon},\\
Q_1Q_2 K L^{4}&<x^{9/17}\frac{(P_1P_2P_3P_4P_5)^{4/3}}{(P_4P_5)^{1/3}}\ll x^{9/17+4/3-2/7\cdot 1/3}<x^{57/32-10\epsilon}.
\end{align*}
Thus all the conditions of Proposition \ref{prpstn:TripleRough} are satisfied, and so Proposition \ref{prpstn:TripleRough} gives the result.

\medskip

\textbf{Case 3: $J=6$.}

Since all the $P_i$ are at least $x^{1/7+10\epsilon}$ in size, we see that the product of any three of them is of size at least $x^{3/7+2\epsilon}$. But then we must have $P_1P_2P_3\in [x^{3/7+\epsilon},x^{4/7-\epsilon}]$, since $P_1P_2P_3,P_4P_5P_6>x^{3/7+2\epsilon}$ and $P_1P_2P_3\ll x/P_4P_5P_6$. Therefore we see that Proposition \ref{prpstn:TypeII} gives the result.
\end{proof}
%
%
%
%
%
%
%
%
Since we have already established Proposition \ref{prpstn:TypeII} assuming Propositions \ref{prpstn:Zhang} and \ref{prpstn:TripleRough}, we are left to establish Propositions \ref{prpstn:SieveAsymptotic}, \ref{prpstn:3Primes}, \ref{prpstn:Zhang} and \ref{prpstn:TripleRough}.
%
%
%
%
%
%
%
%
%
%
%
%
%
%
%
%
%
%
\section{Sieve asymptotics}\label{sec:Asymptotic}
In this section we establish Proposition \ref{prpstn:SieveAsymptotic} using sieve methods, assuming the more technical Proposition \ref{prpstn:VBounds}, given below.
%
%
%
%
%
%
%
%
\begin{prpstn}[Consequence of small factor type II estimate]\label{prpstn:VBounds}
Let $Q_1,Q_2$ satisfy \eqref{eq:Cons1}, \eqref{eq:Cons2} and \eqref{eq:Cons3}. Let $\alpha_d,\beta_e,\gamma_m$ complex sequences with $|\alpha_n|,|\beta_n|,|\gamma_n|\le \tau(n)^{B_0}$ and such that $\gamma_m$ satisfies the Siegel-Walfisz condition \eqref{eq:SiegelWalfisz}. Assume that $D,E,P$ satisfy
\begin{align*}
\frac{x^{1-\epsilon} }{Q_1^{15/8} Q_2^{15/8} }<D E P,\qquad D+E+P<x^{1/7+10\epsilon}.
\end{align*}
Let $M N D E P\asymp x$. Then we have for every $A>0$
\[
\sum_{q_1\sim Q_1}\sum_{\substack{q_2\sim Q_2\\ (q_1q_2,a)=1}}\Bigl|\sum_{d\sim D}\sum_{e\sim E}\sum_{p\sim P}\alpha_{d}\beta_e\sum_{m\sim M}\gamma_m \sum^*_{\substack{n\sim N\\ P^-(d),P^-(n)\ge p}}S_{n m p d e}\Bigr|\ll_{A,B_0} \frac{x}{(\log{x})^A}.
\]
Here $\sum^*$ indicates that the summation is restricted by $O(1)$ inequalities of the form $p^{\alpha_1}d^{\alpha_2}e^{\alpha_3}m^{\alpha_4}n^{\alpha_5}\le B$. The implied constant may depend on all such exponents $\alpha_i$, but none of the quantities $B$.
\end{prpstn}
%
%
%
%
%
%
%
%
We recall that here (and throughout the paper) $S_n$ is defined by
\[
S_n:=\mathbf{1}_{n\equiv a\Mod{q_1q_2}}-\frac{1}{\phi(q_1q_2)}\mathbf{1}_{(n,q_1q_2)=1}
\]
Proposition \ref{prpstn:VBounds} will be established in Section \ref{sec:SmallTypeII}.
%
%
%
%
%
%
%
%
\begin{lmm}[Reduction to fundamental lemma type condition]\label{lmm:Buchstab}
Let $z_1\ge z_2$ and $y\ge 1$. Then there are 1-bounded sequences $\alpha_d$, $\beta_{d}$ supported on $P^-(d)\ge z_2$ and depending only on $d,z_1,z_2$ such that
\[
\mathbf{1}_{P^-(n)>z_1}=\sum_{\substack{m d=n\\ d\le y}}\alpha_d\mathbf{1}_{P^-(m)> z_2}+\sum_{\substack{n=p d m\\ d\le y<d p\\ z_2<p\le z_1\\ P^-(d)\ge p}}\beta_{d}\mathbf{1}_{P^-(m)> p}.
\] 
\end{lmm}
\begin{proof}
This follows from repeated applications of Buchstab's identity. Let $T_0(n)=\mathbf{1}_{P^-(n)>z_2}$ and for each $r\ge 1$ let
\begin{align*}
T_r(n)&:=\sum_{\substack{n=p_1\cdots p_r m\\ z_2<p_r\le \ldots \le p_1\le z_1\\ p_1\cdots p_r\le y }}\mathbf{1}_{P^-(m)>z_2},\\
U_r(n)&:=\sum_{\substack{n=p_1\cdots p_r m\\ z_2<p_r\le \ldots \le p_1\le z_1\\ p_1\cdots p_r\le y }}\mathbf{1}_{P^-(m)>p_r},\\
V_r(n)&:=\sum_{\substack{n=p_1\cdots p_r m\\ z_2<p_r\le \ldots \le p_1\le z_1 \\ p_1\cdots p_r> y\\ p_1\cdots p_{r-1}\le y }}\mathbf{1}_{P^-(m)>p_r}.
\end{align*}
Buchstab's identity then gives
\begin{align*}
\mathbf{1}_{P^-(n)>z_1}&=T_0(n)-U_1(n)-V_1(n),\\
U_j(n)&=T_j(n)-U_{j+1}(n)-V_{j+1}(n).
\end{align*}
Repeatedly applying this and noting that $U_r(n)=T_r(n)=V_r(n)=0$ if $r>2\log{n}$ since $2^{2\log{n}}> n$, we find that
\[
\mathbf{1}_{P^-(n)>z_1}=\sum_{0\le r}(-1)^r (T_r(n)-V_{r+1}(n)).
\]
Thus we obtain the result by letting
\begin{align*}
\alpha_d&:=\sum_{0\le r}(-1)^r\sum_{\substack{d=p_1\cdots p_r \\ z_2<p_r\le \dots \le p_1\le z_1}}1,\\
\beta_d&:=\sum_{1\le r}(-1)^{r}\sum_{\substack{d=p_1\cdots p_{r-1}\\ z_2<p_{r-1}\le \dots \le p_1\le z_1 }}1.
\end{align*}
Since the inner sums are zero unless $r=\Omega(d)$, in which case it is bounded by 1, we see that these are 1-bounded sequences.
\end{proof}
%
%
%
%
%
%
%
%
\begin{lmm}[Fundamental Lemma of the sieve]\label{lmm:FundamentalLemma}
Let $y_0:=x^{1/\log\log{x}}$ and $z\le z_0:=x^{1/(\log\log{x})^3}$. There exists two sequences $\lambda_d^+,\lambda_d^-$ supported on $d\le y_0$ such that
\begin{enumerate}
\item $|\lambda_d^\pm|\le 1$.
\item If $P^-(n)> z$ then
\[
\sum_{d|n}\lambda^+_d=\sum_{d|n}\lambda_d^-=1.
\]
\item If $P^-(n)\le z$ then
\[
\sum_{d|n}\lambda^+_d\le 0\le \sum_{d|n}\lambda_d^+.
\]
\item For any multiplicative function $\omega$ satisfying $0\le \omega(p)\le \kappa$ for all primes $p$, we have
\[
\sum_{d\le y_0}\lambda_d^\pm\frac{\omega(d)}{d}=\prod_{p\le z}\Bigl(1-\frac{\omega(p)}{p}\Bigr)+O_\kappa\Bigl((\log{x})^{-\log\log{x}}\Bigr).
\]
\end{enumerate}
\end{lmm}
\begin{proof}
This is the fundamental lemma of sieve methods, and follows from \cite[Lemma 6.3]{IwaniecKowalski}.
\end{proof}
%
%
%
%
%
%
%
%
\begin{lmm}[Estimates involving a smooth factor]\label{lmm:TypeI}
Let $Q_1,Q_2$ satisfy \eqref{eq:Cons1} and \eqref{eq:Cons2}
 and let $K,N,L$ satisfy $K N L\asymp x$ and
\begin{align*}
N&\le x^{3/7+\epsilon} \frac{x}{(Q_1Q_2)^{15/8}},\\
x^{1/7+10\epsilon}\le K&\le x^{1/4+\epsilon}.
\end{align*}
Let $\alpha_n,\eta_k$ be complex sequences with $|\alpha_n|,|\eta_n|\le \tau(n)^{B_0}$. Then we have for every $A>0$
\[
\sum_{q_1\sim Q_1}\sum_{\substack{q_2\sim Q_2\\ (q_1q_2,a)=1}}\Bigl|\sum_{k\sim K}\sum_{n\sim N}\alpha_n\eta_k \sideset{}{^*}\sum_{\substack{\ell\sim L\\ P^-(\ell)\ge z_0}}S_{n k\ell}\Bigr|\ll_{A,B_0} \frac{x}{(\log{x})^A}.
\]
Here $\sum^*$ means that the summation is restricted to $O(1)$ inequalities of the form $k^{\alpha_1} n^{\alpha_2} \ell^{\alpha_3}\le B$ for some constants $\alpha_1,\alpha_2,\alpha_3$ and some quantity $B$, and the implied constant depends on all such exponents $\alpha_1,\alpha_2,\alpha_3$.
\end{lmm}
\begin{proof}
By Lemma \ref{lmm:Separation}, it suffices to show that for every $A>0$
\[
\sum_{q_1\sim Q_1}\sum_{\substack{q_2\sim Q_2\\ (q_1q_2,a)=1}}\Bigl|\sum_{k\in\mathcal{I}_K}\sum_{n\in\mathcal{I}_N}\sum_{\substack{\ell\in\mathcal{I}_L\\ P^-(\ell)\ge z_0}} \alpha_n \eta_k S_{n k \ell}\Bigr|\ll_A \frac{x}{(\log{x})^A}.
\]
for all choices of intervals $\mathcal{I}_K\subseteq[K,2K]$, $\mathcal{I}_N\subseteq[N,2N]$ and $\mathcal{I}_L\subseteq[L,2L]$. Since $Q_1,Q_2$ satisfies \eqref{eq:Cons1} and \eqref{eq:Cons2}, we have 
\[
Q_1Q_2=(Q_1Q_2^2)^{5/17}(Q_1^{12}Q_2^7)^{1/17}<x^{9/17}.
\]
The result follows from the Bombieri-Vinogradov Theorem if $Q_1Q_2\le x^{1/2-\epsilon}$, so we may assume $x^{1/2-\epsilon}\le Q_1Q_2\le x^{9/17}$. Therefore 
\begin{align*}
Q_1 Q_2 N&\le  \frac{x^{10/7+\epsilon}}{(Q_1Q_2)^{7/8}}\le x^{1-10\epsilon},\\
Q_1^2Q_2^2 N&\le x^{10/7+\epsilon}(Q_1Q_2)^{1/16}<x^{3/2-10\epsilon}.
\end{align*}
 We now wish to apply Lemma \ref{lmm:TripleSmooth}. Recalling that $K<x^{1/4+\epsilon}$ and $L\asymp x/N K$, we see that
\begin{align*}
KL&\gg \frac{x}{N}\gg  x^\epsilon Q_1Q_2,\\
K^3Q_1Q_2&=\frac{K^4 Q_1 Q_2 N}{x^{2-2\epsilon}}\frac{x^{2-2\epsilon}}{N K}\ll\frac{x^{1+1-6\epsilon}}{x^{2-2\epsilon}} L x^{1-2\epsilon}< L x^{1-2\epsilon},\\
K Q_1^2Q_2^2&=\frac{K^2Q_1^2 Q_2^2 N}{x^{2-2\epsilon}}\frac{x^{2-2\epsilon}}{N K}\ll \frac{x^{1/2+3/2-8\epsilon}}{x^{2-2\epsilon}}L x^{1-2\epsilon}<L x^{1-2\epsilon}.
\end{align*}
Thus Lemma \ref{lmm:TripleSmooth} applies, giving the result.
\end{proof}
%
%
%
%
%
%
%
%
\begin{proof}[Proof of Proposition \ref{prpstn:SieveAsymptotic} assuming Proposition \ref{prpstn:TypeII} and Proposition \ref{prpstn:VBounds}]
 We may also assume that $Q_1Q_2\ge x^{1/2-\epsilon}$ since otherwise the result follows immediately from the Bombieri-Vinogradov Theorem.

After expanding $S_{p_1\cdots p_r}(x^{1/7+10\epsilon})$ as a sum of numbers with 3, 4, 5 or 6 prime factors, we see that Proposition \ref{prpstn:TypeII} gives the result if there is a subproduct of $P_1,\dots,P_r$ which lies in $[x^{3/7+\epsilon},x^{4/7-\epsilon}]$. In particular, since $P_1\cdots P_r\le x^{2/3}$ and $P_r\ge x^{1/7+10\epsilon}$, this gives the result unless $P_1\cdots P_{r-1}\le x^{3/7+\epsilon}$, which we now assume.

By Lemma \ref{lmm:Separation}, it suffices to show that for every choice of $\mathcal{I}_i\subseteq[P_i,2P_i]$ and every choice of $B>0$ we have
\[
\sum_{q_1\sim Q_1}\sum_{\substack{ q_2\sim Q_2\\ (q_1q_2,a)=1}} \Bigl|\sum_{\substack{p_1,\dots,p_r\\ p_i\in \mathcal{I}_i}}S_{p_1\cdots p_r}(x^{1/7+10\epsilon})\Bigr|\ll_{B}\frac{x}{(\log{x})^B}.
\]
We recall that $z_1:=x^{1/7+10\epsilon}$ and 
\[
S_{p_1\cdots p_r}(z_1)=\sum_{\substack{m\sim x/(p_1\cdots p_r)}}\mathbf{1}_{P^-(m)>z_1}S_{mp_1\cdots p_r}.
\]
Let $y_1:=x^{1-\epsilon}/(Q_1Q_2)^{15/8}$. Since $Q_1Q_2\ge x^{1/2-\epsilon}$, we see that $y_1\le z_1$. By Lemma \ref{lmm:Buchstab} (taking $y=y_1$, $z_2=z_0$), it suffices to show that for every choice of 1-bounded coefficients $\alpha_d,\beta_d$ we have
\begin{align*}
\sum_{q_1\sim Q_1}\sum_{\substack{ q_2\sim Q_2\\ (q_1q_2,a)=1}} \Bigl|\sum_{\substack{p_1,\dots,p_r\\ p_i\in \mathcal{I}_i}}\sum_{z_0<p'\le z_1}\sum_{\substack{d\le y_1\\ P^-(d)\ge p'\\ d p'>y_1}}\beta_d\sum_{\substack{n\sim x/(d p' p_1\cdots p_r)\\ P^-(n)\ge p}}S_{n d p' p_1\cdots p_r}\Bigr|\ll \frac{x}{(\log{x})^B},\\
\sum_{q_1\sim Q_1}\sum_{\substack{ q_2\sim Q_2\\ (q_1q_2,a)=1}}\Bigl|\sum_{\substack{p_1,\dots,p_r\\ p_i\in \mathcal{I}_i}}\sum_{d\le y_1}\alpha_d\sum_{\substack{m\sim x/(d p_1\cdots p_r)\\ P^-(m)\ge z_0}}S_{m d p_1\cdots p_r}\Bigr|\ll \frac{x}{(\log{x})^B}.
\end{align*}
We begin by considering the first expression above. We note that $d p'>y_1$ and $d,p'<z_1$. Thus, by letting $e=1$ and $m=p_1\cdots p_r$, we see that Proposition \ref{prpstn:VBounds} gives the first estimate.

Similarly, by letting $k=p_r$, $\ell=m$ and $n=d p_1\cdots p_{r-1}$, we see that Lemma \ref{lmm:TypeI} gives the second result if $P_r\le x^{1/4}$. If instead $P_r\ge x^{1/4}$ then we must have that $r=1$.

Finally, if $r=1$ then we note that since $Q_1Q_2\ge x^{1/2-\epsilon}$, for any $d$, $p_1$ occurring we must have
\[
d p_1\le x^{3/7+\epsilon}\frac{x^{1-\epsilon} }{(Q_1Q_2)^{15/8}}\ll \frac{x^{1-\epsilon}}{Q_1 Q_2}.
\]
Thus the range of summation for $m$ has length at least $x^\epsilon Q_1Q_2$. But then by Lemma \ref{lmm:FundamentalLemma}, we have that for any choice of $B>0$ 
\begin{align*}
&\sum_{\substack{m\sim x/(d p_1) \\ P^-(m)\ge z_0}}S_{m d p_1}\le \sum_{\substack{m\sim x/(d p_1) \\ m\equiv a\overline{d p_1}\Mod{q_1q_2}}}\sum_{\substack{e|m\\ e\le y_0}}\lambda^+_d-\frac{1}{\phi(q_1q_2)}\sum_{\substack{m\sim x/(d p_1) \\ (m,q_1q_2)=1}}\sum_{\substack{e|m\\ e\le y_0}}\lambda^-_e\\
&\le \sum_{\substack{e\le y_0\\ (e,q_1q_2)=1}}\lambda^+_e\Bigl(\frac{x}{q_1q_2 d p_1 e}+O(1)\Bigr)-\frac{1}{\phi(q_1q_2)}\sum_{\substack{e\le y_0\\ (e,q_1q_2)=1}}\lambda^-_e\Bigl(\frac{\phi(q_1q_2)x}{q_1q_2 d p_1e}+O(\tau(q_1q_2))\Bigr)\\
&\ll_B \frac{x}{q_1q_2p_1d (\log{x})^B},
\end{align*}
  and similarly we obtain a lower bound by swapping the roles of $\lambda^+_e$ and $\lambda^-_e$. Thus the inner sum over $m$ is suitably small that these terms contribute negligibly when $P_r\ge x^{1/4}$. This gives the result.
\end{proof}
%
%
%
%
%
%
%
%
Since we have already established Proposition \ref{prpstn:TypeII} assuming Propositions \ref{prpstn:Zhang} and \ref{prpstn:TripleRough}, we are left to establish Propositions \ref{prpstn:3Primes}, \ref{prpstn:Zhang}, \ref{prpstn:TripleRough} and \ref{prpstn:VBounds}.
%
%
%
%
%
%
%
%
%
%
%
%
%
%
%
%
%
%
\section{Estimates for numbers with 3 prime factors}\label{sec:3Primes}
In this section we reduce the proof of Proposition \ref{prpstn:3Primes} to Proposition \ref{prpstn:TripleDivisor}, where essentially the indicator function of the primes is replaced by a smooth weight. We do this using  using Proposition \ref{prpstn:VBounds} and repeated Buchstab iterations. The following is our key proposition, which will be established later in Section \ref{sec:TripleDivisor}.
%
%
%
%
%
%
%
%
\begin{prpstn}[Estimate for triple divisor function]\label{prpstn:TripleDivisor}
Let $A>0$. Let $Q,R$ satisfy
\begin{align*}
Q^7 R^9<x^4,\qquad Q^9 R<x^{32/7},\qquad Q R<x^{11/21}.
\end{align*}
Let $x^{3/7}\ge N_1\ge N_2\ge N_3\ge x^\epsilon$ and $M\ge x^\epsilon$ satisfy  $N_1N_2N_3M\asymp x$ and
\[
M<\frac{x^{1-20\epsilon}}{( Q R)^{15/8}}.
\]
Let $|\alpha_m|\le \tau(m)^{B_0}$ be a complex sequence, let $\mathcal{I}_1,\mathcal{I}_2,\mathcal{I}_3$ be intervals with $\mathcal{I}_j\subseteq[N_j,2N_j]$, and let
\[
\Delta_{\mathscr{K}}(q):=\sum_{m\sim M}\alpha_m\sum_{\substack{n_1\in \mathcal{I}_1\\n_2\in \mathcal{I}_2\\ n_3\in \mathcal{I}_3\\ P^-(n_1n_2n_3)\ge z_0}}\Bigl(\mathbf{1}_ {m n_1n_2n_3\equiv a\Mod{q}}-\frac{\mathbf{1}_{(m n_1 n_2 n_3,q)=1}}{\phi(q)}\Bigr).
\]
Then we have
\[
\sum_{\substack{q\sim Q\\ (q,a)=1}}\sum_{\substack{r\sim R\\ (r,a)=1}}|\Delta_{\mathscr{K}}(q r)|\ll_A \frac{x}{(\log{x})^{A}}.
\]
\end{prpstn}
\begin{proof}[Proof of Proposition \ref{prpstn:3Primes} assuming Propositions \ref{prpstn:TripleDivisor}, \ref{prpstn:4Primes} and \ref{prpstn:VBounds}] 
We recall that $S_n$ is defined by
\[
S_n:=\mathbf{1}_{n\equiv a\Mod{q_1q_2}}-\frac{1}{\phi(q_1q_2)}\mathbf{1}_{(n,q_1q_2)=1}
\]
By Lemma \ref{lmm:Separation}, it suffices to show that for every choice of $\mathcal{I}_i\subseteq[P_i,2P_i]$ and every $A>0$
\[
\sum_{q_1\sim Q_1}\sum_{\substack{ q_2\sim Q_2\\ (q_1q_2,a)=1}}\Bigl|\sum_{\substack{p_1,p_2,p_3\\ p_i\in \mathcal{I}_i\forall i}}S_{p_1p_2p_3}\Bigr|\ll_A \frac{x}{(\log{x})^A}.
\]
We recall $z_1=x^{1/7+10\epsilon}$ and $P_i\le x^{3/7+\epsilon}$, so any $m\in \mathcal{I}_i$ can have at most 2 prime factors bigger than $z_1$. Thus, by three applications of Buchstab's identity
\begin{align*}
\sum_{\substack{p_1,p_2,p_3\\ p_i\in \mathcal{I}_i\forall i}}S_{p_1p_2p_3}&=\sum_{\substack{p_1\in \mathcal{I}_1\\ p_2\in \mathcal{I}_2\\ m_3\in \mathcal{I}_3\\ P^-(m_3)>z_1}}S_{p_1 p_2 m_3}-\sum_{\substack{p_1\in\mathcal{I}_1\\ p_2\in \mathcal{I}_2}}\sum_{\substack{z_1<p_3'\le p_3''\\ p_3'p_3''\in\mathcal{I}_3}}S_{p_1p_2p_3'p_3''}\\
&=\sum_{\substack{m_1,m_2,m_3\\ m_i\in \mathcal{I}_i\forall i\\ P^-(m_i)>z_1\forall i}}S_{m_1 m_2 m_3}-\sum_{\substack{p_1\in\mathcal{I}_1\\ p_2\in \mathcal{I}_2}}\sum_{\substack{z_1<p_3'\le p_3''\\ p_3'p_3''\in\mathcal{I}_3}}S_{p_1p_2p_3'p_3''}\\
&-\sum_{\substack{p_1\in\mathcal{I}_1\\ m_3\in \mathcal{I}_3\\ P^-(m_3)>z_1 }}\sum_{\substack{z_1<p_2'\le p_2''\\ p_2'p_2''\in\mathcal{I}_2}}S_{p_1p_2'p_2''m_3}-\sum_{\substack{m_2\in\mathcal{I}_2\\ m_3\in \mathcal{I}_3\\ P^-(m_2 m_3)>z_1 }}\sum_{\substack{z_1<p_1'\le p_1''\\ p_1'p_1''\in\mathcal{I}_1}}S_{p_1p_1' m_2 m_3}.
\end{align*}
Since $\mathcal{I}_i\subseteq[P_i,2P_i]$ and $P_1,P_2,P_3\in [x^{1/4},x^{3/7+\epsilon}]$ and $P_1P_2P_3\asymp x$, the final three terms on the right hand side count integers with 4, 5 or 6 prime factors. Therefore, by Proposition \ref{prpstn:4Primes}, these terms are all negligible when summed over $q_1\sim Q_1$ and $q_2\sim Q_2$ with absolute values. Thus we are left to show that
\[
\sum_{q_1\sim Q_1}\sum_{\substack{ q_2\sim Q_2\\ (q_1q_2,a)=1}}\Bigl|\sum_{\substack{m_1,m_2,m_3\\ m_i\in \mathcal{I}_i\forall i\\ P^-(m_i)>z_1 \forall i}}S_{m_1m_2m_3}\Bigr|\ll_A \frac{x}{(\log{x})^A}.
\]
Our aim is to use Lemma \ref{lmm:Buchstab} and Proposition \ref{prpstn:VBounds} to show that it is sufficient to replace the condition $P^-(m_i)>z_1$ with the condition $P^-(m_i)>z_0$. If we can do this, then the result follows from Proposition \ref{prpstn:TripleDivisor}.

We set $y_1:=x^{1-\epsilon}/(Q_1Q_2)^{15/8}\le z_1$. By Lemma \ref{lmm:Buchstab} (with $y=y_1$ and $z_2=z_0$) applied to $m_1$, it suffices to show that for every 1-bounded sequence $\alpha_d$ and $\beta_d$ we have
\begin{equation}
\sum_{q_1\sim Q_1}\sum_{\substack{ q_2\sim Q_2\\ (q_1q_2,a)=1}}\Bigl(|V^{(1)}(q_1q_2)|+|T^{(1)}(q_1q_2)|\Bigr)\ll_B\frac{x}{(\log{x})^B},
\label{eq:VTBound1}
\end{equation}
where
\begin{align*}
V^{(1)}(q_1q_2)&:=\sum_{\substack{m_2\in\mathcal{I}_2\\ m_3\in \mathcal{I}_3\\ P^-(m_2m_3)>z_1}}\sum_{z_0<p_1\le z_1}\sum_{\substack{d_1\le y_1\\ P^-(d_1)>p_1\\ d_1 p_1>y_1}}\beta_{d_1}\sum_{\substack{d_1 p_1 n_1\in \mathcal{I}_1\\ P^-(n_1)\ge p_1}}S_{n_1 d_1 p_1 m_2 m_3},\\
T^{(1)}(q_1q_2)&:=\sum_{\substack{m_2\in\mathcal{I}_2\\ m_3\in \mathcal{I}_3\\ P^-(m_2m_3)>z_1}}\sum_{d_1\le y_1}\alpha_{d_1}\sum_{\substack{d_1 n_1\in \mathcal{I}_1\\ P^-(n_1)\ge z_0}}S_{n_1 d_1 m_2 m_3}.
\end{align*}
We let $e=1$, $d=d_1$ $p=p_1$, $\ell=n_1$ and $m=m_2m_3$. Since we only consider $d_1p_1>y_1$, and $d_1,p_1\le z_1$, we see that Proposition \ref{prpstn:VBounds} implies that the $V^{(1)}(q_1q_2)$ terms contribute negligibly to \eqref{eq:VTBound1}. Thus we are left to consider the contribution from the $T^{(1)}(q_1q_2)$ terms. 

We now apply Lemma \ref{lmm:Buchstab} to the $m_2$ sum in $T^{(1)}(q_1q_2)$, this time taking $y=y_1/d_1$. Thus it suffices to show for every 1-bounded sequence $\alpha'_{d_2}$ and $\beta'_{d_2}$ that
\begin{equation}
\sum_{q_1\sim Q_1}\sum_{\substack{ q_2\sim Q_2\\ (q_1q_2,a)=1}}\Bigl(|V^{(2)}(q_1q_2)|+|T^{(2)}(q_1q_2)|\Bigr)\ll_B\frac{x}{(\log{x})^B},
\label{eq:VTBound2}
\end{equation}
where
\begin{align*}
V^{(2)}(q_1q_2)&:=\sum_{\substack{d_1,n_1,n_2,m_3\\ m_3\in \mathcal{I}_3\\ d_1\le y_1\\ n_1d_1\in\mathcal{I}_1\\ P^-(n_1)>z_0\\ P^-(m_3)>z_1}}\alpha_{d_1}\sum_{z_0<p_2\le z_1}\sum_{\substack{d_2\le y_1/d_1\\ P^-(d_2)>p_2\\ d_1 d_2 p_2>y_1}}\beta'_{d_2}\sum_{\substack{d_2 p_2 n_2\in \mathcal{I}_2\\ P^-(n_2)\ge p_2}}S_{n_1 d_1 n_2 d_2 p_2 m_3},\\
T^{(2)}(q_1q_2)&:=\sum_{d_1\le y_1}\sum_{\substack{n_1d_1\in\mathcal{I}_1\\ P^-(n_1)>z_0}}\sum_{\substack{m_3\in \mathcal{I}_3\\ P^-(m_3)>z_1}}\sum_{\substack{d_2\le y_1/d_1}}\alpha_{d_1}\alpha'_{d_2}\sum_{\substack{d_2 n_2\in \mathcal{I}_2\\ P^-(n_2)\ge z_0}}S_{n_1 d_1 n_2 d_2 m_3}.
\end{align*}
Again, on letting $e=d_1$, $p=p_2$, $d=d_2$, $n=n_2$ and $m=n_1m_3$, we see that Proposition \ref{prpstn:VBounds} shows that the $V^{(2)}(q_1q_2)$ contribute suitably little to \eqref{eq:VTBound2}, so we are left to consider the $T^{(2)}$ terms.

Finally, we apply Lemma \ref{lmm:Buchstab} once more to the $m_3$ summation of $T^{(2)}(q_1q_2)$ with $y=y_1/d_1d_2$. Thus it suffices to show for every 1-bounded sequence $\alpha''_{d_3}$ and $\beta''_{d_3}$ that
\begin{equation}
\sum_{q_1\sim Q_1}\sum_{\substack{ q_2\sim Q_2\\ (q_1q_2,a)=1}}\Bigl(|V^{(3)}(q_1q_2)|+|T^{(3)}(q_1q_2)|\Bigr)\ll_B\frac{x}{(\log{x})^B},
\label{eq:VTBound3}
\end{equation}
where
\begin{align*}
V^{(3)}(q_1q_2)&:=\sum_{\substack{d_1,d_2,n_1,n_2\\ d_1d_2\le y_1\\ n_1d_1\in\mathcal{I}_1\\ n_2d_2\in\mathcal{I}_2\\ P^-(n_1n_2)>z_0}}\alpha_{d_1}\alpha'_{d_2}\sum_{z_0<p_3\le z_1}\sum_{\substack{d_3\le y_1/(d_1d_2)\\ P^-(d_3)>p_3\\ d_1 d_2 d_3 p_3>y_1}}\beta''_{d_3}\sum_{\substack{d_3 p_3 n_3\in \mathcal{I}_3\\ P^-(n_3)\ge p_3}}S_{n_1n_2n_3d_1d_2d_3 p_3},\\
T^{(3)}(q_1q_2)&:=\sum_{d_1d_2d_3\le y_1}\sum_{\substack{n_1d_1\in\mathcal{I}_1\\ n_2d_2\in\mathcal{I}_2\\ P^-(n_1n_2)>z_0}}\alpha_{d_1}\alpha'_{d_2}\alpha''_{d_3}\sum_{\substack{d_3 n_3\in \mathcal{I}_3\\ P^-(n_3)\ge z_0}}S_{n_1 n_2 n_3 d_1 d_2 d_3}.
\end{align*}
Again, on letting $e=d_1d_2$, $p=p_3$, $d=d_3$, $n=n_3$ and $m=n_1n_2$, we see that Proposition \ref{prpstn:VBounds} shows that the $V^{(3)}$ terms contribute suitably little to \eqref{eq:VTBound3}, so we are left to consider the $T^{(3)}$ terms.

Finally, we split the $T^{(3)}(q_1q_2)$ sum by putting $d_1,d_2,d_3$ into short intervals $d_i\in\mathcal{J}_i=[D_i,D_i(1+\log^{-C_2}{x}) )$ for a suitably large constant $C_2=C_2(A)$. We then see by Lemma \ref{lmm:SmallSets} we may replace the constraint $d_i n_i\in \mathcal{I}_i$ by $D_i n_i\in\mathcal{I}_i$ since the contribution from near the end points is negligible. Therefore it suffices to show that for every $B>0$
\begin{equation}
\sum_{q_1\sim Q_1}\sum_{\substack{ q_2\sim Q_2\\ (q_1q_2,a)=1}}\Bigl|\sum_{\substack{d_1,d_2,d_3\\ d_1d_2d_3\le y_1\\ d_i\in\mathcal{J}_i\forall i}}\alpha_{d_1}\alpha'_{d_2}\alpha''_{d_3}\sum_{\substack{n_1,n_2,n_3\\ D_i n_i\in \mathcal{I}_i\forall i\\ P^-(n_i)>z_0\forall i}}S_{d_1d_2d_3n_1n_2n_3}\Bigr|\ll_B\frac{x}{(\log{x})^B}.
\label{eq:T3Sum}
\end{equation}
Thus, on letting
\[
\alpha_m:=\sum_{\substack{m=d_1d_2d_3\\ d_1d_2d_3\le y_1\\ d_i\in \mathcal{J}_i\,\forall i}}\alpha_{d_1}\alpha'_{d_2}\alpha''_{d_3},
\]
and dividing the $m$-summation into dyadic intervals, we see that \eqref{eq:T3Sum} follows from Proposition \ref{prpstn:TripleDivisor}, as required. This gives the result.
\end{proof}
%
%
%
%
%
%
%
%
We are left to establish Propositions \ref{prpstn:Zhang}, \ref{prpstn:TripleRough}, \ref{prpstn:VBounds} and \ref{prpstn:TripleDivisor}.
%
%
%
%
%
%
%
%
%
%
%
%
%
%
%
%
%
%
\section{Small factor Type II estimate}\label{sec:SmallTypeII}
In this section we reduce Proposition \ref{prpstn:VBounds} to two technical propositions about convolutions, namely Proposition \ref{prpstn:Fouvry} and Proposition \ref{prpstn:SmallDivisor}, given below. 

The first key proposition is a result which is a generalization of work of Fouvry \cite{Fouvry}. This will be established in Section \ref{sec:Fouvry}.
%
%
%
%
%
%
%
%
\begin{prpstn}[Fouvry-style estimate]\label{prpstn:Fouvry}
Let $A>0$ and $C=C(A)$ be sufficiently large in terms of $A$. Assume that $N,M,Q_1,Q_2$ satisfy $NM\asymp x$ and
\begin{align*}
x^\epsilon Q_1&<N,\\
N^6 Q_1^{3}Q_2^2 &<x^{2-15\epsilon},\\
N^{3}Q_1^{3}Q_2^{3}&<x^{2-15\epsilon},\\
N^{3}Q_1^{3}Q_2^{4}&<x^{5/2-15\epsilon},\\
Q_1Q_2&<x^{1/2-2\epsilon}N^{1/2}.
\end{align*}
Let $\beta_m,\alpha_n$ be complex sequences such that $|\alpha_n|,|\beta_n|\le \tau(n)^{B_0}$ and such that $\alpha_{n}$ satisfies the Siegel-Walfisz condition \eqref{eq:SiegelWalfisz} and $\alpha_n$ is supported on $n$ with all prime factors bigger than $z_0:=x^{1/(\log\log{x})^3}$. Let
\[
\Delta(q):=\sum_{m\sim M}\sum_{\substack{n\sim N}}\alpha_n\beta_m\Bigl(\mathbf{1}_{m n\equiv a \Mod{q}}-\frac{\mathbf{1}_{(m n,q)=1}}{\phi(q)}\Bigr).
\]
Then we have
\[
\mathop{\sum_{q_1\sim Q_1}\sum_{q_2\sim Q_2}}\limits_{(q_1q_2,a)=1}|\Delta(q_1q_2)|\ll_{A,B_0} \frac{x}{(\log{x})^A}.
\]
\end{prpstn}
%
%
%
%
%
%
%
%
Our second estimate is a new convolution result tailored for when one factor is a fairly small power of $x$. This will be established in Section \ref{sec:SmallDivisor}.
%
%
%
%
%
%
%
%
\begin{prpstn}[Small divisor estimate]\label{prpstn:SmallDivisor}
Let $A>0$ and $C=C(A)$ be sufficiently large in terms of $A$. Assume that $N,M,Q_1,Q_2$ satisfy $NM\asymp x$ and
\begin{align*}
N^6 Q_1^{7} Q_2^8&<x^{4-13\epsilon},\\
Q_1^2 Q_2&<x^{1-7\epsilon} N.
\end{align*}
Let $\beta_m,\alpha_n$ be complex sequences such that $|\alpha_n|,|\beta_n|\le \tau(n)^{B_0}$ and such that $\alpha_{n}$ satisfies the Siegel-Walfisz condition \eqref{eq:SiegelWalfisz} and $\alpha_n$ is supported on $n$ with all prime factors bigger than $z_0:=x^{1/(\log\log{x})^3}$. Let
\[
\Delta(q):=\sum_{m\sim M}\sum_{\substack{n\sim N}}\alpha_n\beta_m\Bigl(\mathbf{1}_{mn\equiv a \Mod{q}}-\frac{\mathbf{1}_{(mn,q)=1}}{\phi(q)}\Bigr).
\]
Then we have
\[
\mathop{\sum_{q_1\sim Q_1}\sum_{q_2\sim Q_2}}\limits_{(q_1q_2,a)=1}|\Delta(q_1q_2)|\ll_{A,B_0} \frac{x}{(\log{x})^A}.
\]
\end{prpstn}
%
%
%
%
%
%
%
%
Putting together Proposition \ref{prpstn:Fouvry} and Proposition \ref{prpstn:SmallDivisor}, we obtain the following.
%
%
%
%
%
%
%
%
\begin{lmm}[Small Factor Type II estimate for convolutions]\label{lmm:SmallTypeII}
Let $Q_1,Q_2$ satisfy
\begin{align*}
Q_1^{23}Q_2^{31}&<x^{16-40\epsilon},\\
Q_1Q_2&<x^{11/21-20\epsilon},\\
Q_1^{20}Q_2^{19}&<x^{10-60\epsilon},\\
Q_1^3Q_2^2&<x^{8/7-100\epsilon}.
\end{align*}
Let $N$, $M$ be such that $NM\asymp x$  and 
\[
\frac{x^{1-\epsilon}}{(Q_1Q_2)^{15/8}}<N<x^{1/7+10\epsilon}.
\]
Let $\beta_m,\alpha_n$ be complex sequences such that $|\alpha_n|,|\beta_n|\le \tau(n)^{B_0}$ and such that $\alpha_{n}$ satisfies the Siegel-Walfisz condition \eqref{eq:SiegelWalfisz} and $\alpha_n$ is supported on $n$ with all prime factors bigger than $z_0:=x^{1/(\log\log{x})^3}$. Let
\[
\Delta(q):=\sum_{m\sim M}\sum_{\substack{n\sim N}}\alpha_n\beta_m\Bigl(\mathbf{1}_{m n\equiv a \Mod{q}}-\frac{\mathbf{1}_{(m n,q)=1}}{\phi(q)}\Bigr).
\]
Then we have
\[
\sum_{q_1\sim Q_1}\sum_{\substack{ q_2\sim Q_2\\ (q_1q_2,a)=1}}|\Delta(q_1q_2)|\ll_{A,B_0} \frac{x}{(\log{x})^A}.
\]
\end{lmm}
%
%
%
%
%
%
%
%
\begin{proof}[Proof of Lemma \ref{lmm:SmallTypeII} assuming Proposition \ref{prpstn:Fouvry} and Proposition \ref{prpstn:SmallDivisor}]
We only need consider $Q_1Q_2\ge x^{1/2-\epsilon}$ because otherwise the result is immediate from the Bombieri-Vinogradov Theorem.  Proposition \ref{prpstn:Fouvry} gives the result provided 
\[
\max\Bigl(\frac{Q_1^2Q_2^2}{x^{1-4\epsilon} },Q_1 x^\epsilon \Bigr)<N<\min\Bigl(\frac{x^{2/3-5\epsilon}}{Q_1Q_2},\frac{x^{1/3-4\epsilon}}{Q_1^{1/2}Q_2^{1/3}},\frac{x^{5/6-5\epsilon}}{Q_1Q_2^{4/3}}\Bigr).
\]
We can also apply Proposition \ref{prpstn:Fouvry} with a trivial factorization by taking `$Q_2$' to be $Q_1Q_2$ and `$Q_1$' to be 1. This simplifies to gives the result in the range
\[
\frac{Q_1^2Q_2^2}{x^{1-4\epsilon} }<N<\frac{x^{5/6-8\epsilon}}{Q_1^{4/3}Q_2^{4/3}}.
\]
(Here we used the fact that $x^{5/6-8\epsilon}/(Q_1Q_2)^{4/3}<x^{1/3-4\epsilon}/(Q_1Q_2)^{1/3},x^{2/3-5\epsilon}/(Q_1Q_2)$ since $Q_1Q_2>x^{1/2-\epsilon}$.) Finally, we can also apply Proposition \ref{prpstn:SmallDivisor} (with $Q_1$ and $Q_2$ swapped). This gives the result in the range
\[
\frac{Q_1Q_2^2}{x^{1-7\epsilon} }<N<\frac{x^{2/3-3\epsilon}}{Q_1^{8/6}Q_2^{7/6}}.
\]
We see that if 
\begin{equation}
Q_1^{10/3}Q_2^{19/6}<x^{5/3-7\epsilon},
\label{eq:SmallFactorCond1}
\end{equation}
 then
\[
\frac{Q_1^2Q_2^2}{x^{1-4\epsilon}}<\frac{x^{2/3-3\epsilon}}{Q_1^{8/6}Q_2^{7/6}},
\]
and so the final two ranges overlap. Similarly, if 
\begin{equation}
Q_1^{7/3} Q_2^{4/3}<x^{5/6-9\epsilon},
\label{eq:SmallFactorCond2}
\end{equation}
 then
\[
Q_1 x^\epsilon<\frac{x^{5/6-8\epsilon}}{Q_1^{4/3}Q_2^{4/3}},
\]
and so the first two ranges overlap. Thus if \eqref{eq:SmallFactorCond1} and \eqref{eq:SmallFactorCond2} both hold we have the result for
\[
\frac{Q_1 Q_2^2}{x^{1-7\epsilon}}<N<\min\Bigl(\frac{x^{2/3-5\epsilon}}{Q_1Q_2},\frac{x^{1/3-4\epsilon}}{Q_1^{1/2}Q_2^{1/3}},\frac{x^{5/6-5\epsilon}}{Q_1Q_2^{4/3}}\Bigr).
\]
This gives the result in the entire range provided we have 
\begin{align*}
\frac{Q_1Q_2^2}{x^{1-4\epsilon}}<\frac{x^{1-\epsilon}}{(Q_1Q_2)^{15/8}},\qquad
x^{1/7+10\epsilon}<\frac{x^{2/3-5\epsilon}}{Q_1Q_2},\\
x^{1/7+10\epsilon}<\frac{x^{1/3-4\epsilon}}{Q_1^{1/2}Q_2^{1/3}},\qquad x^{1/7+10\epsilon}<\frac{x^{5/6-5\epsilon}}{Q_1Q_2^{4/3}}.
\end{align*}
These conditions hold if we have
\begin{align}
Q_1^{23}Q_2^{31}&<x^{16-40\epsilon}\label{eq:SmallFactorCond3},\\
Q_1Q_2&<x^{11/21-20\epsilon}\label{eq:SmallFactorCond4},\\
Q_1^3Q_2^2&<x^{8/7-100\epsilon},\label{eq:SmallFactorCond5}\\
Q_1^3Q_2^4&<x^{29/14-45\epsilon}.\label{eq:SmallFactorCond6}
\end{align}
Finally, we note that if $Q_1Q_2\ge x^{1/2-\epsilon}$ and \eqref{eq:SmallFactorCond5} holds, then
\[
Q_1^7Q_2^4=\frac{(Q_1^3Q_2^2)^3}{(Q_1Q_2)^2}<\frac{x^{24/7-300\epsilon}}{x^{1-2\epsilon}}<x^{5/2-30\epsilon},
\]
and so  \eqref{eq:SmallFactorCond2} automatically holds. Similarly,
\[
Q_1^3Q_2^4=(Q_1^{23}Q_2^{31})^{6/47}(Q_1^3Q_2^2)^{1/47}<x^{(96+8/7)/47}<x^{29/14-45\epsilon},
\]
so if \eqref{eq:SmallFactorCond3} and \eqref{eq:SmallFactorCond5} both hold then \eqref{eq:SmallFactorCond6} automatically holds. This gives the result.
\end{proof}
%
%
%
%
%
%
%
%
\begin{proof}[Proof of Proposition \ref{prpstn:VBounds} assuming Proposition \ref{prpstn:Fouvry} and Proposition \ref{prpstn:SmallDivisor}] We first note that our assumptions on $Q_1,Q_2$ imply that
\begin{align*}
Q_1^{23}Q_2^{31}&=(Q_1^{12} Q_2^{7})^{15/17} (Q_1 Q_2^2)^{211/17}<x^{271/17}<x^{16-40\epsilon},\\
Q_1Q_2&=(Q_1^{20} Q_2^{19})^{1/21} (Q_1 Q_2^2)^{1/21}<x^{11/21-20\epsilon},\\
Q_1^3 Q_2^2&=(Q_1^{12} Q_2^{7})^{4/17} (Q_1 Q_2^2)^{3/17}<x^{19/17}<x^{8/7-100\epsilon}.
\end{align*}
Thus $Q_1,Q_2$ satisfy all the conditions of Lemma \ref{lmm:SmallTypeII}. By Lemma \ref{lmm:Separation}, we see that it is sufficient to show that
\[
\sum_{q_1\sim Q_1}\sum_{\substack{ q_2\sim Q_2\\ (q_1q_2,a)=1}}\Bigl|\sum_{d\in \mathcal{I}_1}\sum_{e\in\mathcal{I}_2}\sum_{p\in\mathcal{I}_3}\alpha_{d}\beta_e\sum_{m\in\mathcal{I}_4}\gamma_m \sum_{\substack{n\in \mathcal{I}_5\\ P^-(d),P^-(n)\ge p}}S_{n m p d e}\Bigr|\ll_B \frac{x}{(\log{x})^B}.
\]
for every choice of $B>0$ and every choice of intervals $\mathcal{I}_1,\dots,\mathcal{I}_5$ with $\mathcal{I}_1\subseteq[D,2D]$, $\mathcal{I}_2\subseteq[E,2E]$, $\mathcal{I}_3\subseteq[P,2P]$, $\mathcal{I}_4\subseteq[M,2M]$ and $\mathcal{I}_5\subseteq[N,2N]$. We subdivide $\mathcal{I}_3$ into $O(\log^C{x})$ subintervals of length $P/\log^C{x}$, for a suitably large constant $C=C(B)$. Taking the worst such interval, we see it suffices to show that
\[
\sum_{q_1\sim Q_1}\sum_{\substack{ q_2\sim Q_2\\ (q_1q_2,a)=1}}\Bigl|\sum_{d\in \mathcal{I}_1}\sum_{e\in\mathcal{I}_2}\sum_{p\in\mathcal{I}_3}\alpha_{d}\beta_e\sum_{m\in\mathcal{I}_4}\gamma_m \sum_{\substack{n\in \mathcal{I}_5\\ P^-(d),P^-(n)\ge p}}S_{n m p d e}\Bigr|\ll_C \frac{x}{(\log{x})^{3C/2}},
\]
for every interval $\mathcal{I}_3=[P',P'+P/\log^{C}{x}]$ with $P'\in[P,2P]$. We now replace the conditions $P^-(d),P^-(n)\ge p$ with $P^-(d),P^-(n)\ge P'$. This introduces an error from the contribution of when $P^-(d)\in\mathcal{I}_3$ or $P^-(m)\in\mathcal{I}_3$. Since there are $O(DN/\log^{C}{x})$ such choices of $d,n$ and there are $O(P/\log^{C}{x})$ choices of $p$, we see that this counts contributions $S_{r}$ from $r$ coming from a set of size $\ll x/\log^{2C}{x}$. By Lemma \ref{lmm:SmallSets}, this contribution is acceptably small. Thus it suffices to show
\[
\sum_{q_1\sim Q_1}\sum_{\substack{ q_2\sim Q_2\\ (q_1q_2,a)=1}}\Bigl|\sum_{\substack{d\in \mathcal{I}_1\\ P^-(d)\ge P'}}\sum_{e\in\mathcal{I}_2}\sum_{p\in\mathcal{I}_3}\alpha_{d}\beta_e\sum_{m\in\mathcal{I}_4}\gamma_m \sum_{\substack{n\in \mathcal{I}_5\\ P^-(n)\ge P'}}S_{n m p d e}\Bigr|\ll_C \frac{x}{(\log{x})^{3C/2}},
\]
Since we have $D,E,P\le x^{1/7+10\epsilon}$ and $D E P>x^{1-\epsilon}/(Q_1Q_2)^{15/8}$ there is some sub--product of $D,E,P$ which lies in $[x^{1-\epsilon}/(Q_1Q_2)^{15/8},x^{1/7+10\epsilon}]$. Grouping the corresponding variables together, we then see that the result now follows from Lemma \ref{lmm:SmallTypeII}.
\end{proof}
We have already established  Proposition \ref{prpstn:4Primes} assuming Proposition \ref{prpstn:TypeII} and \ref{prpstn:TripleRough}. We have already established Proposition \ref{prpstn:TypeII} assuming Proposition \ref{prpstn:Zhang} and \ref{prpstn:TripleRough}. We have already established Proposition \ref{prpstn:VBounds} assuming Propositions \ref{prpstn:Fouvry} and \ref{prpstn:SmallDivisor}. Thus we are left to establish Propositions \ref{prpstn:Zhang}, \ref{prpstn:TripleRough}, \ref{prpstn:TripleDivisor}, \ref{prpstn:Fouvry} and \ref{prpstn:SmallDivisor}.
%
%
%
%
%
%
%
%
%
%
%
%
%
%
%
%
%
%
%
\section{Preparatory lemmas}\label{sec:Lemmas}
The remainder of the paper is now dedicated to establishing the more technical Propositions \ref{prpstn:Zhang}, \ref{prpstn:TripleRough}, \ref{prpstn:TripleDivisor}, \ref{prpstn:Fouvry} and \ref{prpstn:SmallDivisor} individually. Before we embark on this, we establish various lemmas which we will make use of in the later sections.
%
%
%
%
%
%
%
%
\begin{lmm}[Bezout's identity]\label{lmm:Bezout}
Let $(q_1,q_2)=1$. Then we have
\[
\frac{a}{q_1 q_2}=\frac{a\overline{q_1}}{q_2}+\frac{a\overline{q_2}}{q_1}\Mod{1}.
\]
\end{lmm}
\begin{proof}
Multiplying through by $q_1q_2$, we see that this follows from the fact that $q_1\overline{q_1}+q_2\overline{q_2}\equiv 1\Mod{q_1q_2}$, which is obvious from the Chinese remainder theorem and considering the equation $\Mod{q_1}$ and $\Mod{q_2}$ separately.
\end{proof}
%
%
%
%
%
%
%
%
\begin{lmm}[Splitting into coprime sets]\label{lmm:FouvryDecomposition}
Let $\mathcal{N}\subseteq \mathbb{Z}_{>0}^2$ be a set of pairs $(a,b)$ satisfying:
\begin{enumerate}
\item $a,b\le x^{O(1)}$,
\item $\gcd(a,b)=1$,
\item The number of prime factors of $a$ and of $b$ is $\ll (\log\log{x})^3$. 
\end{enumerate}
Then there is a partition $\mathcal{N}=\mathcal{N}_1\sqcup\mathcal{N}_2\sqcup\dots \sqcup\mathcal{N}_J$ into $J$ disjoint subsets with
\[
J\ll \exp\Bigl(O(\log\log{x})^4\Bigr),
\]
such that if $(a,b)$ and $(a',b')$ are in the same set $\mathcal{N}_j$, then $\gcd(a,b')=\gcd(a',b)=1$. 
\end{lmm}
\begin{proof}
This follows immediately from \cite[Lemme 6]{Fouvry}.
\end{proof}
%
%
%
%
%
%
%
%
\begin{lmm}[Weil bound for Kloosterman sums]\label{lmm:Kloosterman}
Let $S(m,n;c)$ be the standard Kloosterman sum 
\[
S(m,n;c):=\sum_{\substack{b\Mod{c}\\ (b,c)=1}}e\Bigl(\frac{m b+ n \overline{b}}{c}\Bigr).
\]
Then we have that
\[
S(m,n;c)\ll \tau(c)c^{1/2}\gcd(m,n,c)^{1/2}.
\]
\end{lmm}
\begin{proof}
This is \cite[Corollary 11.12]{IwaniecKowalski}.
\end{proof}
%
%
%
%
%
%
%
%
\begin{lmm}[Poisson Summation]\label{lmm:Completion}
Let $C>0$ and  $f:\mathbb{R}\rightarrow\mathbb{R}$ be a smooth function which is supported on $[-10,10]$ and satisfies $\|f^{(j)}\|_\infty\ll_j (\log{x})^{j C}$ for all $j\ge 0$, and let $M,q\le x$. Then we have
\[
\sum_{m\equiv a\Mod{q}} f\Bigl(\frac{m}{M}\Bigr)=\frac{M}{q}\hat{f}(0)+\frac{M}{q}\sum_{1\le |h|\le H}\hat{f}\Bigl(\frac{h M}{q}\Bigr)e\Bigl(\frac{ah}{q}\Bigr)+O_C(x^{-100}),
\]
for any choice of $H>q x^\epsilon/M$. Moreover, if $f$ satisfies $\|f^{(j)}\|_\infty\ll ((j+1)\log{x})^{jC}$ for all $j\ge 0$ then we have the same result with $H\ge q(\log{x})^{2C+1}/M$. In particular,
\[
\sum_{m\equiv a\Mod{q}} \psi_0\Bigl(\frac{m}{M}\Bigr)=\frac{M}{q}\hat{\psi_0}(0)+\frac{M}{q}\sum_{1\le |h|\le H_0}\hat{\psi_0}\Bigl(\frac{h M}{q}\Bigr)e\Bigl(\frac{ah}{q}\Bigr)+O(x^{-100}),
\]
for any choice of $H_0\ge q(\log{x})^5/M$.
\end{lmm}
\begin{proof}
This is a truncated Poisson summation formula. Let $g(n)=f((a+q n)/M)$. Then by Poisson summation we have
\[
\sum_{m\equiv a\Mod{q}}f\Bigl(\frac{m}{M}\Bigr)=\sum_{n\in\mathbb{Z}}g(n)=\sum_{h\in\mathbb{Z}}\hat{g}(h)=\frac{M}{q}\sum_{h\in\mathbb{Z}}\hat{f}\Bigl(\frac{hM}{q}\Bigr)e\Bigl(\frac{ha}{q}\Bigr).
\]
If $\|f^{(j)}\|_\infty\ll_j (\log{x})^{j C}$ then by integration by parts we see that $|\hat{f}(t)|\ll_j (\log{x})^{j C} t^{-j}$ for any $j\ge 1$. Choosing $j$ large enough we see that the terms with $|h|\ge qx^\epsilon/M$ contribute $O_C(x^{-100})$, which gives the first claim. If  $\|f^{(j)}\|_\infty \ll ((j+1)\log{x})^{jC}$ we see the terms with $|h|>q(\log{x})^{2C+1}/M$ contribute
\[
\ll_C \frac{M}{q}\sum_{|h|>q(\log{x})^{2C+1}/M}\frac{(j\log{x})^{jC}}{|2\pi h M/q|^j}\ll_C\frac{(j\log{x})^{jC}}{|2\pi \log^{2C+1}{x}|^{j-1}}.
\]
Choosing $j=\lceil\log{x}\rceil$ then gives the result.
\end{proof}
%
%
%
%
%
%
%
%
\begin{lmm}[Completion of inverses]\label{lmm:InverseCompletion}
Let $C>0$ and $f:\mathbb{R}\rightarrow\mathbb{R}$ be a smooth function which is supported on $[-10,10]$ and satisfies $\|f^{(j)}\|_\infty\ll_j (\log{x})^{j C}$ for all $j\ge 0$. Let $(d,q)=1$. Then we have for any $H\ge x^\epsilon d q/N$
\begin{align*}
\sum_{\substack{(n,q)=1\\ n\equiv n_0\Mod{d}}}&f\Bigl(\frac{n}{N}\Bigr)e\Bigl(\frac{b\overline{n}}{q}\Bigr)=\frac{N\hat{f}(0)}{d q}\sum_{(c,q)=1}e\Bigl(\frac{b c}{q}\Bigr)\\
&+\frac{N}{d q}\sum_{1\le |h|\le H}\hat{f}\Bigl(\frac{h N}{d q}\Bigr)e\Bigl(\frac{n_0\overline{q}h}{d}\Bigr)\sum_{\substack{c\Mod{q}\\ (c,q)=1}} e\Bigl(\frac{b\overline{ dc}+h c}{q}\Bigr)+O_C(x^{-100}).
\end{align*}
Moreover, if $\|f^{(j)}\|_\infty\ll ((j+1)\log{x})^{jC}$ then we have the same result for any $H\ge (\log{x})^{2C+1} dq/N$.
\end{lmm}
\begin{proof}
We first put $n$ into residue classes $\Mod{d q}$ and the apply Lemma \ref{lmm:Completion}. This gives
\begin{align*}
&\sum_{\substack{(n,q)=1\\ n\equiv n_0\Mod{d}}}f\Bigl(\frac{n}{N}\Bigr)e\Bigl(\frac{b\overline{n}}{q}\Bigr)=\sum_{\substack{c\Mod{q}\\ (c,q)=1}}e\Bigl(\frac{b\overline{c}}{q}\Bigr)\sum_{\substack{n\equiv c\Mod{q}\\ n\equiv n_0\Mod{d}}}f\Bigl(\frac{n}{N}\Bigr)\\
&=\frac{N\hat{f}(0)}{dq}\sum_{\substack{c\Mod{q}\\ (c,q)=1}}e\Bigl(\frac{b\overline{c}}{q}\Bigr)+\frac{N}{d q}\sum_{\substack{c\Mod{q}\\ (c,q)=1}}e\Bigl(\frac{b\overline{c}}{q}\Bigr)\sum_{1\le |h|\le H}\hat{f}\Bigl(\frac{h N}{d q}\Bigr)e\Bigl(h\Bigl(\frac{c \overline{d}}{q}+\frac{n_0\overline{q}}{d}\Bigr)\Bigr)\\
&\qquad +O_C(x^{-100}).
\end{align*}
Here we used the Chinese Remainder Theorem to combine the congruence for $n$ in the final line. A change of variables in the $c$-summations then gives the result.
\end{proof}
%
%
%
%
%
%
%
%
\begin{lmm}[Summation with coprimality constraint]\label{lmm:TrivialCompletion}
Let $C>0$ and $f:\mathbb{R}\rightarrow\mathbb{R}$ be a smooth function which is supported on $[-10,10]$ and satisfies $\|f^{(j)}\|_\infty\ll_j (\log{x})^{j C}$ for all $j\ge 0$. Then we have
\[
\sum_{(m,q)=1}f\Bigl(\frac{m}{M}\Bigr)=\frac{\phi(q)}{q}M+O(\tau(q)(\log{x})^{2C}).
\]
\end{lmm}
\begin{proof}
We use M\"obius inversion to rewrite the condition $(m,q)=1$.
\begin{align*}
\sum_{(m,q)=1}f\Bigl(\frac{m}{M}\Bigr)&=\sum_{d|q}\mu(d)\sum_{n}f\Bigl(\frac{d n}{M}\Bigr)\\
&=\sum_{d|q}\mu(d)\Bigl(\frac{M\hat{f}(0)}{d}+O(\log^{2C}{x})\Bigr)=\frac{\phi(q)}{q}M\hat{f}(0)+O(\tau(q)\log^{2C}{x}).
\end{align*}
Here we used Poisson summation and the fact $\hat{f}(t)\ll (\log{x})^{2C}/t^2$ from $\|f^{(2)}\|_\infty \ll (\log{x})^{2C}$ in the final line.
\end{proof}
%
%
%
%
%
%
%
%
\begin{lmm}\label{lmm:SiegelWalfiszMaintain}
Let $C,B>0$ be constants and let $\alpha_n$ be a sequence satisfing the Siegel-Walfisz condition \eqref{eq:SiegelWalfisz}, supported on $n\le 2x$ with $P^-(n)\ge z_0=x^{1/(\log\log{x})^3}$ and satisfying $|\alpha_n|\le \tau(n)^B$. Then $\mathbf{1}_{\tau(n)\le (\log{x})^C}\alpha_n$ also satisfies the Siegel-Walfisz condition.
\end{lmm}
\begin{proof}
since $\alpha_n$ satisfies the Siegel-Walfisz condition, we just need to show that $\mathbf{1}_{\tau(n)\ge(\log{x})^C}\alpha_n$ also does. We see that for any choice of $A>0$ we have
\begin{align*}
\sum_{\substack{n\sim N\\ n\equiv a\Mod{q}\\ \tau(n)\ge (\log{x})^C}}\alpha_n\le \sum_{\substack{n\le \min(2x,N) \\ P^-(n)\ge z_0}}\frac{\tau(n)^{B+A}}{(\log{x})^{AC}}\le \frac{N}{(\log{x})^{AC}}\prod_{z_0<p\le 2x}\Bigl(1+\frac{O_{A,B}(1)}{p}\Bigr)
\end{align*}
The product is $(\log\log{x})^{O_{A,B}(1)}$, so choosing $A$ large enough we see that this gives the Siegel-Walfisz condition (with implied constants depending on $B,C$).
\end{proof}
%
%
%
%
%
%
%
%
\begin{lmm}[Barban-Davenport-Halberstam type estimate]\label{lmm:BarbanDavenportHalberstam}
Let $B_0>0$ and let $\alpha_n$ be a complex sequence which satisfies the Siegel-Walfisz condition \eqref{eq:SiegelWalfisz} and satisfies $|\alpha_n|\le \tau(n)^{B_0}$. Then for any $A>0$ there is a constant $C=C(A,B_0)$ such that if $Q<N/(\log{N})^{C}$ we have
\[
\sum_{q\le Q}\tau(q)^{B_0}\sum_{\substack{b\Mod{q}\\ (b,q)=1}}\Bigl|\sum_{n\sim N}\alpha_n \Bigl(\mathbf{\mathbf{1}}_{n\equiv b\Mod{q}}-\frac{\mathbf{1}_{(n,q)=1}}{\phi(q)}\Bigr)\Bigr|^2\ll_{A,B_0} \frac{N^2}{(\log{N})^A}.
\]
\end{lmm}
\begin{proof}
Without the factor $\tau(q)$ this is part (a) of \cite[Theorem 0]{BFI1}. By Cauchy-Schwarz we have
\[
\sum_{q\le Q}\tau(q)^{B_0}\sum_{\substack{b\Mod{q}\\ (b,q)=1}}\Bigl|\sum_{n\sim N}\alpha_n \Bigl(\mathbf{1}_{n\equiv b\Mod{q}}-\frac{\mathbf{1}_{(n,q)=1}}{\phi(q)}\Bigr)\Bigr|^2\ll \mathscr{S}_1^{1/2}\mathscr{S}_2^{1/2},
\]
where, bounding trivially via Lemma \ref{lmm:Divisor}
\begin{align*}
\mathscr{S}_1&:=\sum_{q\le Q}\tau(q)^{2B_0}\sum_{\substack{b\Mod{q}\\ (b,q)=1}}\Bigl|\sum_{n\sim N}\alpha_n \Bigl(\mathbf{1}_{n\equiv b\Mod{q}}-\frac{\mathbf{1}_{(n,q)=1}}{\phi(q)}\Bigr)\Bigr|^2\\
&\ll N^2\sum_{q\le Q}\frac{(\tau(q)\log{x})^{O_{B_0}(1)}}{q}\ll N^2(\log{x})^{O_{B_0}(1)},
\end{align*}
and by \cite[Theorem 0]{BFI1} for any constant $C>0$
\begin{align*}
\mathscr{S}_2&:=\sum_{q\le Q}\sum_{\substack{b\Mod{q}\\ (b,q)=1}}\Bigl|\sum_{n\sim N}\alpha_n \Bigl(\mathbf{1}_{n\equiv b\Mod{q}}-\frac{1}{\phi(q)}\Bigr)\Bigr|^2\ll_C \frac{N^2}{(\log{N})^{C}}.
\end{align*}
 Choosing $C$ sufficiently large in terms of $A$ and $B_0$ gives the result.
\end{proof}
%
%
%
%
%
%
%
%
\begin{lmm}[Most moduli have small square-full part]\label{lmm:Squarefree}
Let $Q<x^{1-\epsilon}$. Let $\gamma_b,c_q$ be complex sequences satisfying $|\gamma_b|, |c_b|\le \tau(b)^{B_0}$. Let $sq(n)$ denote the square-full part of $n$. (i.e. $sq(n)=\prod_{p:p^2|n}p^{\nu_p(n)}$). Then for every $A>0$ there is a constant $C=C(A,B_0)$ such that
\[
\sum_{\substack{q\sim Q\\ sq(q)\ge (\log{x})^C}}c_q\sum_{b\le  x}\gamma_b\Bigl(\mathbf{1}_{b\equiv a\Mod{q}}-\frac{\mathbf{1}_{(b,q)=1}}{\phi(q)}\Bigr)\ll_{A,B_0} \frac{x}{(\log{x})^A}.
\]
\end{lmm}
\begin{proof}
By Lemma \ref{lmm:Divisor} we have that
\begin{align*}
\sum_{b\le  x}\gamma_b\Bigl(\mathbf{1}_{b\equiv a\Mod{q}}-\frac{\mathbf{1}_{(b,q)=1}}{\phi(q)}\Bigr)&\ll \frac{x(\log{x})^{O_{B_0}(1)}}{q}+\sum_{\substack{b\le x\\ b\equiv a\Mod{q}}}\tau(b)^{B_0}\\
&\ll \frac{x(\tau(q)\log{x})^{O_{B_0}(1)}}{q}.
\end{align*}
Let $q=q_1 q_2$ where $sq(q)=q_1$, and let $q_0$ be the radical of $q_1$, so that $q_0^2|q_1$. Then, using Lemma \ref{lmm:Divisor} again, we see that the sum in the lemma is bounded by
\begin{align*}
\sum_{\substack{q_1\ge (\log{x})^C\\ p|q_1\Rightarrow p^2|q_1}}\sum_{q_2\le 2Q/q_1}\frac{x(\tau(q_1q_2)\log{x})^{O_{B_0}(1)}}{q_1q_2}&\ll x(\log{x})^{O_{B_0}(1)}\sum_{q_0\ge (\log{x})^{C/2}}\frac{\tau(q_0)^{O_{B_0}(1)}}{q_0^2}\\
&\ll \frac{x(\log{x})^{O_{B_0}(1)}}{(\log{x})^{C/2}}.
\end{align*}
Choosing $C$ large enough then gives the result.
\end{proof}
%
%
%
%
%
%
%
%
\begin{lmm}[Most moduli have small $z_0$-smooth part]\label{lmm:RoughModuli}
Let $Q<x^{1-\epsilon}$. Let $\gamma_b,c_q$ be complex sequences with $|\gamma_b|,|c_b|\le \tau(n)^{B_0}$ and recall $z_0:=x^{1/(\log\log{x})^3}$ and $y_0:=x^{1/\log\log{x}}$. Let $\operatorname{sm}(n;z)$ denote the $z$-smooth part of $n$. (i.e. $\operatorname{sm}(n;z)=\prod_{p\le z}p^{\nu_p(n)}$). Then for every $A>0$ we have that
\[
\sum_{\substack{q\sim Q\\ \operatorname{sm}(q;z_0)\ge y_0}}c_q\sum_{b\le  x}\gamma_b\Bigl(\mathbf{1}_{b\equiv a\Mod{q}}-\frac{\mathbf{1}_{(b,q)=1}}{\phi(q)}\Bigr)\ll_{A,B_0} \frac{x}{(\log{x})^A}.
\]
\end{lmm}
\begin{proof}
As in the proof of Lemma \ref{lmm:Squarefree}, by Lemma \ref{lmm:Divisor} we have the trivial bound
\begin{align*}
\sum_{b\le x}\gamma_b\Bigl(\mathbf{1}_{b\equiv a\Mod{q}}-\frac{\mathbf{1}_{(b,q)=1}}{\phi(q)}\Bigr)&\ll \frac{x(\tau(q)\log{x})^{O_{B_0}(1)}}{q}.
\end{align*}
We factor $q=q_1q_2$ into $z_0$-smooth and $z_0$-rough parts. Thus we see that the sum in the lemma is bounded by
\begin{align*}
\sum_{\substack{q_1\ge y_0\\ P^+(q_1)\le z_0}}\sum_{q_2\sim Q_2}\frac{x(\tau(q_1q_2)\log{x})^{O_{B_0}(1)}}{q_1q_2}&\ll x(\log{x})^{O_{B_0}(1)}\sum_{\substack{q_1\ge y_0\\ P^+(q_1)\le z_0}}\frac{\tau(q_1)^{O_{B_0}(1)}}{q_1}.
\end{align*}
Now we us Rankin's trick. Letting $\eta=1/\log{z_0}=1/(\log\log{x})^3$, we have
\begin{align*}
\sum_{\substack{q_1\ge y_0\\ P^+(q_1)\le z_0}}\frac{\tau(q_1)^{B_0}}{q_1}\le y_0^{-\eta}\sum_{P^+(q_1)\le z_0}\frac{\tau(q_1)^{B_0}}{q_1^{1-\eta}}&\ll_{B_0}  y_0^{-\eta}\prod_{\substack{p\le z_0\\ 2^{B_0+3}\le p}}\Bigl(1-\frac{2^{B_0}}{p^{1-\eta}}\Bigr)^{-1}\\
&\ll (\log{z_0})^{O_{B_0}(1)} y_0^{-1/\log{z_0}}.
\end{align*}
Since $y_0^{-1/\log{z_0}}\ll_A (\log{z_0})^{-A}$ for every $A>0$, this gives the result.
\end{proof}
%
%
%
%
%
%
%
%
%
%
%
%
%
%
%
%
%
%
\section{Dispersion estimates}\label{sec:Dispersion}
In this section we perform the initial steps of the Linnik dispersion method (see \cite{Linnik}) to reduce the main problem to that of estimating certain exponential sums. The argument is very similar to that of \cite[\S3-\S7]{BFI1}, but with later applications manipulations in mind it is important for us to treat a slightly more general setup.
%
%
%
%
%
%
%
%
\begin{lmm}[First dispersion sum]\label{lmm:Dispersion1}
Let $\eta_{q,d,r}$ and $\alpha_n$ be bounded complex sequences such that $|\eta_{q,d,r}|\le \tau(q d r)^{B_0}$ and $|\alpha_n|\le \tau(n)^{B_0}$. Let
\begin{align*}
\mathscr{S}_{1}&:=\sum_{e\sim E}\mu^2(e) \sum_{\substack{q\\ (q,a)=1}}\sum_{\substack{d\sim D\\ (d,a)=1}}\sum_{\substack{r_1,r_2\sim R\\ (r_1r_2,a e)=1}}\psi_0\Bigl(\frac{q}{Q}\Bigr)\frac{\phi(q d e)\eta_{q,d,r_1}\overline{\eta_{q,d,r_2}} }{\phi(q d)\phi(q d e r_1)\phi(q d e r_2)}\\
&\qquad\times \sum_{\substack{n_1,n_2\sim N\\ (n_1,q d e r_1)=1\\(n_2,q d e r_2)=1}}\alpha_{n_1}\overline{\alpha_{n_2}}\sum_{\substack{m\\ (m,q d r_1 r_2)=1}} \psi_0\Bigl(\frac{m}{M}\Bigr),
\end{align*}
and for $C=C(A,B_0)$ sufficiently large in terms of $A$ and $B_0$, let
\[
M>(\log{x})^C.
\]
Then
\begin{equation*}
\mathscr{S}_{1}=\mathscr{S}_{MT}+O_{A,B_0}\Bigl(\frac{MN^2}{Q D(\log{x})^{2A}}\Bigr),\label{eq:S21Estimate}
\end{equation*}
where
\begin{align*}
\mathscr{S}_{MT}&:=M\hat{\psi}_0(0)\sum_{e\sim E}\mu^2(e)\sum_{\substack{q\\ (q,a)=1}}\psi_0\Bigl(\frac{q}{Q}\Bigr)\sum_{\substack{d\sim D\\ (d,a)=1}}\sum_{\substack{r_1,r_2\sim R\\ (r_1r_2,ae )=1}}\\
&\qquad\times\frac{\phi(q d e)\eta_{q,d,r_1}\overline{\eta_{q,d,r_2}} \phi(q d r_1r_2)}{\phi(q d)\phi(q d e r_1)\phi(q d e r_2)q d r_1r_2}\sum_{\substack{n_1,n_2\sim N\\ (n_1,q d e r_1)=1\\ (n_2,q d e r_2)=1}}\alpha_{n_1}\overline{\alpha_{n_2}}.
\end{align*}
\end{lmm}
\begin{proof}
We consider the inner sum of $\mathscr{S}_1$. By Lemma \ref{lmm:TrivialCompletion} we have that
\[
\sum_{\substack{m\\ (m,q d r_1r_2)=1}} \psi_0\Bigl(\frac{m}{M}\Bigr)=\frac{\phi(q d r_1 r_2)}{q d r_1 r_2}M\hat{\psi}_0(0)+O(\tau(q d r_1r_2)).
\]
By Lemma \ref{lmm:Divisor}, the error term above contributes to $\mathscr{S}_1$ a total
\begin{align*}
&\ll \sum_{e\sim E}\sum_{q\asymp Q}\sum_{d\sim D}\sum_{r_1,r_2\sim R}\sum_{n_1,n_2\sim N}\frac{\tau( q d r_1r_2)^{2B_0+1}\tau(n_1)^{B_0}\tau(n_2)^{B_0}\log{x}}{Q^2 D^2 R^2 E}\ll \frac{N^2(\log{x})^{O_{B_0}(1)}}{Q D},
\end{align*}
which is $O_A(MN^2/(Q D(\log{x})^{2A}))$ provided
\begin{equation}
M>(\log{x})^C
\end{equation}
 and $C$ is sufficiently large in terms of $A$ and $B_0$. The main term above contributes $\mathscr{S}_{MT}$ to $\mathscr{S}_1$, and this gives the result.
 \end{proof}
 %
%
%
%
%
%
%
%
 \begin{lmm}[Second dispersion sum]\label{lmm:Dispersion2}
Let $\eta_{q,d,r}$ and $\alpha_n$ be complex sequences such that $|\eta_{q,d,r}|\le \tau(q d r)^{B_0}$ and $|\alpha_n|\le \tau(n)^{B_0}$ and such that $\alpha_n$ is supported on $P^-(n)\ge z_0$ and $\eta_{q, d,r}$ is supported on $P^-(r)\ge z_0$. Let
 \begin{align*}
 \mathscr{S}_{2}&:=\sum_{e\sim E}\mu^2(e)\sum_{\substack{q\\ (q,a)=1}}\sum_{\substack{d\sim D\\ (d,a)=1}}\sum_{\substack{r_1,r_2\sim R\\ (r_1r_2,ae)=1}}\psi_0\Bigl(\frac{q}{Q}\Bigr)\frac{\eta_{q,d,r_1}\overline{\eta_{q,d,r_2}} }{\phi(q d e r_2)}\\
 &\qquad\times\sum_{\substack{n_1,n_2\sim N\\ (n_1,q d e r_1)=1\\(n_2,q d e r_2)=1}}\alpha_{n_1}\overline{\alpha_{n_2}}\sum_{\substack{m\\ m n_1\equiv a\Mod{q d r_1}\\ (m,r_2)=1}} \psi_0\Bigl(\frac{m}{M}\Bigr).
 \end{align*}
 Then we have
\begin{equation*}
\mathscr{S}_{2}=\mathscr{S}_{MT}+O(M|\mathscr{E}_{1}|)+O_{A,B_0}\Bigl(\frac{MN^2}{Q D (\log{x})^{2A}}\Bigr),
\end{equation*}
where $\mathscr{S}_{MT}$ is as given by Lemma \ref{lmm:Dispersion1}, $H_1:=\frac{Q D R}{M}\log^5{x}$ and $\mathscr{E}_{1}$ is given by
\begin{align*}
\mathscr{E}_{1}&:=\sum_{e\sim E}\mu^2(e)\sum_{\substack{q\\ (q,a)=1}}\sum_{\substack{d\sim D\\ (d,a)=1}}\sum_{\substack{r_1,r_2\sim R\\ (r_1r_2,a)=1}}\psi_0\Bigl(\frac{q}{Q}\Bigr)\frac{\eta_{q,d,r_1}\overline{\eta_{q,d,r_2}} }{\phi(q d e r_2)q d r_1}\\
&\qquad \times\sum_{\substack{n_1,n_2\sim N\\ (n_1,q d e r_1)=1\\(n_2,q d e  r_2)=1}}\alpha_{n_1}\overline{\alpha_{n_2}}\sum_{1\le |h|\le H_1}\hat{\psi}_0\Bigl(\frac{h M}{q d r_1}\Bigr)e\Bigl( \frac{a h \overline{ n_1}}{q d r_1}\Bigr).
\end{align*}
 \end{lmm}
 \begin{proof}
For notational simplicity, we let $\ell=q d$ throughout the proof. First we remove the condition $(m,r_2)=1$ from $\mathscr{S}_2$, which introduces additional terms which contribute a total
\begin{align*}
&\ll \sum_{f>1}\sum_{e\sim E}\sum_{n_1,n_2\sim N}\sum_{\substack{m\ll M\\ f|m}}\sum_{\ell|m n_1-a}\sum_{r_1|m n_1-a}\sum_{\substack{r_2\sim R\\ f|r_2\\ P^-(r_2)\ge z_0}}\frac{\tau(n_1)^{B_0}\tau(n_2)^{B_0}\tau(\ell r_1r_2)^{2B_0}\log{x}}{Q D E R}\\
&\ll \sum_{f>z_0}\frac{N(\log{x})^{O_{B_0}(1)}}{Q D f}\sum_{n_1\sim N}\sum_{\substack{m\ll M\\ f|m}}\tau(m n_1-a)^{5B_0}\\
&\ll \frac{MN^2(\log{x})^{O_{B_0}(1)}}{Q D z_0}.
\end{align*}
Here we used the fact that $\eta_{q,d,r}$ is supported on $P^-(r)\ge z_0$ in the first line, and so this implies that any divisor $f>1$ of $r_2$ is bigger than $z_0$ in the second line, and we used Lemma \ref{lmm:Divisor} in the final line. This contribution is negligible compared with $MN^2/(Q D\log^{2A}{x})$, so it suffices to consider $\mathscr{S}_{2}$ without the constraint $(m,r_2)=1$. 

Similarly, we remove the constraint $(e,r_1r_2)=1$, which introduces additional terms which contribute
\begin{align*}
&\sum_{f_1f_2>z_0}\sum_{\substack{e\sim E\\ f_1f_2|e}}\sum_{n_1,n_2\sim N}\sum_{\substack{m\ll M\\ m n_1\equiv a\Mod{f_1}}}\sum_{\ell|m n_1-a}\sum_{\substack{r_1|m n_1-a\\ f_1|r_1}}\sum_{\substack{r_2\sim R_2\\ f_2|r_2}}\frac{\tau(n_1)^{B_0}\tau(n_2)^{B_0}\tau(\ell r_1r_2)^{2B_0}\log{x}}{Q D E R}\\
&\ll \sum_{f_1f_2>z_0}\frac{E}{f_1f_2}N^2\Bigl(\frac{M}{f_1}+1\Bigr)(\log{x})^{O_{B_0}(1)}\frac{R}{f_2}\frac{1}{Q D E R}\\
&\ll \frac{MN^2(\log{x})^{O_{B_0(1)}}}{Q D z_0}+\frac{N^2(\log{x})^{O_{B_0(1)}}}{Q D},
\end{align*}
which is acceptably small. Thus we can drop this constraint too.

By Lemma \ref{lmm:Completion}, we have
\[
\sum_{\substack{m\\ m n_1\equiv a\Mod{\ell r_1}}} \psi_0\Bigl(\frac{m}{M}\Bigr)=\frac{M}{\ell r_1 }\hat{\psi_0}(0)+\frac{M}{\ell r_1}\sum_{1\le |h|\le H_1}\hat{\psi}_0\Bigl(\frac{h M}{\ell r_1}\Bigr)e\Bigl( \frac{a h \overline{ n_1}}{\ell r_1}\Bigr)+O(x^{-100}),
\]
where
\[
H_1:=\frac{Q D R}{M} \log^5{x}.
\]
The final term above clearly contributes a negligible amount. The first term above contributes to $\mathscr{S}_2$ a total 
\[
M\hat{\psi_0}(0)\sum_{e\sim E}\mu^2(e) \sum_{\substack{q\\ (q,a)=1}}\sum_{\substack{d\sim D\\ (d,a)=1}}\sum_{\substack{r_1,r_2\sim R\\ (r_1r_2,a)=1}}\psi_0\Bigl(\frac{q}{Q}\Bigr)\frac{\eta_{q,d,r_1}\overline{\eta_{q,d,r_2}} }{\phi(q d e r_2)q d  r_1}\sum_{\substack{n_1,n_2\sim N\\ (n_1,q d e r_1)=1\\(n_2,q d e r_2)=1}}\alpha_{n_1}\overline{\alpha_{n_2}}.
\]
This differs from $\mathscr{S}_{MT}$ by a factor
\[
\frac{\phi(q de )\phi(q d r_1 r_2)}{\phi(q d)\phi(q d e r_1) r_2}
\]
in the summand, and misses the condition $(e,r_1r_2)=1$. Again, since $\eta_{q,d,r_2}$ is supported on $P^-(r_2)\ge z_0$, and $r_2\ll x$, we see that the factor above is 
\[
\prod_{p|r_1r_2}\Bigl(1+\frac{O(1)}{p}\Bigr)=\Bigl(1+O\Bigl(\frac{1}{z_0}\Bigr)\Bigr)^{O(\log{x})}=1+O\Bigl(\frac{\log{x}}{z_0}\Bigr).
\]
 The error term $O((\log{x})/z_0)$ contributes a total
\[
\ll \frac{M\log{x}}{z_0}\sum_{e\sim E}\sum_{r_1,r_2\sim R}\sum_{\ell\ll Q D}\frac{\log{x}}{Q^2 D^2 R^2 E}\sum_{n_1,n_2\sim N}\tau(n_1n_2\ell r_1r_2)^{6B_0}\ll \frac{MN^2(\log{x})^{O_{B_0}(1)}}{QD z_0}.
\]
Finally, we may reintroduce the condition $(e,r_1r_2)=1$ in exactly the same way that we originally removed it, at the cost of a negligible error. Thus, we see that
\begin{equation}
\mathscr{S}_{2}=\mathscr{S}_{MT}+O_{A,B_0}\Bigl(\frac{MN^2}{QD (\log{x})^{2A}}\Bigr)+O(M|\mathscr{E}_{1}|),\label{eq:S22Estimate}
\end{equation}
where $\mathscr{E}_{1}$ is the middle term contribution, given in the statement of the lemma. 
\end{proof}
%
%
%
%
%
%
%
%
\begin{lmm}[Third dispersion sum]\label{lmm:Dispersion3}
Let $\eta_{q,d,r}$ and $\alpha_n$ be complex sequences such that $|\eta_{q,d,r}|\le \tau(q d r)^{B_0}$ and $|\alpha_n|\le \tau(n)^{B_0}$, $\eta_{q,d,r}$ is supported on $P^-(r)\ge z_0$, $\alpha_n$ is supported on $P^-(n)\ge z_0$ and such that $\alpha_n$ satisfies the Siegel-Walfisz condition \eqref{eq:SiegelWalfisz}. Let
\begin{align*}
\mathscr{S}_{3}&:=\sum_{e\sim E}\mu^2(e)\sum_{\substack{q\\ (q,a)=1}}\psi_0\Bigl(\frac{q}{Q}\Bigr)\sum_{\substack{d\sim D\\ (d,a)=1}}\sum_{\substack{r_1,r_2\sim R\\ (r_1r_2,ae)=1}}\eta_{q,d,r_1}\overline{\eta_{q,d,r_2}}\\
&\qquad\times \sum_{\substack{n_1,n_2\sim N\\ n_1\equiv n_2\Mod{qd e}\\ (n_1,q d e r_1)=1\\(n_2,q d e r_2)=1}}\alpha_{n_1}\overline{\alpha_{n_2}}\sum_{\substack{m\\ m n_1\equiv a\Mod{q d r_1}\\ m n_2\equiv a\Mod{q d r_2}}} \psi_0\Bigl(\frac{m}{M}\Bigr).
\end{align*}
Let $C=C(A,B_0)$ be sufficiently large in terms of $A$ and $B_0$, and let
\begin{equation}
Q D E(\log{x})^{C}<N.
\label{eq:DispersionCond}
\end{equation}
Then we have
\begin{equation*}
\mathscr{S}_{3}=\mathscr{S}_{MT}+O(M|\mathscr{E}_2|)+O_{A,B_0}\Bigl(\frac{MN^2}{Q D(\log{x})^{2A}}\Bigr),
\end{equation*}
where $\mathscr{S}_{MT}$ is as given by Lemma \ref{lmm:Dispersion1}, $H_2:=\frac{Q D R^2}{M}\log^5{x}$ and where
\begin{align*}
\mathscr{E}_2&:=\sum_{e\sim E}\mu^2(e)\sum_{\substack{q\\ (q,a)=1}}\psi_0\Bigl(\frac{q}{Q}\Bigr)\sum_{\substack{d\sim D\\ (d,a)=1}}\sum_{\substack{r_1,r_2\sim R\\ (r_1,a r_2)=1\\ (r_2,a q d r_1)=1}}\frac{\eta_{q,d,r_1}\overline{\eta_{q,d,r_2}}}{q d r_1 r_2}\\
&\qquad \times\sum_{\substack{n_1,n_2\sim N\\ n_1\equiv n_2\Mod{q d e}\\ (n_1,n_2 e q d r_1)=1\\(n_2,n_1 e q d r_2)=1\\ |n_1-n_2|\ge N/(\log{x})^C}}\alpha_{n_1}\overline{\alpha_{n_2}}\sum_{1\le |h|\le H_2}\hat{\psi}_0\Bigl(\frac{h M}{q d r_1 r_2}\Bigr)e\Bigl(\frac{ah\overline{n_1r_2}}{q d r_1}+\frac{ah\overline{n_2 q d r_1}}{r_2}\Bigr).
\end{align*}
\end{lmm}
\begin{proof}
Again, for notational simplicity, we let $\ell=q d$ throughout the proof. First we wish to modify the the summation conditions slightly. We first remove terms with $|\alpha_{n_1}\alpha_{n_2}\eta_{q,d,r_1}\eta_{q,d,r_2}|\ge (\log{x})^{C_1}$ or $\tau(m n_1-a)\tau(m n_2-a)\ge (\log{x})^{C_1}$. These terms contribute a total
\begin{align*}
&\ll \sum_{e\sim E}\sum_{\ell\asymp Q D}\sum_{\substack{n_1,n_2\sim N\\ n_1\equiv n_2\Mod{\ell e}}}\sum_{\substack{m\sim M\\ m n_1\equiv a\Mod{\ell}\\ m n_2\equiv a\Mod{\ell}}}\sum_{r_1|m n_1-a}\\
&\qquad\times\sum_{r_2|m n_2-a}\frac{\tau(m n_1-a)\tau(m n_2-a)(\tau(n_1)\tau(n_2)\tau(\ell)^{2}\tau(r_1)\tau(r_2))^{2B_0}}{(\log{x})^{C_1} }\\
&\ll \sum_{n_1\sim N}\sum_{\substack{m\sim  M}}\sum_{\substack{\ell\asymp Q D\\ \ell|m n_1-a}}\frac{\tau(\ell)^{4B_0}\tau(m n_1-a)^{4B_0+2}\tau(n_1)^{4B_0}}{(\log{x})^{C_1} }\sum_{e\sim E}\sum_{\substack{n_2\sim N\\ n_2\equiv n_1\Mod{\ell e}}}1\\
&\ll E\Bigl(\frac{N}{Q D E}+1\Bigr)\sum_{n_1\sim N}\sum_{m\sim M}\frac{\tau(m n_1-a)^{8B_0+3}\tau(n_1)^{4B_0}}{(\log{x})^{C_1} }\\
&\ll \frac{MN^2(\log{x})^{O_{B_0}(1)-C_1}}{Q D}+E MN(\log{x})^{O_{B_0}(1)-C_1}.
\end{align*}
Here we used Cauchy-Schwarz and the symmetry of $n_1$ and $n_2$ to pass to the second line, and used Lemma \ref{lmm:Divisor} in the final line. Thus, if we fix $C_1=C_1(A,B_0)$ to be sufficiently large in terms of $A$ and $B_0$, by \eqref{eq:DispersionCond} these terms contribute $O_{A,B_0}(MN^2/(Q D\log^{2A}{x}))$. We denote the restriction to $|\alpha_{n_1}\alpha_{n_2}\eta_{q,d,r_1}\eta_{q,d,r_2}|\le (\log{x})^{C_1}$ and $\tau(m n_1-a)\tau(m n_2-a)\le (\log{x})^{C_1}$ by $\sum^*$.

Since $\beta_n$ and $\eta_{q,d,r}$ are supported on integers with $P^-(n),P^-(r)\ge z_0$, we expect that $n_1$ and $n_2$ should be typically coprime to one another, and similarly $r_1$ and $r_2$ and $q d$ and $r_2$. Indeed, the contribution with $(r_1,r_2)>1$ to $\mathscr{S}_{3}$ is
\begin{align*}
&\sum_{f>z_0}\sum_{e\sim E}\sum_{\substack{r_1,r_2\sim R\\ f|r_1,\,f|r_2}}\sum_{\ell\asymp Q D}\sum_{\substack{n_1,n_2\sim N\\ n_1\equiv n_2\Mod{\ell e}}}\,\sideset{}{^*}\sum_{\substack{m\ll M\\ m n_1 \equiv a\Mod{\ell r_1}\\ m n_2\equiv a\Mod{\ell r_2} }}(\log{x})^{C_1}\\
&\ll\sum_{\substack{n_1\sim N}}\,\sideset{}{^*}\sum_{m\sim M}\sum_{\substack{f>z_0\\ f|m n_1-a}}\sum_{e\sim E}\sum_{\substack{\ell\asymp Q D\\ \ell|m n_1-a}}\,\,\sideset{}{^*}\sum_{\substack{n_2\sim N\\ n_1\equiv n_2\Mod{\ell [f, e]}}}\sum_{r_1'|m n_1-a}\sum_{r_2'|m n_2-a}(\log{x})^{C_1}\\
&\ll \sum_{n_1\sim N}\,\sideset{}{^*}\sum_{m\sim M}\sum_{\substack{f>z_0\\ f|mn_1-a}}\sum_{e\sim E}\Bigl(\frac{(e,f)N}{Q D E z_0}+1\Bigr)(\log{x})^{3C_1}\\
&\ll MN\Bigl(\frac{N}{Q Dz_0}+E\Bigr)(\log{x})^{5C_1}\\
&\ll \frac{MN^2(\log{x})^{5C_1} }{Q D z_0}+EMN(\log{x})^{5C_1}.
\end{align*}
Thus, assuming \eqref{eq:DispersionCond} holds with $C>5C_1+2A$, this is $O_{A,B_0}(MN^2/(Q D\log^{2A}{x}))$, and so negligible. Similarly, we restrict to $(n_1,n_2)=1$. The contribution to $\mathscr{S}_{3}$ from terms with $(n_1,n_2)>1$ is
\begin{align*}
&\ll \sum_{f>z_0}\sum_{e\sim E}\sum_{\substack{n_1\sim N\\ f|n_1}}\,\sideset{}{^*}\sum_{m\sim  M}\sum_{\substack{\ell\asymp Q D\\ \ell|mn_1-a}}\,\,\sideset{}{^*}\sum_{\substack{n_2\sim N\\ n_1\equiv n_2\Mod{\ell [e,f]}}}\sum_{r_1|m n_1-a}\sum_{r_2|m n_2-a}(\log{x})^{C_1}\\
&\ll \sum_{f>z_0}\sum_{\substack{n_1\sim N\\ f|n_1}}\sum_{m\ll  M}\sum_{e\sim E}\Bigl(\frac{(e,f)N}{Q D E z_0}+1\Bigr)(\log{x})^{3C_1}\\
&\ll  \frac{MN^2(\log{x})^{3C_1+2}}{Q D z_0}+EMN(\log{x})^{3C_1+2}.
\end{align*}
Similarly, we restrict to $(qd,r_2)=1$. Terms with $(q d,r_2)>1$ contribute a total
\begin{align*}
&\ll \sum_{e\sim E}\sum_{\substack{n_1\sim N}}\,\sideset{}{^*}\sum_{m\sim  M}\sum_{\substack{\ell\asymp Q D\\ \ell|m n_1-a\\ (\ell,a)=1}}\sum_{\substack{f>z_0\\ f|\ell}}\,\sideset{}{^*}\sum_{\substack{n_2\sim N\\ m n_2\equiv a\Mod{\ell f}\\ n_2\equiv n_1\Mod{\ell e}}}\sum_{r_1|m n_1-a}\sum_{r_2'|m n_2-a}(\log{x})^{C_1}\\
&\ll \sum_{n_1\sim N}\,\sideset{}{^*}\sum_{m\sim M}\sum_{f|mn_1-a}\sum_{e\sim E}\Bigl(\frac{N(e,f)}{Q D E z_0}+1\Bigr)(\log{x})^{3C_1}\\
&\ll  \frac{MN^2(\log{x})^{5C_1}}{Q D z_0}+E MN(\log{x})^{5C_1}.
\end{align*}
Similarly, we restrict to $|n_1-n_2|\ge N/(\log{x})^C$. Terms with $|n_1-n_2|\le N/(\log{x})^C$ contribute a total
\begin{align*}
&\ll \sum_{e\sim E}\sum_{\substack{n_1\sim N}}\,\sideset{}{^*}\sum_{m\sim  M}\sum_{\substack{\ell\asymp Q D\\ \ell|m n_1-a}}\,\sideset{}{^*}\sum_{\substack{n_2\sim N\\ n_1\equiv n_2\Mod{\ell e}\\ |n_1-n_2|\le N/(\log{x})^C}}\sum_{r_1|m n_1-a}\sum_{r_2|m n_2-a}(\log{x})^{C_1}\\
&\ll ENM(\log{x})^{C_1}\Bigl(\frac{N}{Q D E(\log{x})^C}+1\Bigr)(\log{x})^{2C_1}\\
&\ll  \frac{MN^2(\log{x})^{3C_1}}{Q D (\log{x})^C}+E MN(\log{x})^{3C_1}.
\end{align*}
Finally, we remove the condition $(e,r_1r_2)=1$. Terms with $(e_1,r_2)>1$ contribute
\begin{align*}
&\sum_{f>z_0}\sum_{\substack{e\sim E\\ f|e}}\sum_{n_1\sim N}\,\sideset{}{^*}\sum_{m\sim M}\sum_{\substack{\ell\asymp Q D\\ \ell|m n_1-a}}\sum_{\substack{n_2\sim N\\ n_2\equiv n_1\Mod{\ell e}\\ m n_2\equiv a \Mod{\ell f}}}\sum_{r_1|m n_1-a}\sum_{r_2'|m n_2-a}(\log{x})^{C_1}\\
&\ll  \,\sideset{}{^*}\sum_{m\sim M}\sum_{n_1\sim N}\sum_{\substack{f\\ f|m n_1-a}}\frac{E}{z_0}\Bigl(\frac{N}{QDE}+1\Bigr)(\log{x})^{3C_1}\\
&\ll \frac{MN^2(\log{x})^{4C_1}}{Q D z_0}+\frac{E MN(\log{x})^{4C_1}}{z_0}.
\end{align*}
Provided $C>5C_1+2A+2$ all of these are negligible. After making these restrictions, we re-insert terms with $\tau(m n_1-a)\tau(m n_2a)\ge (\log{x})^{C_1}$ and $|\alpha_{n_1}\alpha_{n_2}\eta_{q,d,r_1}\eta_{q,d,r_2}|\ge(\log{x})^{C_1}$, which we have already seen is a negligible contribution. Thus, assuming \eqref{eq:DispersionCond} holds with $C>5C_1+2A+2$, we see that 
\begin{equation}
\mathscr{S}_{3}=\mathscr{S}_{4}+O_{A,B_0}\Bigl(\frac{MN^2}{Q D (\log{x})^{2A}}\Bigr),\label{eq:S23Estimate}
\end{equation}
where
\begin{align*}
\mathscr{S}_{4}:=\sum_{e\sim E}\mu^2(e)\sum_{\substack{q\\ (q,a)=1}}\psi_0\Bigl(\frac{q}{Q}\Bigr)\sum_{\substack{d\sim D\\ (d,a)=1}}&\sum_{\substack{r_1,r_2\sim R\\ (r_1,a r_2)=1\\ (r_2,a q d r_1)=1}}\eta_{q,d,r_1}\overline{\eta_{q,d,r_2}} \sum_{\substack{n_1,n_2\sim N\\ (n_1,n_2 q d e r_1)=1\\(n_2,n_1 q d e r_2)=1\\ |n_1-n_2|\ge N/(\log{x})^C\\ n_1\equiv n_2\Mod{qde}}}\alpha_{n_1}\overline{\alpha_{n_2}}\\
&\qquad\times\sum_{\substack{m\\ m n_1\equiv a\Mod{q d r_1}\\ m n_2\equiv a\Mod{q d r_2}}} \psi_0\Bigl(\frac{m}{M}\Bigr).
\end{align*}
We see since $n_1\equiv n_2\Mod{\ell}$, then $m$ is fixed to lie in a certain residue class $b\Mod{\ell r_1r_2}$ where $b\equiv a\overline{n_1}\Mod{\ell r_1}$ and $b\equiv a\overline{n_2}\Mod{r_2}$ by the Chinese Remainder Theorem and the fact that $(\ell r_1,r_2)=1$. In particular, we see that
\[
\frac{b}{\ell r_1 r_2}=\frac{a\overline{n_1 r_2}}{\ell r_1}+\frac{a \overline{n_2 \ell r_1}}{r_2}\Mod{1}.
\]
Thus, by Lemma \ref{lmm:Completion}, we have
\begin{align*}
\sum_{m\equiv b\Mod{\ell r_1 r_2}}&\psi_0\Bigl(\frac{m}{M}\Bigr)=\frac{M}{\ell r_1 r_2}\hat{\psi}_0(0)\\
&+\frac{M}{\ell r_1 r_2}\sum_{1\le |h|\le H_2}\hat{\psi}_0\Bigl(\frac{h M}{\ell r_1 r_2}\Bigr)e\Bigl(ah\Bigl(\frac{\overline{n_1r_2}}{\ell r_1}+\frac{\overline{n_2 \ell r_1}}{r_2}\Bigr)\Bigr)+O(x^{-10}),
\end{align*}
where
\[
H_2:=\frac{Q D R^2}{M}\log^5{x}.
\]
The final term clearly makes a negligible contribution. The contribution to $\mathscr{S}_{4}$ from the first term is
\begin{align*}
M\hat{\psi}_0(0)\sum_{e\sim E}\mu^2(e)\sum_{\substack{q\\ (q,a)=1}}\psi_0\Bigl(\frac{q}{Q}\Bigr)\sum_{\substack{d\sim D\\ (d,a)=1}}\sum_{\substack{r_1,r_2\sim R\\ (r_1,a r_2)=1\\ (r_2,a q d r_1)=1}}\frac{\eta_{q,d,r_1}\overline{\eta_{q,d,r_2}} }{q d r_1 r_2}\hspace{-0.5cm}\sum_{\substack{n_1,n_2\sim N\\ n_1\equiv n_2\Mod{q d e}\\ (n_1,q d e r_1)=1\\(n_2,q d e r_2)=1\\ (n_1,n_2)=1\\ |n_1-n_2|\ge N/(\log{x})^C}}\hspace{-0.5cm}\alpha_{n_1}\overline{\alpha_{n_2}}.
\end{align*}
We multiply the summand above by a factor
\[
\frac{\phi(q d e)^2\phi(q d r_1 r_2)}{\phi(q d)\phi(q d e r_1)\phi(q d e r_2)}=\prod_{\substack{p|(r_1r_2,e)\\ p\nmid qd}}\Bigl(\frac{p-1}{p}\Bigr)\prod_{\substack{p|(r_1,r_2)\\ p\nmid qde}}\Bigl(\frac{p}{p-1}\Bigr)=1+O\Bigl(\frac{\log{x}}{z_0}\Bigr),
\]
where we have used the fact that $P^-(r)\ge z_0$ from the support of $\eta$. This introduces an error term which contributes 
\begin{align*}
&\ll \frac{M \log{x}}{z_0}\sum_{e\sim E}\sum_{\substack{\ell \asymp Q D}}\sum_{\substack{r_1,r_2\sim R}}\frac{(\tau(r_1)\tau(r_2)\tau(\ell)^2)^{B_0}}{Q D R^2}\sum_{\substack{n_1,n_2\sim N\\ n_1\equiv n_2\Mod{\ell e} \\ |n_1-n_2|\ge N/(\log{x})^C}}(\tau(n_1)\tau(n_2))^{B_0}\\
&\ll \frac{MN^2(\log{x})^{O_{B_0}(1)}}{Q D z_0}+\frac{E M N(\log{x})^{O_{B_0}(1)}}{z_0}.
\end{align*}
This is negligible assuming \eqref{eq:DispersionCond} holds with $C$ sufficiently large in terms of $B_0$ and $A$. Thus we see that
\begin{equation}
\mathscr{S}_{4}=\mathscr{S}_5+O(M|\mathscr{E}_2|)+O_{A,B_0}\Bigl(\frac{MN^2}{Q D (\log{x})^{2A}}\Bigr), \label{eq:S23*Estimate}
\end{equation}
where $\mathscr{E}_2$ is is given by the statement of the lemma, and where
\begin{align*}
\mathscr{S}_5&:=M\hat{\psi}_0(0)\sum_{e\sim E}\mu^2(e)\sum_{\substack{q\\ (q,a)=1}}\psi_0\Bigl(\frac{q}{Q}\Bigr)\sum_{\substack{d\sim D\\ (d,a)=1}}\sum_{\substack{r_1,r_2\sim R\\ (r_1,a r_2)=1\\ (r_2,a q d r_1)=1}}\\
&\qquad\times\frac{\eta_{q,d,r_1}\overline{\eta_{q,d,r_2}} \phi(q d e)^2\phi(q d r_1 r_2)}{\phi(q d)\phi(q d e r_1)\phi(q d e r_2)q d r_1 r_2}\sum_{\substack{n_1,n_2\sim N\\ n_1\equiv n_2\Mod{q d e}\\ (n_1,n_2 q d e r_1)=1\\(n_2,n_1 q d e r_2)=1}}\alpha_{n_1}\overline{\alpha_{n_2}}.
\end{align*}
We see that we may drop the conditions $(r_2 n_1,r_1 n_2)=1$ and $(r_2,\ell r_1)=1$ in $\mathscr{S}_5$ at the cost of an error of size
\[
\sum_{z_0<f<x}\frac{EMN(\log{x})^{O_{B_0}(1)}}{f} \Bigl(\frac{N}{f Q D E}+1\Bigr)\ll  \frac{MN^2(\log{x})^{O_{B_0}(1)}}{Q D z_0}+E MN(\log{x})^{O_{B_0}(1)}.
\]
Similarly, we may drop the conditions $|n_1-n_2|\ge N/(\log{x})^C$ at the cost of an error of size 
\[
\ll \frac{MN^2(\log{x})^{O_{B_0}(1)}}{Q D(\log{x})^{C}}+E M N(\log{x})^{O_{B_0}(1)},
\]
and we may re-introduce the condition $(e,r_1 r_2)=1$ in $\mathscr{S}_5$ at the cost of an error term of size
\[
\ll \frac{MN^2(\log{x})^{O_{B_0}(1)}}{Q D(\log{x})^{C}}+E M N(\log{x})^{O_{B_0}(1)},
\]
 We then see that if the condition $n_1\equiv n_2\Mod{q d e}$ was replaced by a factor $1/\phi(q d e)$ then we would obtain $\mathscr{S}_{MT}$ plus the contribution from some terms with $(e,r_1r_2)(n_1,n_2)(r_2,q d r_1)>1$, which we have just seen are negligible. In particular, this implies that
\begin{equation}
\mathscr{S}_5=\mathscr{S}_{MT}+\mathscr{E}_3+O_A\Bigl(\frac{MN^2}{Q D (\log{x})^{2A}}\Bigr),\label{eq:S4Estimate}
\end{equation}
where $\mathscr{E}_3$ is given by
\begin{align*}
&M\hat{\psi}_0(0)\sum_{e\sim E}\mu^2(e)\sum_{\substack{q\\ (q,a)=1}}\psi_0\Bigl(\frac{q}{Q}\Bigr)\sum_{\substack{d\sim D\\ (d,a)=1}}\sum_{\substack{r_1,r_2\sim R\\ (r_1r_2,a)=1}}\frac{\eta_{q,d,r_1}\overline{\eta_{q,d,r_2}} \phi(q d)\phi(q d r_1 r_2)}{\phi(q d r_1)\phi(q d r_2)\ell r_1 r_2}\sum_{\substack{b\Mod{q d e}\\ (b,q d e)=1}}\\
&\qquad\times\sum_{n_1,n_2\sim N}\alpha_{n_1}\overline{\alpha_{n_2}}\Bigl(\mathbf{1}_{n_1\equiv b\Mod{q d e}}-\frac{\mathbf{1}_{(n_1,q d e)=1}}{\phi(q d e)}\Bigr)\Bigl(\mathbf{1}_{n_2\equiv b\Mod{q d e}}-\frac{\mathbf{1}_{(n_2,q d e)=1}}{\phi(q d e)}\Bigr).
\end{align*}
Trivially bounding the $r_1,r_2$ summation and letting $s=q d e$, we see that this simplifies to give the bound
\[
\mathscr{E}_3\ll \frac{M(\log{x})^{O_{B_0}(1)}}{Q D}\sum_{\substack{s \le 8 Q D E }}\tau(s)^{3B_0}\sum_{\substack{b\Mod{s}\\ (b,s)=1}}\Bigl|\sum_{n\sim N}\alpha_{n}\Bigl(\mathbf{1}_{n\equiv b\Mod{s}}-\frac{\mathbf{1}_{(n,s)=1}}{\phi(s)}\Bigr)\Bigr|^2.
\]
Since $\alpha_n$ satisfies the Siegel-Walfisz condition \eqref{eq:SiegelWalfisz}, by Lemma \ref{lmm:BarbanDavenportHalberstam} we see that (using \eqref{eq:DispersionCond}) we have
\begin{equation}
\mathscr{E}_3\ll_{A,B_0} \frac{MN^2}{Q D(\log{x})^{2A}}.\label{eq:E4Bound}
\end{equation}
Putting together \eqref{eq:S23Estimate}, \eqref{eq:S23*Estimate}, \eqref{eq:S4Estimate} and \eqref{eq:E4Bound} then gives the result.
\end{proof}
%
%
%
%
%
%
%
%
\begin{prpstn}[Reduction to exponential sums]\label{prpstn:GeneralDispersion}
Let $\alpha_n,\beta_m,\gamma_{q,d},\lambda_{q,d,r}$ be complex sequences with $|\alpha_n|,|\beta_n|\le \tau(n)^{B_0}$ and $|\gamma_{q,d}|\le \tau(q d)^{B_0}$ and $|\lambda_{q,d,r}|\le \tau(q d r)^{B_0}$. Let $\alpha_n$ and $\lambda_{q,d,r}$ be supported on integers with $P^-(n)\ge z_0$ and $P^-(r)\ge z_0$, and let $\alpha_n$ satisfy the Siegel-Walfisz condition \eqref{eq:SiegelWalfisz}. Let
\[
\mathscr{S}:=\sum_{\substack{d\sim D\\ (d,a)=1}}\sum_{\substack{q\sim Q\\ (q,a)=1}}\sum_{\substack{r\sim R\\ (r,a)=1}}\lambda_{q,d,r}\gamma_{q,d}\sum_{m\sim M}\beta_m\sum_{n\sim N}\alpha_n\Bigl(\mathbf{1}_{m n\equiv a\Mod{q r d}}-\frac{\mathbf{1}_{(m n,q r d)=1}}{\phi(q r d)}\Bigr).
\]
Let $A>0$, $1\le E\le x$ and $C=C(A,B_0)$ be sufficiently large in terms of $A,B_0$, and let $N,M$ satisfy
\[
N>Q D E(\log{x})^{C},\qquad M>(\log{x})^C.
\]
Then we have
\[
|\mathscr{S}|\ll_{A,B_0} \frac{x}{(\log{x})^A}+M D^{1/2}Q^{1/2}(\log{x})^{O_{B_0}(1)}\Bigl(|\mathscr{E}_1|^{1/2}+|\mathscr{E}_2|^{1/2}\Bigr),
\]
where
\begin{align*}
\mathscr{E}_{1}&:=\sum_{e\sim E}\mu^2(e)\sum_{\substack{q\\ (q,a)=1}}\sum_{\substack{d\sim D\\ (d,a)=1}}\sum_{\substack{r_1,r_2\sim R\\ (r_1r_2,a)=1}}\psi_0\Bigl(\frac{q}{Q}\Bigr)\frac{\lambda_{q,d,r_1}\overline{\lambda_{q,d,r_2}} }{\phi(q d e r_2)q d r_1}\\
&\qquad \times\sum_{\substack{n_1,n_2\sim N\\ (n_1,q d e r_1)=1\\(n_2,q d e  r_2)=1}}\alpha_{n_1}\overline{\alpha_{n_2}}\sum_{1\le |h|\le H_1}\hat{\psi}_0\Bigl(\frac{h M}{q d r_1}\Bigr)e\Bigl( \frac{a h \overline{ n_1}}{q d r_1}\Bigr),\\
\mathscr{E}_2&:=\sum_{e\sim E}\mu^2(e)\sum_{\substack{q\\ (q,a)=1}}\psi_0\Bigl(\frac{q}{Q}\Bigr)\sum_{\substack{d\sim D\\ (d,a)=1}}\sum_{\substack{r_1,r_2\sim R\\ (r_1,a r_2)=1\\ (r_2,a q d r_1)=1}}\frac{\lambda_{q,d,r_1}\overline{\lambda_{q,d,r_2}}}{q d r_1 r_2}\\
&\qquad \times\sum_{\substack{n_1,n_2\sim N\\ n_1\equiv n_2\Mod{q d e}\\ (n_1,n_2 e q d r_1)=1\\(n_2,n_1 e q d r_2)=1\\ |n_1-n_2|\ge N/(\log{x})^C}}\alpha_{n_1}\overline{\alpha_{n_2}}\sum_{1\le |h|\le H_2}\hat{\psi}_0\Bigl(\frac{h M}{q d r_1 r_2}\Bigr)e\Bigl(\frac{ah\overline{n_1r_2}}{q d r_1}+\frac{ah\overline{n_2 q d r_1}}{r_2}\Bigr),\\
H_1&:=\frac{Q D R}{M}\log^5{x},\\
H_2&:=\frac{Q D R^2}{M}\log^5{x}.
\end{align*}
\end{prpstn}
%
%
%
%
%
%
%
%
\begin{rmk}
To avoid confusion, we emphasize to the reader that in the above proposition (and elsewhere in the paper), $e$ is the integer in the outer summation, except for when we use the function $e(x)=e^{2\pi i x}$. The distinction between the two should be clear from the context, and in most situations when this is used the variable $e$ will just be the constant 1.
\end{rmk}
%
%
%
%
%
%
%
%
\begin{proof}
We first do some slightly artifical looking manipulations (which will be vital later on), where we split $\mathscr{S}$ according to residue classes for additional moduli $e\sim E$. (This corresponds to the variable labelled $c$ in Section \ref{sec:OutlineModified}.)

Given an integer $e\sim E$, since $\alpha_n$ is supported on $P^-(n)>z_0$, we see that the contribution from $(e,n)\ne 1$ to $\mathscr{S}$ is 
\begin{align*}
&\ll \sum_{\substack{f>z_0\\ f|e}}\sum_{d\sim D}\sum_{q\sim Q}\sum_{r\sim R}\sum_{m\sim M}\sum_{\substack{n\sim N\\ f|n\\ (n m,q d r)=1}}\tau(q d r n m)^{2B_0}\Bigl(\mathbf{1}_{m n\equiv a\Mod{q d r}}+\frac{1}{\phi(q d r)}\Bigr)\\
&\ll  \sum_{\substack{f>z_0\\ f|e}}D Q R \frac{x \tau(f)^{O_{B_0}(1)}}{Q D R f}(\log{x})^{O_{B_0}(1)}\\
&\ll \frac{x(\log{x})^{O_{B_0(1)}}}{z_0^{1/2}}\sum_{f|e}\frac{\tau(f)^{O_{B_0}(1)}}{f^{1/2}}\\
&\ll \frac{x(\log{x})^{O_{B_0(1)}}}{z_0^{1/2}}\exp(O_{B_0}((\log{x})^{1/2})).
\end{align*}
Similarly, the contribution with $(e,r)\ne 1$ is
\begin{align*}
&\ll \sum_{\substack{f>z_0\\ f|e}}\sum_{d\sim D}\sum_{q\sim Q}\sum_{\substack{r\sim R\\ f|r}}\sum_{m\sim M}\sum_{\substack{n\sim N}}\tau(q d r n m)^{2B_0}\Bigl(\mathbf{1}_{mn\equiv a\Mod{qdr}}+\frac{1}{\phi(qdr)}\Bigr)\\
&\ll x(\log{x})^{O_{B_0(1)}}\sum_{\substack{f>z_0\\ f|e}}\frac{\tau(f)^{O_{B_0}(1)}}{f}\\
&\ll \frac{x(\log{x})^{O_{B_0(1)}}}{z_0^{1/2}}\exp(O_{B_0}((\log{x})^{1/2}))
\end{align*}
Since $z_0=x^{1/(\log\log{x})^3}$, we see that these error terms are negligible, and so we may insert the conditions $(n,e)=1$ and $(r,e)=1$ at a negligible cost. We see that all terms are supported only on $(m,q d)=1$, and so we also insert this explicitly. We now split the terms according to the residue class of $n\Mod{qde}$. Together this gives
\begin{align*}
\mathscr{S}&=\sum_{\substack{d\sim D\\ (d,a)=1}}\sum_{\substack{q\sim Q\\ (q,a)=1}}\sum_{\substack{r\sim R\\ (r,a e)=1}}\lambda_{q,d,r}\gamma_{q,d}\sum_{\substack{m\sim M\\ (m,q d)=1}}\beta_m\sum_{\substack{b\Mod{q d e}\\ b\equiv a\overline{m}\Mod{q d}\\ (b,q d e )=1}}\\
&\qquad \times\sum_{\substack{n\sim N\\ (n,e)=1}}\alpha_n\Bigl(\mathbf{1}_{\substack{m n\equiv a\Mod{q r d}\\ n\equiv b\Mod{q d e}}}-\frac{\mathbf{1}_{(m n,q r d)=1}}{\phi(q r d e)}\Bigr)\\
&\qquad+O_A\Bigl(\frac{x}{(\log{x})^A}\Bigr).
\end{align*}
(Here we used the fact that there are precisely $\phi(q d e)/\phi(q d)$ choices of $b$ in the summation.)

We now average this expression over squarefree $e\sim E$. This gives
\begin{align*}
\mathscr{S}&=\frac{1}{\mathscr{E}_0}\sum_{e\sim E}\mu^2(e)\sum_{\substack{d\sim D\\ (d,a)=1}}\sum_{\substack{q\sim Q\\ (q,a)=1}}\sum_{\substack{r\sim R\\ (r,a e)=1}}\lambda_{q,d,r}\gamma_{q,d}\sum_{m\sim M}\beta_m\sum_{\substack{b\Mod{q d e}\\ b\equiv a\overline{m}\Mod{q d}\\ (b,q d e)=1}}\\
&\qquad\times\sum_{\substack{n\sim N\\ (n,e)=1}}\alpha_n\Bigl(\mathbf{1}_{\substack{m n\equiv a\Mod{q r d}\\ n\equiv b\Mod{q d e}}}-\frac{\mathbf{1}_{(m n,q r d)=1}}{\phi(q r d e)}\Bigr)\\
&\qquad +O_A\Bigl(\frac{x}{(\log{x})^A}\Bigr),
\end{align*}
where
\[
\mathscr{E}_0:=\sum_{e\sim E} \mu^2(e)\asymp E.
\]
We now apply Cauchy-Schwarz in the $e$, $b$, $m$, $q$ and $d$ variables, which allows us to eliminate the $\beta_m$ and $\gamma_{q,d}$ coefficients. This gives
\begin{align*}
|\mathscr{S}|^2
&\le \frac{1}{\mathscr{E}_0^2}\Bigl(\sum_{e\sim E}\sum_{q\sim Q} \sum_{d\sim D}\sum_{\substack{m\sim M\\ (m,q d)=1}}\sum_{\substack{b\Mod{q d e}\\ b\equiv a\overline{m}\Mod{q d}}}|\gamma_{q,d}|^2|\beta_m|^2\Bigr)  \mathscr{S}'+O_A\Bigl(\frac{x^2}{(\log{x})^{2A}}\Bigr)\\
&\ll M Q D(\log{x})^{O_{B_0}(1)} \mathscr{S}'+O_A\Bigl(\frac{x^2}{(\log{x})^{2A}}\Bigr),
\end{align*}
where 
\begin{align*}
\mathscr{S}'&:=\sum_{e\sim E}\mu^2(e)\sum_{\substack{q\sim Q\\ (q,a)=1}}\sum_{\substack{d\sim D\\ (d,a)=1}}\sum_{m\sim M}\sum_{\substack{b\Mod{q d e}\\ b\equiv a\overline{m}\Mod{q d}\\ (b,q d e )=1}}\\
&\qquad\times\Bigl|\sum_{\substack{r\sim R\\ (r,a e)=1}}\lambda_{q,d,r}\sum_{\substack{n\sim N\\ (n,e)=1}}\alpha_n\Bigl(\mathbf{1}_{\substack{m n\equiv a\Mod{q d r}\\ n\equiv b\Mod{q d e}}}-\frac{\mathbf{1}_{(m n,qd r)=1}}{\phi(q d r e)}\Bigr)\Bigr|^2.
\end{align*}
We now extend the $m$ and $q$ summations using $\psi_0$ for an upper bound, and then open out the square. This gives
\begin{align}
\mathscr{S}'&\le \sum_{\substack{q\\ (q,a)=1}}\sum_{\substack{d\sim D\\ (d,a)=1}}\sum_m \psi_0\Bigl(\frac{q}{Q}\Bigr)\psi_0\Bigl(\frac{m}{M}\Bigr)\sum_{e\sim E}\mu^2(e)\sum_{\substack{b\Mod{q d e}\\ b\equiv a\overline{m}\Mod{q d}\\ (b,q d e )=1}}\nonumber\\
&\qquad\times \Bigl|\sum_{\substack{r\sim R\\ (r,a e)=1}}\lambda_{q,d,r}\sum_{\substack{n\sim N\\ (n,e)=1}}\alpha_n\Bigl(\mathbf{1}_{\substack{m n\equiv a\Mod{q d r}\\ n\equiv b\Mod{q d e}}}-\frac{\mathbf{1}_{(m n,qd r)=1}}{\phi(q d r e)}\Bigr)\Bigr|^2\nonumber \\
&=  \mathscr{S}_{1}-2\Re(\mathscr{S}_{2})+\mathscr{S}_{3},\label{eq:S1Expansion}
\end{align}
where, after performing the $b$ summation, we see that $\mathscr{S}_1,\mathscr{S}_2,\mathscr{S}_3$ are as given by Lemma \ref{lmm:Dispersion1}, \ref{lmm:Dispersion2} and \ref{lmm:Dispersion3} with $\eta_{q,d,r}:=\lambda_{q,d,r}$. These lemmas then give the result.
\end{proof}
%
%
%
%
%
%
%
%
\begin{lmm}[Simplification of exponential sum]\label{lmm:Simplification}
Let $N,M,Q,R \le x$ with $NM\asymp x$ and 
\begin{align}
Q R&<x^{2/3},\label{eq:CrudeSize}\\
Q R^2&< M x^{1-2\epsilon}.\label{eq:CrudeSize2}
\end{align}
Let $\lambda_{q,r}$ and $\alpha_n$ be complex sequences supported on $P^-(n),P^-(r)\ge z_0$ with $|\lambda_{q,r}|\le \tau(qr)^{B_0}$ and $|\alpha_n|\le \tau(n)^{B_0}$. Let $H:=\frac{Q R^2}{M}\log^5{x}$ and let
\begin{align*}
\mathscr{E}&:=\sum_{\substack{ (q,a)=1}}\psi_0\Bigl(\frac{q}{Q}\Bigr)\sum_{\substack{r_1,r_2\sim R\\ (r_1,a r_2)=1\\ (r_2,a q r_2)=1}}\frac{\lambda_{q,r_1}\overline{\lambda_{q,r_2}}}{q r_1 r_2}\sum_{\substack{n_1,n_2\sim N\\ n_1\equiv n_2\Mod{q}\\ (n_1,n_2qr_1)=1\\(n_2,n_1qr_2)=1\\ |n_1-n_2|\ge N/(\log{x})^C}}\alpha_{n_1}\overline{\alpha_{n_2}}\\
&\qquad\qquad \times\sum_{1\le |h|\le H}\hat{\psi}_0\Bigl(\frac{h M}{q r_1 r_2}\Bigr)e\Bigl(\frac{ah\overline{n_1 r_2}}{q r_1}+\frac{ah\overline{n_2 q r_1}}{r_2}\Bigr).
\end{align*}
Then we have (uniformly in $C$)
\[
\mathscr{E}\ll_{B_0}\exp((\log\log{x})^5)\sup_{\substack{H'\le H\\ Q'\le 2Q\\ R_1,R_2\le 2R}}|\mathscr{E}'|+\frac{N^2}{Qx^\epsilon},
\]
where
\[
\mathscr{E}'=\sum_{\substack{Q\le q\le Q'\\ (q,a)=1}}\sum_{\substack{R\le r_1\le  R_1\\ R\le r_2\le R_2\\ (r_1a r_2)=1\\ (r_2,a q r_1)=1}}\frac{\lambda_{q,r_1}\overline{\lambda_{q,r_2}}}{q r_1 r_2}\sum_{\substack{n_1,n_2\sim N\\ n_1\equiv n_2\Mod{q}\\ (n_1,qr_1n_2)=1\\ (n_2,qr_2n_1)=1\\ (n_1r_2,n_2)\in\mathcal{N}\\ |n_1-n_2|\ge N/(\log{x})^C}}\alpha_{n_1}\overline{\alpha_{n_2}}\sum_{1\le |h| \le H'} e\Bigl(\frac{ ah\overline{n_2 q r_1}(n_1-n_2)}{n_1 r_2}\Bigr),
\]
and $\mathcal{N}$ is a set with the property that if $(a,b)\in\mathcal{N}$ and $(a',b')\in\mathcal{N}$ then we have $\gcd(a,b')=\gcd(a',b)=1$.
\end{lmm}
\begin{proof}
We first wish to simplify the exponential term in the summand. By Lemma \ref{lmm:Bezout}, we have
\begin{align*}
\frac{ah\overline{n_1 r_2}}{q r_1}&=\frac{-ah\overline{q r_1}}{n_1 r_2}+\frac{ah}{q r_1 r_2 n_1}\Mod{1}.
\end{align*}
Since $|h|\le H=(QR\log^5{x})/M$, the final fraction is of size $O(\log^5{x}/x)$, and so see that
\[
e\Bigl(\frac{ah\overline{n_1r_2}}{q r_1}+\frac{ah \overline{n_2 qr_1}}{r_2}\Big)=e\Bigl(\frac{ ah\overline{n_2 q r_1}(n_1-n_2)}{n_1 r_2}\Bigr)+O\Bigl(\frac{\log^6{x}}{x}\Bigr).
\]
The error term above contributes to $\mathscr{E}$ a total
\[
\ll N\Bigl(\frac{N}{Q}+1\Bigr)\frac{Q R^2(\log{x})^{O_{B_0}(1)}}{M x}\ll \frac{N^2}{Q}(\log{x})^{O_{B_0}(1)}\Bigl(\frac{Q  R^2}{M x}+\frac{Q^2 R^2}{x^2}\Bigr).
\]
Therefore if \eqref{eq:CrudeSize} and \eqref{eq:CrudeSize2} hold, both terms in parentheses are $\ll_{B_0}x^{-2\epsilon}$. Thus we see that
\[
\mathscr{E}=\mathscr{E}_1+O_{B_0}\Bigl(\frac{N^2}{Q x^\epsilon}\Bigr),
\]
where
\begin{align*}
\mathscr{E}_1&:=\sum_{\substack{ (q,a)=1}}\psi_0\Bigl(\frac{q}{Q}\Bigr)\sum_{\substack{r_1,r_2\sim R\\ (r_1,a r_2)=1\\ (r_2,a q r_2)=1}}\frac{\lambda_{q,r_1}\overline{\lambda_{q,r_2}}}{q r_1 r_2}\sum_{\substack{n_1,n_2\sim N\\ n_1\equiv n_2\Mod{q}\\ (n_1,n_2qr_1)=1\\(n_2,n_1qr_2)=1\\ |n_1-n_2|\ge N/(\log{x})^C}}\alpha_{n_1}\overline{\alpha_{n_2}}\\
&\qquad\qquad\times\sum_{1\le |h|\le H}\hat{\psi}_0\Bigl(\frac{h M}{q r_1 r_2}\Bigr)e\Bigl(\frac{ ah\overline{n_2 q r_1}(n_1-n_2)}{n_1 r_2}\Bigr).
\end{align*}
We note that since $\hat{\psi_0}^{(j)}(\xi)\ll_{j,k} |\xi|^{-k}$ and $\psi_0$ is supported on $[1/2,5/2]$, we have that
\[
\frac{\partial^{k_1+k_2+k_3+k_4}}{\partial q^{k_1}\partial r_1{}^{k_2}\partial r_2{}^{k_3} \partial h^{k_4}}\psi_0\Bigl(\frac{q}{Q}\Bigr)\hat{\psi}_0\Bigl(\frac{h M}{q r_1 r_2}\Bigr)\ll_{k_1,k_2,k_3,k_4} \frac{1}{q^{k_1}r_1^{k_2} r_2^{k_3} h^{k_4}}.
\]
Thus we may remove the $\hat{\psi_0}$ term via partial summation, which shows that
\[
\mathscr{E}_1\ll \sup_{\substack{H'\le 2H\\ Q'\le 2Q\\ R_1,R_2\le 2R}}|\mathscr{E}_2|,
\]
where
\[
\mathscr{E}_2:=\sum_{\substack{Q\le q\le Q'\\ (q,a)=1}}\sum_{\substack{R\le r_1\le  R_1\\ R\le r_2\le R_2\\ (r_1,a r_2)=1\\ (r_2,a q r_1)=1}}\frac{\lambda_{q,r_1}\overline{\lambda_{q,r_2}}}{q r_1 r_2}\sum_{\substack{n_1,n_2\sim N\\ n_1\equiv n_2\Mod{q}\\ (n_1,qr_1n_2)=1\\ (n_2,qr_2n_1)=1\\ |n_1-n_2|\ge N/(\log{x})^C}}\alpha_{n_1}\overline{\alpha_{n_2'}}\sum_{1\le |h| \le H'} e\Bigl(\frac{ ah\overline{n_2 q r_1}(n_1-n_2)}{n_1 r_2}\Bigr).
\]
We now apply Lemma \ref{lmm:FouvryDecomposition}, recalling that $\alpha_n$ and $\lambda_{q,r}$ are supported on integers $q,n,r$ with $P^-(n),P^-(r)\ge z_0$, and so $n$ and $r$ have at most $(\log\log{x})^3$ prime factors. This gives us $\ll \exp((\log\log{x})^5)$ different sets $\mathcal{N}_1,\mathcal{N}_2,\dots\subseteq \mathbb{Z}^2$ which cover all possible values of pairs $(n_1r_2,n_2)$ such that if $(n_1r_2,n_2)$ and $(n_1'r_2',n_2')$ are in the same set $\mathcal{N}_j$ then we have $\gcd(n_1r_2,n_2')=\gcd(n_1'r_2',n_2)=1$. Thus, taking the worst such set $\mathcal{N}$, we find
\[
\mathscr{E}_2\ll \exp((\log\log{x})^5)\sup_{\mathcal{N}}\mathscr{E}',
\]
with $\mathscr{E}'$ given by the lemma. Putting everything together then gives the result.
\end{proof}
%
%
%
%
%
%
%
%
%
%
%
%
%
%
%
%
%
%
\section{Fouvry-style estimates near \texorpdfstring{$x^{1/7}$}{x\^{}(1/7)}}\label{sec:Fouvry}
In this section we establish Proposition \ref{prpstn:Fouvry}. The arguments here are a generalization of \cite{Fouvry}, making use of the fact that we consider only moduli with a conveniently sized prime factor to extend the results to a wider region. The case $Q_1=1$ is precisely the argument of \cite{Fouvry} (or \cite[Theorem 3]{BFI1}). The critical case for these estimates is handling convolutions when one factor is of size $x^{1/7}$ and the other of size $x^{6/7}$ or when one factor is of size $x^{1/21}$ and the other of size $x^{20/21}$.
\begin{lmm}[Deshouillers--Iwaniec]\label{lmm:DeshouillersIwaniec}
Let $b_{n,r,s}$ be a 1-bounded sequence and $R,S,N,D,C\ll x^{O(1)}$. Let $g(c,d)=g_0(c/C,d/D)$ where $g_0$ is a smooth function supported on $[1/2,5/2]\times [1/2,5/2]$. Then we have
\[
\sum_{r\sim R} \sum_{\substack{s\sim S\\ (r,s)=1}}\sum_{n\sim N}b_{n,r,s}\sum_{d\sim D}\sum_{\substack{c\sim C\\ (rd,sc)=1}}g(c,d) e\Bigl(\frac{n\overline{dr}}{cs}\Bigr)\ll_{g_0} x^\epsilon \Bigl(\sum_{r\sim R}\sum_{s\sim S}\sum_{n\sim N}|b_{n,r,s}|^2\Bigr)^{1/2}\mathscr{J}.
\]
where
\[
\mathscr{J}^2=CS(RS+N)(C+DR)+C^2 D S\sqrt{(RS+N)R}+D^2NR.
\]
\end{lmm}
\begin{proof}
This is \cite[Theorem 12]{DeshouillersIwaniec} (correcting a minor typo in the last term of $\mathscr{J}^2$ which is written as $D^2NR/S$).
\end{proof}
%
%
%
%
%
%
%
%
\begin{lmm}[Fouvry-style exponential sum estimate]\label{lmm:FouvryEstimate}
Let $N,Q,R$ satisfy
\begin{align}
N^6 Q^{3}R^2 &<x^{2-14\epsilon},\\
N^6 Q^3 R^4&<x^{3-14\epsilon},\\
N^{3}Q^{3}R^{3}&<x^{2-14\epsilon},\\
N^{3}Q^{3}R^{4}&<x^{5/2-14\epsilon},\\
Q R&<x^{1/2-\epsilon}N^{1/2}.
\end{align}
Let $H_2=(Q R^2\log^5{x})/M$ and $\lambda_{q,r}$ and $\alpha_n$ be 1-bounded complex sequences. Let
\begin{align*}
\mathscr{F}&:=\sum_{\substack{ (q,a)=1}}\psi_0\Bigl(\frac{q}{Q}\Bigr)\sum_{\substack{r_1,r_2\sim R\\ (r_1,a r_2)=1\\ (r_2,a q r_1)=1}}\frac{\lambda_{q,r_1}\overline{\lambda_{q,r_2}}}{q r_1 r_2}\sum_{\substack{n_1,n_2\sim N\\ n_1\equiv n_2\Mod{q}\\ (n_1,n_2qr_1)=1\\(n_2,n_1qr_2)=1\\ |n_1-n_2|\ge N/(\log{x})^C}}\alpha_{n_1}\overline{\alpha_{n_2}}\\
&\qquad\qquad\qquad\times\sum_{1\le |h|\le H_2}\hat{\psi}_0\Bigl(\frac{h M}{q r_1 r_2}\Bigr)e\Bigl(\frac{ah\overline{n_1 r_2}}{q r_1}+\frac{ah\overline{n_2 q r_1}}{r_2}\Bigr).
\end{align*}
Then we have
\[
\mathscr{F}\ll \frac{N^2}{Q x^\epsilon }.
\]
\end{lmm}
\begin{proof}
We first apply Lemma \ref{lmm:Simplification} to $\mathscr{F}$. This shows that
\[
\mathscr{F}\ll \frac{N^2}{Q x^\epsilon}+\exp((\log\log{x})^5)\sup_{\substack{H'\le 2H\\ Q'\le 2Q\\ R_1,R_2\le 2R}}|\mathscr{F}_2|,
\]
with $\mathscr{F}_2$ as given by $\mathscr{E}'$ in Lemma \ref{lmm:Simplification}. We now rewrite the congruence $n_1\equiv n_2\Mod{q}$ by $n_2=n_1+k q$ for some $N/(q\log^{C}{x})\le |k|\le K$ where 
\[
K:=\frac{N}{Q}.
\]
Thus we see that 
\begin{align*}
\mathscr{F}_2&=\sum_{\substack{Q\le q\le Q'\\ (q,a)=1}}\sum_{\substack{R\le r_1\le  R_1\\ R\le r_2\le R_2\\ (r_1,a r_2)=1\\ (r_2,a q r_1)=1}}\frac{\lambda_{q,r_1}\overline{\lambda_{q,r_2}}}{q r_1 r_2}\sum_{\substack{n\sim N\\ (n,qr_1)=1}}\sum_{\substack{N/(q\log^{C}{x})\le |k|\le K\\ n+k q\sim N\\ (n+k q,qr_2n_1)=1\\ (n r_2,n+k q)\in\mathcal{N}}}\alpha_{n}\overline{\alpha_{n+k q}}\\
&\qquad\qquad\times\sum_{1\le |h| \le H'} e\Bigl(\frac{- a h k\overline{(n+k q) r_1}}{n r_2}\Bigr).
\end{align*}
We wish to show $\mathscr{F}_2\ll N^2/Q x^{2\epsilon}$. We apply Cauchy-Schwarz in the $q,r_1,r_2$ variables to eliminate the $\lambda_{q,r}$ coefficients, giving
\[
\mathscr{F}_2\ll \frac{|\mathscr{F}_3|^{1/2}}{Q^{1/2} R},
\]
where
\begin{align*}
\mathscr{F}_3:=\sum_{q\sim Q}\sum_{\substack{r_1,r_2\sim R\\ (r_1,r_2)=1}}\Biggl|\sum_{\substack{n\sim N\\ (n,q r_1)=1}}\sum_{\substack{N/(q\log^{C}{x})\le |k|\le K\\ n+kq\sim N\\ (n+k q,q n r_2)=1\\ (n r_2,n+k q)\in\mathcal{N}}}\sum_{1\le |h| \le H'} c_{n,k,q,h} e\Bigl( \frac{-a h k\overline{(n+kq) r_1}}{ n r_2}\Bigr)\Biggr|^2,
\end{align*}
for some 1-bounded coefficients $c_{n,k,q,h}$. We wish to show $\mathscr{F}_3\ll R^2 N^4 /Q x^{4\epsilon}$. (Here we have extended the summation of $q,r_1,r_2$ slightly for an upper bound.)

We insert smooth majorants for $r_1$ and $r_2$, extend the summation for an upper bound, then expand the square and swap the order of summation. After dropping some summation conditions on the outer variables for an upper bound, this leaves us with
\[
\mathscr{F}_3\le \sum_{n_1,n_2\sim N}\sum_{\substack{|k_1|,|k_2|\le K\\ n_1+k_1q\sim N\\ n_2+k_2q\sim N}}\sum_{1\le |h_1|,|h_2|\le H'}\sum_{\substack{q\sim Q\\ (n_1+k_1q,n_1n_2)=1\\ (n_2+k_2q,n_1n_2)=1}}\Bigl|\sum_{\substack{r_1,r_2\\ (r_1,n_1n_2 r_2)=1\\ (r_2,r_1s)=1}}g(r_1,r_2)e\Bigl(\frac{\ell\overline{s r_1}}{n_1n_2r_2}\Bigr)\Bigr|,
\]
where
\begin{align*}
g(r_1,r_2)&:=\psi_0\Bigl(\frac{r_1}{R}\Bigr)\psi_0\Bigl(\frac{r_2}{R}\Bigr),\\
\ell&:=a (h_1 k_1 n_2(n_2+k_2 q)-h_2 k_2 n_1(n_1+k_1 q)),\\
s&:=(n_1+k_1 q)(n_2+k_2q).
\end{align*}
Here we have crucially made use of the fact that $(n_1r_2,n_1+k_1q),(n_2r_2,n_2+k_2q)\in\mathcal{N}$ and so we may assume $\gcd(n_1,n_2+k_2q)=\gcd(n_2,n_1+k_1q)=1$ to give the conditions in the $q$-summation.

We split our upper bound for $\mathscr{F}_3$ according to whether $\ell=0$ or not:
\[
\mathscr{F}_3\le \mathscr{F}_{\ell=0}+\mathscr{F}_{\ell\ne 0}.
\]
We first consider $\mathscr{F}_{\ell=0}$, the contribution from the `diagonal' terms with $\ell=0$. Since $(n_2,n_2+k_2q)=1$ we have $(k_2,n_2+k_2q)=1$ and so $k_2|h_1k_1n_2$ for these terms. Thus
\begin{align*}
\mathscr{F}_{\ell=0}&\ll \sum_{\substack{q\sim Q\\ r_1,r_2\sim R}}\sum_{\substack{|h_1|\le H\\ |k_1|\le K\\ n_2\sim N}}\sum_{k_2|h_1k_1n_2}\,\,\sum_{h_2n_1|h_1k_1n_2(n_2+k_2q)}1\\
&\ll (\log{x})^{O(1)}Q R^2 H K N\\
&\ll (\log{x})^{O(1)} \frac{Q R^4 N^2 }{M}. 
\end{align*}
This gives $\mathscr{F}_{\ell=0}\ll R^2 N^4/Q x^{4\epsilon}$ provided
\[
M>\frac{x^{6\epsilon} Q^2 R^2}{N^2}.
\]
Recalling $MN\asymp x$, this is satisfied if
\begin{equation}
Q R<x^{1/2-4\epsilon} N^{1/2}.
\label{eq:FouvryCond2}
\end{equation}
We now consider the terms with $\ell\ne 0$. We let $t=n_1n_2$ and put each of $\ell$, $t$, $s$ into one of $O(\log^3{x})$ dyadic ranges. Thus 
\begin{equation}
\mathscr{F}_{\ell\ne 0} \ll (\log{x})^3\sup_{\substack{N^2\ll S,T\ll N^2 \\ L\ll R^2 N^3(\log{x})^2/M}}\sum_{s\sim S}\sum_{t\sim T}\sum_{\ell\sim L}b_{\ell,s,t}\sum_{\substack{r_1,r_2\\ (s r_1,r_2 t)=1}}g(r_1,r_2)e\Bigl(\frac{\ell \overline{s r_1}}{r_2 t}\Bigr),
\label{eq:FouvryOffDiag}
\end{equation}
where $b_{\ell,s,t}$ are complex coefficients satisfying
\[
|b_{\ell,s,t}|\le \mathop{\sum_{q\sim Q}\,\, \sum_{|k_1|,|k_2|\le K}\,\,\sum_{1\le |h_1|,|h_2|\le H'}\,\,\sum_{n_1,n_2\sim N}}\limits_{\substack{\ell=a(h_1k_1n_2(n_2+k_2q)-h_2k_2n_1(n_1+k_1q))\\
s=(n_1+k_1q)(n_2+k_2q)\\ t=n_1n_2}}1.
\]
We can now apply Lemma \ref{lmm:DeshouillersIwaniec} to bound the sum \eqref{eq:FouvryOffDiag}. This gives
\begin{align*}
\mathscr{F}_{\ell\ne 0}& \ll x^\epsilon\Bigl(RT(ST+L)(R+RS)+R^3T\sqrt{(ST+L)S}+R^2LS\Bigr)^{1/2}\|b\|_2\\
&\ll x^{2\epsilon}\Biggl(R N^2\Bigl(N^4+\frac{N^3 R^2}{M}\Bigr)(R+R N^2)\\
&\qquad\qquad\qquad+R^3N^2\sqrt{\Bigl(N^4+\frac{N^3 R^2}{M}\Bigr)N^2}+R^2\frac{N^3 R^2}{M}N^2\Biggr)^{1/2}\|b\|_2.
\end{align*}
Since $MN\asymp x$, this simplifies to
\begin{align}
\mathscr{F}_{\ell\ne0}&\ll x^{2\epsilon} \Bigl(R^2 N^8+\frac{R^4 N^8}{x}+R^3N^5+\frac{R^4 N^5}{x^{1/2}}\Bigr)^{1/2}\|b\|_2.\label{eq:FouvryFOffDiag}
\end{align}
We first consider $\|b\|_2$. We see that
\[
\|b\|_2^2\le \sum_{\ell\sim L}\sum_{s\sim S}\sum_{\substack{t\sim T\\ (s,t)=1}}\Bigl(\sum_{\substack{n_1n_2=t\\ n_1,n_2\sim N}}\sum_{\substack{|k_1|,|k_2|\le K\\ q\sim Q \\ s=(n_1+k_1q)(n_2+k_2q)}}\#\mathcal{C}\Bigr)^2,
\]
where $\mathcal{C}=\mathcal{C}(k_1,k_2,q,n_1,n_2,\ell)$ is the set
\begin{align*}
\mathcal{C}=\{(h_1,h_2)\in [1,H]^2:&\,\ell=a(h_1k_1n_2(n_2+k_2q)-h_2k_2n_1(n_1+k_1q))\}.
\end{align*}
There are at most $\tau(t)=x^{o(1)}$ choices of $n_1,n_2$ given $t$. There are at most $\tau(s)=x^{o(1)}$ choices of $n_1+k_1q$ and $n_1+k_2q$ given $s$. There are then at most $\tau(s-n_1)\tau(s-n_2)=x^{o(1)}$ choices of $k_1,k_2,q$ given $n_1+k_1q$ and $n_2+k_2q$ and $n_1,n_2$. Thus there are at most $x^{o(1)}$ choices of $n_1,n_2,k_1,k_2,q$ given $s$ and $t$. Moreover, since $(s,t)=1$ we have $(n_1,k_1)=(n_2,k_2)=1$. Thus, by Cauchy-Schwarz, we see that
\begin{align*}
\|b\|_2^2&\ll x^{o(1)}\sum_{\ell\sim L}\sum_{s\sim S}\sum_{\substack{t\sim T\\ (s,t)=1}}\sum_{\substack{n_1n_2=t\\ n_1,n_2\sim  N}}\sum_{q\sim Q}\sum_{\substack{|k_1|,|k_2|\le K\\ s=(n_1+k_1q)(n_2+k-2q)}}\#\mathcal{C}^2\\
&\ll x^{o(1)} \mathop{\sum_{n_1,n_2,k_1,k_2,q}\,\,\sum_{h_1,h_2,h_3,h_4}}\limits_{\substack{(h_1-h_3)k_1n_2(n_2+k_2q)=(h_2-h_4)k_2n_1(n_1+k_1q)\\ (n_1,k_1)=(n_2,k_2)=1}}1.
\end{align*}
Here we have suppressed the ranges for the variables in the final line for convenience. We see that if $h_1-h_3=0$, then we must have $h_2-h_4=0$, and so the total contribution from terms with $h_1-h_3=0$ is
\begin{align}
&\ll x^{o(1)} H^2 N^2 K^2Q\ll x^{o(1)} \frac{Q^2 R^4}{M^2}N^2\frac{N^2}{Q^2}Q\ll x^{o(1)} \frac{Q R^4N^6}{x^2}.
\label{eq:bDiag}
\end{align}
If instead $h_1-h_3\ne 0$, then we first fix a choice of $h_1$, $h_3$, $n_2$, $k_2$ and $q$, for which there are $O(H^2 N K Q)$ choices. Since $(k_1,n_1)=1$, we have that $n_1(n_1+k_1q)|(h_1-h_3)n_2(n_2+k_2q)$, so there are $O(x^{o(1)})$ choices of $n_1,k_1$. Given such a choice of $n_1,k_1$, we see that $(h_2-h_4)$ is uniquely determined, and so there are $O(H)$ choices of $h_2$ and $h_4$. Thus the total contribution from terms with $h_1-h_3\ne 0$ is 
\begin{equation}
\ll x^{o(1)}H^3 N K Q\ll x^{o(1)} \frac{Q R^4N^6}{x^2}\Bigl(\frac{Q^2R^2}{x N}\Bigr).
\label{eq:bOffDiag}
\end{equation}
But we have $Q^2 R^2<x^{1-8\epsilon}N$ by \eqref{eq:FouvryCond2}, so we find that the bound \eqref{eq:bDiag} is larger than the bound \eqref{eq:bOffDiag}. Thus
\begin{equation}
\|b\|_2^2\ll x^{\epsilon}\frac{Q R^4 N^6}{x^2}.\label{eq:bBound}
\end{equation}
Putting together \eqref{eq:bBound} and \eqref{eq:FouvryFOffDiag}, we find that
\begin{align*}
\mathscr{F}_{\ell\ne 0}&\ll x^{3\epsilon}\Bigl(R^2 N^8+\frac{R^4 N^8}{x}+R^3 N^5+\frac{R^4 N^5}{x^{1/2}}\Bigr)^{1/2}\Bigl(\frac{Q R^4 N^6}{x^2}\Bigr)^{1/2}\\
&\ll x^{3\epsilon}\Bigl(\frac{Q^{1/2} R^3 N^7}{x}+\frac{Q^{1/2} R^4 N^7}{x^{3/2}}+\frac{Q^{1/2}R^{7/2}N^{11/2}}{x}+\frac{Q^{1/2}R^4 N^{11/2}}{x^{5/4}}\Bigr).
\end{align*}
Thus we have $\mathscr{F}_{\ell\ne 0}\ll R^2 N^4/(Q x^{4\epsilon})$ provided we have
\begin{align}
N^6 R^2 Q^{3}&<x^{2-14\epsilon},\\
N^6 R^4 Q^{3}&<x^{3-14\epsilon},\\
N^{3}R^{3}Q^{3}&<x^{2-14\epsilon},\\
N^{3}R^{4}Q^{3}&<x^{5/2-14\epsilon}.
\end{align}
This gives the result.
\end{proof}
%
%
%
%
%
%
%
%
\begin{proof}[Proof of Proposition \ref{prpstn:Fouvry}]
First we note that by Lemma \ref{lmm:Divisor} the set of $n,m$ with $\max(|\alpha_n|,|\beta_m|)\ge(\log{x})^C$ has size $\ll x(\log{x})^{O_{B_0}(1)-C}$, so by Lemma \ref{lmm:SmallSets} these terms contribute negligibly if $C=C(A,B_0)$ is large enough. Thus, after dividing through by $(\log{x})^{2C}$ and replacing $A$ with $A+2C$, it suffices to show the result when $\alpha_n\beta_m$ are 1-bounded ($\alpha_n$ still satisfies \eqref{eq:SiegelWalfisz} by Lemma \ref{lmm:SiegelWalfiszMaintain}.)

We factor $q_2=r d$ where $P^-(r)>z_0\ge P^+(d)$ into parts with large and small prime factors. By putting these in dyadic intervals, we see that it suffices to show for every $A>0$ and every choice of $D R\asymp Q_2$ that
 \begin{align*}
 &\sum_{\substack{q_1\sim Q_1\\ (q_1,a)=1}}\sum_{\substack{r\sim R\\ P^-(d)>z_0\\ (r,a)=1}}\sum_{\substack{d\sim D\\ P^+(d)\le z_0}}|\Delta(q_1 r d)|\ll_{A}\frac{x}{(\log{x})^A}.
 \end{align*}
By Lemma \ref{lmm:RoughModuli} we have the result unless $D\le y_0=x^{1/\log\log{x}}$, and so we may assume that $R=Q_2x^{-o(1)}$. We let $q=q_1d$ and $Q\asymp Q_1D$, and see that it suffices to show
 \begin{align*}
 &\sum_{\substack{q\sim Q\\ (q,a)=1}}\tau(q)\sum_{\substack{r\sim R\\ P^-(d)>z_0\\ (r,a)=1}}|\Delta(q r)|\ll_{A}\frac{x}{(\log{x})^A}.
 \end{align*}
We insert coefficients $c_{q,r}$ to remove the absolute values, and absorb the conditions $P^-(r)>z_0$, $q\sim Q$ into the definition of $c_{q,r}$. We now see that we have a sum of the type considered in Proposition \ref{prpstn:GeneralDispersion}, which we can apply provided $N>Q_1(\log{x})^C$. By the assumptions of the proposition we also have that $N Q_1 Q_2<x^{1-\epsilon}$, so we have $H_1=(Q R\log^5{x})/M<1$ and so the sum $\mathscr{E}_1$ of Proposition \ref{prpstn:GeneralDispersion} vanishes. Therefore, by Proposition \ref{prpstn:GeneralDispersion}, it suffices to show that
\[
\mathscr{E}_2\ll \frac{N^2}{Q x^\epsilon},
\]
where
\begin{align*}
\mathscr{E}_2&:=\sum_{ (q,a)=1}\psi_0\Bigl(\frac{q}{Q}\Bigr)\sum_{\substack{r_1,r_2\sim R\\ (r_1,a r_2)=1\\ (r_2,a q r_1)=1}}\frac{\lambda_{q,r_1}\overline{\lambda_{q,r_2}}}{q r_1 r_2}\sum_{\substack{n_1,n_2\sim N\\ n_1\equiv n_2\Mod{q}\\ (n_1,n_2 q r_1)=1\\(n_2,n_1 q r_2)=1\\ |n_1-n_2|\ge N/(\log{x})^C}}\alpha_{n_1}\overline{\alpha_{n_2}}\\
&\qquad\qquad \times\sum_{1\le |h|\le H_2}\hat{\psi}_0\Bigl(\frac{h M}{q r_1 r_2}\Bigr)e\Bigl(\frac{ah\overline{n_1r_2}}{q r_1}+\frac{ah\overline{n_2 q r_1}}{r_2}\Bigr),
\end{align*}
where $Q=Q_1x^{o(1)}$, $R=Q_2x^{-o(1)}$ and $\lambda_{q,r}=c_{q,r}$ are 1-bounded coefficients supported on $P^-(r)\ge z_0$. We now see that Lemma \ref{lmm:FouvryEstimate} gives the result provided we have
\begin{align*}
Q\log^{C}{x}&<N,\qquad &N^6 R^2 Q^{3}&<x^{2-14\epsilon},\qquad &N^6 R^4 Q^{3}&<x^{3-14\epsilon},\\
N^{3}R^{3}Q^{3}&<x^{2-14\epsilon},\qquad &N^{3}R^{4}Q^{3}&<x^{5/2-14\epsilon},\qquad
&Q R&<x^{1/2-\epsilon}N^{1/2}.
\end{align*}
The result also follows from the Bombieri-Vinogradov Theorem if $QR<x^{1/2-\epsilon}$, so we may assume that $QR>x^{1/2-\epsilon}$. We then see that the fourth condition implies that $N<x^{1/6}$, so the third condition follows from the fifth, and may be dropped. Recalling that $Q=Q_1x^{o(1)}$, and $R=Q_2x^{-o(1)}$ and noting the conditions on $Q_1,Q_2$ then gives the result.
\end{proof}
%
%
%
%
%
%
%
%
%
%
%
%
%
%
%
%
%
%
\section{Small divisor estimates near \texorpdfstring{$x^{1/21}$}{x\^{}(1/21)}}\label{sec:SmallDivisor}
In this section we establish Proposition \ref{prpstn:SmallDivisor}. This is a slightly different arrangement of the sums appearing, but has similarities with the work of Fouvry \cite{Fouvry}. It ultimately relies on the Weil bound, but is only effective when $N$ is a relatively small power of $x$. This is important in handling products of three large smooth factors and one small factor. Roughly similar results via an alternative method were obtained by Fouvry and Radziwi\l\l\, \cite{FouvryRadziwill}, but their results (which hold regardless of factorization properties of the moduli) do not extend to moduli which are large enough for our applications. The critical case is for convolutions when one factor is of size $x^{1/21}$ and the other of size $x^{20/21}$.
%
%
%
%
%
%
%
%
\begin{lmm}[Kloosterman sum bound]\label{lmm:Weil}
\[
\sum_{\substack{n\\ (n,q)=1}}\psi_0\Bigl(\frac{n}{N}\Bigr)e\Bigl(\frac{a\overline{n}}{q}\Bigr)\ll q^{o(1)}\Bigl(q^{1/2}+\frac{N}{q}(a,q)\Bigr).
\]
\end{lmm}
\begin{proof}
This is the classic Ramanujan-Weil bound combined with completion of sums - it follows immediately from  Lemma \ref{lmm:InverseCompletion}, noting that the first term involves a Ramanujan sum so is of size $N(a,q)/q$ and the second term involves a Kloosterman sum, which is bounded by $q^{1/2+o(1)}(h,q)^{1/2}$ by Lemma \ref{lmm:Kloosterman}.
\end{proof}
%
%
%
%
%
%
%
%
\begin{lmm}[Small divisor exponential sum estimate]\label{lmm:SmallDivisorExponential}
Let $M,N,Q,R,D\ge 1$ satisfy $D\le x^{o(1)}$, $NM\asymp x$ and
\begin{align*}
N^6 Q^{7} R^8&<x^{4-12\epsilon},\\
Q^2 R&<x^{1-6\epsilon} N,
\end{align*}
and let $H=(DQ^2R^2\log^5{x})/M$. Let $\eta_{d,q,r}$ be a 1-bounded complex sequence with supported on $qr$ having no prime factors less than $z_0:=x^{1/(\log\log{x})^3}$, and let $\alpha_n$ be a 1-bounded sequence. Let
\begin{align*}
\mathscr{D}&:=\sum_{\substack{ (d,a)=1}}\psi_0\Bigl(\frac{d}{D}\Bigr)\sum_{\substack{q_1,q_2\sim Q \\ (q_1,aq_2)=1\\ (q_2,a d q_1)=1}}\sum_{\substack{r_1,r_2\sim R\\ (r_1,aq_2 r_2)=1\\ (r_2,a d q_1 r_1)=1}}\frac{\eta_{d,q_1,r_1}\overline{\eta_{d,q_2,r_2}}}{d q_1 q_2 r_1 r_2}\\
&\times\sum_{\substack{n_1,n_2\sim N\\ n_1\equiv n_2\Mod{d}\\ (n_1,n_2 d q_1 r_1)=1\\(n_2,n_1 d q_2 r_2)=1\\ |n_1-n_2|\ge N/(\log{x})^C}}\alpha_{n_1}\overline{\alpha_{n_2}}\sum_{1\le |h|\le H}\hat{\psi}_0\Bigl(\frac{h M}{d q_1 q_2 r_1 r_2}\Bigr)e\Bigl(\frac{ah\overline{n_1 q_2 r_2}}{d q_1 r_1}+\frac{ah\overline{n_2 d q_1 r_1}}{q_2 r_2}\Bigr).
\end{align*}
Then we have for every $A>0$ (uniformly in $C>0$)
\[
\mathscr{D}\ll_A \frac{N^2}{D(\log{x})^{2A}}.
\]
\end{lmm}
\begin{proof}
We first note that $\mathscr{D}$ is a special case of the sort of sum considered in Lemma \ref{lmm:Simplification} when the coefficients $\lambda_{q,r}$ are of the form
\[
\lambda_{d,b}=\sum_{\substack{b=qr\\ q\sim Q\\ r\sim R}}\eta_{d,q,r}.
\]
Thus we may apply Lemma \ref{lmm:Simplification}. Taking the worst $d\sim D$ (and removing the $1/d$ factor), this gives
\[
\mathscr{D}\ll \exp((\log\log{x})^5)\sup_{\substack{H'\le H\\ d\le 2D\\ B_1,B_2\le 4QR}}|\mathscr{D}_2|+\frac{N^2}{x^\epsilon},
\]
where 
\begin{align*}
\mathscr{D}_2&:=\sum_{\substack{q_1,q_2\sim Q\\ (q_1,a q_2)=1\\ (q_2,a d q_1)=1}}\sum_{\substack{r_1,r_2\sim R\\ (r_1,a q_2 r_2)=1\\ (r_2,a d q_1 r_1)=1\\ q_1r_1\le B_1\\ q_2r_2\le B_2}}\frac{\eta_{d,r_1,q_1}\overline{\eta_{d,r_2,q_2}}}{r_1r_2q_1q_2}\sum_{\substack{n_1,n_2\sim N \\ n_1\equiv n_2\Mod{d}\\ (n_1,n_2 d r_1 q_1)=1\\ (n_2,n_1 d r_2 q_2)=1\\ (n_1 q_2r_2,n_2)\in\mathcal{N} \\ |n_1-n_2|\ge N/(\log{x})^C}}\alpha_{n_1}\overline{\alpha_{n_2}}\\
&\qquad\qquad \times\sum_{1\le |h| \le H'} e\Bigl(\frac{ah(n_1-n_2)\overline{d n_2 r_1q_1}}{ n_1 r_2 q_2}\Bigr).
\end{align*}
We wish to show $\mathscr{D}\ll N^2/x^\epsilon$, and so we see it is sufficient to show $\mathscr{D}_2\ll N^2/x^{2\epsilon}$. We now introduce a new variable $s=q_1r_1$, and then Cauchy in $s$ and $q_2$ to eliminate the $\eta_{d,r_1,q_1}$ coefficients. This gives
\[
\mathscr{D}_2\ll \frac{\mathscr{D}_3^{1/2}(\log{x})^{O(1)}}{Q R^{3/2}},
\]
where
\begin{align*}
\mathscr{D}_3&:=\sum_{s\sim S}\sum_{\substack{q_2\sim Q\\ (q_2,d s)=1}}\Bigl|\sum_{\substack{r_2\sim R\\ (r_2,ds)=1}}\sum_{\substack{n_1,n_2\sim N\\  (n_1,n_2 d s)=1\\ (n_2,n_1r_2q_2)=1\\ (n_1q_2r_2,n_2)\in\mathcal{N}\\ |n_1-n_2|\ge N/(\log{x})^C}}\sum_{1\le |h| \le H'}\xi_{r_2,q_2,n_1,n_2,h}e\Bigl(\frac{ah(n_1-n_2)\overline{d n_2 s}}{ n_1 q_2 r_2}\Bigr)\Bigr|^2\\
&\le \sum_{s}\psi_0\Bigl(\frac{s}{S}\Bigr)\sum_{\substack{q_2\sim Q\\ (q_2,d s)=1}}\Bigl|\sum_{\substack{r_2\sim R\\ (r_2,d s)=1}}\sum_{\substack{n_1,n_2\sim N\\ (n_1,n_2 d s)=1\\ (n_2,n_1r_2q_2)=1\\(n_1q_2r_2,n_2)\in\mathcal{N}}}\sum_{1\le |h| \le H'}\xi_{r_2,q_2,n_1,n_2,h}e\Bigl(\frac{ah(n_1-n_2)\overline{d n_2 s}}{ n_1 r_2 q_2}\Bigr)\Bigr|^2
\end{align*}
for some suitable quantity 
\[
S\in[QR,4QR]
\]
 and some $1$-bounded sequence $\xi$ which does not depend on $s$ (we have absorbed the conditions on the summation into $\xi$). We wish to show $\mathscr{D}_2\ll N^2/x^{2\epsilon}$, so it is sufficient to show that $\mathscr{D}_3\ll N^2R^3Q^2/x^{5\epsilon}$. We expand the square and swap the order of summation in $\mathscr{D}_3$ to give
\[
\mathscr{D}_3\le \mathscr{D}_4,
\]
where
\begin{align*}
\mathscr{D}_4&:=\sum_{q_2\sim Q}\sum_{\substack{r_2,r_2'\sim R}}\sum_{\substack{n_1,n_1',n_2,n_2'\sim N\\(d n_2n_2',n_1n_1'r_2r_2'q_2)=1}}\sum_{1\le |h|,|h'|\le H}\Bigl|\sum_{\substack{s\\ (s,n_1n_1'q_2r_2r_2')=1}}\psi_0\Bigl(\frac{s}{S}\Bigr)e\Bigl(\frac{a c\overline{d n_2n_2' s}}{n_1n_1'q_2r_2r_2'}\Bigr)\Bigr|,\\
c&=h(n_1-n_2)n_2'r_2'-h'(n_1'-n_2')n_2r_2.
\end{align*}
We have the conditions $(n_1r_2,n_2')=(n_1'r_2',n_2)=1$ from the construction of $\mathcal{N}$. We split $\mathscr{D}_4$ into two subsums depending on whether $c=0$ or not
\begin{equation}
\mathscr{D}_4=\mathscr{D}_{c=0}+\mathscr{D}_{c\ne 0}.\label{eq:D4}
\end{equation}
We first consider $\mathscr{D}_{c=0}$. We first make a choice of $h$, $n_2'$, $r_2'$ and $n_1$, for which there are $O(H N^2R)$ choices. Since $(n_1,n_2)=1$ and $c=0$, we see that $n_2|r_2'h n_2'$, so there are at most $x^{o(1)}$ choices of $n_2$. We then see that $r_2 h'(n_1'-n_2')|h(n_1-n_2)n_2'q_2'$, (and this quantity is non-zero), so there are then at most $x^{o(1)}$ choices of $h',n_1',r_2$. Therefore there are 
\[
\ll x^{o(1)} R H N^2\ll x^\epsilon \frac{Q^2 R^3 N^2}{M}
\]
choices of $n_1,n_1,n_2,n_2',h,h',r_2,r_2'$ with $c=0$. Therefore we have
\begin{align}
\mathscr{D}_{c=0}\ll Q S x^\epsilon \frac{Q^2 R^3 N^2}{M} \ll \frac{Q^4 R^4 N^3}{x^{1-\epsilon}}.\label{eq:DDiag}
\end{align}
We now consider $\mathscr{D}_{c\ne 0}$. We apply Lemma \ref{lmm:Weil} to the inner sum, which gives
\begin{equation}
\sum_{(s,n_1n_1'q_1q_2q_2')=1}\psi_0\Bigl(\frac{s}{S}\Bigr)e\Bigl(\frac{a c\overline{d n_2n_2' s}}{n_1n_1'q_2 r_2r_2'}\Bigr)\ll x^{o(1)}Q^{1/2}N R+\frac{x^{o(1)} Q R}{Q R^2 N^2}(c,n_1n_1'q_2r_2r_2').
\label{eq:SmallFactorWeil}
\end{equation}
The first term in \eqref{eq:SmallFactorWeil} contributes
\begin{align}
&\ll x^{o(1)}\sum_{n_1,n_1',n_2,n_2'\ll N}\sum_{q_2\ll Q}\sum_{r_2,r_2'\ll R}\sum_{|h|,|h'|\ll H}Q^{1/2} N R\nonumber \\
&\ll x^{o(1)} N^5 Q^{3/2} R^3 H^2\nonumber\\
&\ll \frac{N^7 Q^{11/2}R^7}{x^{2-\epsilon}}\label{eq:DOffDiagDiag}
\end{align}
to $\mathscr{D}_{c\ne 0}$. We now consider the second term of  \eqref{eq:SmallFactorWeil}. We note that if $e=(c,n_2)$ then since $(n_2,n_1n_1'r_2r_2')=1$ we must have that $e|h' n_2'$ so $(c,n_2)\le (n_2,h n_2')$. Thus
\[
(c,n_1n_1'n_2n_2'q_2r_2r_2')\le (c,q_2)(c,n_2)(c,n_1n_1'n_2'r_2r_2')\le (c,q_2)(n_2,h n_2')N^3R^2.
\]
Thus the second term in \eqref{eq:SmallFactorWeil} contributes
\begin{align*}
&\ll x^{o(1)}\sum_{n_1,n_1',n_2'\ll N}\sum_{r_2,r_2'\ll R}\sum_{|h|,|h'|\ll H}\frac{N^3 R^2}{R N^2}\sum_{n_2\ll N}(n_2,h n_2')\sum_{q_2\ll Q}(c,q_2)\\
&\ll  x^{o(1)}\sum_{n_1,n_1',n_2'\ll N}\sum_{r_2,r_2'\ll R}\sum_{|h|,|h'|\ll H}NR\cdot x^{o(1)}N\cdot x^{o(1)}Q\\
&\ll x^{o(1)}N^5 Q R^3 H^2
\end{align*}
to $\mathscr{D}_{c\ne 0}$, which is smaller that the bound \eqref{eq:DOffDiagDiag}. Thus
\begin{equation}
\mathscr{D}_{c\ne 0}\ll  \frac{N^7 Q^{11/2}R^7}{x^{2-\epsilon}}.\label{eq:DOffDiag}
\end{equation}
Putting together \eqref{eq:D4}, \eqref{eq:DDiag} and \eqref{eq:DDiag}, we have that
\[
\mathscr{D}_4\ll \frac{Q^4 R^4 N^3}{x^{1-\epsilon}}+\frac{N^7 Q^{11/2}R^7}{x^{2-\epsilon}}.
\]
We wish to show that $\mathscr{D}_4\ll N^4Q^2R^3/x^{5\epsilon}$. The above bound gives this provided
\begin{align}
Q^2R&< x^{1-6\epsilon}N,\\
N^3Q^{7/2}R^4&<x^{2-6\epsilon}.
\end{align}
This gives the result.
\end{proof}
%
%
%
%
%
%
%
%
\begin{proof}[Proof of Proposition \ref{prpstn:SmallDivisor}]
First we note that by Lemma \ref{lmm:Divisor} the set of $n,m$ with $\max(|\alpha_n|,|\beta_m|)\ge(\log{x})^C$ has size $\ll x(\log{x})^{O_{B_0}(1)-C}$, so by Lemma \ref{lmm:SmallSets} these terms contribute negligibly if $C=C(A,B_0)$ is large enough. Thus, by dividing through by $(\log{x})^{2C}$ and considering $A+2C$ in place of $A$, it suffices to show the result when all the sequences are 1-bounded. ($\alpha_n$ still satisfies \eqref{eq:SiegelWalfisz} by Lemma \ref{lmm:SiegelWalfiszMaintain}.)

We factor $q_1=d_1 q$ and $q_2=d_2 r$ where $P^-(r),P^-(q)>z_0\ge P^+(d_1),P^+(d_2)$ into parts with large and small prime factors. By putting these in dyadic intervals, we see that it suffices to show for every $A>0$ and every choice of $D_1 Q\asymp Q_1$ and $D_2 R\asymp Q_2$ that
 \begin{align*}
 &\sum_{\substack{d_1\sim D_1\\ P^+(d)\le z_0}}\sum_{\substack{d_2\sim D_2\\ P^+(d)\le z_0}}\sum_{\substack{q\sim Q\\ P^-(q)>z_0 \\ (q_1,a)=1}}\sum_{\substack{r\sim R\\ P^-(r)>z_0\\ (r,a)=1}}|\Delta(q r d_1 d_2)|\ll_{A}\frac{x}{(\log{x})^A}.
 \end{align*}
By Lemma \ref{lmm:RoughModuli} we have the result unless $D_1,D_2\le y_0=x^{1/\log\log{x}}$, and so we may assume that $Q=Q_1x^{-o(1)}$ and $R=Q_2x^{-o(1)}$. We let $d=d_1d_2$, extend the summation over $d_1,d_2$ to only have the constraint $d\le y_0^2$, and then insert coefficients $c_{q,r,d}$ to remove the absolute values, where we absorb the conditions $P^-(qr)>z_0\ge P^+(d)$, $d\sim D$ into the definition of $c_{q,r,d}$. Thus it suffices to show that
\[
\sum_{\substack{d\le y_0^2\\ (d,a)=1}}\tau(d)\sum_{\substack{q\sim Q\\ (q,a)=1}}\sum_{\substack{r\sim R\\ (r,a)=1}} c_{q,r,d}\Delta(d q r)\ll_{A} \frac{x}{(\log{x})^A}.
\]
If we let 
\[
\lambda_{b_1,b_2,b_3}=\mathbf{1}_{b_1=1}\sum_{q r=b_3}c_{q,r,b_2},
\]
then we see that we have a sum of the type considered in Proposition \ref{prpstn:GeneralDispersion} (taking `$R$' to be $QR$, `$Q$' to be $D$ and `$E$' to be 1). By the assumptions of the proposition, we have that $N Q R\le N Q_1 Q_2<x^{1-\epsilon}$, so we have $H_1=(Q D R\log^5{x})/M<1$ and so the sum $\mathscr{E}_1$ of Proposition \ref{prpstn:GeneralDispersion} vanishes. Therefore, by Proposition \ref{prpstn:GeneralDispersion}, it suffices to show that
\[
\mathscr{E}_2\ll \frac{N^2}{D x^\epsilon},
\]
where $H_2=(D Q^2 R^2\log^5{x})/M$ and
\begin{align*}
\mathscr{E}_2&:=\sum_{\substack{(d,a)=1}}\psi_0\Bigl(\frac{d}{D}\Bigr)\sum_{\substack{q_1,q_2\sim Q \\ (q_1,aq_2)=1\\ (q_2,a d q_1)=1}}\sum_{\substack{r_1,r_2\sim R\\ (r_1,aq_2 r_2)=1\\ (r_2,a d q_1 r_1)=1}}\frac{c_{d,q_1,r_1}\overline{c_{d,q_2,r_2}}}{d q_1 q_2 r_1 r_2}\sum_{\substack{n_1,n_2\sim N\\ n_1\equiv n_2\Mod{d}\\ (n_1,n_2 d q_1 r_1)=1\\(n_2,n_1 d q_2 r_2)=1\\ |n_1-n_2|\ge N/(\log{x})^C}}\alpha_{n_1}\overline{\alpha_{n_2}}\\
&\qquad\times\sum_{1\le |h|\le H_2}\hat{\psi}_0\Bigl(\frac{h M}{d q_1 q_2 r_1 r_2}\Bigr)e\Bigl(\frac{ah\overline{n_1 q_2 r_2}}{d q_1 r_1}+\frac{ah\overline{n_2 q_1 r_1}}{q_2 r_2}\Bigr).
\end{align*}
We now see that Lemma \ref{lmm:SmallDivisorExponential} gives the result provided we have
\begin{align*}
N^6 Q^{7} R^8&<x^{4-12\epsilon},\\
Q^2 R&<x^{1-6\epsilon} N.
\end{align*}
Recalling that $Q=Q_1x^{o(1)}$, and $R=Q_2x^{-o(1)}$ then gives the result.
\end{proof}
%
%
%
%
%
%
%
%
%
%
%
%
%
%
%
%
%
%
\section{Zhang-style estimates near \texorpdfstring{$x^{1/2}$}{x\^{}(1/2)}}\label{sec:Zhang}
In this section we prove Proposition \ref{prpstn:Zhang}. The argument is a refined version of that of Zhang \cite{Zhang}; we take into account cancellation in an additional variable to obtain quantitatively stronger results, but the original work of Zhang (or the similar work of Polymath \cite{Polymath} would produce qualitatively similar results. These estimates are key for estimating convolutions where both factors are close to $x^{1/2}$; the most critical cases are when one factor is of size $Q_1Q_2$ and the other of size $x/(Q_1Q_2)$.
%
%
%
%
%
%
%
%
\begin{lmm}[Zhang exponential sum estimate]\label{lmm:Zhang1}
Let $Q,R,M,N$ satisfy $NM\asymp x$ and
\begin{align*}
Q^{7} R^{12}< x^{4-18\epsilon},\qquad 
Q <N<\frac{x^{1-5\epsilon}}{Q},
\end{align*}
and let  $H\ll QNR^2/x^{1-\epsilon}$. Let $c_{q,r}$ and $\alpha_n$ be $1$-bounded complex sequences with $c_{q,r}$ supported on $r$ with $P^-(r)\ge z_0$. Let
\begin{align*}
\mathscr{Z}&:=\sum_{\substack{q\sim Q\\ (q,a)=1}}\sum_{\substack{r_1,r_2\sim R\\ (r_1,ar_2)=1\\ (r_2,aqr_1)=1}}\sum_{\substack{n_1,n_2\sim N\\ n_1\equiv n_2\Mod{q}\\ (n_1,q r_ 1 n_2)=1\\ (n_2,q r_2 n_1)=1\\ |n_1-n_2|\ge N/(\log{x})^C}}\frac{\alpha_{n_1}\overline{\alpha_{n_2}}c_{q,r_1}\overline{c_{q,r_2}}}{q r_1 r_2}\\
&\qquad\times \sum_{1\le |h|\le H}\hat{\psi_0}\Bigl(\frac{h M}{q r_1 r_2}\Bigr)e\Bigl(ah\Bigl(\frac{\overline{n_1 r_2}}{qr_1}+\frac{\overline{n_2 q r_1}}{r_2}\Bigr)\Bigr),
\end{align*}
Then we have
\[
\mathscr{Z}\ll \frac{N^2 }{Q x^\epsilon}.
\]
\end{lmm}
\begin{proof}
We assume throughout that $H\ll QR^2N/x^{1-\epsilon}$ and that $Q\ll N$, and deduce the other conditions are sufficient to give the result.

Since we only consider $r_1,r_2$ with $P^-(r_1r_2)\ge z_0$, $r_1$ and $r_2$ have at most $(\log\log{x})^3$ prime factors. Therefore, by Lemma \ref{lmm:FouvryDecomposition}, there are $O(\exp(\log\log{x})^5))$ different sets $\mathcal{N}_1,\mathcal{N}_2,\dots$ which cover all possible pairs $(r_1,r_2)$, and such that if $(r_1,r_2),(r_1',r_2')\in\mathcal{N}_j$ then $\gcd(r_1,r_2')=\gcd(r_1',r_2)=1$. Taking the worst such set $\mathcal{N}$, we see that
 \begin{align*}
\mathscr{Z}&\ll \exp((\log\log{x})^5)\Bigl|\sum_{\substack{q\sim Q\\ (q,a)=1}}\sum_{\substack{r_1,r_2\sim R\\ (r_1,a r_2)=1\\ (r_2,a q r_1)=1\\ (r_1,r_2)\in\mathcal{N}}}\sum_{\substack{n_1,n_2\sim N\\ n_1\equiv n_2\Mod{q}\\ (n_1,q r_ 1 n_2)=1\\ (n_2,q r_2 n_1)=1\\ |n_1-n_2|\ge N/(\log{x})^C}}\frac{\alpha_{n_1}\overline{\alpha_{n_2}}c_{q,r_1}\overline{c_{q,r_2}}}{q r_1 r_2}\\
&\qquad\qquad\times\sum_{1\le |h|\le H}\hat{\psi_0}\Bigl(\frac{h M}{q r_1 r_2}\Bigr)e\Bigl(ah\Bigl(\frac{\overline{n_1 r_2}}{qr_1}+\frac{\overline{n_2 q r_1}}{r_2}\Bigr)\Bigr)\Bigr|.
\end{align*}
We now Cauchy in $n_1,n_2$ and $q$ to eliminate the $\alpha$-coefficients and insert a smooth majorant for the $n_1$ and $n_2$ summations. This gives (using $Q\ll N$)
\[
\mathscr{Z}^2\ll \exp(2(\log\log{x})^5)\Bigl(\sum_{q\sim Q}\sum_{\substack{n_1,n_2\sim N\\ n_1\equiv n_2\Mod{q}}}\frac{1}{q^2}\Bigr)|\mathscr{Z}_2|\ll \frac{x^\epsilon N^2}{Q^2}|\mathscr{Z}_2|,
\]
where
\begin{align*}
\mathscr{Z}_2&:=\sum_{\substack{q\sim Q\\ (q,a)=1}}\sum_{\substack{n_1,n_2\\ (n_1n_2,q)=1}}\psi_0\Bigl(\frac{n_1}{N}\Bigr)\psi_0\Bigl(\frac{n_2}{N}\Bigr)\\
&\qquad \times\Bigl|\sum_{\substack{r_1,r_2\sim R\\ (r_1,a r_2 n_1)=1\\ (r_2,a q r_1 n_2)=1\\ (r_1,r_2)\in\mathcal{N}}}\frac{c_{q,r_1}\overline{c_{q,r_2}}}{r_1 r_2}\sum_{1\le |h|\le H}\hat{\psi_0}\Bigl(\frac{h M}{q r_1 r_2}\Bigr)e\Bigl(ah\Bigl(\frac{\overline{n_1 r_2}}{q r_1}+\frac{\overline{ n_2q r_1}}{r_2}\Bigr)\Bigr)\Bigr|^2\\
&\le \frac{1}{R^4}|\mathscr{Z}_3|,
\end{align*}
where
\begin{align*}
\mathscr{Z}_3:=\sum_{q\sim Q}\sum_{\substack{r_1,r_1',r_2,r_2'\sim R\\ (qr_1r_1',r_2r_2')=1}}\sum_{1\le |h|,|h'|\le H}\Bigl|\sum_{\substack{n_1,n_2\\ n_1\equiv n_2\Mod{q} \\ (n_1,qr_1r_1')=1\\ (n_2,r_2r_2')=1}}\psi_0\Bigl(\frac{n_1}{N}\Bigr)\psi_0\Bigl(\frac{n_2}{N}\Bigr)e\Bigl(\frac{c_1\overline{n_1}}{q r_1r_1'}+\frac{c_2\overline{n_2}}{r_2 r_2'}\Bigr)\Bigr|,
\end{align*}
and where $c_1\Mod{q r_1r_1'}$ and $c_2\Mod{r_2r_2'}$ are given by
\begin{align*}
c_1&=a(h r_1'r_2'-h' r_1r_2)\overline{r_2r_2'},\\
c_2&=a (h r_1'r_2'-h' r_1r_2)\overline{q r_1r_1'}.
\end{align*}
(Here we used the fact that $(r_1,r_2),(r_1',r_2')\in\mathcal{N}$ to conclude $(r_1,r_2')=(r_1',r_2)=1$.) In order to establish the desired bound $\mathscr{Z}\ll N^2/(x^\epsilon Q)$, it suffices to show $\mathscr{Z}_2\ll N^2/x^{3\epsilon}$, and so it suffices to prove
\begin{equation}
\mathscr{Z}_3\ll \frac{N^2 R^4}{x^{3\epsilon}}.\label{eq:ZhangE4}
\end{equation}
We separate the diagonal terms $\mathscr{Z}_{=}$ with $hr_1'r_2'=h'r_1r_2$ and the off-diagonal terms $\mathscr{Z}_{\ne}$ with $h r_1'r_2'\ne h'r_1r_2$.
\begin{equation}
\mathscr{Z}_3\ll \mathscr{Z}_{=}+\mathscr{Z}_{\ne}.
\label{eq:Z3Bound}
\end{equation}
We first consider the diagonal terms $\mathscr{Z}_{=}$. Given a choice of $h,r_1',r_2'$ there are $x^{o(1)}$ choices of $h',r_1,r_2$ by the divisor bound. Thus, estimating the remaining sums trivially we have (using $Q\ll N$ and $H\ll NQR^2/x^{1-\epsilon}$)
\begin{equation}
\mathscr{Z}_{=}\ll x^{o(1)} Q R^2 H N \Bigl(\frac{N}{Q}+1\Bigr)\ll \frac{N^3 Q R^4}{x^{1-2\epsilon}}.
\label{eq:ZEq}
\end{equation}
Now we consider the off-diagonal terms $\mathscr{Z}_{\ne}$. By Lemma \ref{lmm:InverseCompletion}, we have that
\begin{align*}
&\sum_{\substack{n_2\equiv n_1\Mod{q}\\ (n_2,r_2r_2')=1}}\psi_0\Bigl(\frac{n_2}{N}\Bigr)e\Bigl(\frac{c_2\overline{n_2}}{r_2 r_2'}\Bigr)\\
&\qquad=\frac{N}{q r_2 r_2'}\sum_{|\ell_2|\le x^\epsilon QR^2/N}\hat{\psi_0}\Bigl(\frac{\ell_2 N}{q r_2r_2'}\Bigr)S(c_2,\ell_2\overline{q};r_2r_2')e\Bigl(\frac{\ell_2 n_1 \overline{r_2r_2'}}{q}\Bigr) +O(x^{-100}).
\end{align*}
here $S(m,n;c)$ is the standard Kloosterman sum. By Lemma \ref{lmm:InverseCompletion} again, we have that
\begin{align*}
&\sum_{(n_1,q r_1 r_1')=1}\psi_0\Bigl(\frac{n_1}{N}\Bigr)e\Bigl(\frac{c_1\overline{n_2}+\ell_2 r_1 r_1' n_1\overline{r_2r_2'} }{q r_1 r_1'}\Bigr)\\
&=\frac{N}{q r_1r_1'}\sum_{|\ell_1|\le x^\epsilon Q R^2/N}\hat{\psi_0}\Bigl(\frac{\ell_1 N}{q r_1 r_1'}\Bigr)S(c_1,\ell_1\overline{q};r_1 r_1')S(c_1,\ell_1\overline{r_1 r_1'}+\ell_2\overline{r_2r_2'};q)+O(x^{-100}).
\end{align*}
Thus, we see that $\mathscr{Z}_{3}$ is a sum of Kloosterman sums. By the standard Kloosterman sum bound of Lemma \ref{lmm:Kloosterman} $S(m,n;c)\ll \tau(c) c^{1/2}(m,n,c)^{1/2}\ll c^{1/2+o(1)}(m,c)^{1/2}$, the inner sum has the bound
\begin{align*}
\sum_{\substack{n_1,n_2\\ n_1\equiv n_2\Mod{q} \\ (n_1,qr_1r_1')=1\\ (n_2,r_2r_2')=1}}&\psi_0\Bigl(\frac{n_1}{N}\Bigr)\psi_0\Bigl(\frac{n_2}{N}\Bigr)e\Bigl(\frac{c_1\overline{n_1}}{q r_1r_1'}+\frac{c_2\overline{n_2}}{r_2 r_2'}\Bigr)\\
&\ll \frac{x^{o(1)}N^2}{Q^2 R^4}\sum_{\substack{|\ell_1|\le x^\epsilon QR^2/N\\ |\ell_2|\le x^\epsilon QR^2/N}}Q^{1/2}R^2 (c_2,r_2r_2')^{1/2}(c_1,r_1r_1')^{1/2}(c_1,q)^{1/2}\\
&\ll x^{3\epsilon}Q^{1/2}R^2 (hr_1'r_2'-h'r_1r_2,qr_1r_1'r_2r_2')^{1/2}\\
&\ll x^{3\epsilon}Q^{1/2}R^2 (h,r_1r_2)^{1/2}(h',r_1'r_2')^{1/2}(r_1'r_2',r_1r_2) (hr_1'r_2'-h'r_1r_2,q)^{1/2}.
\end{align*}
Substituting this into our expression for $\mathscr{Z}_{\ne}$ gives
\begin{align}
\mathscr{Z}_{\ne}&\ll x^{3\epsilon} Q^{1/2} R^2\sum_{r_1,r_1'\sim R}\sum_{r_2,r_2'\sim R}(r_1r_1',r_2r_2')\sum_{\substack{1\le |h|,|h'|\le H\\ hr_1'r_2'\ne h'r_1r_2}}(h,r_1r_2)(h',r_1'r_2')\nonumber\\
&\qquad \times\sum_{q\sim Q}(hr_1'r_2'-h'r_1r_2,q)^{1/2}\nonumber\\
&\ll x^{4\epsilon} Q^{3/2} R^6 H^2\nonumber\\
&\ll \frac{N^2 Q^{7/2} R^{10}}{x^{2-6\epsilon}}.\label{eq:ZNeq}
\end{align}
Substituting \eqref{eq:ZEq} and \eqref{eq:ZNeq} into \eqref{eq:Z3Bound} then gives
\[
\mathscr{Z}_3\ll \frac{N^3 Q R^4}{x^{1-2\epsilon}}+\frac{N^2 Q^{7/2} R^{10}}{x^{2-6\epsilon}}.
\]
This gives the desired bound \eqref{eq:ZhangE4} provided we have
\begin{align}
N&<\frac{x^{1-5\epsilon} }{Q},\\
Q^{7} R^{12}&<x^{4-18\epsilon}.
\end{align}
This gives the result.
\end{proof}
%
%
%
%
%
%
%
%
\begin{lmm}[Second exponential sum estimate]\label{lmm:Zhang2}
Let $Q,R,M,N\le x^{O(1)}$ satisfy $NM\asymp x$ and
\begin{align*}
N Q < x^{1-4\epsilon},\qquad
N Q^{5/2}R^3< x^{2-4\epsilon}, \qquad
N^2 Q R &<x^{2-4\epsilon}.
\end{align*}
Let $\alpha_n$, $\lambda_{q,r}$ be $1$-bounded complex sequences, $H_1=(QR\log^5{x})/M$ and 
\begin{align*}
\widetilde{\mathscr{Z}^{}}&:=\sum_{\substack{q\sim Q \\ (q,a)=1}}\sum_{\substack{r_1,r_2\sim R\\ (r_1 r_2,a)=1}}\frac{\lambda_{q,r_1}\overline{\lambda_{q,r_2}}}{\phi(q r_2)q r_1}\sum_{\substack{n_1,n_2\sim N\\ (n_1,q r_1)=1\\(n_2,q r_2)=1}}\alpha_{n_1}\overline{\alpha_{n_2}}\sum_{1\le |h|\le H_1}\hat{\psi}_0\Bigl(\frac{h M}{q r_1}\Bigr)e\Bigl( \frac{a h \overline{ n_1}}{q r_1}\Bigr).
\end{align*}
Then we have
\[
\widetilde{\mathscr{Z}^{}}\ll \frac{N^2}{Q x^\epsilon}.
\]
\end{lmm}
\begin{proof}
This is essentially \cite[\S 9]{Zhang}, and is an easier variant of the sum considered above. We first Cauchy in the $n_1,n_2,q,r_2$ variables to eliminate the $\alpha$ coefficients, and insert a smooth majorant for $n_1$. This gives
\begin{align*}
|\widetilde{\mathscr{Z}^{}}|^2&\le \frac{N^2\log{x}}{Q^3 R^3}\sum_{q\sim Q}\sum_{r_2\sim R}\sum_{n_2\sim N}\sum_{\substack{n_1 \\ (n_1,q)=1}}\psi_{0}\Bigl(\frac{n_1}{N}\Bigr)\Bigl|\sum_{\substack{r_1\sim R\\ (r_1,a r_2 n_1)=1}}\sum_{1\le|h|\le H_1}c_{q,r_1,r_2,h}e\Bigl(\frac{a h \overline{n_1}}{q r_1}\Bigr)\Bigr|^2\\
&\le \frac{N^3\log{x}}{Q^3 R^2}\widetilde{\mathscr{Z}_2},
\end{align*}
where $c_{q,r_1,r_2,h}$ are some 1-bounded coefficients and where
\begin{align*}
\widetilde{\mathscr{Z}_2}&:=\sum_{q\sim Q} \sum_{r_1,r_1'\sim R}\sum_{1\le |h|,|h'|\le H_1}\Bigl|\sum_{\substack{n\\ (n,q r_1 r_1')=1}}\psi_0\Bigl(\frac{n}{N}\Bigr)e\Bigl(\frac{a \overline{n_1}\ell}{q r_1r_1'}\Bigr)\Bigr|,\\
\ell&:=h r_1'-h' r_1.
\end{align*}
In order to show $\widetilde{\mathscr{Z}}\ll N^2/(Q x^\epsilon)$ it is sufficient to show $\widetilde{\mathscr{Z}_2}\ll N Q R^2/x^{3\epsilon}$. We split the sum according to whether $\ell=0$ or not
\begin{equation}
\widetilde{\mathscr{Z}_2}=\widetilde{\mathscr{Z}}_{\ell=0}+\widetilde{\mathscr{Z}}_{\ell\ne 0}.
\label{eq:Z2}
\end{equation}
We first consider $\widetilde{\mathscr{Z}}_{\ell=0}$, where $h r_1'=h' r_1$. Given $h$, $r_1'$ there are $x^{o(1)}$ choices of $h'$, $r_1$. Therefore we find
\begin{equation}
\widetilde{\mathscr{Z}}_{\ell=0}\ll x^{o(1)} Q R H_1 N\ll x^{o(1)}\frac{Q^2 R^2 N^2}{x}.
\label{eq:Z21}
\end{equation}
We now consider $\widetilde{\mathscr{Z}}_{\ell\ne 0}$. By Lemma \ref{lmm:Weil} we have that
\begin{equation}
\sum_{\substack{n\\ (n,q r_1 r_1')=1}}\psi_0\Bigl(\frac{n}{N}\Bigr)e\Bigl(\frac{a \overline{n_1}\ell}{q r_1r_1'}\Bigr)\ll x^{o(1)}Q^{1/2} R + x^{o(1)}\frac{N}{Q R^2}(\ell,qr_1r_1').\label{eq:Z2Weil}
\end{equation}
The first term of \eqref{eq:Z2Weil} contributes a total
\begin{equation}
\ll x^{o(1)} Q R^2 H_1^2 Q^{1/2} R \ll x^{o(1)}\frac{N^2 Q^{7/2} R^5}{x^2}\label{eq:Z22}
\end{equation}
to $\widetilde{\mathscr{Z}}_{\ell\ne 0}$. We now consider the second term of \eqref{eq:Z2Weil}. We see that $(\ell,q r_1 r_1')\le (\ell,q)(h,r_1)(h',r_1')(r_1,r_1')^2$. Thus these terms contribute
\begin{align}
\ll \frac{x^{o(1)} N}{Q R^2}\sum_{r_1,r_1'\sim R}(r_1,r_1')^2\sum_{1\le |h|,|h'|\le H_1}(h,r_1)(h',r_1')\sum_{q\sim Q}(\ell,q)
&\ll x^{o(1)}N H_1^2 R \nonumber \\
&\ll x^{o(1)}\frac{N^3 Q^2 R^3}{x^2}.
\label{eq:Z23}
\end{align}
Putting together together \eqref{eq:Z2}, \eqref{eq:Z21}, \eqref{eq:Z22} and \eqref{eq:Z23}, we see that
\[
\widetilde{\mathscr{Z}}_2=\widetilde{\mathscr{Z}}_{\ell=0}+\widetilde{\mathscr{Z}}_{\ell\ne 0}\ll x^{o(1)}\frac{N^2 Q^{7/2}R^5}{x^2}+x^{o(1)} \frac{N^2 Q^2 R^2}{x}+x^{o(1)}\frac{N^3 Q^2 R^3}{x^2}.
\]
We recall that it is sufficient to show that $\widetilde{\mathscr{Z}}_2\ll N Q R^2/x^{3\epsilon}$. Thus we are done provided we have
\begin{align}
N Q^{5/2}R^3 &< x^{2-4\epsilon},\\
N Q&<x^{1-4\epsilon}, \\
N^2 Q R&<x^{2-4\epsilon}.
\end{align} 
This gives the result.
\end{proof}
%
%
%
%
%
%
%
%
Finally, we are in a position to prove Proposition \ref{prpstn:Zhang}.
%
%
%
%
%
%
%
%
\begin{proof}[Proof of Proposition \ref{prpstn:Zhang}]
First we note that by Lemma \ref{lmm:Divisor} the set of $n,m$ with $\max(|\alpha_n|,|\beta_m|)\ge(\log{x})^C$ has size $\ll x(\log{x})^{O_{B_0}(1)-C}$, so by Lemma \ref{lmm:SmallSets} these terms contribute negligibly if $C=C(A,B_0)$ is large enough. Thus, by dividing through by $(\log{x})^{2C}$ and considering $A+2C$ in place of $A$, it suffices to show the result when all the sequences are 1-bounded. ($\alpha_n$ still satisfies \eqref{eq:SiegelWalfisz} by Lemma \ref{lmm:SiegelWalfiszMaintain}.)

We factor $q_1=r d$ where $P^-(r)>z_0\ge P^+(d)$ into parts with large and small prime factors. By putting these in dyadic intervals, we see that it suffices to show for every $A>0$ and every choice of $D R\asymp Q_1$ that
 \begin{align*}
 &\sum_{\substack{q_2\sim Q_2\\ (q_1,a)=1}}\sum_{\substack{r\sim R\\ P^-(d)>z_0\\ (r,a)=1}}\sum_{\substack{d\sim D\\ P^+(d)\le z_0}}|\Delta(q_2 r d)|\ll_{A}\frac{x}{(\log{x})^A}.
 \end{align*}
By Lemma \ref{lmm:RoughModuli} we have the result unless $D\le y_0=x^{1/\log\log{x}}$, and so we may assume that $R=Q_1 x^{-o(1)}$. We let $q=q_2 d$, and so it suffices to show that for $Q=Q_2 x^{o(1)}$
\[
\sum_{\substack{q\sim Q\\ (q,a)=1}}\tau(q)\sum_{\substack{r\sim R\\ P^-(r)\ge z_0\\ (r,a)=1}}|\Delta(q r)|\ll_A\frac{x}{(\log{x})^A}.
\]
We insert coefficients $c_{q,r}$ to remove the absolute values, and absorb the condition $P^-(r)>z_0$ into the definition of $c_{q,r}$. We now see that we have a sum of the type considered in Proposition \ref{prpstn:GeneralDispersion} with $D=E=1$. Therefore, since $N>x^\epsilon Q$ we may apply Proposition \ref{prpstn:GeneralDispersion}, and it suffices to show that
\[
\mathscr{E}_1,\mathscr{E}_2\ll \frac{N^2}{Q x^\epsilon},
\]
where $H_1=(QR\log^5{x})/M$, $H_2=(QR^2\log^5{x})/M$ and where
\begin{align*}
\mathscr{E}_{1}&:=\sum_{\substack{ (q,a)=1}}\psi_0\Bigl(\frac{q}{Q}\Bigr)\sum_{\substack{r_1,r_2\sim R\\ (r_1 r_2,a)=1}}\frac{c_{q,r_1}\overline{c_{q,r_2}}}{\phi(q r_2)q r_1}\sum_{\substack{n_1,n_2\sim N\\ (n_1,q r_1)=1\\(n_2,q r_2)=1}}\alpha_{n_1}\overline{\alpha_{n_2}}\sum_{1\le |h|\le H_1}\hat{\psi}_0\Bigl(\frac{h M}{q r_1}\Bigr)e\Bigl( \frac{a h \overline{ n_1}}{q r_1}\Bigr),\\
\mathscr{E}_2&:=\sum_{\substack{(q,a)=1}}\psi_0\Bigl(\frac{q}{Q}\Bigr)\sum_{\substack{r_1,r_2\sim R\\ (r_1,a r_2)=1\\ (r_2,a q r_1)=1}}\frac{c_{q,r_1}\overline{c_{q,r_2}}}{q r_1 r_2}\sum_{\substack{n_1,n_2\sim N\\ n_1\equiv n_2\Mod{q}\\ (n_1,n_2 q r_1)=1\\(n_2,n_1 q r_2)=1\\ |n_1-n_2|\ge N/(\log{x})^C}}\alpha_{n_1}\overline{\alpha_{n_2}}\\
&\qquad\qquad \times\sum_{1\le |h|\le H_2}\hat{\psi}_0\Bigl(\frac{h M}{q r_1 r_2}\Bigr)e\Bigl(\frac{ah\overline{n_1r_2}}{q r_1}+\frac{ah\overline{n_2 q r_1}}{r_2}\Bigr).
\end{align*}
Absorbing the $\psi_0(q/Q)$ factors into the coefficients $c_{q,r}$, we see these are precisely the sums $\widetilde{\mathscr{Z}}$ and $\mathscr{Z}$ considered in Lemma \ref{lmm:Zhang1} and Lemma \ref{lmm:Zhang2}. Thus, these lemmas give the result provided we have
\begin{align*}
Q^{7} R^{12}< x^{4-18\epsilon},\qquad 
Q<N<\frac{x^{1-5\epsilon}}{Q},\\
N Q < x^{1-4\epsilon},\qquad
N Q^{5/2}R^3< x^{2-4\epsilon},\qquad
N^2 Q R<x^{2-4\epsilon}.
\end{align*}
We see that the first two conditions imply the final three (since we may assume $QR\ge x^{1/2-\epsilon}$ or else the result follows from the Bombieri-Vinogradov theorem). Recalling that $Q=Q_2 x^{o(1)}$ and $R=Q_1 x^{-o(1)}$ then gives the result.
\end{proof}
%
%
%
%
%
%
%
%
%
%
%
%
%
%
%
%
%
%
%
\section{Bombieri--Friedlander--Iwaniec-style estimates near \texorpdfstring{$x^{1/5}$}{x\^{}(1/5)}}\label{sec:BFI}
In this section we prove Proposition \ref{prpstn:TripleRough}, which is a refinement of \cite[Theorem 4]{BFI2} by Bombieri, Friedlander and Iwaniec. Our argument is similar to the work of Bombieri, Friedlander and Iwaniec, but crucially we use the extra flexibility from our amplification set-up to reduce the contribution of some diagonal terms in the critical situation of five factors all of length $x^{1/5}$.
%
%
%
%
%
%
%
%
\begin{lmm}[Deshouillers--Iwaniec Bound]\label{lmm:DeshouillersIwaniec2}
Let $b_m,a_n$ be complex sequences, and let $g$ be a smooth function with $\|g^{(j)}\|_\infty\ll_j 1$. Let $r\in[R,2R]$, $s\in[S,2S]$ and let $\theta_{q}=\max(0,1-4\lambda_1(q))$, where $\lambda_1(q)$ is the least eigenvalue of the congruence subgroup $\Gamma_0(q)$.

We have
\begin{align*}
&\sum_{m\sim M}b_m\sum_{n\sim N}a_n\sum_{(c,r)=1}g\Bigl(\frac{c}{C}\Bigr)S(m\overline{r},n,sc)\\
&\ll x^{o(1)}\Bigl(1+\sqrt\frac{S^2 C^2 R}{MN}\Bigr)^{\theta_{rs}}\|b_m\|\|a_n\|\Bigl(S^2RC^2+MN+S MC^2+S NC^2+\frac{ M N C^2}{R}\Bigr)^{1/2}.
\end{align*}
\end{lmm}
\begin{proof}
This follows from \cite[Theorem 9]{DeshouillersIwaniec} after correcting a typo. (In \cite[Theorem 9]{DeshouillersIwaniec} there is a factor $(1+\sqrt{S^2 C R/(MN)})^{\theta_{rs}}$ instead of $(1+\sqrt{S^2 C^2 R/(MN)})^{\theta_{rs}}$).
\end{proof}
%
%
%
%
%
%
%
%
\begin{lmm}[Kim--Sarnak eigenvalue bound]\label{lmm:KimSarnak}
Let $q\in\mathbb{Z}_{>0}$ and $\theta_q$ be as in Lemma \ref{lmm:DeshouillersIwaniec2}. Then $\theta_q\le 7/32$.
\end{lmm}
\begin{proof}
This follows from \cite[Appendix, Proposition 2]{KimSarnak}.
\end{proof}
%
%
%
%
%
%
%
%
\begin{lmm}[Simpler exponential sum estimate]\label{lmm:BFISimple}
Let $1\le N,M,Q$ with $NM \asymp x$ and
\begin{align}
N^{3/2}&< x^{1-2\epsilon},\qquad Q< x^{1-2\epsilon}.
\end{align}
Let $\alpha_n$, be a complex sequence with $|\alpha_n|\le x^{o(1)}$. Let $H_1:=Q N(\log{x})^5/x$ and let 
\begin{align*}
\widetilde{\mathscr{B}}&:=\sum_{e\sim E}\mu^2(e)\sum_{\substack{q\\ (q,a)=1}}\psi_0\Bigl(\frac{q}{Q}\Bigr)\frac{1}{\phi(q e)q}\sum_{\substack{n_1,n_2\sim N\\ (n_1n_2,q e)=1}}\alpha_{n_1}\overline{\alpha_{n_2}}\sum_{1\le |h|\le H_1}\hat{\psi}_0\Bigl(\frac{-h M}{q}\Bigr)e\Bigl( \frac{a h \overline{ n_1}}{q}\Bigr).
\end{align*}
Then we have
\[
\widetilde{\mathscr{B}}\ll\frac{N^2}{Q x^\epsilon}.
\]
\end{lmm}
\begin{proof}
This is essentially just the argument of \cite[\S5]{BFI1}. By Bezout's identity (Lemma \ref{lmm:Bezout}) we have
\[
\frac{ah \overline{n_1}}{q}=-\frac{ah \overline{q}}{n_1}+\frac{a h}{n_1q} \Mod{1}.
\]
Since $h\ll QN(\log{x})^5/x$, this implies that
\[
e\Bigl(\frac{a h \overline{n_1}}{q}\Bigr)=e\Bigl(\frac{- a h \overline{q}}{n_1}\Bigr)+O\Bigl(\frac{\log^5{x}}{x}\Bigr).
\]
The error term above contributes $\ll x^{o(1)} N^2 H_1 / (x Q)\ll N^2/(x^\epsilon Q)$ since $H_1=Q N/x^{1-o(1)}<x^{1-2\epsilon}$. Thus we find that
\[
\widetilde{\mathscr{B}}\ll x^{o(1)}\widetilde{\mathscr{B}_2}+\frac{N^2}{Q x^\epsilon},
\] 
where (replacing $h$ with $-h$ for convenience)
\begin{align*}
\widetilde{\mathscr{B}_2}&:=\sum_{e\sim E}\sum_{1\le |h|\le H_1}\sum_{\substack{n_1,n_2\sim N}}\Bigl|\sum_{\substack{q\\ (q,a n_1 n_2)=1}}\frac{1}{q\phi(q e)}\psi_0\Bigl(\frac{q}{Q}\Bigr)\hat{\psi_0}\Bigl(\frac{h M}{q}\Bigr)e\Bigl(\frac{a h \overline{q}}{n_1}\Bigr)\Bigr|.
\end{align*}
Using the identity
\[
\frac{1}{\phi(q e)}=\frac{1}{q\phi(e)}\sum_{\substack{f_1|q\\ f_1\nmid e}}\frac{\mu(f_1)^2}{\phi(f_1)}
\]
and M\"obius inversion to remove the condition $(q,an_2)=1$, we see that the inner sum is bounded by
\[
\frac{\log{x}}{E}\sum_{\substack{f_1\ll x\\(f_1,n_1e)=1}}\sum_{\substack{f_2|a n_2\\ (f_2,n_1)=1}}\frac{\mu(f_1)^2}{\phi(f_1)}\Bigl|\sum_{\substack{q\\ (q,n_1)=1\\ f_1 f_2|q}}\frac{1}{q^2}\psi_0\Bigl(\frac{q}{Q}\Bigr)\hat{\psi_0}\Bigl(\frac{h M}{q}\Bigr)e\Bigl(\frac{a h \overline{q}}{n_1}\Bigr)\Bigr|.
\]
We let $q=f_1f_2 q'$, and note that
\[
\frac{\partial^j}{(\partial q')^j}\Bigl(\psi_0\Bigl(\frac{q'f_1f_2}{Q}\Bigr)\widehat{\psi_0}\Bigl(\frac{-hM}{q'f_1f_2}\Bigr)\Bigr)\ll_j (q')^{-j}.
\]
Thus, using partial summation to remove the $\psi_0(q/Q)\hat{\psi_0}(h M/q)/q^2$ weight, we see that this is
\[
\ll \frac{\log{x}}{Q^2 E}\sum_{\substack{f_1\ll x\\(f_1,n_1e)=1}}\sum_{\substack{f_2|a n_2\\ (f_2,n_1)=1}}\frac{\mu(f_1)^2}{\phi(f_1)}\sup_{Q',Q''\asymp Q/f_1 f_2}\Bigl|\sum_{\substack{Q' \le q'\le Q''\\ (q',n_1)=1}}e\Bigl(\frac{a h \overline{f_1f_2q'}}{n_1}\Bigr)\Bigr|.
\]
Finally, Lemma \ref{lmm:Weil} shows that the inner sum is $\ll N^{1/2+o(1)}+Q N^{-1}(h,n_1)$. Thus we find that
\begin{align*}
\widetilde{\mathscr{B}_2}&\ll  \frac{x^{o(1)} }{Q^2 E}\sum_{e\sim E}\sum_{1\le |h|\le H_1}\sum_{n_1,n_2\sim N} \Bigl(N^{1/2}+\frac{Q}{N}(h,n_1)\Bigr)\\
&\ll x^{o(1)}\frac{H_1 N^{5/2}}{ Q^2}+x^{o(1)}\frac{H_1 N }{Q}\\
&\ll \frac{N^2}{Q x^{3\epsilon/2}}\Bigl(\frac{N^{3/2}}{ x^{1-2\epsilon}}+\frac{Q}{x^{1-2\epsilon} }\Bigr).
\end{align*}
It suffices to show that $\widetilde{\mathscr{B}_2}\ll N^2 R/(Q x^{3\epsilon/2})$, and so we are done provided
\begin{align}
N^{3/2}&< x^{1-2\epsilon},\qquad
Q< x^{1-2\epsilon}.
\end{align}
This gives the result.
\end{proof}
%
%
%
%
%
%
%
%
\begin{lmm}[First reduction of exponential sum]\label{lmm:BFI0}
Let $M\ge x^\epsilon$, $Q\le x^{1-\epsilon}$, $N\asymp KL$, $MN\asymp x$, $H:=Q N (\log{x})^5/x$, $1\le E\le N/Q$, and let
\[
\alpha_{n}=\sum_{\substack{k\ell=n\\ k\sim K\\ \ell\sim L}}\beta_k\gamma_\ell
\]
for some 1-bounded coefficients $\beta_k,\gamma_\ell$.  Define
\begin{align*}
\mathscr{B}&:=\sum_{e\sim E}\mu^2(e)\sum_{\substack{q\\  (q,a)=1}}\frac{1}{q}\psi_0\Bigl(\frac{q}{Q}\Bigr)\sum_{\substack{n_1,n_2\sim N\\ n_1\equiv n_2\Mod{q e}\\ (n_1,q n_2)=1\\ (n_2,q n_1)=1\\ |n_1-n_2|\ge N/(\log{x})^C}}\alpha_{n_1}\overline{\alpha_{n_2}}\sum_{1\le |h|\le H}\hat{\psi_0}\Bigl(\frac{-M h}{q}\Bigr)e\Bigl(\frac{ah\overline{n_1}}{q}\Bigr).
\end{align*}
Then we have for any $A>0$
\[
\mathscr{B}\ll \frac{(\log{x})^{C+O(1)}N^{1/2}K^{1/2} }{Q}\sup_{\substack{H'\le H\\ R\le 10aN/QE\\ E'\le 2E\\ \theta\in[0,1]\\ a'|a}}|\mathscr{B}_0'|^{1/2}+O_{A}\Bigl(\frac{N^2}{Q(\log{x})^{2A}}\Bigr),
\]
where $\mathscr{B}'_0=\mathscr{B}_0'(\theta,a',R,H',E')$ is given by
\[
\mathscr{B}_0'=\sum_{r'\sim R}\sum_{n_1\sim N}\sum_{\substack{k\sim K\\ (k,r' n_1)=1}}r'\Bigl|\sum_{\substack{E\le e\le E'\\ (e,k n_1)=1}}\mu^2(e)\sum_{\substack{\ell\sim L\\ \ell \equiv \overline{k}n_1\Mod{r' e} \\ (\ell,n_1)=1}}\overline{\gamma_\ell}e(\ell\theta)\sum_{ h\sim H'} e\Bigl(\frac{a' h r' e\overline{\ell k}}{n_1}\Bigr)\Bigr|^2.
\]
\end{lmm}
\begin{proof}
This is similar to the initial proof of \cite[Theorem 4]{BFI2} or \cite{FouvryRadziwill}, but with a slightly different setup. An essentially identical proof goes through, but for completeness we have included an explicit proof.

 We note that if $(n_1,n_2)=1$ then we automatically have $(n_1n_2,q)=1$ from $n_1\equiv n_2\Mod{q}$, and so we may drop the conditions $(n_1,q)=(n_2,q)=1$.
 
 First we remove the condition $(q,a)=1$ by M\"obius inversion. This gives
\[
\mathscr{B}=\sum_{e\sim E}\mu^2(e)\sum_{d|a}\mu(d)\sum_{\substack{q\\  d|q}}\frac{1}{q}\psi_0\Bigl(\frac{q}{Q}\Bigr)\sum_{\substack{n_1,n_2\sim N\\ n_1\equiv n_2\Mod{q e}\\ (n_1,n_2)=1\\ |n_1-n_2|\ge N/(\log{x})^C}}\alpha_{n_1}\overline{\alpha_{n_2}}\sum_{1\le |h|\le H}\hat{\psi_0}\Bigl(\frac{-M h}{q}\Bigr)e\Bigl(\frac{ah\overline{n_1}}{q}\Bigr).
\]
We now simplify the exponential. By Bezout's identity (Lemma \ref{lmm:Bezout}) we have
\[
\frac{ah \overline{n_1}}{q}=-\frac{ah \overline{q}}{n_1}+\frac{a h}{n_1 q}\Mod{1}.
\]
Since $h\ll (Q\log^5{x})/M$ and $NM \asymp x$ this implies that
\[
e\Bigl(\frac{ah \overline{n_1}}{q}\Bigr)=e\Bigl(\frac{-ah \overline{q}}{n_1}\Bigr)+O\Bigl(\frac{\log^5{x}}{x}\Bigr).
\]
The error term above contributes a total
\[
\ll \sum_{n_1\ne n_2\sim N}\sum_{\substack{q,e\\ eq|n_1-n_2}}\frac{H\log^5{x}}{Qx}\ll \frac{N^2 H\log^{O(1)}{x}}{Q x}.
\]
This is acceptably small since $H=Q N/x^{1-o(1)}<x^{1-2\epsilon}$. Thus we may replace the exponential with $e(-ah\overline{q}/m)$.

We now change variables. Since $n_1\equiv n_2\Mod{q e}$ and $(n_1n_2,q)=1$ and $|n_1-n_2|>N/(\log{x})^C$, we have $n_1-n_2=q r e$ for some $r$ satisfying
\[
\frac{N}{10 QE(\log{x})^C}\le |r|\le \frac{10 N}{QE}.
\]
We therefore replace the $q$-summation with a summation over $r$. Putting $|r|$ into dyadic ranges, we find that $\mathscr{B}$ is given by
\begin{equation}
\mathscr{B}=\sum_{d|a}\mu(d)\sum_{\substack{R=2^j\\ N/(10QE\log^C{x})\le R\le 10N/QE}}\mathscr{B}_2(d,R)+O_{A}\Bigl(\frac{N^2}{Q\log^{2A}{x}}\Bigr),
\end{equation}
where $\mathscr{B}_2=\mathscr{B}_2(d,R)$ is given by
\begin{align*}
\mathscr{B}_2:=&\sum_{e\sim E}\mu^2(e)\sum_{|r| \sim R}\sum_{\substack{n_1,n_2\sim N\\ n_1\equiv n_2\Mod{d r e}\\ (n_1,n_2)=1\\ (n_1-n_2)/r>0\\ |n_1-n_2|\ge N/(\log{x})^C}}\frac{r e\alpha_{n_1}\overline{\alpha_{n_2}}}{n_1-n_2}\psi_0\Bigl(\frac{n_1-n_2}{e r Q} \Bigr)\\
&\qquad\times\sum_{1\le |h|\le H}\hat{\psi_0}\Bigl(\frac{-M h r e}{n_1-n_2}\Bigr)e\Bigl(\frac{-ah\overline{(n_1-n_2)/(re)}}{n_1}\Bigr).
\end{align*}
Since $(n_1,n_2)=1$, we see that $\overline{(n_1-n_2)/(re)}\equiv - r e\overline{n_2}\Mod{n_1}$ so we can simplify the argument of the exponential to $ahre\overline{n_2}/n_1$. 

We now wish to remove some of the dependencies between the variables $h,n_2,e$ and $r,n_1$. We separate them in $\psi_0,\hat{\psi_0}$ by noting that
\begin{align*}
\frac{\partial^{j_1+j_2+j_3}}{\partial h^{j_1}\partial n_2^{j_2} \partial^{j_3}e}\Bigl(\frac{r e}{n_1-n_2}\psi_0\Bigl(\frac{n_1-n_2}{e r Q} \Bigr)\hat{\psi_0}\Bigl(\frac{ -h M r e}{n_1-n_2}\Bigr)\Bigr)&\ll_{j_1,j_2}\frac{1}{Q} |h|^{-j_1}|n_1-n_2|^{-j_2}|e|^{-j_3}\\
&\ll \frac{(\log{x})^{C j_2}}{Q}|h|^{-j_1}|n_2|^{-j_2}|e|^{-j_3}.
\end{align*}
Therefore, by partial summation we find that
\[
\mathscr{B}_2\ll \frac{(\log{x})^{C+1}}{Q}\sup_{\substack{H''\le H\\ N'\le 2N\\ E'\le 2E}}|\mathscr{B}_3|,
\]
where
\[
\mathscr{B}_3:=\sum_{r\sim R}\sum_{n_1\sim N}\tau(n_1)\Bigl|\sum_{E\le e\le E'}\mu^2(e)\sum_{\substack{N\le n_2\le N'\\ n_2\equiv n_1\Mod{d r e}\\ (n_2,n_1)=1 \\ (n_1-n_2)/r>0\\ |n_1-n_2|\ge N/(\log{x})^C}}\overline{\alpha_{n_2}}\sum_{1\le |h|\le H''}e\Bigl(\frac{ahre\overline{n_2}}{n_1}\Bigr)\Bigr|.
\]
(Here we used the fact that $\beta,\gamma$ are 1-bounded so $|\alpha_{n_1}|\le \tau(n_1)$ and the symmetry in $r$ and $-r$ to just consider postive $r$.) Finally, we recall that
\[
\alpha_{n}=\sum_{\substack{k\ell=n\\ k\sim K\\ \ell\sim L}}\beta_k\gamma_\ell
\]
for some 1-bounded coefficients $\beta_k,\gamma_\ell$. Substituting this for $\alpha_{n_2}$, putting $h$ in dyadic intervals and using the symmetry between $h$ and $-h$, we see that for some $H'\le H$
\begin{align*}
\mathscr{B}_3&\le \log{x}\sum_{r \sim R}\sum_{n_1\sim N}\tau(n_1)\sum_{\substack{k\sim K\\ (k,d r n_1)=1}}\Bigl|\sum_{\substack{E\le e\le E'\\ (e,k n_1)=1}}\mu^2(e)\sum_{\substack{\ell\sim L\\ \ell \equiv \overline{k}n_1\Mod{d r e} \\ (\ell,n_1)=1\\ \ell\in\mathcal{I}(n_1,r,k)}}\overline{\gamma_\ell}\sum_{h\sim H'} e\Bigl(\frac{ah r e\overline{\ell k}}{n_1}\Bigr)\Bigr|\\
&=\log{x}\sum_{r \sim R}\sum_{n_1\sim N}\sum_{\substack{k\sim K\\ (k,dre n_1)=1}}c_{n_1,r,k}\sum_{\substack{E\le e\le E'\\ (e,k n_1)=1}}\mu^2(e)\sum_{\substack{\ell\sim L\\ \ell \equiv \overline{k}n_1\Mod{d r e} \\ (\ell,n_1)=1\\ \ell\in\mathcal{I}(n_1,r,k)}}\overline{\gamma_\ell}\sum_{h\sim H'} e\Bigl(\frac{ah r e\overline{\ell k}}{n_1}\Bigr),
\end{align*}
for some suitable coefficients $|c_{n_1,r,k}|\le \tau(n_1)$ and where $\mathcal{I}(n_1,r,k)$ is the interval in $[L,2L]$ such that $\ell$ satisfies
\[
(n_1-k\ell)/r>0,\quad |n_1-k\ell|\ge\frac{N}{(\log{x})^C},\quad N\le k\ell\le N'.
\]
(We note that $(n_1,n_2)=1$ implies that $(n_1n_2,dre)=1$ since $n_1\equiv n_2\Mod{dre}$, so we can insert the conditions $(k,dre)=1$ and $(e,n_1)=1$.) 
We now remove the dependency between $\ell$ and $n_1,r,k$ caused by $\mathcal{I}$ by noting that
\begin{align*}
\mathbf{1}_{\ell\in\mathcal{I}(n_1,r,k)}&=\int_0^1 e(\ell\theta)\Bigl(\sum_{j\in\mathcal{I}(n_1,r,k)}e(-j\theta)\Bigr)d\theta\\
&=\int_0^1e(\ell\theta) c'_{n_1,r,k,\theta}\min(L,|\theta|^{-1})d\theta,
\end{align*}
for some suitable 1-bounded coefficients $c'_{n_1,r,k,\theta}$. (Here we used the standard bound for an exponential sum over an interval.)  This gives
\begin{align*}
\mathscr{B}_3&\le \log{x}\int_0^1\min(L,|\theta|^{-1})\sum_{|r| \sim R}\sum_{n_1\sim N}\tau(n_1)\sum_{\substack{k\sim K\\ (k,d r e n_1)=1}}|\mathscr{B}_4|d\theta
\end{align*}
where
\begin{align*}
\mathscr{B}_4:=\sum_{\substack{E\le e\le E'\\ (e,k n_1)=1}}\mu^2(e)\sum_{\substack{\ell\sim L\\ \ell \equiv \overline{k}n_1\Mod{d r e} \\ (\ell,n_1)=1}}\overline{\gamma_\ell}e(\ell\theta)\sum_{h\sim H'} e\Bigl(\frac{ah r e\overline{\ell k}}{n_1}\Bigr).
\end{align*}
Using the $L^1$ bound
\[
\int_0^1 \min(L,|\theta|^{-1})d\theta\ll \log{x},
\]
we obtain
\begin{align*}
\mathscr{B}_3&\ll (\log{x})^2\sup_{\theta}\sum_{r \sim R}\sum_{n_1\sim N}\tau(n_1)\sum_{\substack{k\sim K\\ (k,r n_1)=1}}|\mathscr{B}_4|.
\end{align*}
Finally, applying Cauchy-Schwarz to the outer variables, and replacing $r$ with $r'=r d$ and $a$ with $a'=a/d$ we see that
\[
\mathscr{B}_3\ll (\log{x})^5 N K \sup_{R'\asymp R}\sum_{r'\sim R'}\sum_{n_1\sim N}\sum_{\substack{k\sim K\\ (k,r n_1)=1}}r' |\mathscr{B}_5|^2,
\]
where $\mathscr{B}_5$ is $\mathscr{B}_4$ with $d$ replaced by 1, $a$ replaced with $a'$ and $r$ replaced with $r'$. This gives the result.
\end{proof}
%
%
%
%
%
%
%
%
\begin{lmm}[Improved BFI exponential sum bound, Part I]\label{lmm:BFI1}
Let $\mathscr{B}'=\mathscr{B}'(C,D,K,E,H,N)$ be given by
\[
\mathscr{B}':=\sum_{ c\sim  C}\sum_{\substack{ d\sim D\\ (c,d)=1}}\sum_{\substack{ k\sim K \\ (k,c)=1}}k\Biggl|\sum_{\substack{e\sim E\\ (e,ck)=1}}\sum_{ h\sim H}\sum_{\substack{ n\sim N\\ d n\equiv c\Mod{k e}\\ (n,k e c)=1}}\beta(h,e,n)e\Bigl(\frac{a h k e \overline{d n}}{c}\Bigr)\Biggr|^2.
\]
Then, if $|\beta(h,e,n)|\le 1$ is supported on square-free $e$ and $H\le N\le C$ and $C,D,K\ll x^{O(1)}$, we have
\begin{align*}
\mathscr{B}'&\ll (C D H K N + C E H K^2 N+D H^2 N^2+D H^2 K N)\log^{O(1)}{x}\\
&\qquad +\frac{x^{o(1)} D }{C E^2}\sup_{\substack{S_1\le E_0\le E\\ S_2\le K \\ L\le (C E^2 K \log^5{x})/(D E_0 S_1 S_2)}}E_0(\mathscr{B}''_=+\mathscr{B}''_{\ne}),
\end{align*}
where $\mathscr{B}_{=}''=\mathscr{B}''_{=}(S_1,S_2,K,E_0,E,N,H,C,L)$ and $\mathscr{B}_{\ne}''=\mathscr{B}''_{\ne}(S_1,S_2,K,E_0,E,N,H,C,L)$ are given by
\begin{align*}
\mathscr{B}''_=&:=\sum_{\substack{s_1\sim S_1\\ s_2\sim S_2}}\sum_{\substack{k'\sim K/s_2}} \sum_{\substack{e_0'\sim E_0/s_1}}\sum_{\substack{e_1',e_2'\sim E/e_0'}}\sum_{\substack{n\sim N\\ (n,s_1s_2 k' e_0'e_1'e_2')=1}}\sup_{\theta\in\mathbb{R}}\sum_{\substack{h_1,h_2\sim H\\ h_1e_1'\ne h_2e_2'}}\\
&\times\Bigl|\sum_{\substack{c'\\ (c',n e_1'e_2')=1}}\sum_{\ell\sim L}g_{0}\Bigl(\frac{c'}{C/(S_1S_2)}\Bigr)e(\ell \theta)S(a(h_1e_1'-h_2e_2')\overline{n e_1' e_2'},\ell;c' s_1 s_2)\Bigr|,\\
\mathscr{B}''_{\ne}&:=\sum_{\substack{s_1\sim S_1\\ s_2\sim S_2}}\sum_{\substack{k'\sim K/s_2}} \sum_{\substack{e_0'\sim E_0/s_1}}\sum_{\substack{e_1',e_2'\sim E/e_0'}}\sum_{\substack{n_1,n_2\sim N\\ n_1\equiv n_2\Mod{s_1 s_2 k' e_0'}\\ (n_1n_2,s_1s_2 k' e_0')=1\\ (n_1,e_1')=1=(n_2,e_2')\\ n_1\ne n_2}}\\
&\times\sup_{\theta\in\mathbb{R}}\sum_{\substack{h_1,h_2\sim H\\ h_1e_1'n_2\ne h_2e_2'n_1}}\Bigl|\sum_{\substack{c'\\ (c',n_1n_2e_1'e_2')=1}}\sum_{\ell\sim L}g_{0}\Bigl(\frac{c'}{C/(S_1S_2)}\Bigr)e(\ell \theta)S(r,\ell;c' s_1 s_2)\Bigr|,\\
r&:=a (h_1e_1'n_2-h_2e_2'n_1)\overline{n_1n_2e_1'e_2'},
\end{align*}
for some smooth function $g_0(t)$ supported on $t\asymp 1$ and satisfying $\|g_0^{(j)}\|_\infty\ll_j 1$ for each $j\ge 0$.
\end{lmm}
\begin{proof}
First we insert smooth majorants for the $c$ and $d$ summation using the function
\begin{equation}
f_0(t):=\int_0^\infty \psi_0(y)\psi_0\Bigl(\frac{t}{y}\Bigr)\frac{dy}{y}.
\label{eq:F0Def}
\end{equation}
We note that $f_0$ is supported on $[1/10,10]$, $f_0\ge 1$ on $[1,2]$ and $\|f_0^{(j)}\|_\infty\ll_j 1$ for all $j\ge 1$. We then expand out the square, giving
\begin{align*}
\mathscr{B}'&\le \sum_{c}f_0\Bigl(\frac{c}{C}\Bigr)\sum_{\substack{d\\ (c,d)=1}}f_0\Bigl(\frac{d}{D}\Bigr)\sum_{\substack{k\sim K \\ (k,c)=1}}k\Biggl|\sum_{\substack{e\sim E\\ (e,ck)=1}}\sum_{ h\sim H}\sum_{\substack{ n\sim N\\ d n\equiv c\Mod{k e}\\ (n,c)=1}}\beta(h,e,n)e\Bigl(\frac{a h k e \overline{d n}}{c}\Bigr)\Biggr|^2\\
&=\sum_{k\sim K}k\sum_{\substack{e_1,e_2\sim E\\ (e_1e_2,k)=1}}\sum_{\substack{n_1,n_2\sim N\\ n_1\equiv n_2\Mod{k}\\ (n_1,ke_1)=1\\ (n_2,ke_2)=1}}\sum_{h_1,h_2\sim H}\beta(h_1,e_1,n_1)\overline{\beta(h_2,e_2,n_2)}\\
& \times \sum_{\substack{c\\ (c,e_1e_2n_1n_2k)=1}}f_0\Bigl(\frac{c}{C}\Bigr)\sum_{\substack{d\\ d\equiv c\overline{n_1}\Mod{ke_1}\\ d\equiv c\overline{n_2} \Mod{ke_2}}}f_0\Bigl(\frac{d}{D}\Bigr)e\Bigl(\frac{ak(h_1e_1n_2-h_2e_2n_1)\overline{n_1n_2d}}{c}\Bigr).
\end{align*}	
We consider contribution from `diagonal' terms with $h_1e_1n_2=h_2e_2n_1$ and off-diagonal terms with $h_1e_1n_2\ne h_2e_2n_1$ separately. We let $\mathscr{B}_1'$ denote the diagonal terms and $\mathscr{B}_2'$ denote the off-diagonal terms, so we have the bound
\[
\mathscr{B}'\le \mathscr{B}'_1+\mathscr{B}'_2.
\]
We first consider the contribution from $\mathscr{B}_1'$, which we bound trivially. We see that
\begin{align*}
\mathscr{B}'_1&\ll \sum_{k\sim K}k\sum_{\substack{e_1,e_2\sim E\\ n_1,n_2\sim N\\ h_1, h_2\sim H\\ h_1e_1n_2=h_2e_2n_1}}\sum_{c}f_0\Bigl(\frac{c}{C}\Bigr)\sum_{\substack{d\equiv c\overline{n_1}\Mod{k e_1}}}f_0\Bigl(\frac{d}{D}\Bigr)\\
&\ll K^2C\sum_{b\le 8E H N}\tau_3(b)\Bigl(\frac{D}{K E}+1\Bigr)\\
&\ll (C D H K N+C E H K^2 N)\log^{O(1)}{x}.
\end{align*}
This gives the first two terms of the lemma.

We now consider $\mathscr{B}'_2$, the terms with $h_1e_1n_2\ne h_2e_2n_1$. Let $e_0=(e_1,e_2)$ and $e_1=e_0e_1'$, $e_2=e_0e_2'$ for some $e_1',e_2',e_0$ which we may assume to be pairwise coprime since $\beta(h,e,n)$ is supported on square-free $e$. We see that the inner sum is 0 unless $n_1\equiv n_2\Mod{e_0}$. Putting $e_0$ into dyadic ranges, we see that
\begin{align*}
\mathscr{B}'_2&\ll \log{x}\sup_{\substack{E_0\le E}}\Bigl|\sum_{\substack{k\sim K}}k\sum_{\substack{e_0\sim E_0\\ (e_0,k)=1}}\sum_{\substack{e_1',e_2'\sim E/e_0\\ (e_1' e_2',e_0k)=1\\ (e_1',e_2')=1}}\sum_{\substack{n_1,n_2\sim N\\ n_1\equiv n_2\Mod{k e_0}\\ (n_1 n_2,k e_0)=1\\ (n_1,e_1')=(n_2,e_2')=1}}\\
&\qquad\times\sum_{\substack{h_1,h_2\sim H\\ h_1e_1'n_2\ne h_2e_2'n_1}}\beta(h_1,e_0 e_1',n_1)\overline{\beta(h_2,e_0 e_2',n_2)}\sum_{\substack{c\\ (c,e_0e_1'e_2' k n_1n_2)=1}}f_0\Bigl(\frac{c}{C}\Bigr)\\
&\qquad \times \sum_{\substack{d\\ d\equiv c\overline{n_1}\Mod{k e_0 e_1'}\\ d\equiv c\overline{n_2} \Mod{k e_0 e_2'}}}f_0\Bigl(\frac{d}{D}\Bigr)e\Bigl(\frac{a k e_0(h_1 e_1' n_2 - h_2 e_2' n_1)\overline{n_1 n_2 d}}{c}\Bigr)\Bigr|.
\end{align*}
By Lemma \ref{lmm:InverseCompletion} we have that
\begin{align*}
&\sum_{\substack{d\\ d\equiv c\overline{n_1}\Mod{k e_0e_1'}\\ d\equiv c\overline{n_2} \Mod{e_2'}}}f_0\Bigl(\frac{d}{D}\Bigr)e\Bigl(\frac{a k e_0(h_1e_1'n_2-h_2e_2'n_1)\overline{n_1n_2d}}{c}\Bigr)=\frac{\hat{f_0}(0)D}{c k e_0 e_1' e_2'}\sum_{\substack{b\Mod{d}\\ (b,d)=1}}e\Bigl(\frac{r b}{c}\Bigr)\\
&\qquad\qquad+\frac{D}{c k e_0e_1'e_2'}\sum_{0<|\ell|<L_0}\hat{f_0}\Bigl(\frac{\ell D}{c k e_0 e_1' e_2'}\Bigr)e\Bigl(\frac{\ell\mu}{k e_0e_1'e_2'}\Bigr)S(r,\ell;c)+O(x^{-100}),
\end{align*}
where $r=r(h_1,h_2,e_1',e_2',n_1,n_2)$ and $L_0$ are given by
\begin{align*}
r&:=a (h_1e_1'n_2-h_2e_2'n_1)\overline{n_1n_2e_1'e_2'},\\
L_0&:=\frac{C E^2 K}{D E_0}\log^5{x},
\end{align*}
and $\mu=\mu(n_1,n_2,k,e_0,e_1',e_2')$ is the solution $\Mod{ke_0e_1'e_2'}$ of the congruences
\begin{align*}
\mu&\equiv \overline{n_1}\Mod{ke_0e_1'},\\
\mu&\equiv \overline{n_2}\Mod{ke_0e_2'}.
\end{align*}
(There is a unique solution by the Chinese Remainder Theorem since $n_1\equiv n_2\Mod{ke_0}$ and $(e_1',e_2')=1$.) We note that $\mu$ does not depend on $c$ or $\ell$. Substituting this into $\mathscr{B}_2'$, we find that
\[
\mathscr{B}_2'\le \log{x}\sup_{E_0\le E}(\mathscr{B}_3'+\mathscr{B}_4')+O(x^{-90}),
\]
where $\mathscr{B}_3'$ is the contribution from the first term involving $\hat{f}_0(0)$ term, and $\mathscr{B}_4'$ is the contribution from the sum over $\ell$, given explicitly by
\begin{align*}
\mathscr{B}_4'&:=
\sum_{\substack{k\sim K}}k\sum_{\substack{e_0\sim E_0\\ (e_0,k)=1}}\sum_{\substack{e_1',e_2'\sim E/e_0\\ (e_0k,e_1'e_2')=1\\(e_1',e_2')=1}}\sum_{\substack{n_1,n_2\sim N\\ n_1\equiv n_2\Mod{k e_0}\\ (n_1n_2, k e_0)=1\\ (n_1,e_1')=1=(n_2,e_2')}}\sum_{\substack{h_1,h_2\sim H\\ h_1e_1'n_2\ne h_2e_2'n_1}}\\
&\qquad\times\beta(h_1,e_0e_1',n_1)\overline{\beta(h_2,e_0e_2',n_2)}\sum_{\substack{c\\ (c,e_0 e_1'e_2' k n_1n_2)=1}}f_0\Bigl(\frac{c}{C}\Bigr)\frac{D}{c k e_0e_1'e_2'}\\
&\qquad\times\sum_{0<\ell<L_0}\hat{f_0}\Bigl(\frac{\ell D}{c k e_0e_1'e_2'}\Bigr)e\Bigl(\frac{\ell\mu}{k e_0e_1'e_2'}\Bigr)S(r,\ell;c').
\end{align*}
 First we consider $\mathscr{B}_3'$. The exponential sum over $b$ is a Ramanujan sum, and of size $(c,r)$. Therefore, bounding all terms trivially, we find that
\begin{align*}
\mathscr{B}_3'&\ll \sum_{e_0\sim E_0}\sum_{\substack{k\sim K}}K\sum_{\substack{e_1',e_2'\sim E/e_0}}\sum_{\substack{n_1,n_2\sim N\\ n_1\equiv n_2\Mod{k e_0}}}\sum_{h_1,h_2\sim H}\sum_{\substack{c\\ (c,n_1n_2)=1}}f_0\Bigl(\frac{c}{C}\Bigr)\frac{D E_0 (c,r)}{C E^2 K}\\
&\ll E_0 K^2 \frac{E^2}{E_0^2} N\Bigl(\frac{N}{K E_0}+1\Bigr) H^2 C \frac{D E_0}{C E^2 K}(\log{x})^{O(1)}\\
&\ll ( D H^2 N^2+D H^2 K N) (\log{x})^{O(1)}.
\end{align*}
This gives the third and fourth terms of the lemma. 

Finally, we are left to consider $\mathscr{B}_4'$. First we wish to remove the condition $(c,e_0k)=1$. We do this by M\"obius inversion $\sum_{s_1|c,e_0}\mu(s_1)\sum_{s_2|k,c}\mu(s_2)$, and write $c=c' s_1 s_2$, $k=k' s_2$ and $e_0=e_0's_1$. Putting $\ell$, $s_1$ and $s_2$ into dyadic ranges then gives 
\[
\mathscr{B}_4'\le K x^{o(1)}\sup_{\substack{S_1,S_2\\ L\le L_0}}\mathscr{B}'_5,
\]
where
\begin{align*}
\mathscr{B}'_5&:=\sum_{\substack{s_1\sim S_1\\ s_2\sim S_2\\ (s_1,s_2)=1}}\sum_{\substack{k'\sim K/s_2\\ (k',s_1)=1}}\sum_{\substack{e_0'\sim E_0/s_1\\ (e_0',k' s_2)=1}}\sum_{\substack{e_1',e_2'\sim E/e_0's_1\\ (k' e_0's_1s_2,e_1'e_2')=1\\(e_1',e_2')=1}}\sum_{\substack{n_1,n_2\sim N\\ n_1\equiv n_2\Mod{s_1s_2 k' e_0'}\\ (n_1n_2, s_1s_2k' e_0')=1\\ (n_1,e_1')=1=(n_2,e_2')}}\mathscr{B}_6'\\
\mathscr{B}_6'&:=\sum_{\substack{h_1,h_2\sim H\\ h_1e_1'n_2\ne h_2e_2'n_1}}\Bigl|\sum_{\substack{c'\\ (c',e_1'e_2'n_1n_2)=1}}f_0\Bigl(\frac{c' s_1 s_2}{C}\Bigr)\frac{D}{c' k' e_0' s_1^2s_2^2 e_1'e_2'}\sum_{\ell\sim L}\hat{f_0}\Bigl(\frac{\ell D}{c' k' s_1^2s_2^2 e_0'e_1'e_2'}\Bigr)\\
&\qquad\times e\Bigl(\frac{\ell\mu}{s_1s_2k' e_0'e_1'e_2'}\Bigr)S(r,\ell;c's_1s_2)\Bigr|.
\end{align*}
This is almost in the correct form for the lemma, but we need to separate the $\ell$ and $c'$ variables from the others. Recalling the definition \eqref{eq:F0Def} of $f_0$, we note that for $\lambda\in\mathbb{R}_{>0}$ we have
\begin{align*}
\hat{f_0}\Bigl(\frac{\lambda\ell}{c'}\Bigr)&=\int_{0}^{\infty}\int_{0}^{\infty}\psi_0(y)\psi_0\Bigl(\frac{x}{y}\Bigr)e\Bigl(-\ell\frac{\lambda x}{c'}\Bigr)\Bigr)\frac{dy}{y}dx\\
&=\frac{c'}{\lambda}\int_{0}^{\infty}\int_{0}^{\infty}\psi_0(c' v)\psi_0\Bigl(\frac{u}{\lambda v}\Bigr)e(-\ell u)\frac{dv}{v}du,
\end{align*}
where we made the substitutions $u=\lambda x/c'$, $v=y/c'$. In particular, since $\psi_0$ is supported on $[1/2,5/2]$, we see that
\begin{align*}
\frac{D}{s_1^2s_2^2 c' k' e_0'e_1'e_2'}&\hat{f_0}\Bigl(\frac{\ell D}{s_1^2s_2^2 c' k' e_0'e_1'e_2'}\Bigr)\\
&=\int_{V}^{10V}\int_{W}^{10W}\psi_0(c' v)\psi_0\Bigl(\frac{ws_1^2s_2^2k'e_0'e_1'e_2'}{Dv}\Bigr)e(-\ell w)\frac{dv}{v}dw,
\end{align*}
where $V\asymp S_1S_2/C$ and $W\asymp DE_0/(C E^2 K)$. Similarly, we have
\[
f_0\Bigl(\frac{c's_1s_2}{C}\Bigr)=\int_{U}^{10U} \psi_0(s_1s_2 u)\psi_0\Bigl(\frac{c' }{C u}\Bigr)\frac{du}{u}
\]
where $U\asymp 1/(S_1S_2)$. Substituting these into our definition for $\mathscr{B}_6'$ and taking the worst $w$ we find that
\[
\mathscr{B}_6'\ll \frac{DE_0}{C E^2 K}\int_{U}^{10U}\int_V^{10V}\sup_{\theta\in\mathbb{R}}|\mathscr{B}_7'|\frac{du dv}{uv},
\]
where
\[
\mathscr{B}_7':=\sum_{\substack{h_1,h_2\sim H\\ h_1e_1'n_2\ne h_2e_2'n_1}}\Bigl|\sum_{\substack{c'\\ (c',n_1n_2e_1'e_2')=1}}\sum_{\ell\sim L}g_{u,v}\Bigl(\frac{c'}{C/(S_1S_2)}\Bigr)e(\ell \theta)S(r,\ell;c' s_1 s_2)\Bigr|,
\]
and where $g_{u,v}(t)$ is the smooth function supported on $t\asymp 1$ defined by
\[
g_{u,v}(t):=\psi_0\Bigl(\frac{vtC}{S_1S_2}\Bigr)\psi_0\Bigl(\frac{t}{S_1S_2 u}\Bigr).
\]
Here we crucially made use of the fact that $\mu$ doesn't depend on $c'$ or $\ell$. Taking the worst $u,v$, and substituting this to give a bound for $\mathscr{B}_5'$, we find that
\begin{align*}
\mathscr{B}_5'&\ll x^{o(1)}\frac{ D E_0}{C E^2 K}\sup_{\substack{u\asymp 1/(S_1S_2)\\ v\asymp S_1S_2/C}}\sum_{\substack{s_1\sim S_1\\ s_2\sim S_2\\ (s_1,s_2)=1}}\sum_{\substack{k'\sim K/s_2\\ (k',s_1)=1}} \sum_{\substack{e_0'\sim E_0/s_1\\ (e_0',k' s_2)=1}}\sum_{\substack{e_1',e_2'\sim E/e_0's_1\\ (k'e_0's_1s_2,e_1'e_2')=1\\ (e_1',e_2')=1}}\\
&\qquad \times\sum_{\substack{n_1,n_2\sim N\\ n_1\equiv n_2\Mod{s_1 s_2 k' e_0'}\\ (n_1n_2,s_1s_2 k' e_0')=1\\ (n_1,e_1')=1=(n_2,e_2')}}\sup_{\theta\in \mathbb{R}}|\mathscr{B}_7'|.
\end{align*}
Finally, we separate $\mathscr{B}_5'$ into the terms $\mathscr{B}''_=$ with $n_1=n_2$ and the terms $\mathscr{B}''_{\ne}$ with $n_1\ne n_2$. After dropping some summation constraints on $s_1,s_2,k',e_0',e_1',e_2'$ for an upper bound, we see this gives the result of the lemma.
\end{proof}
%
%
%
%
%
%
%
%
\begin{lmm}[Improved BFI exponential sum bound, Part II]\label{lmm:BFI2}
Let $\mathscr{B}''_{\ne}$ be as in Lemma \ref{lmm:BFI1}, and assume that $H\ll EN$. Then we have that
\begin{align*}
\mathscr{B}_{\ne}''&\ll \frac{x^\epsilon C E^2}{D E_0}\Bigl(E^{4} H^2 K N^{4}\Bigr)^{1/2}\Bigl(1+\frac{C D N}{E H K}\Bigr)^{7/64}\Bigl(C D N^2+ E H K N+C^2 K\Bigr)^{1/2}.
\end{align*}
\end{lmm}
\begin{proof}
Let $n_0=(n_1,n_2)$, and write $n_1=n_0n_1'$ and $n_2=n_0n_2'$ for some $(n_1',n_2')=1$. We put $n_0$ into a dyadic region, and insert 1-bounded coefficients to remove the absolute values. Thus we see that
\[
\mathscr{B}''_{\ne}\le \log{x}\sup_{N_0\le N}\mathscr{B}''_{2},
\]
where
\begin{align*}
\mathscr{B}''_{2}&:=\sum_{\substack{s_1\sim S_1\\ s_2\sim S_2}}\sum_{\substack{k'\sim K/s_2}} \sum_{\substack{e_0'\sim E_0/s_1}}\sum_{\substack{e_1',e_2'\sim E/e_0's_1}}\sum_{\substack{n_0\sim N_0\\ (n_0,s_1s_2k'e_0'e_1'e_2')=1}}\sum_{\substack{n_1',n_2'\sim N/n_0\\ n_1'\equiv n_2'\Mod{s_1s_2 k' e_0'}\\ (n_1'n_2',k'e_0's_1s_2)=1\\ (n_1'e_2',e_1'n_2')=1\\ n_1\ne n_2}}\\
&\times \sum_{\substack{h_1,h_2\sim H\\ h_1e_1'n_2'\ne h_2e_2'n_1'}}\xi\sum_{\substack{(c',r_2)=1}}g_0\Bigl(\frac{c'S_1S_2}{C}\Bigr)\sum_{\ell\sim L}e(\ell\theta)S(m\overline{r_2},\ell;c' s_1 s_2),
\end{align*}
where $\xi=\xi(s_1,s_2,k',e_0',n_0,e_1',e_2',n_1',n_2',h_1,h_2)\in \mathbb{C}$ is bounded by 1, where $\theta=\theta(n_0,n_1',n_2',e_0,e_1',e_2',k',s_1,s_2)\in\mathbb{R},$ and where
\begin{align*}
m&:=a(h_1e_1'n_2'-h_2e_2'n_1'),\\
r_2&:=n_0n_1'n_2'e_1'e_2'.
\end{align*}
Let $b_m=b_m(e_1',e_2',n_1',n_2',s_1,s_2,k',e_0,n_0)$ be the sequence given by
\[
b_{m}:=
\sum_{\substack{h_1,h_2\sim H \\ m=a(h_1e_1'n_2'-h_2e_2'n_1')}}\xi.\]
We consider dyadic ranges $r_2\sim R$ and $m\sim M$ separately, so taking the worst range and dropping some of the summation constraints for an upper bound, we find
\[
\mathscr{B}_2''\le x^{o(1)}\sup_{\substack{M\ll HEN/E_0N_0\\ R\ll N^2E^2/N_0E_0^2\\ L'\le L}}\sum_{\substack{s_1\sim S_1\\ s_2\sim S_2}}\sum_{\substack{k'\sim K/s_2}} \sum_{\substack{e_0'\sim E_0/s_1}}\sum_{\substack{e_1',e_2'\sim E/e_0's_1}}\sum_{\substack{n_0\sim N_0}}\sum_{\substack{n_1',n_2'\sim N/n_0\\ n_1'\equiv n_2'\Mod{s_1s_2 k' e_0'}\\ (n_1'e_2',e_1'n_2')=1\\ n_1\ne n_2}}|\mathscr{B}_3''|,
\]
where
\[
\mathscr{B}_3'':=\sum_{m\sim M}b_{m}\sum_{\ell\sim L}e(\ell\theta)\sum_{\substack{c'\\(c',r_2)=1}}g_0\Bigl(\frac{c'S_1S_2}{C}\Bigr)S(m\overline{r_2},\ell;c's_1s_2).
\]
By Lemma \ref{lmm:DeshouillersIwaniec2} and Lemma \ref{lmm:KimSarnak}, we have that (noting that $c'$ is of size $C/(S_1S_2)$)
\[
\mathscr{B}_3''\ll x^{o(1)}\Bigl(1+\frac{C^2 R}{L M}\Bigr)^{7/64}\Bigl(C^2 R+L M+\frac{C^2 (M+L)}{S_1 S_2}+\frac{C^2 L M}{R S_1 S_2}\Bigr)^{1/2}L^{1/2}\|b_{m}\|.
\]
We see that
\begin{align*}
\sum_{s_1\sim S_1}&\sum_{s_2\sim S_2}\sum_{k'\sim K/s_2}\sum_{e_0'\sim E_0/s_1}\sum_{n_0\sim N_0}\sum_{\substack{e_1',e_2'\sim E/(e_0's_1)}}\sum_{\substack{n_1',n_2'\sim N/n_0\\ n_1'\equiv n_2'\Mod{s_1s_2 k' e_0'}\\n_1'\ne n_2'\\ (e_2'n_1',e_1'n_2')=1}}\|b_{m}\|^2\\
&\ll N_0 \sum_{e_1',e_2'\ll E/E_0}\sum_{\substack{n_1',n_2'\ll N/N_0\\ (e_2'n_1',e_1'n_2')=1\\ n_1'\ne n_2'}}\tau_4(n_1'-n_2')\sum_{m\sim M}\Bigl|\sum_{\substack{h_1,h_2\sim H\\ a(h_1e_1'n_2'-h_2e_2'n_1')=m}}1\Bigr|^2\\
&\ll x^{o(1)}N_0\sum_{m\sim M}\sum_{n_2'\ll N/N_0}\sum_{e_1'\ll E/E_0}\sum_{h_1\sim H}\sum_{\substack{h_2,e_2',n_1'\\ h_2e_2'n_1'|m+h_1e_1'n_2'\\ (e_2'n_1',e_1'n_2')=1}}\\
&\qquad\times\sum_{\substack{h_1'\sim H\\ h_1'e_1'n_2'\equiv m\Mod{e_2'n_1'}}}\sum_{\substack{h_2'\sim H\\ h_2'e_2'n_1'=m+h_1'n_2'e_1'}}1\\
&\ll x^{o(1)}N_0 \sum_{m\sim M}\sum_{n_2'\ll N/N_0}\sum_{e_1'\ll E/E_0}H \cdot x^{o(1)}\cdot \Bigl(1+\frac{H E_0 N_0}{E N}\Bigr)\cdot 1\\
&\ll x^{o(1)} E H M N\Bigl(1+\frac{h N_0}{E N}\Bigr).
\end{align*}
Thus
\begin{align*}
&\sum_{s_1\sim S_1}\sum_{s_2\sim S_2}\sum_{k'\sim K/s_2}\sum_{e_0'\sim E_0/s_1}\sum_{n_0\sim N_0}\sum_{\substack{e_1',e_2'\sim E/(e_0's_1)}}\sum_{\substack{n_1',n_2'\sim N/n_0\\ n_1'\equiv n_2'\Mod{s_1s_2 k' e_0'}\\ n_1'\ne n_2'\\ (e_2'n_1',e_1'n_2')=1}}\|b_{m}\|\\
&\qquad\ll x^{o(1)} \Bigl(1+ \frac{H N_0}{E N}\Bigr)^{1/2}\Bigl(\frac{E^3 H N^3 M}{E_0^2 N_0}\Bigr)^{1/2}.
\end{align*}
Recalling that $L\ll x^{o(1)}C E^2 K/(D E_0)$, $R\asymp  E^2 N^2/(E_0^2 N_0^2)$, $M\ll E H N/(E_0 N_0)$, $S_1,S_2\ge 1$, $S_1S_2\ll E_0 K$ and $N_0,E_0\ge 1$, our bound simplifies to give
\begin{align*}
&\sum_{s_1\sim S_1}\sum_{s_2\sim S_2}\sum_{\substack{k'\sim K/s_2}}\sum_{\substack{e_0'\sim E_0/s_1}}\sum_{n_0\sim N_0}\sum_{\substack{e_1',e_2'\sim E/(e_0's_1)}}\sum_{\substack{n_1',n_2'\sim N/n_0\\ n_1'\equiv n_2'\Mod{s_1s_2k'e_0'}\\ n_1'\ne n_2' \\ (e_2'n_1',e_1'n_2')=1\\ r_2\sim R}}|\mathscr{B}_3''|\\
&\ll  x^{o(1)} \Bigl(1+ \frac{H N_0}{E N}\Bigr)^{1/2}\Bigl(\frac{E^3 H N^3 M}{E_0^2 N_0}\Bigr)^{1/2}\\
&\qquad \times  \Bigl(1+\frac{C^2 R}{L M}\Bigr)^{7/64}L^{1/2}\Bigl(C^2 R+L M+\frac{C^2(M+L)}{S_1 S_2}+\frac{C^2 L M}{R S_1 S_2}\Bigr)^{1/2}\\
&\ll x^{o(1)}\Bigl(1+\frac{H}{EN}\Bigr)^{1/2}\frac{E^2 H N^2}{E_0}\Bigl(1+\frac{C D N}{E H K}\Bigr)^{7/64}\Bigl(\frac{C E^2 K}{D}\Bigr)^{1/2} \\
&\qquad\qquad \times\Bigr( C^2 E^2 N^2+ \frac{C E^3 H K N}{D}+C^2 E H N+\frac{C^3 E^2 K}{D}+\frac{C^3 E H K}{D N}\Bigr)^{1/2}.
\end{align*}
In the final line above we have observed that the maximum occurs when $M$ and $L$ take their largest values and $E_0,N_0,S_1,S_2$ take their smallest values (even after extracting a factor $1/E_0$).

Since $H\ll EN$ by assumption, we see that $C^2E^2N^2\gg C^2EHN$ and $C^3E^2K/D\gg C^3EHK/(DN)$, so the first term is larger than the third term and the fourth term is larger than the final term in the final set of parentheses. Thus we may drop the third and final term in the final set of parentheses. Similarly, the $(1+H/(EN))$ factors are $O(1)$ and so can be dropped. Thus, simplifying the terms slightly we find that
\begin{align*}
\frac{D E_0}{ C E^2}\mathscr{B}_2''&\ll  x^{o(1)}E^2 H K^{1/2} N^{2} \Bigl(1+\frac{C D N}{E H K}\Bigr)^{7/64}\Bigl( C D N^2+E H K N+C^2K\Bigr)^{1/2}.
\end{align*}
This gives the result.
\end{proof}
%
%
%
%
%
%
%
%
\begin{lmm}[Improved BFI exponential sum bound, Part III]\label{lmm:BFI3}
Let $\mathscr{B}''_{=}$ be as in Lemma \ref{lmm:BFI1} with $H\ll E N$. Then we have that
\begin{align*}
\mathscr{B}_{=}''&\ll \frac{C E^2 x^{o(1)}}{D E_0}\Bigl(E^4 H^2 K^3 N^2\Bigr)^{1/2}\Bigl(1+\frac{H}{E}\Bigr)^{1/2}\Bigl(1+\frac{C D N}{E H K}\Bigr)^{7/64}\\
&\qquad \times \Bigl(C D N+ E H K+C^2 K\Bigr)^{1/2}.
\end{align*}
\end{lmm}
\begin{proof}
First we recall that
\begin{align*}
\mathscr{B}''_=&:=\sum_{\substack{s_1\sim S_1\\ s_2\sim S_2}}\sum_{\substack{k'\sim K/s_2}} \sum_{\substack{e_0'\sim E_0/s_1}}\sum_{\substack{e_1',e_2'\sim E/e_0\\ (e_1',e_2')=1}}\sum_{\substack{n\sim N}}\\
&\times\sup_{\theta\in\mathbb{R}}\Bigl|\sum_{\substack{h_1,h_2\sim H}}\xi_{h_1,h_2}\sum_{\substack{c'\\ (c',n e_1'e_2')=1}}\sum_{\ell}g_{0}\Bigl(\frac{c'}{C/(S_1S_2)}\Bigr)e(\ell \theta)S(m\overline{r},\ell;c's_1s_2)\Bigr|,
\end{align*}
for some 1-bounded coefficients $\xi_{h_1,h_2}$ (also depending on $s_1,s_2,k',e_0',e_1',e_2',n$) and where
\begin{align*}
m=a(h_1e_1'-h_2e_2'),\qquad 
r=n e_1'e_2'.
\end{align*}
Let $b_m'$ be the sequence
\[
b_m':=\sum_{\substack{h_1,h_2\sim H\\ a(h_1e_1'-h_2'e_2')=m}}\xi_{h_1,h_2}.
\]
Thus we have
\begin{align*}
&\sum_{h_1,h_2\sim H}\xi_{h_1,h_2}\sum_{(c',n e_1'e_2')=1}\sum_{\ell}g_{0}\Bigl(\frac{c'}{C/(S_1S_2)}\Bigr)e(\ell \theta)S(m\overline{r},\ell;c's_1s_2)\\
&=\sum_{\substack{M=2^j\\ M\ll HE/E_0}}\sum_{m\sim M}b_m'\sum_{\ell\sim L}e(\ell\theta)\sum_{(c',n e_1'e_2')=1}g_{0}\Bigl(\frac{c'}{C/(S_1S_2)}\Bigr)S(m\overline{r},\ell;c's_1s_2).
\end{align*}
By Lemma \ref{lmm:DeshouillersIwaniec2} and Lemma \ref{lmm:KimSarnak} we have
\begin{align*}
&\sum_{m\sim M}b_m'\sum_{\ell\sim L}e(\ell\theta)\sum_{(c',n e_1'e_2')=1}g_{0}\Bigl(\frac{c'}{C/(S_1S_2)}\Bigr)S(m\overline{r},\ell;c's_1s_2)\\
&\ll x^{o(1)} L^{1/2}\Bigl(1+\frac{C^2 R}{L M}\Bigr)^{7/64}\Bigl(C^2 R+L M+\frac{C^2 (M+L)}{S_1 S_2}+\frac{C^2 L M}{R S_1 S_2}\Bigr)^{1/2}\|b_m'\|.
\end{align*}
We have that
\begin{align*}
\sum_{\substack{s_1\sim S_1}}\sum_{s_2\sim S_2}&\sum_{\substack{k'\sim K/s_2}} \sum_{\substack{e_0'\sim E_0/s_1}}\sum_{\substack{e_1',e_2'\sim E/e_0\\ (e_1',e_2')=1}}\sum_{\substack{n\sim N}}\|b_m'\|^2\\
&\ll x^{o(1)} K E_0 N \sum_{m\sim M}\sum_{h_1\sim H}\sum_{e_1'\ll E/E_0}\sum_{\substack{e_2'\ll E/E_0\\ (e_1',e_2')=1\\ e_2|m+ah_1e_1'}}\\
&\qquad\times\sum_{\substack{h_2\sim H\\ ah_2e_2'=m+ah_1e_1'}}\sum_{\substack{h_1'\sim H\\ a h_1'\equiv m\overline{e_1'}\Mod{e_2'}}}\sum_{\substack{h_2'\sim H\\ ah_2'e_2'=m+ah_1'e_1'}}1\\
&\ll x^{o(1)} E H K M N\Bigl(1+\frac{E_0 H}{E}\Bigr).
\end{align*}
Thus we find that
\begin{align*}
\mathscr{B}_{=}''&\ll \sup_{\substack{M\ll E H/E_0 \\ L\ll x^{o(1)}C E^2 K/(DE_0) \\ 1\ll S_1S_2\ll E_0 K\\ R\asymp E^2 N/E_0^2}}x^{o(1)} \Bigl(E H K M N\Bigr)^{1/2}\Bigl(1+\frac{E_0 H}{E}\Bigr)^{1/2}\Bigl(\frac{E^2 K N}{E_0}\Bigr)^{1/2} \\
&\qquad \times \Bigl(1+\frac{C^2 R }{L M}\Bigr)^{7/64}(L^{1/2})\Bigl(RC^2+L M+\frac{C^2(M+L)}{S_1S_2}+\frac{C^2 ML}{RS_1S_2}\Bigr)^{1/2}.
\end{align*}
The maximum of the expression above clearly occurs when $M$ and $L$ take their largest values and when $S_1,S_2$ take their smallest values. After removing a factor $1/E_0$, we see that the expression is maximized when $E_0=1$. Thus we find that
\begin{align*}
\mathscr{B}_{=}''&\ll \frac{x^{o(1)}}{E_0}\Bigl(E^2 H^2 K N\Bigr)^{1/2}\Bigl(1+\frac{H}{E}\Bigr)^{1/2}\Bigl(E^2 K N\Bigr)^{1/2} \Bigl(\frac{C E^2 K}{D}\Bigr)^{1/2}\Bigl(1+\frac{C D N}{E H K}\Bigr)^{7/64}\\
&\qquad \times \Bigl(C^2 E^2 N+ \frac{C E^3 H K}{D}+C^2 E H+\frac{C^3 E^2 K}{D}+\frac{C^3 E H K}{D N}\Bigr)^{1/2}.
\end{align*}
Again, since $H\ll EN$ by assumption, we see that in the final set of parentheses the first term is larger than the third and the fourth term is larger than the final term. Thus, we obtain
\begin{align*}
\mathscr{B}_{=}''&\ll x^{o(1)}\frac{C E^2 }{D E_0}\Bigl(E^4 H^2 K^3 N^2\Bigr)^{1/2}\Bigl(1+\frac{H}{E}\Bigr)^{1/2}\Bigl(1+\frac{C D N}{E H K}\Bigr)^{7/64}\\
&\qquad \times\Bigl(C D N+ E H K+C^2 K\Bigr)^{1/2}.
\end{align*}
This gives the result.
\end{proof}
%
%
%
%
%
%
%
%
Putting together Lemma \ref{lmm:BFI1}, \ref{lmm:BFI2} and \ref{lmm:BFI3}, we obtain the following result.
%
%
%
%
%
%
%
%
\begin{lmm}[Improved BFI estimate]\label{lmm:BFICombined}
Let $\mathscr{B}'=\mathscr{B}'(C,D,K,E,H,N)$ be as in Lemma \ref{lmm:BFI1}. Then, if $|\beta(h,e,n)|\le 1$ is supported on square-free $e$ and $H\ll EN$ and $C,D,K\ll x^{O(1)}$, we have
\begin{align*}
\mathscr{B}'&\ll (C D H K N+C E H K^2 N+D H^2 N^2+DH^2KN)\log^{O(1)}{x}\\
&\quad+x^\epsilon (E K^2+EN^3+H K^2)^{1/2}\Bigl(E^3 H^2 K N^2\Bigr)^{1/2}\Bigl(1+\frac{C D N}{E H K}\Bigr)^{7/64}\\
&\qquad\qquad\times \Bigl( C D N +E H K\Bigr)^{1/2} \\
&\quad+x^\epsilon (E K^2 +E N^2+H K^2)^{1/2}\Bigl(C^2 E^3 H^2 K^2 N^2\Bigr)^{1/2}\Bigl(1+\frac{C D N}{E H K}\Bigr)^{7/64}.
\end{align*}
\end{lmm}
\begin{proof}
Applying Lemmas \ref{lmm:BFI1}, \ref{lmm:BFI2} and \ref{lmm:BFI3} in turn gives
\begin{align*}
\mathscr{B}'&\ll (C D H K N + C E H K^2 N+D H^2 N^2+D H^2 K N)\log^{O(1)}{x}\\
&+ x^\epsilon\Bigl(E^{4} H^2 K N^{4}\Bigr)^{1/2}\Bigl(1+\frac{C D N}{E H K}\Bigr)^{7/64}\Bigl(C D N^2+ E H K N+C^2 K\Bigr)^{1/2} \\
&+x^\epsilon\Bigl(E^4 H^2 K^3 N^2\Bigr)^{1/2}\Bigl(1+\frac{H}{E}\Bigr)^{1/2}\Bigl(1+\frac{C D N}{E H K}\Bigr)^{7/64}\Bigl(C D N+ E H K+C^2 K\Bigr)^{1/2}.
\end{align*}
The first line above matches the first line of the lemma. We note that
\begin{align*}
&\Bigl(E^4 H^2 K N^{4}\Bigr)^{1/2}\Bigl(1+\frac{C D N}{E H K}\Bigr)^{7/64}\Bigl( C D N^2+ E H K N\Bigr)^{1/2}\\
&\qquad+\Bigl(E^4 H^2 K^3 N^2\Bigr)^{1/2}\Bigl(1+\frac{H}{E}\Bigr)^{1/2}\Bigl(1+\frac{C D N}{E H K}\Bigr)^{7/64}\Bigl(C D N+ E H K\Bigr)^{1/2}\\
&=  (E K^2+EN^3+H K^2)^{1/2}\Bigl(E^3 H^2 K N^2\Bigr)^{1/2}\Bigl(1+\frac{C D N}{E H K}\Bigr)^{7/64}\Bigl( C D N +E H K\Bigr)^{1/2},
\end{align*}
which gives the second line of the lemma. Similarly
\begin{align*}
&\Bigl(E^{4} H^2 K N^{4}\Bigr)^{1/2}\Bigl(1+\frac{C D N}{E H K}\Bigr)^{7/64}\Bigl(C^2 K\Bigr)^{1/2} \\
& \qquad+\Bigl(E^4 H^2 K^3 N^2\Bigr)^{1/2}\Bigl(1+\frac{H}{E}\Bigr)^{1/2} \Bigl(1+\frac{CD N}{E H K}\Bigr)^{7/64}
\Bigl(C^2 K\Bigr)^{1/2}\\
&=(E N^2 + H K^2+EK^2)^{1/2}\Bigl(C^2 E^3 H^2 K^2 N^2\Bigr)^{1/2}\Bigl(1+\frac{C D N}{E H K}\Bigr)^{7/64}.
\end{align*}
Thus the result holds.
\end{proof}
%
%
%
%
%
%
%
%
\begin{lmm}\label{lmm:NewBFI}
Let $\mathscr{B}'=\mathscr{B}'(C,D,K,E,H,N)$ be as in Lemma \ref{lmm:BFI1}. Let $x^{1/2-\epsilon/2}<Q<x^{7/10}$, $H\ll (\log{x})^5 Q K L/x$, and let $K,L$ satisfy
\begin{align}
Qx^\epsilon &< KL,\\
K&<\frac{x^{1-2\epsilon}}{Q},\\
KL &< \frac{x^{153/224-10\epsilon}}{Q^{1/7}},\\
K L^{4}&<\frac{x^{57/32-10\epsilon}}{Q}.
\end{align}
Then there is a choice of $E$ satisfying $1\le E\le KL/(x^\epsilon Q)$ such that
\[
\sup_{\substack{N\ll KL\\ R\ll KL/(EQ)}}\mathscr{B}'(N,\,K,\, R,\,E,\,H,\,L)\ll \frac{K^2L^3}{x^{\epsilon/2}}.
\]
\end{lmm}
\begin{proof}
By Lemma \ref{lmm:BFICombined}, the upper bounds for $N$, $R$ and our bound $H\ll x^{o(1)}QKL/x$ for $H$, we see that $\mathscr{B}'=\mathscr{B}'(N,K,R,E,H,L)$ satisfies
\begin{align*}
\mathscr{B}'&\ll x^{o(1)}\Bigl(KL K\frac{Q K L}{x} \frac{KL}{E Q}L+K L E\frac{Q K L}{x} \frac{K^2 L^2}{E^2 Q^2} L+K \frac{Q^2K^2L^2}{x^2} L^2\Bigr)\\
&\qquad\quad\qquad+x^{o(1)} K\frac{Q^2K^2L^2}{x^2} \frac{KL}{E Q} L\\
&\quad+x^\epsilon \Bigl(E \frac{K^2 L^2}{E^2 Q^2}+E L^3+\frac{Q K L}{x} \frac{K^2 L^2}{E^2 Q^2}\Bigr)^{1/2}\Bigl(E^3 \frac{Q^2K^2L^2}{x^2} \frac{KL}{E Q} L^2\Bigr)^{1/2}\\
&\qquad\qquad\times\Bigl(1+\frac{K L K L}{E \frac{Q K L}{x}\frac{KL}{EQ}}\Bigr)^{7/64} \Bigl( KL K L +E \frac{Q K L}{x} \frac{KL}{E Q}\Bigr)^{1/2} \\
&\quad+x^\epsilon \Bigl(E\frac{K^2 L^2}{E^2 Q^2} +E L^2+\frac{Q K L}{x} \frac{K^2 L^2}{E^2 Q^2}\Bigr)^{1/2}\Bigl(K^2 L^2 E^3 \frac{Q^2 K^2 L^2 }{x^2 } \frac{K^2 L^2}{E^2 Q^2} L^2\Bigr)^{1/2}\\
&\qquad\qquad \times\Bigl(1+\frac{KLK L}{E \frac{Q K L}{x} \frac{K L}{E Q}}\Bigr)^{7/64}.
\end{align*}
This simplifies to give
\begin{align*}
\mathscr{B}'&\ll x^{o(1)}\Bigl(\frac{K^4 L^4}{E x} +\frac{K^4 L^5}{E Q x}+\frac{K^3 L^4 Q^2}{x^2}+\frac{K^4 L^4 Q}{E x^2}\Bigr)\\
&\quad+x^\epsilon \Bigl(\frac{K^2 L^2}{E Q^2}+E L^3+\frac{K^3 L^3}{E^2 Q x}\Bigr)^{1/2}\Bigl( \frac{E^2 K^3 L^5 Q}{x^2} \Bigr)^{1/2}x^{7/64}\Bigl( K^2 L^2 \Bigr)^{1/2} \\
&\quad+x^\epsilon \Bigl(\frac{K^2 L^2}{E Q^2} +E L^2+ \frac{K^3 L^3}{E^2 Q x}\Bigr)^{1/2}\Bigl( \frac{E K^6 L^8}{x^2}\Bigr)^{1/2}x^{7/64}.
\end{align*}
We recall that $x^{1/2-\epsilon/2}<Q$, and $E Q<x^\epsilon K L<x$ so that $x>K L$, $x>Q E$ and $K^3L^3/(E^2 Q x)>K^2L^2/(E Q^2)$. Moreover, the conditions on the proposition imply that $L,K<x^{1/2-\epsilon/2}<Q$, and so $K,L>L K /Q$ and $EL^3> L^3 K^3/(E^2 Q x)$. Thus
\begin{align*}
\mathscr{B}'&\ll x^{o(1)}\Bigl(\frac{K^4 L^4}{E x} +\frac{Q^2 K^3 L^4}{x^2}\Bigr)+x^\epsilon \Bigl(E L^3\Bigr)^{1/2}\Bigl( \frac{Q K^3 L^5 E^2}{x^2} \Bigr)^{1/2}x^{7/64}\Bigl( K^2 L^2 \Bigr)^{1/2} \\
&\quad+x^\epsilon \Bigl(E L^2+ \frac{K^3 L^3}{E^2 Q x}\Bigr)^{1/2}\Bigl( \frac{E K^6 L^8}{ x^2}\Bigr)^{1/2}x^{7/64}.
\end{align*}
We wish to show that $\mathscr{B}'\ll K^2L^3x^{-\epsilon/2}$. Most of the important terms are increasing in $E$, but we need to choose $E>K^2 L/x^{1-\epsilon/2}$ to handle the first term. Therefore we choose 
\[
E=\max\Bigl(\frac{K^2 L}{x^{1-\epsilon}},1\Bigr).
\]
For this to be valid we require that $x^{\epsilon} E Q<KL$, and so we impose the conditions
\begin{align}
KL&>Q x^\epsilon,\label{eq:BFICond1}\\
K&<\frac{x^{1-2\epsilon} }{Q}.\label{eq:BFICond2}
\end{align}
If this holds, we then have
\begin{align*}
\frac{\mathscr{B}'}{K^2L^3} &\ll x^{o(1)}\Bigl(\frac{1}{x^\epsilon} +\frac{Q^2 K L}{x^2}\Bigr)+x^{3\epsilon} \Bigl( \frac{Q K^{7} L^{7} }{x^5}+
\frac{Q K L^{4} }{x^2} \Bigr)^{1/2}x^{7/64}\\
&\quad +x^{3\epsilon}\Bigl(\frac{K^2 L^4}{x^2}+\frac{K^6 L^{6} }{x^4}+\min\Bigl(\frac{K^5 L^{5} }{Q x^3}, \frac{K^3 L^4 }{Q x^2}\Bigr)\Bigr)^{1/2}x^{7/64}\\
&\ll x^{\epsilon/2}\Bigl(\frac{1}{x^\epsilon} +\frac{Q^2 K L}{x^2}\Bigr)+x^{3\epsilon} \Bigl( \frac{Q^{32} K^{224} L^{224} }{x^{153} }+\frac{Q^{32}K^{32}L^{128}}{x^{57}}\Bigr)^{1/64}\\
&\quad +x^{3\epsilon}\Bigl(\frac{K^{64} L^{128}}{x^{57}}+\frac{K^{192} L^{192}}{x^{121}}+\frac{K^{160} L^{160} }{Q^{32} x^{89}} \Bigr)^{1/64}.
\end{align*}
This shows that $\mathscr{B}'\ll K^2 L^3 x^{-\epsilon/2}$ provided we have
\begin{align}
K L&<\min\Bigl(\frac{ x^{2-2\epsilon}}{Q^2},\frac{x^{153/224-2\epsilon}}{Q^{1/7}},x^{121/192-2\epsilon},x^{89/160-2\epsilon}Q^{1/5}\Bigr),\label{eq:BFICond3}\\
KL^2&<x^{57/64-4\epsilon},\label{eq:BFICond4}\\
KL^4&<\frac{x^{57/32-7\epsilon}}{Q}.\label{eq:BFICond5}
\end{align}
We note that for $Q\in[x^{2/5},x^{7/10}]$ we have that 
\[
\min\Bigl(\frac{ x^{2}}{Q^2},\frac{x^{153/224}}{Q^{1/7}},x^{121/192},x^{89/160}Q^{1/5}\Bigr)=\frac{x^{153/223}}{Q^{1/7}},
\]
so the constraint \eqref{eq:BFICond3} simplifies to 
\begin{equation}
KL\ll \frac{x^{153/224-2\epsilon}}{Q^{1/7}}.\label{eq:BFICond6}
\end{equation}
 Moreover, if $K,L$ satisfy \eqref{eq:BFICond6} and \eqref{eq:BFICond5} then we see that since $Q\ge x^{1/2-\epsilon}$
\[
KL^2= (KL)^{2/3}(KL^4)^{1/3}\ll \frac{x^{51/112}}{Q^{2/21}}\cdot \frac{x^{19/32}}{Q^{1/3}}=\frac{x^{235/224}}{Q^{3/7}}<x^{57/64-10\epsilon},
\]
so \eqref{eq:BFICond4} is automatically satified. Thus all the conditions are satisfied provided \eqref{eq:BFICond1}, \eqref{eq:BFICond2}, \eqref{eq:BFICond5} and \eqref{eq:BFICond6} hold, giving the result.
\end{proof}
%
%
%
%
%
%
%
%
\begin{proof}[Proof of Proposition \ref{prpstn:TripleRough}]
This is similar to \cite[Theorem 4]{BFI2}, but using our refinements Lemma \ref{lmm:NewBFI} and Proposition \ref{prpstn:GeneralDispersion} given above and keeping track of some of the quantities slightly more carefully.

First we note that by Lemma \ref{lmm:Divisor} the set of $k,l,m$ with $\max(|\eta_n|,|\lambda_\ell|,|\beta_m|)\ge(\log{x})^B$ has size $\ll x(\log{x})^{O_{B_0}(1)-B}$, so by Lemma \ref{lmm:SmallSets} these terms contribute negligibly if $B=B(A,B_0)$ is large enough. Thus, by dividing through by $(\log{x})^{3B}$ and considering $A+3B$ in place of $A$, it suffices to show the result when all the sequences are 1-bounded. ($\alpha_n$ still satisfies \eqref{eq:SiegelWalfisz} by Lemma \ref{lmm:SiegelWalfiszMaintain}.)

We insert coefficients $\gamma_q$ to remove the absolute values, and see that it suffices to show
\[
\sum_{\substack{q\sim Q\\ (q,a)=1}}\gamma_q\Delta_{\mathscr{B}}(q)\ll_{A} \frac{x}{(\log{x})^A}.
\]
This is a special case of the type of sum considered in Proposition \ref{prpstn:GeneralDispersion}, where $D=R=1$, $\lambda_{q,d,r}=1$ and
\[
\alpha_n=\sum_{\substack{k\ell=n\\ k\sim K\\ \ell\sim L}}\eta_k\lambda_\ell,
\]
which satisfies the Siegel-Walfisz estimate since $\eta_k$ does. Since $KL>Q(\log{x})^C$ by assumption, we have that $N>KL>Q(\log{x})^C$. The conditions of the proposition ensure that
\begin{align}
N\ll KL \ll  \frac{x^{153/224}}{Q^{1/7}} \ll x^{5/8},\label{eq:NBound}
\end{align}
and so $M\gg x/N \gg x^\epsilon$. Thus, by Proposition \ref{prpstn:GeneralDispersion}, it suffices to show that for a suitable choice of $E\in [1,KL(\log{x})^{-C_1}/Q]$ we have
\[
\mathscr{E}_1,\mathscr{E}_2\ll_{A} \frac{N^2}{Q (\log{x})^{C_1}},
\]
where $C_1=C_1(A)$ is a constant depending only on $A$, $H=Q(\log{x})^5/M$, and $\mathscr{E}_{1}$ is given by
\begin{align*}
\sum_{e\sim E}\mu^2(e)\sum_{\substack{q\\ (q,a)=1}}\psi_0\Bigl(\frac{q}{Q}\Bigr)\frac{1 }{\phi(q e)q}\sum_{\substack{n_1,n_2\sim N\\ (n_1n_2,q e)=1}}\alpha_{n_1}\overline{\alpha_{n_2}}\sum_{1\le |h|\le H_1}\hat{\psi}_0\Bigl(\frac{-h M}{q}\Bigr)e\Bigl( \frac{a h \overline{ n_1}}{q}\Bigr),
\end{align*}
and $\mathscr{E}_2$ is given by
\begin{align*}
\sum_{e\sim E}\mu^2(e)\sum_{\substack{q\\ (q,a)=1}}\psi_0\Bigl(\frac{q}{Q}\Bigr)\frac{1}{q}\sum_{\substack{n_1,n_2\sim N\\ n_1\equiv n_2\Mod{q e}\\ (n_1n_2,eq)=1\\(n_2,n_1)=1\\ |n_1-n_2|\ge N/(\log{x})^C}}\alpha_{n_1}\overline{\alpha_{n_2}}\sum_{1\le |h|\le H_2}\hat{\psi}_0\Bigl(\frac{-h M}{q}\Bigr)e\Bigl(\frac{ah\overline{n_1}}{q}\Bigr).
\end{align*}
$\mathscr{E}_1$ is of the form considered in Lemma \ref{lmm:BFISimple}. The conditions of the proposition imply that $Q<x^{2/3}$ and $N\ll x^{5/8}$ by \eqref{eq:NBound}, so we may apply Lemma \ref{lmm:BFISimple} which shows that
\[
\mathscr{E}_1\ll \frac{N^2}{Q x^\epsilon}.
\]
Thus it suffices to consider $\mathscr{E}_2$. This is precisely the sum considered in Lemma \ref{lmm:BFI0}, and so by Lemma \ref{lmm:BFI0} it suffices to show for a suitably large constant $C_2=C_2(C_1)$ depending only on $C_1$ that
\[
\mathscr{B}_0'\ll_{C_2}\frac{L^3 K^2}{(\log{x})^{C_2}},
\]
where $H'\le H$, $R\le 10aN/Q$, $\theta\in[0,1]$, $a'|a$ and $\mathscr{B}_0'=\mathscr{B}'_0(\theta,a',E',R,H')$ is given by
\[
\mathscr{B}_0'=\sum_{r\sim R}\sum_{n_1\sim N}\sum_{\substack{k\sim K\\ (k,n_1)=1}}r\Bigl|\sum_{\substack{E\le e\le E'\\ (e,kn_1)=1}}\mu^2(e)\sum_{\substack{\ell\sim L\\ \ell \equiv \overline{k}n_1\Mod{r e} \\ (\ell,n_1)=1}}\overline{\gamma_\ell}e(\ell\theta)\sum_{h\sim  H'} e\Bigl(\frac{a'h r e\overline{\ell k}}{n_1}\Bigr)\Bigr|^2.
\]
We see that $\mathscr{B}_0'$ is of the form $\mathscr{B}'(N,K,R,E,H,L)$ considered in Lemma \ref{lmm:BFI1} for a suitable choice of coefficients $\beta(h,e,n)$. Moreover, the conditions of the proposition imply that we can apply Lemma \ref{lmm:NewBFI} to bound this, which then gives the result.
\end{proof}
%
%
%
%
%
%
%
%
%
%
%
%
%
%
%
%
%
%
%
\section{Smoothing preparations}\label{sec:Smoothing}
Before we embark on the proof of Proposition \ref{prpstn:TripleDivisor}, we establish some basic lemmas which allow us to pass from sums over integers in an interval with no prime factors larger than $z_0$ to a smooth sum of integers.
%
%
%
%
%
%
%
%
\begin{lmm}[Smooth partition of unity]\label{lmm:Partition}
Let $C\ge 3$. There exists smooth non-negative functions $\widetilde{\psi}_1,\dots,\widetilde{\psi}_J$ with $J\le (\log{x})^C+2$ such that
\begin{enumerate}
\item $\|\widetilde{\psi}_i^{(j)}\|_\infty\ll_{C} ((j+1)\log{x})^{j C}$ for each $1\le i\le J$ and each $j\ge 0$.
\item We have that
\[
\sum_{j=1}^J \widetilde{\psi}_j(t)=\begin{cases}
0,\qquad &\text{if }t\le 1-1/(\log{x})^C,\\
O(1), &\text{if }1-1/(\log{x})^C\le t\le N,\\
1,&\text{if }1\le t\le 2,\\
O(1), &\text{if }2\le t\le 2+1/(\log{x})^C,\\
0, &\text{if }2+1/(\log{x})^C\le t.\\
\end{cases}
\]
\end{enumerate}
\end{lmm}
\begin{proof}
Recall that $\psi_0$ is a smooth function supported on $[1/2,5/2]$ and equal to $1$ on $[1,2]$ which satisfies $\|\psi_0^{(j)}\|_\infty\le 64^j j!^2$. Define $J:=\lceil(\log{x})^C+1\rceil$ and
\[
\widetilde{\psi}_i(t):=\begin{cases}
\psi_0((t-1)(\log{x})^C-i+2),&1+\frac{i-3/2}{(\log{x})^C}\le t\le 1+\frac{i-1}{(\log{x})^C},\\
1,&1+\frac{i-1}{(\log{x})^C}\le t \le 1+\frac{i-1/2}{(\log{x})^C},\\
1-\psi_0((t-1)(\log{x})^C-i+1),&1+\frac{i-1/2}{(\log{x})^C}\le t\le 1+\frac{i}{(\log{x})^C},\\
0,&\text{otherwise.}
\end{cases}
\]
Then it is easy to verify that $\widetilde{\psi}_i(t)$ is smooth, non-negative, satisfies $\|\widetilde{\psi}_i^{(j)}\|_\infty\ll ((j+1)\log{x})^{j C}$ and that $\sum_{j=1}^J\widetilde{\psi_j}$ satisfies the bounds of the lemma. 
\end{proof}
%
%
%
%
%
%
%
%
\begin{lmm}[Reduction to smoothed sums]\label{lmm:SmoothReduction}
Let $N\ge x^\epsilon$, $NM\asymp x$ and $z\le z_0$. Let $\alpha_m$, $c_q$ be 1-bounded complex sequences.

Imagine that for every choice of $N',D,A,C>0$ with $N' D\asymp N$ and $D\le y_0$, and every smooth function $f$ supported on $[1/2,5/2]$ satisfying $\|f^{(j)}\|_\infty\ll_j (\log{x})^{j C}$, and for every $1$-bounded complex sequence $\beta_d$ we have the estimate
\[
\sum_{q\sim Q} c_q\sum_{m\sim M}\alpha_m\sum_{d\sim D}\beta_d\sum_{n'}f\Bigl(\frac{n'}{N'}\Bigr)\Bigl(\mathbf{1}_{m n' d\equiv a\Mod{q}}-\frac{\mathbf{1}_{(m n' d,q)=1}}{\phi(q)}\Bigr)\ll_{A,C} \frac{x}{(\log{x})^A}.
\]
Then for any $B>0$ and every interval $\mathcal{I}\subseteq [N,2N]$ we have
\[
\sum_{q\sim Q}c_q \sum_{m\sim M}\alpha_m\sum_{\substack{n\in\mathcal{I}\\ P^-(n)>z}}\Bigl(\mathbf{1}_{mn\equiv a\Mod{q}}-\frac{\mathbf{1}_{(m n,q)=1}}{\phi(q)}\Bigr)\ll_{B} \frac{x}{(\log{x})^B}.
\]
\end{lmm}

\begin{proof}
We first note that by Lemma \ref{lmm:FundamentalLemma}
\begin{align*}
\sum_{\substack{n\in\mathcal{I}\\ P^-(n)>z}}&\Bigl(\mathbf{1}_{m n\equiv a\Mod{q}}-\frac{\mathbf{1}_{(m n,q)=1}}{\phi(q)}\Bigr)\\
&\le \sum_{\substack{n\in\mathcal{I}}}\Bigl(\sum_{d|n}\lambda^+_d\Bigr) \mathbf{1}_{m n\equiv a\Mod{q}}-\sum_{\substack{n\in\mathcal{I}}}\Bigl(\sum_{d|n}\lambda^-_d\Bigr) \frac{\mathbf{1}_{(m n,q)=1}}{\phi(q)}\\
&= \sum_{d\le y_0}\lambda^+_d \sum_{\substack{d n'\in\mathcal{I}}}\Bigl(\mathbf{1}_{m d n'\equiv a\Mod{q}}-\frac{\mathbf{1}_{(m d n',q)=1}}{\phi(q)}\Bigr)\\
&\qquad +O\Bigl(\Bigl| \sum_{d\le y_0}(\lambda^+_d-\lambda^-_d) \sum_{\substack{d n'\in\mathcal{I}}}\frac{\mathbf{1}_{(m d n',q)=1}}{\phi(q)}\Bigr)\Bigr|\Bigr),
\end{align*}
and similarly we obtain a lower bound with the roles of $\lambda^+_d$ and $\lambda^-_d$ reversed. By Lemma \ref{lmm:FundamentalLemma} (and recalling that $MN\asymp x$, $N\ge x^\epsilon$), the final term above contributes
\begin{align*}
&\ll \sum_{q\sim Q}\frac{1}{\phi(q)}\sum_{m\sim M}\Bigl|\sum_{\substack{d\le y_0\\ (d,q)=1}} (\lambda^+_d-\lambda_d^-)\Bigl(\frac{\#\mathcal{I}\phi(q)}{d q}+O(\tau(q))\Bigr)\Bigr|\\
&\ll \frac{MN}{(\log{x})^{\log\log{x}}}+M x^{o(1)}\ll_B \frac{x}{(\log{x})^B}.
\end{align*}
Thus it suffices to show that
\[
\sum_{q\sim Q}c_q \sum_{m\sim M}\alpha_m\sum_{d\le y_0}\lambda^+_d \sum_{\substack{n'\sim N/d\\ dn'\in\mathcal{I}}}\Bigl(\mathbf{1}_{m d n'\equiv a\Mod{q}}-\frac{\mathbf{1}_{(m d n',q)=1}}{\phi(q)}\Bigr)\ll_B \frac{x}{(\log{x})^B},
\]
and similarly with $\lambda_d^-$ in place of $\lambda_d^+$. Let $\mathcal{I}=[N_1,N_2]$. We give the proof for the case of $\lambda_d^+$; the other case is entirely analogous. We first apply Lemma \ref{lmm:Partition} to see that
\[
\sum_{\substack{N_1/d\le n'\le N_2/d\\m d n'\equiv a\Mod{q}}}1=\sum_{j\ll (\log{x})^{2B} }\sum_{\substack{n'\\ m d n'\equiv a\Mod{q}}}\widetilde{\psi}_j\Bigl(\frac{n' d}{N}\Bigr)+O\Bigl(\sum_{\substack{n'\in \mathcal{I}_d\\ m d n'\equiv a\Mod{q}}}1\Bigr),
\]
where 
\[
\mathcal{I}_d=\Bigl[\frac{N_1}{d}-\frac{N_1}{d(\log{x})^{2B}},\frac{N_1}{d}\Bigr]\cup\Bigl[\frac{N_2}{d},\frac{N_2}{d}+\frac{2N}{d(\log{x})^{2B}}\Bigr].
\]
The error term above contributes a total
\[
\ll \sum_{d\le y_0}\sum_{n'\in\mathcal{I}_d}\sum_{m\sim M}\Bigl(1+\sum_{q| m n' d-a}1\Bigr)\ll \frac{MN(\log{x})^{O(1)}}{(\log{x})^{2B}}\ll_B \frac{x}{(\log{x})^B},
\]
provided $B$ is large enough (as we may assume). We now split the summation over $d$ into one of $O(\log{x})^{10 B+1}$ subsums of the form $d\in [(1+\eta)^k,(1+\eta)^{k+1})$ where $\eta:=(\log{x})^{-10 B}$. Taking the worst subsum and the worst value of $j\ll (\log{x})^{2B}$, it suffices to show that
\begin{align*}
\sum_{q\sim Q}c_q \sum_{m\sim M}\alpha_m\sum_{\substack{d\in [d_0,d_0(1+\eta))\\ d\le y_0}}\lambda^+_d \sum_{\substack{n'}}\widetilde{\psi}_j\Bigl(\frac{n' d}{N}\Bigr)\Bigl(\mathbf{1}_{m d n'\equiv a\Mod{q}}-\frac{\mathbf{1}_{(m d n',q)=1}}{\phi(q)}\Bigr)\\
\ll_B \frac{x}{(\log{x})^{13 B}},
\end{align*}
Since $\|\widetilde{\psi}_j'\|_\infty\ll (\log{x})^{2B}$, we can replace $\widetilde{\psi}_j(n'd/N)$ with $\widetilde{\psi}_j(n' d_0/N)$ at the cost of an error
\[
\ll \sum_{m\sim M}\sum_{\substack{d\in [d_0,d_0(1+\eta))}}\sum_{\substack{n'\asymp N}}\frac{(\log{x})^{2B}}{(\log{x})^{10B}}\Bigl(1+\sum_{q|m n' d-a}1\Bigr)\ll \frac{x(\log{x})^{O(1)}}{(\log{x})^{18B}},
\]
and so this is also acceptable. Finally, let $\tilde{\lambda}_d$ be $\lambda^+_d$ with support restricted to $d\in [d_0,d_0(1+\eta)]\cap[1,y_0]$, let $D:=d_0$ and let $N':=N/d_0$. We see that it suffices to show that
\[
\sum_{q\sim Q}c_q \sum_{m\sim M}\alpha_m\sum_{\substack{d\sim D}}\tilde{\lambda}_d \sum_{\substack{n'}}\widetilde{\psi}_j\Bigl(\frac{n'}{N'}\Bigr)\Bigl(\mathbf{1}_{m d n'\equiv a\Mod{q}}-\frac{\mathbf{1}_{(m d n',q)=1}}{\phi(q)}\Bigr)\ll_B \frac{x}{(\log{x})^{13 B}}.
\]
This now follows from the assumption of the lemma.
\end{proof}
%
%
%
%
%
%
%
%
%
%
%
%
%
%
%
%
%
%
%
\section{Triple divisor function estimates}\label{sec:TripleDivisor}
In this section we prove Proposition \ref{prpstn:TripleDivisor}. This is a generalization on work of showing $d_3(n)$ is equidistributed in arithmetic progressions to moduli of size greater than $x^{1/2}$. Such a result was first obtained by Friedlander and Iwaniec \cite{FIDivisor} (with an appendix by Birch and Bombieri), was improved by Heath-Brown \cite{HBDivisor} and Fouvry, Kowalski and Michel \cite{FKMDivisor} and Polymath \cite{Polymath}. It is worth noting that our results hold with moduli as large as $x^{11/21}$, which was the state of the art until the recent work of \cite{FKMDivisor} and \cite{Polymath} (and is still the record for general residue classes).

As with all of these previous results, we rely crucially on Deligne's results \cite{Deligne1,Deligne2} on the Weil Conjectures for general varieties (and its generalizations) to estimate certain exponential sums. We use these results as a black box in the form of results for correlations of hyper-Kloosterman sums as given by \cite{Polymath}, based on the trace function formalism developed by Fouvry, Kowalski, Michel and others (see \cite{FKMS} for more details). To optimize our results we require a different arrangement of the manipulations to that of previous works, exploiting the fact that in our situation we have moduli which have a relatively small factor. The critical case for these estimates is handling convolutions of three smooth sequences of length $x^{1/3-1/168}$ and one rough sequence of length $x^{1/56}$, although this is not the bottleneck in our final results.

We begin with some notation. Let
\[
\Kl_3(a;q):=\frac{1}{q}\sum_{\substack{b_1,b_2,b_3\in \mathbb{Z}/q\mathbb{Z}\\ b_1b_2b_3=a}}e\Bigl(\frac{b_1+b_2+b_3}{q}\Bigr).
\]
be the hyper-Kloosterman sum. We let
\begin{equation}
F(h_1,h_2,h_3;a,q)=\sum_{\substack{b_1,b_2,b_3\in(\mathbb{Z}/q\mathbb{Z})^\times \\ b_1b_2b_3=a}}e\Bigl(\frac{h_1b_1+h_2b_2+h_3b_3}{q}\Bigr)
\label{eq:FDef}
\end{equation}
be a closely related sum.
%
%
%
%
%
%
%
%
\begin{lmm}[Deligne bound]\label{lmm:Deligne}
Let $(a,q)=1$ and $\mu^2(q)=1$. Then we have
\[
|\Kl_3(a;q)|\ll \tau_3(q).
\]
\end{lmm}
\begin{proof}
This is proven in \cite[\S 7]{DeligneApp}, for example.
\end{proof}
%
%
%
%
%
%
%
%
\begin{lmm}[Correlations of Hyper-Kloosterman sums]\label{lmm:KloostermanCorrelation}
Let $s,r_1,r_2$ be squarefree integers with $(s,r_1)=(s,r_2)=1$. Let $a_1,a_2$ be integers with $(a_1,r_1s)=1=(a_2,r_2s)$. Then we have
\begin{align*}
&\sum_{\substack{h\in\mathbb{Z}\\ (h,s r_1 r_2)=1}}\psi\Bigl(\frac{h}{H}\Bigr)\Kl_3(a_1h;r_1s)\overline{\Kl_3(a_2h;r_2s)}\\
&\ll (Hr_1r_2s)^\epsilon\Bigl(\frac{H}{[r_1,r_2]s}+1\Bigr)s^{1/2}[r_1,r_2]^{1/2}\gcd(a_2-a_1,r_1,r_2)^{1/2}\gcd(a_2r_1^3-a_1r_2^3,s)^{1/2}.
\end{align*}
\end{lmm}
\begin{proof}
This is \cite[Corollary 6.26]{Polymath}. We have omitted the condition $H\ll (s[r_1,r_2])^{O(1)}$ there by adding a factor $H^\epsilon$ in the bound.
\end{proof}
%
%
%
%
%
%
%
%
\begin{lmm}[Properties of $F$ sum]\label{lmm:FProperties}
We have the following:
\begin{enumerate}
\item If $(q_1,q_2)=1$ then
\[
F(h_1,h_2,h_3;a;q_1q_2)=F(h_1,h_2,h_3;a \overline{q_1}^3;q_2)F(h_1,h_2,h_3;a\overline{q_2}^3;q_1).
\]
\item If $(b,q)=1$ then
\[
F(h_1,h_2,h_3;a;q)=F(b h_1,b h_2,b h_3;a \overline{b}^3;q).
\]
\item If $(a,q)\ne 1$ then 
\[
F(h_1,h_2,h_3;a;q)=0.
\]
\item If $(a,q)=1$ and $\gcd(h_1,h_2,h_3,q)=d$ then 
\[
F(h_1,h_2,h_3; a ;q)=\frac{\phi(q)^2}{\phi(q/d)^2} F\Bigl(\frac{h_1}{d},\frac{h_2}{d},\frac{h_3}{d};a;\frac{q}{d}\Bigr).
\]
\item If $(a,q)=1$ and $\gcd(h_1h_2h_3,q)=1$ then 
\[
F(h_1,h_2,h_3;a;q)=q\Kl_3(ah_1h_2h_3;q).
\]
\item If $(a,q)=1$ and $\gcd(h_1h_2h_3,q)\ne 1$ and $\gcd(h_1,h_2,h_3,q)=1$ and $\mu^2(q)=0$  then 
\[
F(h_1,h_2,h_3;a;q)=0.
\]
\item If $(a,q)=1$ and $q|h_1h_2h_3$ and $\gcd(h_1,h_2,h_3,q)=1$ and $\mu^2(q)=1$ then $F(h_1,h_2,h_3;a;q)$ depends only on $(h_1,q)$, $(h_2,q)$, $(h_3,q)$ and $q$, and satisfies
\[
|F(h_1,h_2,h_3;a,q)|\ll \frac{(h_1,q)(h_2,q)(h_3,q)}{q}.
\]
\end{enumerate}
\end{lmm}
\begin{proof}
Statement $(1)$ follows from the Chinese Remainder Theorem, letting $b_i=b_{i,1}q_2+b_{i,2}q_1$ for unique $b_{i,1}\in(\mathbb{Z}/q_1\mathbb{Z})^\times$ and $b_{i,2}\in (\mathbb{Z}/q_2\mathbb{Z})^\times$.

Statement $(2)$ follows from making the change of variables $b_i'=b_i\overline{b}$ in the summation.

Statement $(3)$ is immediate since the sum is supported on $b_1b_2b_3\in (\mathbb{Z}/q\mathbb{Z})^\times$.

Let $p$ be a prime such that $p^\ell||\gcd(h_1,h_2,h_3,q)$, and let $q=p^jq_1$ with $(q_1,p)=1$. By parts $(1)$ and $(2)$, we see that
\begin{align*}
F(h_1,h_2,h_3;a;q)&=F(h_1,h_2,h_3;a\overline{q_1}^3;p^j)F(h_1,h_2,h_3;a\overline{p}^{3j};q_1)\\
&=F(h_1,h_2,h_3;a\overline{q_1}^3;p^j)F\Bigl(\frac{h_1}{p^\ell},\frac{h_2}{p^\ell},\frac{h_3}{p^\ell};a\overline{p}^{3j-3\ell};q_1\Bigr),\\
F\Bigl(\frac{h_1}{p^\ell},\frac{h_2}{p^\ell},\frac{h_3}{p^\ell};a;\frac{q}{p^\ell}\Bigr)&=F\Bigl(\frac{h_1}{p^\ell},\frac{h_2}{p^\ell},\frac{h_3}{p^\ell};a\overline{q_1}^3;p^{j-\ell}\Bigr)F\Bigl(\frac{h_1}{p^\ell},\frac{h_2}{p^\ell},\frac{h_3}{p^\ell};a\overline{p}^{3j-3\ell};q_1\Bigr).
\end{align*}
Thus it suffices to show $(4)$ when $q=p^j$ is a prime power. This is trivial if $j=\ell$ since then all summands are equal to 1. Otherwise let 
\[
b_i=b_{i}'+p^{j-\ell}b_{i}''
\]
for some $b_{i}'\in(\mathbb{Z}/p^{j-\ell}\mathbb{Z})^\times$ and $b_{i}''\in\mathbb{Z}/p^\ell \mathbb{Z}$. The summands in $F$ only depend on $b_{1}'$, $b_{2}'$ and $b_{3}'$, so $F(h_1,h_2,h_3;a;p^j)$ is given by
\[
\sum_{\substack{b_{1}',b_{2}',b_{3}'\in (\mathbb{Z}/p^{j-\ell}\mathbb{Z})^\times\\ b_{1}'b_{2}'b_{3}'\equiv a\Mod{p^{j-\ell}}}}e\Bigl(\frac{h_1b_1'+h_2b_2'+h_3b_3'}{p^j}\Bigr)\sum_{\substack{b_{1}'',b_{2}'',b_{3}''\in\mathbb{Z}/p^{\ell}\mathbb{Z}\\ \prod_{i=1}^3 (b_{i}'+p^{j-\ell}b_{i}'')\equiv a\Mod{p^{j}}}}1.
\]
By considering $b_i''\Mod{p^{j-\ell}}$, then $b_i''\Mod{p^{2j-2\ell}}$ etc in turn, we see that at each stage the congruence condition reduces to a linear constraint, and so the inner sum is $p^{2\ell}$. This gives $(4)$.

Statement $(5)$ follows from the change of variables $b_i'=b_i h_i$.

From statement $(1)$, it suffices to show statement $(6)$ when $q=p^j$ is a prime power with $j\ge 2$. Without loss of generality let $p|h_1$, and let $b_i=b_i'+p^{j-1}b_i''$. Then the summands don't depend on $b_1''$ so $F(h_1,h_2,h_3;a;p^j)$ is given by
\[
\sum_{\substack{b_1',b_2',b_3'\in(\mathbb{Z}/p^{j-1}\mathbb{Z})^\times\\ b_1'b_2'b_3'\equiv a\Mod{p^{j-1}}}}e\Bigl(\frac{h_1b_1'+h_2b_2'+h_3b_3'}{p^j}\Bigr)\hspace{-0.5cm}\sum_{\substack{b_1'',b_2'',b_3''\in\mathbb{Z}/p\mathbb{Z}\\ \prod_{i=1}^3 (b_i'+p^{j-1}b_i'')\equiv a\Mod{p^{j}}}}\hspace{-0.5cm}e\Bigl(\frac{h_2b_2''+h_3b_3''}{p}\Bigr).
\]
In the inner sum, the congruence condition simplifies to a linear constraint on $b_1'',b_2'',b_3''$ which always has a unique solution $b_1''$. Thus $b_2''$ and $b_3''$ are unconstrained, so the sum is zero unless $p|h_2 $ and $p|h_3$, which is impossible if $\gcd(h_1,h_2,h_3,q)=1$. This gives $(6)$.

Finally, by statement $(1)$ it suffices to establish $(7)$ when $q=p$ is a prime. Without loss of generality, let $p|h_1$. In this case we see that $F(h_1,h_2,h_3;a;q)$ is given by
\[
\sum_{\substack{b_1,b_2,b_3\in(\mathbb{Z}/p\mathbb{Z})^\times\\ b_1b_2b_3\equiv a}}e\Bigl(\frac{h_2b_2+h_3b_3}{p}\Bigr)=\Big(\sum_{b_2\in(\mathbb{Z}/p\mathbb{Z})^\times}e\Bigl(\frac{h_2b_2}{p}\Bigr)\Bigr)\cdot \Bigl(\sum_{b_3\in(\mathbb{Z}/p\mathbb{Z})^\times}e\Bigl(\frac{h_3 b_3}{p}\Bigr)\Bigr).
\]
The right hand side is a product of Ramanujan sums. These only depend on whether $p|h_i$ or not, and the right hand side is bounded by $(h_2,p)(h_3,p)=(h_1,p)(h_2,p)(h_3,p)/p$. This gives $(7)$.
\end{proof}
%
%
%
%
%
%
%
%
\begin{lmm}[Reduction to an exponential sum]\label{lmm:K1}
Let $M N_1 N_2 N_3\asymp x$ with $x^{2\epsilon} \le N_1\le N_2\le N_3$, and let $Q\ge x^{2\epsilon}$. Let $d\in[D,2D]$ with $(d,a)=1$ and let $\psi_1,\psi_2,\psi_3$ be smooth functions satisfying $\|\psi_1^{(j)}\|_\infty,\|\psi_2^{(j)}\|_\infty,\|\psi_3^{(j)}\|_\infty\ll_j (\log{x})^{C j}$ and let $\mathscr{K}=\mathscr{K}(d)$ be defined by
\begin{align*}
\mathscr{K}&:=\sum_{\substack{q\sim Q\\ (q,a d)=1}}c_q\sum_{\substack{m\sim M}}\alpha_m\sum_{\substack{n_1,n_2,n_3}}\psi_1\Bigl(\frac{n_1}{N_1}\Bigr)\psi_2\Bigl(\frac{n_2}{N_2}\Bigr)\psi_3\Bigl(\frac{n_3}{N_3}\Bigr)\\
&\qquad \times\Bigl(\mathbf{1}_ {m n_1n_2n_3\equiv a\Mod{q d}}-\frac{\mathbf{1}_{(m n_1 n_2 n_3,q d)=1}}{\phi(q d)}\Bigr),
\end{align*}
where $c_q,\alpha_m$ are 1-bounded complex sequences (possibly depending on $d$). Then we have for any $A>0$
\begin{align*}
\mathscr{K}&\ll_{A,C} \frac{x}{D\log^A{x}}+\frac{N_1 N_2 N_3 D^4(\log{x})^{O_C(1)}}{Q^3}\sup_{\substack{H_1\le x^\epsilon Q D/N_1 \\ H_2\le x^\epsilon Q D/N_2 \\ H_3\le x^\epsilon Q D/N_3\\ b_q,b_m,b_1,b_2,b_3\Mod{d}\\ (b_q b_m,d)=1}} |\mathscr{K}'|,
\end{align*}
where
\begin{align*}
\mathscr{K}':&=\sum_{\substack{q\sim Q\\ q\equiv b_q\Mod{d} \\ (q,a)=1}}\frac{Q^3 c_{q}}{q^3}\sum_{\substack{m\sim M\\ (m,q)=1\\ m\equiv b_m\Mod{d}}}\alpha_m\sum_{\substack{1\le |h_1|\le H_1 \\ 1\le |h_2|\le H_2\\ 1\le |h_3|\le H_3\\ h_i\equiv b_i\Mod{d}}}F(h_1,h_2,h_3;a\overline{md^3};q),
\end{align*}
and where $F$ is the function defined in \eqref{eq:FDef}.
\end{lmm}
\begin{proof}
This essentially follows from completing the sums in $n_1,n_2,n_3$. We wish to estimate $\mathscr{K}=\mathscr{K}_1-\mathscr{K}_2$, where
\begin{align*}
\mathscr{K}_1&:=\sum_{\substack{q\sim Q}}\frac{c_q}{\phi(q d)}\sum_{m\sim M}\alpha_m\sum_{\substack{n_1,n_2,n_3\\ (m n_1n_2n_3,q d)=1}}\psi_1\Bigl(\frac{n_1}{N_1}\Bigr)\psi_2\Bigl(\frac{n_2}{N_2}\Bigr)\psi_3\Bigl(\frac{n_3}{N_3}\Bigr),\\
\mathscr{K}_2&:=\sum_{\substack{q\sim Q}}c_{q}\sum_{m\sim M}\alpha_m\sum_{\substack{n_1,n_2,n_3\\ m n_1n_2n_3\equiv a\Mod{q d}}}\psi_1\Bigl(\frac{n_1}{N_1}\Bigr)\psi_2\Bigl(\frac{n_2}{N_2}\Bigr)\psi_3\Bigl(\frac{n_3}{N_3}\Bigr).
\end{align*}
First we consider $\mathscr{K}_1$. By Lemma \ref{lmm:TrivialCompletion}, we have
\[
\sum_{(n,q d)=1}\psi_1\Bigl(\frac{n}{N}\Bigr)=N\hat{\psi_1(0)}\frac{\phi(q d)}{q d}+O_C(x^{o(1)}\tau(q d)).
\]
Therefore, since $N_1\le N_2\le N_3$, we have
\begin{align*}
\sum_{\substack{n_1,n_2,n_3\\ (n_1n_2n_3,q d)=1}}\psi_1\Bigl(\frac{n_1}{N_1}\Bigr)\psi_2\Bigl(\frac{n_2}{N_2}\Bigr)\psi_3\Bigl(\frac{n_3}{N_3}\Bigr)&=N_1N_2N_3\frac{\phi(q d)^3}{q^3 d^3}\hat{\psi_1}(0)\hat{\psi_2}(0)\hat{\psi_3}(0)\\
&\qquad+O_C(N_2N_3x^{o(1)}).
\end{align*}
Thus
\[
\mathscr{K}_1=\mathscr{K}_{MT}+O_C\Bigl(\frac{x^{1+o(1)}}{D N_1}\Bigr),
\]
where
\[
\mathscr{K}_{MT}:=N_1N_2N_3\hat{\psi_1}(0)\hat{\psi_2}(0)\hat{\psi_3}(0)\sum_{\substack{q\sim Q\\ (q,a d)=1}}\frac{\phi(q d)^2 c_{q}}{q^3 d^3}\sum_{\substack{m\sim M\\ (m,q d)=1}}\alpha_m.
\]
Now we consider $\mathscr{K}_2$. We find by Lemma \ref{lmm:Completion} that for $H_1:=Q D x^\epsilon/N_1$
\[
\sum_{n_1\equiv a\overline{b}\Mod{q d}}\psi_1\Bigl(\frac{n_1}{N_1}\Bigr)=\frac{N_1}{q d}\hat{\psi_1(0)}+\frac{N_1}{q d}\sum_{1\le |h_1|\le H_1}\hat{\psi_1}\Bigl(\frac{h_1 N_1}{q d}\Bigr)e\Bigl(\frac{ah_1\overline{b}}{q d}\Bigr)+O_C(x^{-10}).
\]
We first apply this with $b=m n_2n_3$ to estimate the $n_1$ sum in $\mathscr{K}_2$. By Lemma \ref{lmm:TrivialCompletion}, the contribution from the last term above to $\mathscr{K}_2$ is negligible. The contribution from the first term to $\mathscr{K}_2$ is
\[
\sum_{\substack{q\sim Q\\ (q,ad)=1}}\frac{N_1\hat{\psi_1}(0)c_{q}}{q d}\sum_{\substack{m\sim M\\ (m,q d)=1}}\alpha_m\sum_{\substack{n_2,n_3\\ (n_2 n_3,q d)=1}}\psi_2\Bigl(\frac{n_2}{N_2}\Bigr)\psi_3\Bigl(\frac{n_3}{N_3}\Bigr)=\mathscr{K}_{MT}+O_C\Bigl(\frac{x^{1+o(1)}}{D N_2}\Bigr).
\]
Thus we are left to consider the contribution from the middle term sum over $1\le |h_1|\le H_1$. Lemma \ref{lmm:InverseCompletion} shows that for $H_2:=Q D x^\epsilon/N_2$
\begin{align*}
\sum_{(n_2,q d)=1}\psi_2&\Bigl(\frac{n_2}{N_2}\Bigr)e\Bigl(\frac{b\overline{n_2}}{q d}\Bigr)=\frac{N_2\hat{\psi_2}(0)}{q d}\sum_{(b_2,q d)=1}e\Bigl(\frac{b b_2}{q d}\Bigr)\\
&+\frac{N}{q d}\sum_{1\le |h_2|\ll H_2}\hat{\psi_2}\Bigl(\frac{h_2 N_2}{q d}\Bigr)\sum_{\substack{b_2\Mod{q d}\\ (b_2,q d)=1}} e\Bigl(\frac{b\overline{b_2}+h_2 b_2}{q d}\Bigr)+O_C(x^{-10}).
\end{align*}
The first term above is a multiple of a Ramanujan sum, and so $\ll N_2\gcd(b,q d)/q d$. We apply this with $b=ah_1\overline{m n_3}$ to evaluate the $n_2$ sum, and then similarly to evaluate the $n_3$ sum. This gives with $H_3:=Q D x^\epsilon/N_3$ 
\begin{align*}
&\sum_{\substack{n_2,n_3\\ (n_2n_3,q d)=1}}\psi_2\Bigl(\frac{n_2}{N_2}\Bigr)\psi_3\Bigl(\frac{n_3}{N_3}\Bigr)e\Bigl(\frac{ah_1\overline{m n_2n_3}}{q d}\Bigr)\\
&=\frac{N_2N_3}{q^2 d^2}\sum_{\substack{1\le |h_2|\le H_2\\ 1\le |h_3|\le H_3}}\hat{\psi_2}\Bigl(\frac{h_2 N_2}{q d}\Bigr)\hat{\psi_3}\Bigl(\frac{h_3 N_3}{q d}\Bigr)\sum_{b_2,b_3\in(\mathbb{Z}/q d\mathbb{Z})^\times }e\Bigl(\frac{a\overline{b_2 b_3 m}+h_2 b_2+h_3 b_3}{q d}\Bigr)\\
&\qquad+O_C\Bigl(\frac{N_2N_3}{Q D}(h_1,q d)\Bigr).
\end{align*}
Putting this all together, we obtain
\begin{align*}
\mathscr{K}_2&=\mathscr{K}_{MT}+\frac{N_1N_2N_3}{Q^3 d^3}\mathscr{K}_3+O_C\Bigl(\frac{x^{1+o(1)}}{D N_1}\Bigr),
\end{align*}
where
\begin{align*}
\mathscr{K}_3:&=\sum_{\substack{q\sim Q\\ (q,a d)=1}}\frac{Q^3 c_q}{q^3}\sum_{\substack{m\sim M\\ (m,q d)=1}}\alpha_m\sum_{\substack{1\le |h_1|\le H_1 \\ 1\le |h_2|\le H_2\\ 1\le |h_3|\le H_3}}\hat{\psi_1}\Bigl(\frac{N_1 h_1}{q d}\Bigr)\hat{\psi_2}\Bigl(\frac{N_2h_2}{q d}\Bigr)\hat{\psi_3}\Bigl(\frac{N_3h_3}{q d}\Bigr)\\
&\qquad\times\sum_{\substack{b_1,b_2,b_3\in(\mathbb{Z}/q d\mathbb{Z})^\times\\ b_1b_2b_3\equiv a\overline{m}\Mod{q d}}}e\Bigl(\frac{h_1b_2+h_2b_2+h_3b_3}{q d}\Bigr).
\end{align*}
We see that the final sum over $b_1,b_2,b_3$ is $F(h_1,h_2,h_3;a\overline{m};q d)$. We note that for $\psi\in\{\psi_1,\psi_2,\psi_3\}$ we have
\[
\frac{\partial^{j_1+j_2}}{\partial q^{j_1}\partial h^{j_2}}\hat{\psi}\Bigl(\frac{h N}{q d}\Bigr)\ll_{j_1,j_2}\frac{(\log{x})^{C(j_1+j_2)}}{q^{j_1}h^{j_2}},
\]
and so we may remove the $\hat{\psi}$ factors from $\mathscr{K}'$ by partial summation. This gives
\[
\mathscr{K}_3\ll_C (\log{x})^{O_C(1)}\sup_{\substack{Q'\le 2Q\\H_1'\le x^\epsilon Q D/N_1\\ H_2'\le x^\epsilon Q D/N_2\\ H_3'\le x^\epsilon Q D/N_3}}|\mathscr{K}_4|,
\]
where
\[
\mathscr{K}_4:=\sum_{\substack{Q\le q \le Q'\\ (q,a d)=1}}\frac{Q^3 c_q}{q^3}\sum_{\substack{m\sim M\\ (m,q d)=1}}\alpha_m\sum_{\substack{1\le |h_1|\le H_1' \\ 1\le |h_2|\le H_2'\\ 1\le |h_3|\le H_3'}}F(h_1,h_2,h_3;a\overline{m};q d).
\]
Finally, by Lemma \ref{lmm:FProperties} and $(q,d)=1$ we have that 
\[
F(h_1,h_2,h_3;a\overline{m};q d)=F(h_1,h_2,h_3;a\overline{m d^3};q)F(h_1,h_2,h_3;a\overline{m q^3};d).
\]
We note that $F(h_1,h_2,h_3;a\overline{m q^3};d)$ depends only on the values of $h_1,h_2,h_,m,q\Mod{d}$ and is trivially bounded by $d^2$. Therefore, putting these variables into residue classes $\Mod{d}$ and taking the worst residue classes gives the result.
\end{proof}
%
%
%
%
%
%
%
%
\begin{lmm}[Dealing with common factors]\label{lmm:K2}
Let $\mathscr{K}'_1$ be given by
\[
\mathscr{K}_1':=\sum_{\substack{Q\le q \le Q'}}c_q'\sum_{\substack{m\sim M\\ (m,q)=1}}\alpha'_m\sum_{\substack{1\le |h_1|\le H_1 \\ 1\le |h_2|\le H_2\\ 1\le |h_3|\le H_3\\ h_i\equiv b_i\Mod{d}\,\forall i}}F(h_1,h_2,h_3;a\overline{m d^3};q),
\]
where $c_q'$ and $\alpha_m'$ are 1-bounded complex sequences with $c_q'$ supported on square-free numbers coprime to $ad$.
 Let $\mathscr{K}'_2=\mathscr{K}'_2(d_0,d_1,e_1,e_2,e_3,Q'',Q''')$ be the related sum, given by
\begin{align*}
\mathscr{K}'_2&:=\sum_{\substack{Q''\le q\le Q'''\\ (q,f_1f_2)=1}}\tilde{c}_{q}\sum_{\substack{m\sim M\\ (m,q)=1}}\alpha'_m\sum_{\substack{1\le |h_1'|\le H_1/(d_0e_1)\\ 1\le |h_2'|\le H_2/(d_0e_2) \\ 1\le |h_3'|\le H_3/(d_0e_3)\\ h_i'\equiv b_i'\Mod{d}\\ (h_i',q d_1/e_i)=1\,\forall i}}\Kl_3\Bigl(f_1 h_1'h_2'h_3'\overline{m f_2};q\Bigr),
\end{align*}
where $f_1= a e_1e_2e_3$, $f_2=d^3d_1^3$ and where
\begin{align*}
\frac{Q}{d_0d_1}\le Q''\le Q'''\le \frac{Q'}{d_0d_1},\qquad \tilde{c}_{q}:=\frac{q d_0d_1 c_{q d_0 d_1}'}{2 Q}. 
\end{align*}
If we have the bound
\[
\mathscr{K}'_2\ll \frac{d_1^3}{e_1 e_2 e_3}\frac{Q'''{}^2 M}{x^{2\epsilon}}
\]
for every such choice of $Q'',Q''',d_0,d_1,e_1,e_2,e_3$ and $b_1',b_2',b_3'\Mod{d}$, then we have that
\[
\mathscr{K}'\ll \frac{M Q^3}{x^\epsilon}.
\]
\end{lmm}
\begin{proof}
To handle possible common factors between the $h_i$ and $q$ we use Lemma \ref{lmm:FProperties}. Let $d_0=\gcd(h_1,h_2,h_3,q)$. Then $q=d_0q'$, $h_i=d_0h_i'$ for some $h_1',h_2',h_3',q'$ with $\gcd(h_1',h_2',h_3',q')=1$. Since $c_q$ is supported on square-free $q$, we only need to consider $(q',d_0)=1$, so
\[
F(h_1,h_2,h_3;a\overline{m d^3},q)=\phi(d_0)^2 F(h_1',h_2',h_3';a\overline{m d^3};q').
\]
Now let $d_1=\gcd(h_1'h_2'h_3',q')$, and $q'=d_1 q''$ where $(q'',d_1)=1$ (since we only need consider $q$ square-free). We see that
\[
F(h_1',h_2',h_3';a\overline{m d^3};q')=F(h_1',h_2',h_3';a\overline{m d_1^3};q'')F(h_1',h_2',h_3';a\overline{m d^3 q''{}^3};d_1).
\]
Since $(h_1h_2h_3,q'')=1$, we have that
\[
F(h_1',h_2',h_3';a\overline{m d^3 d_1^3};q'')=q''\Kl_3(h_1'h_2'h_3a\overline{m d^3 d_1^3};q'').
\]
Let $h_i'=h_i''e_i$ where $e_i=\gcd(h_i',d_1)$. Then we see that 
\[
F(h_1',h_2',h_3';a\overline{m d^3 q''{}^3};d_1)=H(e_1,e_2,e_3;d_1)
\]
for some function $H$ which depends only on $e_1,e_2,e_3$ and $d_1$ and is bounded by $e_1e_2e_3/d_1$.

Thus $\mathscr{K}'$ is given by
\begin{align*}
\mathscr{K}'=&\sum_{\substack{1\le d_0\le H_1\\ (d_0,ad)=1}}\sum_{\substack{1\le d_1\le Q/d_0\\ (d_1,ad)=1}}\mu^2(d_0 d_1)\sum_{\substack{1\le e_1\le H_1/d_0\\ 1\le e_2\le H_2/d_0\\ 1\le e_3\le H_3/d_0\\ (e_1e_2e_3,d_1)=d_1\\ (e_1,e_2,e_3)=1\\ e_i|d_1\,\forall i}}H(e_1,e_2,e_3;d_1)\sum_{\substack{Q/d_0d_1\le q''\le Q'/(d_0d_1)\\ (q'',e_1e_2e_3)=1}}c'_{d_0 d_1 q''}\\
&\times\sum_{\substack{m\sim M\\ (m,q'' d_0d_1)=1}}\alpha'_m\sum_{\substack{1\le |h_1''|\le H_1/d_0e_1\\ 1\le |h_2''|\le H_2/d_0e_2 \\ 1\le |h_3''|\le H_3/d_0e_3\\ d_0 e_i h_i''\equiv b_i\Mod{d}\,\forall i\\  (h_i'',d_1/e_i)=1\,\forall i\\ (h_i'',q'')=1\,\forall i}}\phi(d_0)^2q'' \Kl_3\Bigl(d_0^2e_1e_2e_3h_1''h_2''h_3''a\overline{m d^3 d_1{}^3};q''\Bigr).
\end{align*}
Let $H_i':=H_i/d_0e_i$. We recall that $H(e_1,e_2,e_3)\ll e_1e_2e_3/d_1$. Then we have
\begin{align*}
\mathscr{K}'\ll & Q\sum_{\substack{1\le d_0\le x\\ 1\le d_1\le x}}\frac{d_0\mu^2(d_0 d_1)}{d_1}\sum_{\substack{e_1,e_2,e_3 \\ e_i|d_1 \forall i}}\frac{e_1 e_2 e_3}{d_1}\sup_{\substack{Q/d_0d_1\le Q''\le Q'''\le Q'/d_0d_1\\ b_1'',b_2'',b_3''\Mod{d}}}|\mathscr{K}'_2|,
\end{align*}
where $\mathscr{K}'_2=\mathscr{K}'_2(d_0,d_1,e_1,e_2,e_3)$ is given by
\begin{align*}
\mathscr{K}'_2&:=\sum_{\substack{Q''\le q''\le Q'''\\ (q'',e_1e_2e_3)=1}}\tilde{c}_{q''}\sum_{\substack{m\sim M\\ (m,q''d_0d_1)=1}}\alpha'_m\sum_{\substack{1\le |h_1''|\le H_1'\\ 1\le |h_2''|\le H_2' \\ 1\le |h_3''|\le H_3'\\ h_i''\equiv b_i''\Mod{d}\,\forall i\\ (h_i'',d_1/e_i)=1\,\forall i\\ (h_i'',q'')=1\,\forall i}}\Kl_3\Bigl(f_1 h_1''h_2''h_3''\overline{m f_2};q''\Bigr),
\end{align*}
where $f_1=a e_1e_2e_3$ and $f_2=d^3d_1^3$ and $\tilde{c}_{q''}=d_0d_1q'' c_{q''d_0 d_1}/(2Q)$.  Therefore, if 
\begin{align}
\mathscr{K}'_2\ll \frac{Q'''{}^2 M d_1^3}{e_1e_2e_3 x^{2\epsilon}} \ll \frac{ Q^2 M d_1}{d_0^2 e_1e_2e_3 x^{2\epsilon}},
\label{eq:E2Target}
\end{align}
then, since $N_1N_2N_3 M\asymp x$, we find
\[
\mathscr{K}'\ll \frac{M Q^3}{x^{2\epsilon}}\sum_{d_0,d_1\ll x}\frac{1}{d_0d_1}\sum_{e_1,e_2,e_3|d_1}1\ll \frac{M Q^3}{x^{\epsilon}}.
\]
This gives the result.
\end{proof}
%
%
%
%
%
%
%
%
\begin{lmm}[Bound for the exponential sum]\label{lmm:K3}
Let $x^{2\epsilon}\le N_1\le N_2\le N_3$ and let $H_1,H_2,H_3$ satisfy $H_i\ll x^{2\epsilon}QR/(d_0e_i N_i)$. Let $\mathscr{K}''$ be given by
\[
\mathscr{K}'':=\sum_{\substack{q\sim Q\\ r\sim R\\(q r,f_1f_2)=1}}\tilde{c}_{q,r}\sum_{\substack{m\sim M\\ (m,q r)=1}}\alpha'_m\sum_{\substack{1\le |h_1|\le H_1\\ 1\le |h_2|\le H_2 \\ 1\le |h_3|\le H_3\\ h_i\equiv b_i'\Mod{d}\\ (h_i,q r d_1/e_i)=1\,\forall i}}\Kl_3\Bigl(f_1 h_1 h_2 h_3 \overline{m f_2};q r\Bigr),
\]
where $\tilde{c}_{q,r}$ are 1-bounded coefficients supported on pairs with $qr$ squarefree, and $\alpha_m$ is a 1-bounded complex sequence. Let $N_3,Q,R,M$ satisfy
\begin{align}
\frac{M Q^{5/2} R^3}{x^{1-14\epsilon}}&<N_3< \frac{x^{2-14\epsilon} }{Q^3 R^2 M},
\label{eq:N3Constraint}\\
M^2 Q^{7/2} R^{3}&<x^{2-20\epsilon},
\label{eq:MConstraint}
\end{align}
Then we have that
\[
\mathscr{K}''\ll \frac{Q^2 R^2 M d_1^3}{ e_1e_2e_3 x^\epsilon }.
\]
\end{lmm}
\begin{proof}
We Cauchy in $q$, $h_1$ and $h_2$ and combine $h_1$ and $h_2$ into a single variable $h=h_1h_2$. This gives (dropping some coprimality constraints on the outer variables)
\[
\mathscr{K}''{}^2\ll H_1 H_2 Q (\log{x} )^{O(1)}\mathscr{K}_2'',
\]
where
\begin{align*}
\mathscr{K}_2'':=
\sum_{\substack{q\sim Q\\ (q,f_1f_2)=1}}\sum_{\substack{1\le |h|\le H_1H_2\\ (h,q)=1}}\Bigl|\sum_{\substack{1\le |h_3|\le H_3 \\ h_3\equiv b_3'\Mod{d}\\ (h_3,q d_1/e_3)=1}}\sum_{\substack{r\sim R\\ (r,h h_3 f_2)=1}}\sum_{\substack{m\sim M\\ (m,q r)=1}}\tilde{c}_{q,r}\alpha_m \Kl_3\Bigl(f_1 h h_3\overline{m f_2};q r\Bigr)\Bigr|^2.
\end{align*}
In order to establish the result (recalling $M N_1 N_2 N_3\asymp x$ and $H_i\ll x^{2\epsilon} QR/(d_0e_i N_i)$), we see it is sufficient to show
\begin{equation}
\mathscr{K}_2''\ll \frac{Q R^2M x^{1-7\epsilon} d_1^6 d_0^2}{e_1 e_2 e_3^2 N_3}.
\label{eq:E4Target}
\end{equation}
By symmetry it suffices to just consider $h>0$. We put $h$ into dyadic regions $h\sim H$ with $H\ll H_1H_2$, take the worst range $H$. We then insert a smooth majorant for this $h$ summation for an upper bound, and expand out the square. Dropping some summation constraints on the outer variables for an upper bound, this gives
\[
\mathscr{K}_2''\le (\log{x})\sup_{H\le H_1H_2}\sum_{\substack{r_1,r_2 \sim R\\ (r_1r_2,f_1f_2)=1}}\sum_{\substack{q\sim Q\\ (q,f_1f_2)=1}}\sum_{\substack{m_1,m_2\sim M\\ (m_1,qr_1)=1 \\ (m_2,q r_2)=1}}\sum_{\substack{1\le |h_3|,|h_3'|\le H_3\\ (h_3,qr_1)=1\\ (h_3',q r_2)=1}}|\mathscr{K}_3''|,
\]
where $\mathscr{K}_3''=\mathscr{K}_3''(r_1,r_2,q,m_1,m_2,h_3,h_3')$ is given by
\[
\mathscr{K}_3'':=\sum_{\substack{ (h,q r_1 r_2)=1}}\psi\Bigl(\frac{h}{H}\Bigr)\Kl_3(f_1 h h_3\overline{m_1 f_2};q r_1)\overline{\Kl_3(f_1 h h_3'\overline{m_2 f_2};q r_2)}.
\]
We consider separately the terms with $r_1^3h_3'm_1= r_2^3 h_3m_2$ and those with $r_1^3 h_3' m\ne r_2^3 h_3 m_2$.

For the `diagonal terms' with $r_2^3 h_3'm_1= r_2^3 h_3 m_2$, we use the `trivial bound' $\mathscr{K}_3''\ll x^{o(1)}H$ which follows from Lemma \ref{lmm:Deligne}. Given $r_1,h_3',m_1$ there are $x^{o(1)}$ choices of $r_2,h_3$ and $m_2$. Thus these terms contribute to $\mathscr{K}_2''$ a total
\begin{align}
\ll\sum_{q\sim Q}\sum_{r_1\sim R}\sum_{1\le |h_3'|\le H_3}\sum_{m_1\sim M}x^{o(1)} H
&\ll x^{o(1)}Q R M H_1 H_2 H_3\nonumber\\
&\ll x^{7\epsilon}\frac{Q^4 R^4 M}{N_1 N_2 N_3 d_0^3 e_1e_2e_3}\nonumber\\
&\ll \frac{Q^4 R^4 M^2}{x^{1-7\epsilon} d_0^3 e_1e_2e_3}.\label{eq:KBound1}
\end{align}
We now consider the terms when $r_1^3 h_3' m_1\ne r_2^3 h_3 m_2$. In this situation we apply Lemma \ref{lmm:KloostermanCorrelation}, which gives that 
\begin{equation}
\mathscr{K}_4'\ll x^\epsilon Q^{1/2} R d_3+\frac{x^{\epsilon}H_1 H_2 }{Q^{1/2}[r_1,r_2]^{1/2}}d_3 d_4,
\label{eq:E5Bound}
\end{equation}
where $d_3=d_3(q,r_1,r_2,m_1,m_2,h_3,h_3')$ and $d_4=d_4(r_1,r_2,h_3'm_1-h_3m_2)$ are given by
\begin{align*}
d_3&=\gcd(r_1^3 h_3'm_1-r_2^3 h_3 m_2,q),\\
d_4&=\gcd(h_3'm_1-h_3m_2,r_1,r_2).
\end{align*}
Thus the contribution of the first term on the right hand side of \eqref{eq:E5Bound} to $\mathscr{K}_3'$ is
\begin{align}
&\ll x^\epsilon Q^{1/2} R \sum_{r_1,r_2\sim R}\sum_{m_1,m_2\sim M} \sum_{\substack{1\le |h_3|,|h_3'|\le H_3\\ r_1^3 h_3'm_1\ne r_2^3 h_3'm_2}}\sum_{q\sim Q}d_3\nonumber\\
&\ll x^\epsilon Q^{3/2}R \sum_{r_1,r_2\sim R}\sum_{m_1,m_2\sim M} \sum_{\substack{1\le |h_3|,|h_3'|\le H_3\\ r_1^3 h_3'm_1\ne r_2^3 h_3'm_2}}\tau(r_1^3 h_3'm_1-r_2^3 h_3m_2)\nonumber\\
&\ll x^{2\epsilon} Q^{3/2} R^3 M^2 H_3^2\nonumber\\
&\ll x^{6\epsilon} \frac{M^2 Q^{7/2} R^5}{N_3^2 e_3^2 d_0^{2}}.\label{eq:KBound2}
\end{align}
The contribution from the last term on the right hand side of \eqref{eq:E5Bound} to $\mathscr{K}_3'$ is
\begin{align}
&\ll \frac{x^\epsilon H_1H_2}{Q^{1/2} R}\sum_{1\le |h_3|,|h_3'|\le H_3}\sum_{\substack{m_1,m_2\sim M\\ r_1^3 h_3' m_1\ne r_2^3 h_3m_2}}\sum_{\substack{r_1,r_2\sim R}}(r_1,r_2)^{1/2}d_4\sum_{q\sim Q}d_3\nonumber\\
&\ll x^{2\epsilon} \frac{H_1 H_2}{Q^{1/2}R} R^2 H_3^2 M^2 Q\nonumber\\
&\ll \frac{Q^{9/2} R^5 M^3}{e_1e_2e_3 d_0^{4} x^{1-10\epsilon}N_3 }.\label{eq:KBound3}
\end{align}
Thus, recalling that $d_1\ge e_1,e_2,e_3$, together \eqref{eq:KBound1}, \eqref{eq:KBound2} and \eqref{eq:KBound3} give
\[
\mathscr{K}_3'\ll \frac{x^{1-7\epsilon}d_1^6 d_0^2 Q R^2 M}{e_1 e_2 e_3^2 N_3}\Bigl( \frac{Q^{5/2} R^3 M}{x^{1-13\epsilon} N_3}+\frac{Q^{7/2} R^3 M^2}{ x^{2-17\epsilon} }+\frac{Q^3 R^2 M N_3}{x^{2-14\epsilon} }\Bigr).
\]
This is acceptable for \eqref{eq:E4Target} if we have
\begin{align}
N_3&< \frac{x^{2-14\epsilon} }{Q^3 R^2 M},\\
\frac{M Q^{5/2} R^3}{x^{1-14\epsilon}}&<N_3,\\
M^2 Q^{7/2} R^{3}&<x^{2-20\epsilon}.
\end{align}
This gives the result.
\end{proof}
%
%
%
%
%
%
%
%
\begin{lmm}\label{lmm:KRough}
Let $x^{2\epsilon} \le N_1\le N_2\le N_3$ and $x^\epsilon\le M$ and $Q_1,Q_2\ge 1$ be such that $Q_1Q_2\le x^{1-\epsilon}$, $N_1N_2N_3M\asymp x$ and
\[
\frac{M Q_1^{5/2}Q_2^3}{x^{1-15\epsilon}}\le N_3\le \frac{x^{2-15\epsilon}}{Q_1^3Q_2^2 M}.
\]
Let $\alpha_m$ be a 1-bounded complex sequence, $\mathcal{I}_j\subseteq[N_j,2N_j]$ an interval and 
\[
\Delta_{\mathscr{K}}(q):=\sum_{m\sim M}\alpha_m\sum_{\substack{n_1\in \mathcal{I}_1\\n_2\in \mathcal{I}_2\\ n_3\in \mathcal{I}_3\\ P^-(n_1n_2n_3)\ge z_0}}\Bigl(\mathbf{1}_ {m n_1n_2n_3\equiv a\Mod{q}}-\frac{\mathbf{1}_{(m n_1 n_2 n_3,q)=1}}{\phi(q)}\Bigr).
\]
Then for every $A>0$ we have
\[
\sum_{\substack{q_1\sim Q_1\\ (q_1,a)=1}}\sum_{\substack{q_2\sim Q_2\\ (q_2,a)=1}}\Bigl|\Delta_{\mathscr{K}}(q_1 q_2)\Bigr|\ll_A \frac{x}{(\log{x})^{A}}.
\]
\end{lmm}
\begin{proof}
We first simplify the moduli slightly. Let $q_0$ be the square-full part of $q_1q_2$, and $q'$ be the square-free part. Then $q'=q_1'q_2'$ where $q_1'=(q_1,q')$ and $q_2'=(q_2,q')$. Then we see it clearly suffices to show that
\[
\sum_{\substack{q_0\sim Q_0\\ q_0\text{ square-full}}}\sum_{q_1'\sim Q_1'}\sum_{\substack{q_2'\sim Q_2'\\ (q_0q_1'q_2',a)=1\\ q_1'q_2' \text{ square-free}}}\Bigl|\Delta_{\mathscr{K}}(q_0q_1'q_2')\Bigr|\ll_A\frac{x}{(\log{x})^{A}},
\]
for every choice of $Q_0,Q_1',Q_2'$ with $Q_0Q_1'Q_2'\asymp Q_1Q_2$ and $Q_1\gg Q_1/Q_0$. By Lemma \ref{lmm:Squarefree} we have the result unless $Q_0<z_0$, and so $Q_1'=Q_1 x^{-o(1)}, Q_2'=Q_2x^{-o(1)}$. In particular, restricting to finer than dyadic intervals and inserting coefficients $c_{q_0,q_1',q_2'}$ to remove the absolute values, it suffices to show that
\[
\sum_{\substack{q_0\sim Q_0}}\sum_{q_1\in[Q_1',4Q_1'/3)}\sum_{\substack{q_2\in[Q_2',4Q_2'/3)\\ (q_1q_2,a q_0)=1\\ q_1q_2\text{ square-free}}}c_{q_0,q_1,q_2}\Delta_\mathscr{K}(q_0 q_1 q_2)\ll_A \frac{x}{(\log{x})^{A}}
\]
for $Q_0=x^{o(1)}$, $Q_1'=Q_1x^{o(1)}$, $Q_2'=Q_2x^{o(1)}$ and for some $1$-bounded coefficients $c_{q_0,q_1,q_2}$ supported on $q_1q_2$ squarefree with $(q_1q_2,a q_0)=(q_0,a)=1$.

Now we reduce to a smoothed situation. By three applications of Lemma \ref{lmm:SmoothReduction}, we see that it is sufficient to show that for every choice of $B>0$ and $Q_0\le z_0$
\[
\sup_{\substack{q_0\sim Q_0}}\Bigl|\sum_{q_1\in[Q_1',4Q_1'/3)}\sum_{q_2\in[Q_2',4Q_2'/3)}c_{q_0,q_1,q_2}\widetilde{\Delta}_{\mathscr{K}}(q_0 q_1 q_2)\Bigr|\ll_B \frac{x}{Q_0(\log{x})^B},
\]
where
\begin{align*}
\widetilde{\Delta}_{\mathscr{K}}(q)&:=\sum_{\substack{d_1\sim D_1\\ d_2\sim D_2\\ d_3\sim D_3}}\widetilde{\lambda}_{d_1,d_2,d_3}\sum_{m\sim M}\alpha_m\sum_{\substack{n_1,n_2,n_3}}f_1\Bigl(\frac{n_1}{N_1'}\Bigr)f_2\Bigl(\frac{n_2}{N_2'}\Bigr)f_3\Bigl(\frac{n_3}{N_3'}\Bigr)\\
&\qquad \times\Bigl(\mathbf{1}_ {m d_1d_2d_3 n_1n_2n_3\equiv a\Mod{q}}-\frac{\mathbf{1}_{(md_1d_2d_3 n_1 n_2 n_3,q)=1}}{\phi(q)}\Bigr).
\end{align*}
and $\tilde{\lambda}_{d_1,d_2,d_3}$ is a 1-bounded sequence and where $D_iN_i'\asymp N_i$ for $1\le i\le 3$ and $D_i\le x^{o(1)}$, and where $f_i$ are smooth functions supported on $[1/2,5/2]$ satisfying $f_i^{(j)}\ll_j (\log{x})^{jC}$ for all $1\le i\le 3$ and $j\ge 0$ where $C=C(A)$ is a constant depending only on $A$. If we let 
\[
\alpha'_{b}:=\sum_{\substack{d_1d_2d_3m=b\\ d_i\sim D_i\\ m\sim M}}\alpha_m\widetilde{\lambda}_{d_1,d_2,d_3},
\]
then we see it suffices to establish the result for $\Delta'_{\mathscr{K}}$ in place of $\widetilde{\Delta}_{\mathscr{K}}$, where $\Delta'_{\mathscr{K}}(q)$ is given by
\begin{align*}
\sum_{m\sim M'}\alpha'_{m}\sum_{\substack{n_1,n_2,n_3}}f_1\Bigl(\frac{n_1}{N_1'}\Bigr)f_2\Bigl(\frac{n_2}{N_2'}\Bigr)f_3\Bigl(\frac{n_3}{N_3'}\Bigr)\Bigl(\mathbf{1}_ {m n_1n_2n_3\equiv a\Mod{q}}-\frac{\mathbf{1}_{(m n_1 n_2 n_3,q)=1}}{\phi(q)}\Bigr).
\end{align*}
We let
\[
c_q:=\sum_{\substack{q_1q_2=q\\ q_1\in[Q_1',4Q_1'/3)\\ q_2\in[Q_2',4Q_2'/3)}}c_{q_0,q_1,q_2},
\]
which is a $1$-bounded sequence supported on $q\sim Q$. By applying Lemma \ref{lmm:K1}, we have that
\[
\sup_{\substack{q_0\sim Q_0}}\Bigl|\sum_{q_1\sim Q_1'}\sum_{q_2\sim Q_2'}c_{q_0,q_1,q_2}\Delta'_{\mathscr{K}}(q_0 q_1 q_2)\Bigr|\ll_{A,B} \frac{x}{Q_0(\log{x})^B}+\frac{N_1 N_2 N_3 x^{o(1)}}{(Q_1'Q_2')^3}\sup_{q_0\sim Q_0}\mathscr{K}',
\]
where 
\begin{align*}
\mathscr{K}'&:=\sum_{\substack{q_1\in[Q_1',4Q_1'/3)\\ q_2\in [Q_2',4Q_2'/3)}}\frac{(Q_1'Q_2')^3 c'_{q_0,q_1,q_2}}{q_1^3 q_2^3}\sum_{m\sim M'}\alpha_m''\sum_{\substack{1\le |h_1|\le H_1\\ 1\le |h_2|\le H_2\\ 1\le |h_3|\le H_3\\ h_i\equiv b_i\Mod{q_0}}}F(h_1,h_2,h_3;a\overline{mq_0^3};q_1q_2),\\
H_i&:=\frac{x^\epsilon Q_0Q_1'Q_2'}{N_i}\le \frac{x^{2\epsilon}Q_1'Q_2'}{N_i}\qquad \text{for }i\in\{1,2,3\},
\end{align*}
and where $\alpha_m''$ is $\alpha'_m$ restricted to a residue class $m\equiv b_m\Mod{q_0}$ and $c'_{q_0,q_1,q_2}$ is $c_{q_0,q_1,q_2}$ restricted to a residue class $q_1q_2\equiv b_q\Mod{q_0}$. By applying Lemma \ref{lmm:K2} we then see it suffices to show that for every choice of $e_1,e_2,e_3,d_0,d_1$ with $d_1\ge e_1,e_2,e_3$, and every choice of $Q_1'',Q_2''$ with $Q_1''\le Q_1'$ and $Q_2''\le Q_2$ and every choice of $Q'',Q'''\asymp Q_1'Q_2'/d_0d_1$ and $H_i'\le x^{2\epsilon} Q_1'Q_2'/(N_i' d_0e_i)$ we have
\[
\sum_{\substack{q_1'\sim Q_1''\\ q_2'\sim Q_2''}}c''_{q_0,q_1',q_2'}\sum_{\substack{m\sim M'\\ (m,q_1'q_2')=1}}\alpha_m''\sum_{\substack{1\le |h_1'|\le H_1'\\ 1\le |h_2'|\le H_2'\\ 1\le |h_3'|\le H_3'\\ h_i\equiv b_i\Mod{q_0}\\ (h_i',q_1q_2 d_1/e_i)=1\,\forall i}}\Kl_3(f_1h_1'h_2'h_3'\overline{m f_2};q_1q_2)\ll \frac{d_1^3(Q_1'Q_2')^2 M}{e_1e_2e_3x^{2\epsilon}},
\]
where 
\[
c_{q_0,q_1',q_2'}''=\mathbf{1}_{\substack{Q''\le q_1'q_2'\le Q'''\\ (q_1'q_2',d_0d_1)=1}}\sum_{a_1a_2=d_0d_1}\frac{q_1' q_2' d_0d_1 c'_{q_0,a_1q_1',a_2q_2'}}{Q_1' Q_2'}.
\]
This bound now follows from Lemma \ref{lmm:K3} provided we have
\begin{align*}
\frac{M' Q_1''{}^{5/2}Q_2''{}^3}{x^{1-14\epsilon}}&\le N_3'\le \frac{x^{2-14\epsilon}}{Q_1''{}^3Q_2''{}^2 M'},\\
M'{}^2 Q_1''{}^{7/2}Q_2''{}^3&<x^{2-20\epsilon}.
\end{align*}
Since $N_3'=N_3x^{o(1)}$ and $Q_i''\le Q_i'\le Q_i$ and $M'=M x^{o(1)}$ we see that the above bounds certainly hold provided we have
\begin{align*}
\frac{M Q_1^{5/2}Q_2^3}{x^{1-15\epsilon}}&\le N_3\le \frac{x^{2-15\epsilon}}{Q_1^3Q_2^2 M},\\
M^2 Q_1^{7/2}Q_2^3&<x^{2-21\epsilon}.
\end{align*}
Finally, if the range for $N_3$ is non-trivial then we see we must have
\[
M^2 Q_1^{11/2}Q_2^{5}\le x^{3-30\epsilon},
\]
and so the first condition implies the second whenever $Q_1Q_2\ge x^{1/2-\epsilon/2}$. If instead $Q_1Q_2\le x^{1/2-\epsilon/2}$ then the whole result follows immediately from the Bombieri-Vinogradov theorem. This completes the proof.
\end{proof}
%
%
%
%
%
%
%
%
\begin{proof}[Proof of Proposition \ref{prpstn:TripleDivisor}]
First we note that by Lemma \ref{lmm:Divisor} the set of $m$ with $|\alpha_m|\ge(\log{x})^C$ has size $\ll x(\log{x})^{O_{B_0}(1)-C}$, so by Lemma \ref{lmm:SmallSets} these terms contribute negligibly if $C=C(A,B_0)$ is large enough. Thus, by dividing through by $(\log{x})^{C}$ and considering $A+C$ in place of $A$, it suffices to show the result when $|\alpha_m|\le 1$.

Since $N_1N_2N_3M\asymp x$ and $N_3\ge N_2\ge N_1$ we have
\[
N_3\gg \frac{x^{1/3}}{M^{1/3}}.
\]
We first apply Lemma \ref{lmm:KRough} with the trivial factorization $Q_1=Q R$, $Q_2=1$. This gives the result provided
\begin{equation}
\frac{M Q^{5/2} R^{5/2} }{x^{1-15\epsilon}}<N_3<\frac{x^{2-15\epsilon} }{Q^3 R^3 M}.\label{eq:N3Range1}
\end{equation}
We now apply Lemma \ref{lmm:KRough} with the factorization $Q_1=Q$, $Q_2=R$, which gives the result provided
\begin{equation}
\frac{M Q^{5/2}R^{3}}{x^{1-15\epsilon}}<N_3<\frac{x^{2-15\epsilon}}{Q^3 R^2 M}.\label{eq:N3Range2}
\end{equation}
We see that if $Q,R$ satisfy $Q^7 R^9<x^4$ and $M<x^{1-20\epsilon}/(QR)^{15/8}$ then
\begin{align*}
\frac{M Q^{5/2} R^3}{x^{1-15\epsilon}}&=\frac{x^{2-15\epsilon} }{M Q^3 R^3}\Bigl(\frac{M^2 Q^{11/2} R^6}{x^{3-30\epsilon}}\Bigr)\\
&<\frac{x^{2-15\epsilon}}{M Q^3 R^3}\Bigl(\frac{Q^{7/4}R^{9/4}}{x^{1+10\epsilon}}\Bigr)\\
&<\frac{x^{2-15\epsilon} }{M Q^3 R^3}.
\end{align*}
Thus the ranges \eqref{eq:N3Range1} and \eqref{eq:N3Range2} overlap, giving the result for the range
\[
\frac{M Q^{5/2}R^{5/2}}{x^{1-15\epsilon}}<N_3<\frac{x^{2-15\epsilon} }{Q^3 R^2 M}.
\]
Since $M<x^{1-20\epsilon} Q^{-15/8}R^{-15/8}$, we see that 
\begin{equation}
\frac{M Q^{5/2} R^{5/2}}{x^{1-15\epsilon}}<\frac{x^{1/3-\epsilon}}{M^{1/3}}.\label{eq:N3Lower}
\end{equation}
Since $M<x^{1-20\epsilon} Q^{-15/8}R^{-15/8}$ and $Q^9 R<x^{32/7}$, we have
\begin{equation}
\frac{x^{2-15\epsilon}}{M Q^3 R^2}>\frac{x^{1+\epsilon} }{Q^{9/8}R^{1/8}}>x^{3/7+\epsilon}.\label{eq:N3Upper}
\end{equation}
Therefore, using \eqref{eq:N3Lower} and \eqref{eq:N3Upper}, we see that \eqref{eq:N3Range1} and \eqref{eq:N3Range2} cover the range
\[
\frac{x^{1/3-\epsilon }}{M^{1/3}}\le N_3\le x^{3/7+\epsilon}.
\]
This gives the result.
\end{proof}
%
%
%
%
%
%
%
%
We have now completed the proof of Propositions \ref{prpstn:Zhang}, \ref{prpstn:TripleRough}, \ref{prpstn:TripleDivisor}, \ref{prpstn:Fouvry} and \ref{prpstn:SmallDivisor}, and so completed the proof of Theorem \ref{thrm:MainTheorem}.
%
%
%
%
%
%
%
%
%
\bibliographystyle{plain}
\bibliography{Bibliography}

\end{document}